\documentclass[a4paper, oneside, openright]{book}
\usepackage[T1]{fontenc}
\usepackage{lmodern}
\usepackage{fix-cm} 

\usepackage{bm}
\usepackage{multicol} 
\usepackage[utf8]{inputenc} 
\usepackage[english]{babel} 
\usepackage{blindtext} 
\usepackage{graphicx} 
\usepackage[a4paper,top=2.5cm,bottom=2.5cm,left=3cm,right=3cm]{geometry} 
\usepackage{placeins} 
\usepackage[fontsize=13pt]{scrextend} 
\usepackage{graphicx}
\usepackage{mathdots}
\usepackage[parfill]{parskip} 
\usepackage{titlesec}
\usepackage{minted} 
\usepackage{float}
\usepackage[font=scriptsize, skip=5pt]{caption} 
\usepackage[immediate]{silence}
\usepackage{sectsty}
\DeactivateWarningFilters[temp] 
\usepackage{hyperref} 
\usepackage[justification=centering]{caption} 
\usepackage{csquotes} 
\usepackage[bottom]{footmisc} 
\usepackage{epsfig,amsfonts,amsmath,amssymb,amscd,float}
\usepackage[dvipsnames]{xcolor}
\usepackage{graphicx,caption,float,subcaption}
\usepackage{booktabs,multirow,siunitx}
\usepackage{amsthm}
\usepackage{amsmath}
\usepackage{tikz}
\usepackage{parskip}
\usepackage{cases}
\usepackage{algorithm, algorithmicx, algpseudocode}
\usepackage[sort,nocompress]{cite}
\usepackage{colortbl,hhline}
\usepackage{tikz}
\usetikzlibrary{decorations.pathreplacing}
\newtheorem{proposition}{Proposition}[chapter]
\newtheorem{theorem}[proposition]{Theorem}
\newtheorem{corollary}[proposition]{Corollary}
\newtheorem{conjecture}{Conjecture}[section]
\newtheorem{remark}[proposition]{Remark}
\theoremstyle{definition}
\newtheorem{definition}{Definition}
\titleformat{\chapter}[display]
    {\normalfont\huge\bfseries}{\chaptertitlename\ \thechapter}{10pt}{\LARGE}
\titlespacing*{\chapter}{0pt}{0pt}{20pt}
\chaptertitlefont{\fontsize{22pt}{30pt}\selectfont}

\hypersetup{ 
    colorlinks,
    citecolor=black,
    filecolor=black,
    linkcolor=black,
    urlcolor=black
}

\DeclareUnicodeCharacter{02BC}{}

\addto\captionsenglish{%
}
\begin{document} 
\pagenumbering{gobble} 
\begin{titlepage}
\begin{center}
    \vspace*{1.5cm}
    {\normalsize University of Insubria}\\[2mm]
    {\normalsize Department of Theoretical and Applied Sciences}\\[3mm]
    {\normalsize Ph.D. in Computer Science and Mathematics of Computation}\\[1.8cm]

   \includegraphics[width=5cm]{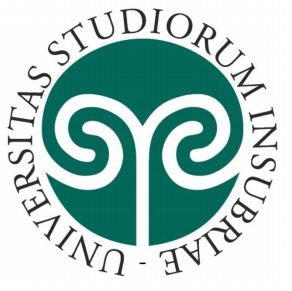}

    {\LARGE\bfseries GLT matrix-sequences and few emblematic applications}\\
    \vspace*{\fill}

    \begin{minipage}{0.45\textwidth}
        \begin{flushleft}
            {\normalsize\textbf{Supervisor:}}\\[2mm]
            {\normalsize Prof. Stefano Serra-Capizzano}
        \end{flushleft}
    \end{minipage}%
    \hfill
    \begin{minipage}{0.45\textwidth}
        \begin{flushright}
            {\normalsize\textbf{Presented by:}}\\[2mm]
            {\normalsize Muhammad Faisal Khan}\\
            {\normalsize 755852}
        \end{flushright}
    \end{minipage}

    \vspace{1.5cm}

    {\normalsize\textbf{Academic year:}}\\[2mm]
    {\normalsize 2024/2025}
\end{center}
\end{titlepage} 
\chapter*{Dedication}
\begin{center}
\thispagestyle{empty}
\vspace{1cm}

I dedicate this thesis to the loving memory of my late father, whose guidance and values continue to inspire me every day.\\[0.5cm]
To my mother, brother, and sisters, for their unconditional love, support, and encouragement throughout this journey.\\[0.5cm]
And to my supervisor, Stefano Serra-Capizzano, for his invaluable guidance, mentorship, and trust, which made this work possible.

\end{center}
\clearpage
\chapter*{Acknowledgments}

\begin{center}
\thispagestyle{empty}
\vspace{1cm}
I would like to express my deepest gratitude to my wonderful supervisor, Stefano Serra-Capizzano, 
for his invaluable technical guidance and his constant moral support. Especially at the beginning of my Ph.D., when everything seemed uncertain, 
he was always there to encourage me and to give me hope. 
He has been like a father to me, and I will always remain grateful for his kindness and mentorship.

My heartfelt thanks go to my beloved friend, Uswah Yasin. 
Her patience, support, and constant care have been my anchor throughout this journey. She stood by me in every challenge, always encouraging me and inspiring me with her kindness.

I am also grateful to my dear friends, Asim Ilyas and Yasir Khan. 
They have been like brothers to me, always supporting me in every stage of this journey. 
Their encouragement and companionship have been a source of strength, 
especially during a very difficult period in October 2024, when their care and loyalty 
helped me overcome one of the hardest phases of my life.  

To all of you, I owe my deepest gratitude.  

\end{center}
\chapter*{Abstract}
This thesis advances the spectral theory of structured matrix-sequences within the framework of Generalized Locally Toeplitz (GLT) $\ast$-algebras, focusing on the geometric mean of Hermitian positive definite (HPD) GLT matrix-sequences and its applications in mathematical physics. For two HPD sequences \(\{A_n\}_n \sim_{\mathrm{GLT}} \kappa\) and \(\{B_n\}_n \sim_{\mathrm{GLT}} \xi\) in the same \(d\)-level, \(r\)-block GLT \(*\)-algebra, we prove that, when \(\kappa\) and \(\xi\) commute, the sequence of geometric means \(\{G(A_n,B_n)\}_n\) is a GLT sequence with symbol \((\kappa\xi)^{1/2}\), without requiring the almost-everywhere invertibility of either symbol, thereby settling \cite[Conjecture 10.1]{garoni2017} for \(r=1\), \(d \geq 1\). In degenerate cases, where symbols vanish on sets of positive measure, we identify conditions ensuring that the geometric mean retains a GLT structure in the commuting setting, so having \(\{G(A_n, B_n)\}_n \sim_{\mathrm{GLT}} G(\kappa, \xi)\). Conversely, for \(r>1\) with degenerate, non-commuting symbols, we provide numerical evidence that the resulting sequence still admits a spectral symbol, with \(G(\kappa,\xi)\) being not well defined. The latter implies that the result \(\{G(A_n, B_n)\}_n \sim_{\mathrm{GLT}} G(\kappa, \xi)\) in the commuting setting is maximal. Numerical experiments in scalar and block settings, in one and two dimensions, confirm the theoretical predictions and illustrate spectral behaviour. We sketch also the case of $k\geq 2$ matrix-sequences, by considering the Karcher mean. Preliminary results and numerical experiments indicate that \(\{G(A_n^{(1)},\ldots A_n^{(k)})\}_n \sim_{\mathrm{GLT}} G(\kappa_{1},\ldots \kappa_{k})\), if \(\{A_n^{(j)}\}_n \sim_{\mathrm{GLT}} \kappa_{j}\), for \( j = 1,\ldots,k \).

The GLT framework is further applied to mean-field quantum spin systems, with particular attention to the quantum Curie--Weiss model. In this context, we show that the structured matrices arising from the model form GLT sequences, enabling an explicit determination of their spectral distributions in both unrestricted and symmetry-restricted cases. Numerical simulations validate the analysis and reveal additional spectral features such as eigenvalue localization and extremal behaviour.

In terms of mathematical tools, we use the axioms characterizing the \(d\)-level, \(r\)-block GLT \(\ast\)-algebra, the notion of approximating classes of sequences and the important two-sided ideal of zero-distributed matrix-sequences.

\tableofcontents  
\newpage 
\pagenumbering{arabic}
\setcounter{chapter}{0} 
\listoffigures
\addcontentsline{toc}{chapter}{List of figures}
\listoftables
\addcontentsline{toc}{chapter}{List of tables}
\chapter*{General notations and acronyms}

\begin{itemize}
    \item $\mathbb{N}$, $\mathbb{R}$, $\mathbb{C}$: Sets of natural, real, and complex numbers, respectively.
    \item $\mathbb{R}^+$, $\mathbb{N}^+$: Sets of positive real and natural numbers.
    \item $\mathbb{Z}$: Set of integers.
    \item $\mathbb{K}$: A field, usually $\mathbb{R}$ or $\mathbb{C}$.
    \item $\mathbb{K}^r$: Space of vectors of length $r$ with entries in $\mathbb{K}$.
    \item $\mathbb{K}^{r \times s}$: Space of $r \times s$ matrices with entries in $\mathbb{K}$.
    \item $M_n(\mathbb{K})$: Set of $n \times n$ matrices over $\mathbb{K}$.
    \item $\bm{x}, \bm{y}$: Vectors in $\mathbb{K}^r$ (bold letters usually denote vectors).
    \item $A, B$: Matrices in $M_n(\mathbb{K})$ unless specified otherwise.
    \item $\mathbb{I}_n$ or $I_n$: $n \times n$ identity matrix.
    \item $\mathbf{0}_n$ or $0_n$: $n \times n$ zero matrix.
    \item $(\cdot)^{\top}$: denotes the transpose of a vector or matrix.
    \item $(\cdot)^*$: denotes the conjugate transpose of a vector or matrix.
    \item $(\cdot)^\dagger$: denotes the Moore-Penrose pseudoinverse of a square matrix.
    \item $\lambda_j(A)$: denotes the $j$-th eigenvalue of $A \in M_n(\mathbb{K})$, for $j = 1, \dots, n$.
    \item $\sigma_j(A)$: denotes the $j$-th singular value of $A \in M_n(\mathbb{K})$, for $j = 1, \dots, n$.
     \item $\mathcal{H}$: Hilbert space.
  \item $\sigma_x, \sigma_y, \sigma_z$: Pauli matrices.
  \item $\sigma^{(i)}$: Pauli operator acting on the $i$-th spin.
  \item $H$: Hamiltonian operator of the system.
  \item $\rho(H)$: empirical spectral measure of $H$.
  \item $\chi$: Magnetic susceptibility.
  \item $\beta$: Inverse temperature parameter.
  \item $Z$: Partition function.
  \item $\langle \cdot \rangle$: Expectation with respect to the Gibbs measure.
    \item $A \geq B$, where $A, B \in M_n(\mathbb{K})$, means that $A - B$ is positive semidefinite. Likewise, $A \leq B$ means that $A - B$ is negative semidefinite.
    \item $A > B$, where $A, B \in M_n(\mathbb{K})$, means that $A - B$ is positive definite. Likewise, $A < B$ means that $A - B$ is negative definite.
    \item $A \sim B$, where $A, B \in M_n(\mathbb{K})$, denotes that $A$ and $B$ are similar matrices (i.e., there exists an invertible $P$ such that $A = PBP^{-1}$).

    \item $\| \cdot \|_p$: denotes the standard $p$-norm of a vector or matrix, $1 \leq p \leq \infty$.
    \item $\| \cdot \|$: denotes the spectral norm (largest singular value), coincides with $\| \cdot \|_2$. If the matrix is normal, the spectral norm coincides with the spectral radius.
   
    \item $\otimes$: Kronecker (tensor) product. For $A \in \mathbb{C}^{p \times q}$, $B \in \mathbb{C}^{r \times s}$,
          \[
            A \otimes B =
            \begin{bmatrix}
                a_{11} B & \cdots & a_{1q} B \\
                \vdots   &        & \vdots   \\
                a_{p1} B & \cdots & a_{pq} B
            \end{bmatrix}
            \in \mathbb{C}^{pr \times qs}
          \]
    \item $L^p(D)$: Space of $p$-integrable functions on domain $D$.
    \item $C([a,b])$: Space of continuous functions on $[a,b]$.
    \item $\mathcal{O}(\cdot)$, $o(\cdot)$: Big-$O$ and little-$o$ notation.
    \item $\mathcal{C}_c(\mathbb{R})$, $\mathcal{C}_c(\mathbb{C})$: The set of continuous, complex-valued functions with compact support on $\mathbb{R}$ or $\mathbb{C}$.
    \end{itemize}
    \textbf{Multi-index notation:}
\begin{itemize}
    \item $\bm{i}$ denotes a multi–index $(i_1, \ldots, i_d) \in \mathbb{Z}^d$. Its size will be clear from context.
    \item $\bm{0}, \bm{1}, \bm{2}, \ldots$ are vectors of all zeros, ones, twos, etc.
    \item $\bm{h} + \bm{k}$ is the multi–index whose components are $h_r + k_r$ for $r = 1, \ldots, d$. More generally, operations between multi-indices such as subtraction, multiplication, division, and so on, are performed componentwise.
    \item $\bm{h} \leq \bm{k}$ means that $h_r \leq k_r$ for all $r = 1, \ldots, d$. More generally, relations between multi-indices are evaluated componentwise.
    \item $\bm{h}, \ldots, \bm{k}$ denotes the multi-index range $\{\bm{j} \in \mathbb{Z}^d : \bm{h} \leq \bm{j} \leq \bm{k}\}$. With $\bm{j} = \bm{h}, \ldots, \bm{k}$ we mean that $\bm{j}$ varies in this range from $\bm{h}$ to $\bm{k}$ according to the standard lexicographic ordering:
    \[
    \left[
        \cdots
        \left[
            \left[
                (j_{1},\ldots,j_{d})\right]_{j_d = h_d, \ldots, k_d}
        \right]_{j_{d-1} = h_{d-1}, \ldots, k_{d-1}}
        \cdots
    \right]_{j_1 = h_1, \ldots, k_1}
    .\]
    \item $\bm{m} \to \infty$ means that $\min(\bm{m}) = \min_{j=1, \ldots, d} \bm{m}_j \to \infty$, or equivalently, that all the components tend to infinity.
    \item $N(\bm{m})$ is the product of all the components of $\bm{m}$, i.e., $N(\bm{m}) := \prod_{j=1}^d m_j$.
\end{itemize}
\begin{itemize}
    \item $\{A_n\}_n$: Sequence of matrices indexed by size/discretization parameter $n$.
    \item $T_n(f)$: Toeplitz matrix of size $n$ generated by symbol $f$.
    \item $C_n(f)$: Circulant matrix of size $n$ generated by symbol $f$.
    \item $H_n(f)$: Hankel matrix of size $n$ generated by symbol $f$.
     \item Symbol $(\kappa)$: Measurable function associated with a matrix-sequence, e.g., $\{A_n\}_n \sim_{\mathrm{GLT}} \kappa$.
    \item Symbol $(\xi)$: Measurable function associated with a matrix-sequence, e.g., $\{B_n\}_n \sim_{\mathrm{GLT}} \xi$.
    \item $D = [0,1]^d \times [-\pi, \pi]^d$: Typical domain for the symbol of a $d$-level GLT sequence.
    \item $\sim_\sigma$: Singular value distribution.
    \item $\sim_\lambda$: Eigenvalue distribution.
    \item $\sim_{\mathrm{GLT}}$: Matrix-sequence is a GLT sequence with symbol.
    \item $\operatorname{ran} (M)$: The column space (range) of a matrix $M$.
    \item $\operatorname{ess\,Ran}(\cdot)$: The essential range of a Grassmannian-valued measurable map, i.e., the set of subspaces which are not avoided by the function on a set of positive measure.
\end{itemize}

\textbf{Acronyms and Abbreviations:}
\begin{itemize}
    \item GLT: Generalized Locally Toeplitz
    \item LT: Locally Toeplitz
    \item HPD: Hermitian positive definite
    \item SPD: Symmetric positive definite
    \item a.c.s.: Approximating class of sequences
    \item FDE / PDE: Fractional / Partial Differential Equation
    \item CW: Curie–Weiss (model)
    \item TFIM: Transverse-field Ising model.
  \item LSD: Limiting spectral distribution.
  \item a.e.: almost everywhere.
  \item w.h.p.: with high probability.
\end{itemize}

\addcontentsline{toc}{chapter}{General notations and acronyms}
\chapter*{Introduction}

The asymptotic spectral distribution of large, structured matrices is a central theme in matrix analysis and computational mathematics, particularly in the context of discretized differential equations. This theme dates back to the early twentieth century with the pioneering work of Gábor Szegő, whose limit theorem rigorously describes the eigenvalue distribution of $n \times n$ Hermitian Toeplitz matrices generated by a real-valued symbol $f \in L^\infty([-\pi, \pi])$. More precisely, for $f$ real-valued and essentially bounded, the classical Szegő Limit Theorem as discussed in \cite{GrenanderSzego1984} states that for the Toeplitz sequence $\{T_n(f)\}_n$, it holds
\begin{equation}
    \lim_{n \to \infty} \frac{1}{n} \sum_{j=1}^n F\big(\lambda_j(T_n(f))\big) = \frac{1}{2\pi} \int_{-\pi}^{\pi} F(f(\theta)) \, d\theta,
\end{equation}
for every continuous function $F$, where $\lambda_j(T_n(f))$ are the eigenvalues of $T_n(f)$; see \cite[Theorem~1.1]{Tilli1998}. This profound results laid the groundwork for connecting discrete matrix structures with continuous spectral data, a connection of fundamental importance in many applications, including the approximation of differential and integral operators.

Over time, the need to extend these foundational results to broader function spaces and more general matrix structures became apparent. A major advancement was made by Eugene Tyrtyshnikov in 1996 \cite{Tyrtyshnikov1996}, who developed a unifying approach for the spectral and singular value distribution of Toeplitz and multilevel Toeplitz matrices generated by symbols in $L^2([-\pi, \pi]),  d\geq1$. In particular, Tyrtyshnikov proved that for the wide class of Toeplitz matrix-sequences with symbol $f \in L^2([-\pi, \pi]^{d})$, the limiting singular value distribution is still governed by integrals over the symbol, even when $f$ is unbounded or complex-valued. The eigenvalue distribution is also proved for $\{T_{n}(f)\}_n$ and $f \in L^2([-\pi, \pi]^{d})$, but with the restriction that $f$ is essentially real-valued. By relaxing the assumption on the generating function to the minimal requirement $f\in L^1([-\pi,\pi]^{d})$, the singular value distribution was extended in \cite{tyrtyshnikov1998}, while, for the eigenvalue distribution, the function $f$ has to be essentially real-valued. This framework not only extends the reach of the classical Szegő and Avram–Parter theorems \cite{Bottcher2008}, but also justifies the use of weak spectral limits and singular value analysis in a broad array of applications, from signal processing to the numerical analysis of partial differential equations (PDEs).

In terms of mathematical techniques, in \cite{tyrtyshnikov1998}, Tyrtyshnikov and Zamarashkin introduced matrix-theoretical approximation tools which anticipate the notion of approximating classes of sequences \cite{SerraCapizzano2001Dist}. A further significant contribution which is mathematically very elegant was provided by Paolo Tilli in \cite{tilli1998}, where he extended Szegő-type distribution in the case where the generating function is Lebesgue integrable and matrix-valued. Again, the eigenvalue distribution is proven for Hermitian $d$-level, $r$-block Toeplitz matrix-sequences, $d,r \geq 1$. Interesting results for the eigenvalue distribution in the non-Hermitian setting can be found in \cite{Bottcher2005,Bogoya2024Fast,Bogoya2024Matrixless,serra97,Tilli1999Complex}, where completely different techniques and tools are used. In such cases, the resulting matrices exhibit Toeplitz or Toeplitz-like structures, and the above theory applies directly. But when the operator involves variable coefficients the translation invariance is lost, and the resulting matrices no longer belong to the Toeplitz class. This limitation naturally motivated the introduction of a more general framework capable of capturing the asymptotic spectral behavior of matrices arising from the discretization of variable-coefficient PDEs. To address this challenge, Tilli introduced the theory of Locally Toeplitz (LT) sequences in \cite{Tilli1998}. This framework extends classical Toeplitz theory by associating each LT sequence with a pair of measurable functions $(a,f)$, called the local symbol, where $f$ plays the role of the generating function, analogous to the symbol in Toeplitz theory, while $a$ acts as a weight function capturing spatial variations due to variable coefficients. Tilli proved that LT sequences admit weighted Szegő-type formulas for both eigenvalues and singular values, thus providing a robust mathematical foundation for the analysis of matrix-sequences arising from the discretization of PDEs with variable coefficients. Importantly, classical Toeplitz sequences $\{T_{n}(f)\}_{n}$ are recovered as a special case of LT sequences by taking $a(x) \equiv 1,$ when $f \in L^2([-\pi, \pi])$ . This theory marked a paradigm shift in structured matrix analysis, enabling the application of spectral distribution results to a vastly broader class of discretization schemes and structured matrices relevant in modern numerical analysis.

Since these foundational contributions by Tyrtyshnikov and Tilli, there has been a surge of interest in the spectral analysis of structured matrix-sequences, leading to the development of the comprehensive theory of Generalized Locally Toeplitz (GLT) sequences~\cite{SerraCapizzano2003, SerraCapizzano2006}. This theoretical framework has rapidly become central in both pure and applied mathematics, due to its remarkable ability to capture the asymptotic singular value and eigenvalue distributions of matrices arising from virtually any meaningful approximation of (fractional) partial differential equations ((F)PDEs); see, e.g., the books and review papers~\cite{barbarino2020block1d, barbarino2020blockmulti, garoni2017, garoni2018, tutorial1} and references therein. In fact, virtually any practical discretization scheme, such as finite differences, finite elements of any order, discontinuous Galerkin methods, finite volumes, or isogeometric analysis, produces GLT matrix-sequences. Mathematically, for fixed positive integers $d$ and $r$, the set of $d$-level, $r$-block GLT matrix-sequences forms a maximal $*$-algebra of matrix-sequences, which is isometrically equivalent to the space of $d$-variate, $r \times r$ matrix-valued measurable functions defined on $[0,1]^d \times [-\pi,\pi]^d$. Each such GLT sequence $\{A_n\}_n$ is uniquely associated with a measurable, matrix-valued function $\kappa$ the GLT symbol on the domain $D = [0,1]^d \times [-\pi, \pi]^d$.
Notice that the set $[0, 1]^d$ can be replaced by any bounded Peano-Jordan measurable subset of $\mathbb{R}^d$ 
as occurring with the notion of reduced GLT $*$-algebras, see \cite[pp. 398-399, formula (59)]{SerraCapizzano2003} for the first occurrence with applications of approximated PDEs on general non-Cartesian domains in $d$ dimensions, \cite[Section 3.1.4]{SerraCapizzano2006} for the first formal proposal, and \cite{Barbarino2022} for an exhaustive treatment, containing both the $*$-algebra theoretical results and a number of applications. This symbol provides a powerful tool for analyzing the singular value and eigenvalue distributions when the matrices $A_n$ are part of a matrix-sequence of increasing size. The notation $\{A_n\}_n \sim_{\mathrm{GLT}} \kappa$ indicates that $\{A_n\}_n$ is a GLT sequence with symbol $\kappa$. Notably, the symbol of a GLT sequence is unique in the sense that if $\{A_n\}_n \sim_{\mathrm{GLT}} \kappa$ and $\{A_n\}_n \sim_{\mathrm{GLT}} \xi$, then $\kappa = \xi$ almost everywhere in $[0, 1]^d \times [-\pi, \pi]^d$ \cite{barbarino2020block1d,barbarino2020blockmulti,garoni2017,garoni2018}. Furthermore, by the $*$-algebra structure, $\{A_n\}_n \sim_{\mathrm{GLT}} \kappa$ and $\{B_n\}_n \sim_{\mathrm{GLT}} \kappa$ implies that $\{A_n-B_n\}_n \sim_{\mathrm{GLT}} 0$, i.e., the matrix-sequence $\{A_n-B_n\}_n$ is zero-distributed and the latter is very important for building explicit matrix-sequences approximating a GLT matrix-sequence and whose inversion is computationally cheap, in the context of preconditioning of large linear systems. It must be noticed that the $d$-level, $r$-block GLT $\ast$-algebra includes all Toeplitz sequences $\{T_n(f)\}_n$ with $f$ matrix-valued and $f \in L^{1}([-\pi,\pi]^{d})$. Hence, even when $r=d=1$, the LT matrix-sequences are strictly contained in the GLT class with $r=d=1$.

For other structured matrix-sequences such as Hankel, flipped Toeplitz, and sampling dense matrix-sequences, further distributional results exist and can be related to the study of GLT matrix-sequences; see \cite{FasinoTilli2000} for Hankel structures generated by $f \in L^2([-\pi,\pi]^d)$, $d \ge 1$, also of rectangular matrix-valued type. For flipped and symmetrized Toeplitz structures, and for the related development of preconditioning techniques and asymptotic spectral analyses applied to evolutionary and optimal-control PDEs, see \cite{FungHon2024, FerrariFurciHon2019, FerrariFurci2021, HonFungDong2024, MazzaPestana2021, HonDongSerra2023, HonLiSormani2025, LiHon2024,MazzaPestana2019, McWathen2020, McPestana2018, Pestana2019, PestanaWathen2015}, where spectral properties of flipped or symmetrized Toeplitz and multilevel Toeplitz matrices are analyzed, and structure-exploiting preconditioners are proposed and tested. Finally, see \cite{AlFhaidSerra2014,SalinelliSerra2016} for dense sampling matrix-sequences with applications to integral operators, including models arising in mathematical finance.


For other structured matrix-sequences like Hankel, flipped Toeplitz, sampling dense matrix-sequences, further distributional results exist and they can be
related to the study of GLT matrix-sequences; see \cite{FasinoTilli2000} for Hankel structures generated by $f \in L^2([-\pi, \pi]^d), d \ge 1$, also of rectangular matrix-valued type, see \cite{FerrariFurciHon2019,FerrariFurci2021,HonFungDong2024,MazzaPestana2021} for flipped Toeplitz structures and related applications to discretized
evolutionary PDEs, see \cite{AlFhaidSerra2014,SalinelliSerra2016} for dense
sampling matrix-sequences with applications to integral operators also in finance models.

A particularly rich connection of the GLT theory is with that related to the concept of geometric means of matrices. This notion was originally formalized by Ando, Li, and Mathias (ALM)~\cite{Ando2004}. The concept of the matrix geometric mean has attracted considerable attention from a large number of researchers in recent decades due to its elegant theoretical foundations and growing significance in a wide range of fields, including mathematics, engineering, and applied sciences. More in detail, matrix geometric means appear in diffusion tensor imaging (DTI)~\cite{Batchelor2005}, radar detection~\cite{Yang2010,LapuyadeLahorgue2008}, image processing~\cite{rathi2007}, elasticity~\cite{Moakher2006}, machine learning~\cite{Iannazzo2019}, brain-computer interfaces~\cite{Yger2017}, network analysis~\cite{Fasi2018}, and many more applications.
The geometric mean provides a meaningful way of averaging positive definite matrices while preserving critical structural properties. Historically, the notion of the geometric mean dates back to ancient mathematics, where it was primarily applied to positive numbers. For positive real numbers $a_1, a_2, \ldots, a_k$, the geometric mean is defined as
\[
g = \left( \prod_{i=1}^k a_i \right)^{1/k},
\]
a concept that can be extended to matrices~\cite{Moakher2005, bhatia2007}. However, this extension is far from straightforward, primarily due to the lack of commutativity in matrix multiplication. 

Initially introduced implicitly in the context of a functional calculus for sesquilinear maps by Pusz and Woronowicz~\cite{PuszWoronowicz1975}, but popularized as a mean by Kubo and Ando~\cite{KuboAndo1980}, the matrix geometric mean was rigorously formalized as follows: for two positive definite matrices $A$ and $B$, the geometric mean $G(A, B)$ is defined as
\[
G(A, B) = A^{1/2} (A^{-1/2} B A^{-1/2})^{1/2} A^{1/2} = G(B, A).
\]
The rigorous definition has been analyzed by Ando, Li, and Mathias (ALM)~\cite{Ando2004}, who also identified essential axiomatic properties~\cite[Section~3]{bini2024survey} that a proper matrix geometric mean should satisfy. These axioms serve as a foundation for several established definitions, including the ALM mean itself, the Nakamura–Bini–Meini–Poloni (NBMP) mean~\cite{Nakamura2009}, and the Karcher mean, recognized as a matrix geometric mean through Riemannian geometry as introduced in~\cite{Bini2013, Moakher2005}. Further detailed historical developments and generalizations can be found in~\cite{bhatia2007}.

Building upon these foundational works, it becomes essential to address the challenges that arise when considering more than two matrices. In principle, if $k > 2$, the ALM mean can be computed through a recursive process, where at each step the geometric mean is formed by reducing the problem to $k-1$ matrices. However, this method has significant limitations: it only achieves linear convergence and entails high computational cost, as each recursive step requires a large number of iterations. As a result, extending the ALM geometric mean beyond two matrices quickly becomes computationally infeasible~\cite{Poloni2010}. To overcome these challenges, the Karcher mean~\cite{Bini2013} was introduced as a true generalization of the geometric mean to multiple positive definite matrices. For HPD matrices $A_1, A_2, \ldots, A_k$, the Karcher mean is defined as the unique positive definite solution 
\[
\sum_{i=1}^k \log(A_i^{-1} X) = 0,
\]
as established by Moakher~\cite[Proposition 3.4]{Moakher2005}. 

With these generalizations in hand, a natural and central question arises: How does the geometric mean behave asymptotically when applied to sequences of structured matrices, such as those in the GLT class? In particular, does the geometric mean of two (or more) GLT sequences itself form a GLT sequence, and what is its symbol?

These foundational questions form the cornerstone of the present thesis. The main questions addressed in this thesis concern the asymptotic behavior and spectral distribution of geometric means of GLT matrix-sequences, which are the focus of Chapters~\ref{GM GL1}, \ref{GM GLT2}. The results can be summarized as follows: the extension of distribution results to the geometric mean of HPD GLT matrix-sequences, establishing that under appropriate conditions the sequence of geometric means is itself a GLT sequence, with a symbol given by the geometric mean of the input symbols; the removal of unnecessary assumptions on invertibility and clarification of the role of commutativity in the distribution results. This is done in a complete way for $k=2$, while for the case of more than two matrix-sequences is still open, even though partial theoretical results and numerical tests indicate that such a generalization is possible.

In addition, the thesis includes an application-oriented Chapter~\ref{chapter:quantum} that demonstrates the power of GLT theory in the spectral analysis of large quantum spin systems. For instance, in mean-field models such as the Curie–Weiss model, the asymptotic spectral properties of large Hamiltonian matrices play a central role in understanding phase transitions and collective phenomena in the thermodynamic limit. Traditionally, analyzing these spectral properties relied on specialized analytical or combinatorial techniques. By interpreting discretized Hamiltonians as GLT matrix-sequences, however, one can leverage the powerful machinery of structured matrix theory: the spectral distribution of the quantum system is described by the symbol associated with the GLT sequence. This approach enables precise predictions of eigenvalue distributions and asymptotic behavior, bridging the gap between abstract mathematical theory and concrete models in mathematical physics.
\section*{Thesis structure}

The thesis is organized as follows:

\begin{itemize}
    \item Chapter~\ref{chapter:Structured Matrix Sequences} introduces foundational concepts and definitions related to matrix-sequences, focusing on singular value and eigenvalue distributions, approximation classes, and multi-level structures. This chapter sets the mathematical background necessary for the study of structured matrices and their spectral properties.

    \item Chapter~\ref{chap:Structured matrices} reviews classical structured matrices, including Toeplitz, circulant, $\omega$-circulant, $\tau$, and Hankel matrices. The chapter details their definitions, algebraic and spectral properties, and discusses their importance in numerical linear algebra and asymptotic linear algebra.

    \item Chapter~\ref{GLT} develops the theoretical framework of Generalized Locally Toeplitz (GLT) sequences, which serves as the mathematical backbone for the spectral analysis of structured matrix-sequences. The chapter introduces the GLT notion and presents GLT algebras and their key properties, establishing the essential axioms and algebraic structure that define the behavior of GLT sequences. It distinguishes between matrix-sequences with explicit and hidden asymptotic structures and discusses fundamental concepts such as zero-distributed sequences, block Toeplitz and block diagonal sampling matrices, and the construction of $d$-level $r$-block GLT $\ast$-algebras.

    \item Chapter~\ref{GM GL1} investigates the spectral distribution of geometric means of Hermitian positive definite GLT matrix-sequences. For the scalar case $(r = 1,\ d \geq 1)$, it is shown that the geometric mean of two GLT sequences remains in the GLT class, with the symbol given by the geometric mean of the input symbols. This result is extended to the block and multilevel case $(r,\ d \geq 1)$. Extensions to multiple matrices via the Karcher mean are also discussed. Theoretical findings are supported by numerical experiments in one and two dimensions, confirming the asymptotic spectral behavior.

    \item Chapter~\ref{GM GLT2} extends the study of the spectral distribution of geometric means of Hermitian positive definite GLT matrix-sequences to the maximal setting, including degenerate and non-commuting GLT symbols. It proves that if the GLT symbols commute, the geometric mean remains a GLT sequence with symbol given by the geometric mean of the input symbols, without requiring invertibility almost everywhere, completing the proof of a long-standing conjecture for the scalar and block cases $(r, d \geq 1)$. Through detailed numerical experiments, the chapter also demonstrates that when the symbols do not commute or are both degenerate, the classical relation may fail, revealing new phenomena and providing a candidate symbol for the geometric mean in these settings. 

    \item Chapter~\ref{chapter:quantum} applies the GLT framework to mean-field quantum spin models, especially the Curie–Weiss model. This chapter demonstrates how GLT theory and spectral distribution results can be used to analyze the asymptotic spectral properties of large Hamiltonians, connecting abstract theory with quantum statistical mechanics.
\end{itemize}
In terms of mathematical tools, Theorem~\ref{4: th:two - r=1,d general GM1}, Theorem~\ref{4: th:two - r,d general GM1}, and Theorem~\ref{main result-bis} are proved only by using the GLT axioms. The case where both GLT symbols are degenerate requires an elegant use of the notion of approximating classes of sequences combined with the GLT axioms; see Theorem~\ref{theorem 1} and Theorem~\ref{th:two - r,d general}. Finally, Theorem~\ref{main result} has been the most difficult, since the verification that the considered matrix-sequences are GLT with zero symbol has been performed through an explicit computation.

The thesis ends with a final section where we summarize the results and indicate open questions and future directions research.

\addcontentsline{toc}{chapter}{Introduction}
\chapter{Matrices and matrix-sequences
}\label{chapter:Structured Matrix Sequences}
 Matrix-sequences are essential to the mathematical investigation and practical application of numerical techniques for differential equations along with applications.  A matrix-sequence is formally defined as a collection $\{A_n\}_n $ of square matrices, where each matrix $A_n$ is of size $d_n \times d_n$, and the dimension $d_n$ increases with $n$.  This increasing dimension signifies the enhancement of discretizations, demonstrated by the reduction of mesh size in finite difference or finite element methods to attain improved accuracy.\\
 The origins of matrix-sequences originated in the discretization of continuous mathematical models.  For example, when an ordinary or partial differential equation is discretized on a grid with $n$ points, a linear system is often produced, with the coefficient matrix size specified by the number of unknowns in the discretization.  As the discretization refines (i.e., as the grid size diminishes or the number of nodes increases), the resultant matrices expand in size, creating a sequence indexed by the discretization parameter. These sequences are not random; the matrices frequently possess algebraic or analytic features derived from the original differential operator or the problem's geometry.  Prominent cases contain Toeplitz and circulant matrices in scenarios involving constant coefficients and periodic domains, as well as broader categories such as banded, block, or multi-level matrices.  Recently, the theory has been widened to include GLT sequences, which can represent a broader range of operators.\\
 The main motivation for examining matrix-sequences is their spectral characteristics as the matrix dimensions increase. The asymptotic behavior of eigenvalues and singular values, expressed through matrix symbols and limiting spectral measures, provides valuable insight into the stability and convergence of numerical schemes. For instance, the link between the von Neumann stability of linear approximation methods and the asymptotic spectral distribution results is discussed in \cite[Section 3.4]{SerraCapizzano2006}. When considering stationary iterative solvers, the convergence can be read in the symbol. If the eigenvalues remain confined within bounded regions of the complex plane and the spectral radius of the iteration matrix is strictly less than one, then the resulting iterative technique is convergent. In terms of the symbol, a necessary condition is that its infinity norm is less than one. In contrast, broad spectral distributions or eigenvalue accumulation near zero usually indicate poor conditioning, reduced convergence speed, or numerical instability. Furthermore, comprehending the fundamental structure of a matrix-sequence is essential for developing efficient algorithms, including rapid direct solvers, preconditioned iterative approaches (both stationary and of Krylov type), and multi-grid techniques or multi-iterative methods.\\
 This chapter seeks to explain the essential principles and motives underlying the study of matrix-sequences, thereby preparing the reader for the more complex content in subsequent chapters. Section~\ref{sig-eig} focuses on the concepts of eigenvalue and singular value distributions, which are essential instruments for the asymptotic spectral analysis of matrix-sequences.  Section~\ref{acs} explains the notion of convergence in the realm of matrix-sequences, emphasizing the robust framework for approximating classes of sequences. Section~\ref{multi level} ultimately examines the multilevel structural dimensions of the matrices presented in this thesis. Readers who desire a comprehensive and formal exploration of these topics are directed to the following {\rm \cite[Chapter 2]{garoni2018}}.

 \section{Singular value and eigenvalue distributions of a matrix-sequence}\label{sig-eig}
 In numerous cases, a matrix-sequence may be associated with an eigenvalue or singular value distribution.  We provide a formal definition of this idea and clarify its intuitive meaning.
 
\begin{definition}\label{99}
    (Singular value and eigenvalue distribution of a matrix-sequence). Let $\{A_n\}_n$ be a matrix-sequence, with $A_n$ of size $d_n$, and let $\psi : D \subset \mathbb{R}^t \to \mathbb{C}^{r \times r}$ be a measurable function defined on a set $D$ with $0 < \mu_t(D) < \infty, $ where $ \mu_t(\cdot) $  denotes the standard Lebesgue measure in $\mathbb{R}^t$.
\begin{itemize}
    \item We say that $\{A_n\}_n$ has an (asymptotic) singular value distribution described by $\psi$, and we write $\{A_n\}_n \sim_\sigma \psi$, if
    \[
    \lim_{n \to \infty} \frac{1}{d_n} \sum_{i=1}^{d_n} F(\sigma_i(A_n)) = \frac{1}{\mu_t(D)} \int_D \frac{1}{r} \sum_{i=1}^r F(\sigma_i(\psi(\bm{x}))) \, d\bm{x}, \quad \forall F \in C_c(\mathbb{R}).
    \]

    \item We say that $\{A_n\}_n$ has an (asymptotic) spectral (or eigenvalue) distribution described by $\psi$, and we write $\{A_n\}_n \sim_\lambda \psi$, if
    \[
    \lim_{n \to \infty} \frac{1}{d_n} \sum_{i=1}^{d_n} F(\lambda_i(A_n)) = \frac{1}{\mu_t(D)} \int_D \frac{1}{r} \sum_{i=1}^r F(\lambda_i(\psi(\bm{x}))) \, d\bm{x}, \quad \forall F \in C_c(\mathbb{C}).
    \]

    \item If $\psi$ describes both the singular value and eigenvalue distribution of $\{A_n\}_n$, we write $\{A_n\}_n \sim_{\sigma, \lambda} \psi$.
\end{itemize}
    In this case, the function $\psi$ is referred to as the \textit{eigenvalue (or spectral) symbol} of $\{A_n\}_n$.\\
    The same definition when the considered matrix-sequence shows a multilevel structure. In that case $n$ is replaced by $\bm{n}$ uniformly in $A_n$ and $d_n$.
\end{definition}

The formal definitions shown above define precisely the mathematical constraints under which a function, commonly referred to as the symbol, characterizes the asymptotic spectral (eigenvalue) or singular value distribution of a sequence of matrices.

The informal meaning behind the spectral distribution definition is as follows: if $\psi$ is continuous, then a suitable ordering of the eigenvalues $\{\lambda_j(A_n)\}_{j=1,\dots,d_n}$, assigned in correspondence with an equispaced grid on $D$, reconstructs approximately the $r$ surfaces $\bm{x} \mapsto \lambda_i(\psi(\bm{x}))$, $i = 1, \dots, r$. For example, in the simplest case where $t = 1$ and $D = [a, b]$, $d_n = nr$, the eigenvalues of $A_n$ are approximately equal up to a few potential outliers to $\lambda_i(\psi(x_j))$, where
\[
x_j = a + j \frac{(b-a)}{n}, \quad j = 1, \dots, n, \quad i = 1, \dots, r.
\]
If $t = 2$ and $D = [a_1, b_1] \times [a_2, b_2]$, $d_n = n^2 r$, the eigenvalues of $A_n$ are approximately equal again up to a few potential outliers to $\lambda_i(\psi(x_{j_1}, y_{j_2}))$, where
\[
x_{j_1} = a_1 + j_1 \frac{(b_1 - a_1)}{n}, \quad y_{j_2} = a_2 + j_2 \frac{(b_2 - a_2)}{n}, \quad j_1, j_2 = 1, \dots, n, \quad i = 1, \dots, r.
\]
A analogous understanding pertains to the definition of the singular value distribution.\\
If the considered structure is two-level then the subscript is $\bm{n}=(n_1,n_2)$ and $d_n = n_1 n_2 r$.  
Furthermore, for $t \geq 3$, a similar reasoning applies. 

Finally we report an observation which is useful in the following derivations.
\begin{remark}\label{rem: range}
The relation $\{A_n\}_n \sim_\lambda \psi$ and $\Lambda(A_n) \subseteq S$ for all $n$ imply that the range of $\psi$ is a subset of the closure $\bar{S}$ of 
$S$. In particular $\{A_n\}_n \sim_\lambda \psi$ and $A_n$ positive definite for all $n$ imply that $\psi$ is nonnegative definite almost everywhere, simply nonnegative almost everywhere if $r=1$. The same applies when a multilevel matrix-sequence $\{A_{\bm{n}}\}_{\bm{n}}$ is considered and similar statements hold when singular values are taken into account.
\end{remark}
As we will experience, computing the spectral distribution is often simpler for Hermitian
matrix-sequences. However, there exist several effective tools for analyzing matrix-sequences
where the non-Hermitian component is somehow negligible; as in the following proposition
stating that small-norm perturbations of Hermitian sequences do not alter the spectral
distribution.
\begin{proposition}[{\!\!\cite[Corollary 2]{Barbarino2020}}]
Let $\{A_{n}\}_n$ be a matrix-sequence such that  $A_{n} = B_{n} + C_{n}$, where $B_{n}$ is Hermitian $\forall n \in \mathbb{N}$. Suppose that
\begin{itemize}
    \item $\{B_{n}\}_n \sim_\lambda \psi$,
    \item $\|C_{n}\|=o(1)$.
\end{itemize}
Then $\{A_{n}\}_n \sim_\lambda \psi$.
\end{proposition}

The non-uniqueness of the distribution function for a matrix-sequence occurs because any measurable function that yields the same integral relationship when evaluated against all compactly supported continuous functions is also considered a distribution.  Once a distribution function $\psi$ is established, there exist infinitely many alternative functions $\phi$, potentially specified on distinct domains or possessing varying dimensionalities, that produce identical asymptotic spectral information.  For instance, any permutation of $\psi$, specifically any measurable function $\varphi : E \subset \mathbb{R}^s \to \mathbb{C}$ that meets the specified condition
\begin{equation*}
\frac{1}{\mu_t(D)} \int_D F(\psi(\bm{x}))\, d\bm{x} = \frac{1}{\mu_s(E)} \int_E F(\varphi(\bm{x}))\, d\bm{x}, \quad \forall F \in C_c(\mathbb{C}),
\end{equation*}
is also classified as a distribution. Alternative representations, which may involve varying numbers of variables or be matrix-valued; see e.g.\ \cite{barbarino2020block1d, barbarino2020blockmulti}, are deemed valid if they fulfill the aforementioned integral criterion.
\section{Approximating classes of sequences}\label{acs}

We present the idea of approximation class of sequences (a.c.s.), which is essential in the contemporary theory of matrix-sequences.  Generally, an (a.c.s.)  scheme for a specified matrix-sequence $\{A_n\}_n$ comprises a family of matrix sequences that can increasingly approximate $\{A_n\}_n$, aside from the potential existence of error terms: one of low rank relative to the overall matrix dimensions, and another possessing a negligible norm.

This concept can be regarded as a particular form of convergence, or more formally, as a method of establishing a topology on the space of matrix-sequences. The formal concept of approximation classes of sequences and their associated convergence results was initially defined in prior research see e.g.\ \cite{SerraCapizzano2001, Tilli1998}.  For additional literature and contemporary finding, refer to, for instance \cite{garoni2017, garoni2018}.
\begin{definition}[Approximating class of sequences]\label{def:acs}
Let $\{A_n\}_n$ be a matrix-sequence and let $\{\{B_{n,j}\}_n\}_j$ be a class of matrix-sequences, with $A_n$ and $B_{n,j}$ of size $d_n$. We say that $\{\{B_{n,j}\}_n\}_j$ is an approximating class of sequences (a.c.s.) for $\{A_n\}_n$ if the following condition is met: for every $j$, there exists $n_j$ such that, for every $n \geq n_j$,
\[
A_n = B_{n,j} + R_{n,j} + N_{n,j},
\]
\[
\text{rank}(R_{n,j}) \leq c(j) d_n \quad \text{and} \quad \|N_{n,j}\| \leq \omega(j),
\]
where $n_j$, $c(j)$, and $\omega(j)$ depend only on $j$, and
\[
\lim_{j \to \infty} c(j) = \lim_{j \to \infty} \omega(j) = 0.
\]
$
\{\{B_{n,j}\}_n\}_j \xrightarrow{\text{a.c.s. wrt } j} \{A_n\}_n
$ denotes that $\{\{B_{n,j}\}_n\}_j$ is an a.c.s. for $\{A_n\}_n$.
\end{definition}

In more intuitive words, we assert that $\{\{B_{n,j}\}_n\}_j$ constitutes an approximating class of sequences (a.c.s.) for $\{A_n\}_n$ if, for any sufficiently large value of $j$, the sequence $\{B_{n,j}\}_n$ yields an asymptotic approximation to $\{A_n\}_n$.  This indicates that each $A_n$ can be expressed as the sum of the corresponding $B_{n,j}$, a matrix with a rank significantly lower than that of $A_n$, and an additional matrix with a tiny norm.  As $j$ rises, the rank and norm of these error terms diminish significantly for large $n$.

The following theorem represents the expression of a related convergence theory and it is a powerful tool
used, for example, in the construction of the GLT $\ast$-Algebra.
\begin{theorem}\label{th: acs}
Let $\{A_n\}_n$, $\{B_{n,j}\}_n$, with $j, n \in \mathbb{N}$, be matrix-sequences and let $\psi, \psi_j : D \subset \mathbb{R}^d \to \mathbb{C}^{r \times r}$ be measurable functions defined on a set $D$ with positive and finite Lebesgue measure. Suppose that:
\begin{enumerate}
    \item $\{B_{n,j}\}_n \sim_{\sigma} \psi_j$ for every $j$;
    \item $\{\{B_{n,j}\}_n\}_j \xrightarrow{\text{a.c.s. wrt } j} \{A_n\}_n$;
    \item $\psi_j \to \psi$ in measure.
\end{enumerate}
Then
\[
\{A_n\}_n \sim_{\sigma} \psi.
\]
\end{theorem}

\begin{theorem}\label{th: acs eig}
Given a Hermitian matrix-sequence $\{A_n\}_n$, let $\{B_{n,j}\}_n$, be a class of
Hermitian matrix-sequences and let $\psi, \psi_j : D \subset \mathbb{R}^d \to \mathbb{C}^{r \times r}$ be measurable functions defined on a set $D$ with positive and finite Lebesgue measure. Suppose that:
\begin{enumerate}
    \item $\{B_{n,j}\}_n \sim_{\lambda} \psi_j$ for every $j$;
    \item $\{\{B_{n,j}\}_n\}_j \xrightarrow{\text{a.c.s. wrt } j} \{A_n\}_n$;
    \item $\psi_j \to \psi$ in measure.
\end{enumerate}
Then
\[
\{A_n\}_n \sim_{\lambda} \psi.
\]

\end{theorem}
We end this section by observing that the same definition can be given and corresponding results (with obvious changes) hold, when the involved matrix-sequences show a multilevel structure. In that case $n$ is replaced by $\bm{n}$ uniformly in $A_n$, $B_{n,j}$, $d_n$.

To prevent ambiguity, it is essential to specify that, in some instances, we regard ``constant'' classes of sequences of the kind $\{\{B_{n}\}_n\}_j$, where the fundamental sequence $\{B_n\}_n$ is independent of the parameter $j$.  Thus, for each value of $j$, the sequence stays invariant.  For instance, in the class $\{\{I_n\}_n\}_j$, each member is simply the series of identity matrices $\{I_n\}_n$ for all $j \in \mathbb{N}$.  
\section{Multilevel discretization matrices}\label{multi level}

The discretization of partial differential equations (PDEs) on multi-dimensional domains, such as rectangles, cubes, or higher-dimensional hyperrectangles, inherently results in the formation of multi-level matrices. These matrices are characterized by their hierarchical structure, reflecting the underlying geometry and grid of the computing domain. Multi-level matrices are not simply collections of numbers; they are recursively structured into blocks: each matrix is segmented into larger blocks, which can be further subdivided into smaller sub-blocks, and so on. When a PDE is discretized on a $d$-dimensional grid comprising $n_1, n_2, \ldots, n_d$ points in each spatial direction, the resultant matrix is generally of size $N \times N$, where $N = n_1 n_2 \cdots n_d$.

At the highest level, the matrix consists of $n_1^2$ blocks of size $N/n_1 $, corresponding to subdivisions along the first coordinate; each block at this level is further broken down into $n_2^2$ blocks of size $N/(n_1 n_2) $, according to the second coordinate, and this recursive process continues until, at the final level, we arrive at individual scalar entries. In general, at the $d$-th level, each block is partitioned into $n_d^2$ blocks of order $N/(n_1 \cdots n_d)$. This multi-level structure is not just a mathematical curiosity: it is vital for the efficient storage, computation, and theoretical analysis of the matrices that arise from multi-dimensional problems.

Using the conventional notation of two scalar indices $i, j = 1, \ldots, N$ to refer to the entries of a multi-level matrix quickly becomes impractical. To address this, mathematicians use multi-index notation, where each position in the matrix or vector is labeled by a $d$-tuple, or multi-index, $\bm{i} = (i_1, \ldots, i_d)$ in $\mathbb{Z}^{d}$, that records its location along every dimension of the grid. This approach not only makes the notation more concise and natural but also aligns perfectly with the recursive construction of multi-level matrices. Using multi-indices, we can systematically and clearly describe the positions of entries in both vectors and matrices: for example, a vector $\bm{x} \in \mathbb{C}^{N(\bm{n})}$ with $ N(\bm{n}) =n_1, \ldots, n_d$ can be indexed as $\bm{x} = [ x_{\bm{i}}]_{\bm{i=1}}^{\bm{n}}$, The first component is $x_{(1,\ldots,1)}$, the second is $x_{(1,\ldots,1,2)}$, and so on, until $x_{(n_1,\ldots,n_d)}$. For instance, the component in position $n_d + 3$ will be $x_{(1,\ldots,1,2,3)}$. Likewise, a $d$-level matrix $A \in M_{N(\bm{n})}(\mathbb{C})$ can be written as $A = [a_{\bm{i}\bm{j}}]_{\bm{i},\bm{j}=1}^{\bm{n}}$, where $\bm{n} = (n_1, \ldots, n_d)$.

In summary, a multi-level matrix-sequence is a family of matrices $\{A_{\bm{n}}\}_{\bm{n}}$, whose size and structure are governed by a multi-index $\bm{n} \to +\infty$ that reflects the refinement of the underlying grid in each coordinate direction, with $A \in M_{N(\bm{n})}(\mathbb{C})$.  With this framework in place, we are now prepared to explore the different types of structured matrices that play a central role in the numerical analysis of PDEs on multi-dimensional domains.
\chapter{Structured matrices}\label{chap:Structured matrices}

Matrices generated by common discretization techniques, such as finite difference schemes, often possess distinctive structures whose theoretical analysis is crucial for efficient numerical computation. Indeed, it is very important to understand these structures from a mathematical point of view because it leads to faster numerical processes and new ideas about how numerical schemes work.  Toeplitz matrices, which have constant entries along each diagonal, are one of the most well-known of these types.  This diagonal constant is a discretized form of the translation invariance quality that is built into partial differential equations (PDEs) and fractional differential equations (FDEs) with constant coefficients. A standard Toeplitz matrix is automatically made when a uniform grid spacing is used in a one-dimensional discretization setting.  This is easy because the discrete approximation of derivative operators keeps the stencil patterns the same along the grid, which means that diagonals have the same entries.

Toeplitz structures with more than one level are created when discretizing PDEs or FDEs in higher dimensions using uniform grid spacing over rectangular areas.  As these matrices get more complicated, they turn into block Toeplitz matrices, where each block has its own Toeplitz structure.  Each block is basically a Toeplitz matrix with fewer dimensions that is nested inside a matrix structure with more dimensions. These multi-level Toeplitz matrices have a direct relationship between the number of dimensions of the subject being studied and the number of recursive levels.  For instance, problems in two dimensions produce matrices with two levels of recursion, while problems in three dimensions produce matrices with three levels of stacked levels.  Understanding and using this multi-level Toeplitz structure is important for making numerical methods work well and for making the best use of computing strategies in real-world situations.

This chapter begins with Section \ref{TM}. It contains a rigorous definition of multi-level Toeplitz matrices, followed by a concise overview of their key spectral and algebraic characteristics, based on the substantial research that has developed in this area over time. Unless otherwise noted, these foundational results are comprehensively discussed in \cite[Chapter~3]{garoni2018}. Additionally, the chapter explores several related matrix structures some of which are specific subclasses of Toeplitz matrices, while others exhibit strong algebraic and spectral links to Toeplitz matrices highlighting their significance in numerical computations. In particular, circulant and $\omega$-circulant matrices are introduced in Sections~\ref{cir mat} and \ref{omega cir mat}, respectively, whereas Sections~\ref{tau-mat} and \ref{hank matr} focus on $\tau$ matrices and Hankel matrices.
\section{Toeplitz matrices}\label{TM}
Given $\bm{n} \in \mathbb{N}^d$, a matrix of the form
\[
T_{\bm{n}} := [a_{\bm{i}-\bm{j}}]_{\bm{i},\bm{j}=1}^{\bm{n}} \in M_{N(\bm{n})}(\mathbb{C}),
\]
where the element in position $(\bm{i}, \bm{j})$ depends only on the difference $\bm{i}-\bm{j}$, is called a multi-level or $d$-level Toeplitz matrix.

In the simplest case where $d = 1$ and $\bm{n} = n$, a 1-level or uni-level Toeplitz matrix is just a classical Toeplitz matrix, whose entries remain constant along the diagonals:
\[
T_n = [a_{i-j}]_{i,j=1}^n =
\begin{bmatrix}
a_0 & a_{-1} & a_{-2} & \cdots & \cdots & a_{1-n} \\
a_1 & a_0 & a_{-1} & a_{-2}&  & \vdots \\
a_2 & a_1 & \ddots &\ddots & \ddots & \vdots \\
\vdots & a_2&\ddots&\ddots&\ddots&a_{-2}\\
\vdots &  & \ddots & \ddots & \ddots &a_{-1} \\
a_{n-1} & \cdots & \cdots & a_{2} & a_1&a_{0}
\end{bmatrix}
\in M_n(\mathbb{C}).
\]
Note that $T_n$ is completely specified by its $2n - 1$ coefficients $a_k$, $k = 1-n, \ldots, n-1$, or equivalently by its first row and column.

For $d = 2$ and $\bm{n} = (n_1, n_2)$, we have a 2-level Toeplitz matrix:
\[
T_n = \bigg[\left[ a_{(i_1 - j_1,\, i_2 - j_2)} \right]_{i_2, j_2 = 1}^{n_2}\bigg]_{i_1, j_1 = 1}^{n_1}
= \left[ A_{i_1 - j_1} \right]_{i_1, j_1 = 1}^{n_1}\]
\[
= 
\begin{bmatrix}
A_0 & A_{-1} & A_{-2} & \cdots & \cdots & A_{1-n_{1}} \\
A_1 & A_0 & A_{-1} & A_{-2}&  & \vdots \\
A_2 & A_1 & \ddots &\ddots & \ddots & \vdots \\
\vdots & A_2&\ddots&\ddots&\ddots&A_{-2}\\
\vdots &  & \ddots & \ddots & \ddots &A_{-1} \\
A_{n_{1}-1} & \cdots & \cdots & A_{2} & A_1&A_{0}
\end{bmatrix}
\in M_{n_1 n_2}(\mathbb{C}),
\]
where each block $A_k$ is itself a Toeplitz matrix:\\
$
A_k = \left[ a_{(k,\, i_2 - j_2)} \right]_{i_2, j_2 = 1}^{n_2}$

\[=
\begin{bmatrix}
a_{(k, 0)} & a_{(k, -1)} & a_{(k, -2)} & \cdots & \cdots & a_{(k, 1-n_{2})} \\
a_{(k, 1)} & a_{(k, 0)} & a_{(k, -1)} & a_{(k, -2)}&  & \vdots \\
a_{(k, 2)} & a_{(k, 1)} & \ddots &\ddots & \ddots & \vdots \\
\vdots & a_{(k, 2)}&\ddots&\ddots&\ddots&a_{(k, -2)}\\
\vdots &  & \ddots & \ddots & \ddots &a_{(k, -1)} \\
a_{(k, n_{2}-1)} & \cdots & \cdots & a_{(k, 2)} & a_{(k, 1)}&a_{(k, 0)}
\end{bmatrix}\in M_{n_2}(\mathbb{C}).
\]
This type of matrix is also known as a \emph{block Toeplitz matrix with Toeplitz blocks} or BTTB matrix.\\
In the case where $d = 3$, a 3-level Toeplitz matrix is made up of blocks, and each block has its own 2-level Toeplitz structure.  A more general way to describe a $d$-level Toeplitz matrix is as a block Toeplitz matrix whose blocks are $(d-1)$-level Toeplitz matrices.  In this recursive design, the coefficients in the first row and first column of the matrix tell you everything you need to know about it, just like in the case of a single-level matrix.
From the computational perspective, this makes the storage requirement for Toeplitz matrices drop to only $O(N(\bm{n}))$, compared to $O(N(\bm{n})^2)$ for a generic dense matrix.\\
A compact expression for a multi-level Toeplitz matrix can be formulated by defining, for $n \in \mathbb{N}$ and $k \in \mathbb{Z}$, the matrix $J_n^{(k)}$ as follows:

\[
[ J_n^{(k)}]_{i,j}:= \delta_{i-j-k}=\begin{cases}
  1 & \text{if } i - j = k, \\\\
  0 & \text{otherwise},
\end{cases}
\qquad i, j = 1, \ldots, n,
\]
where $\delta_{r}$ is the Kronecker delta function. In other words, $J_n^{(k)}$ is a matrix with all ones on the $k$-th diagonal and zeros elsewhere. The main diagonal is numbered as 0, diagonals above as $-1, -2, ..., -n+1$, and those below as $1, 2, ..., n-1$. For the multi-level case, where $\bm{n} \in \mathbb{N}^{d}$ and $\bm{k} \in \mathbb{Z}^{d}$, are both $d$-dimensional vectors (for example, $\bm{n} = (n_1, ..., n_d)$ and $\bm{k} = (k_1, ..., k_d)$), we define its multi-level version:
\[
J_{\bm{n}}^{(\bm{k})} = J_{n_1}^{(k_1)} \otimes J_{n_2}^{(k_2)} \otimes \cdots \otimes J_{n_d}^{(k_d)}.\]

Then, $T_{\bm{n}}$ admits the following representation:
\begin{equation}\label{T comp}
T_{\bm{n}} = [a_{\bm{i}-\bm{j}}]_{\bm{i},\bm{j=1}}^n = \sum_{\bm{k=1-n}}^{\bm{n-1}} a_{\bm{k}} J_{\bm{n}}^{(\bm{k})}.
\end{equation}
This means the matrix is a sum over all possible diagonals, where each term involves the coefficient $a_{\bm{k}}$ multiplied by the matrix $J_{\bm{n}}^{(\bm{k})}$, which places ones along the $\bm{k}$-th diagonal and zeros elsewhere. This result can be demonstrated by examining each entry $(\bm{i}, \bm{j})$ of the matrix on the right-hand side and using the definition of $J_{\bm{n}}^{(\bm{k})}$ together with the properties of the Kronecker product.\\

Given a function $f : [-\pi, \pi]^d \to \mathbb{C}$ belonging to the space of integrable functions $L^1([-\pi, \pi]^d)$, it is possible to construct a class of multi-level Toeplitz matrices associated with $f$. Let $f_{\bm{k}}$ denote the $\bm{k}$-th Fourier coefficient of $f$, defined for any $\bm{k} \in \mathbb{Z}^d$ as
\[
f_{\bm{k}} := \frac{1}{(2\pi)^d} \int_{[-\pi,\pi]^d} f(\bm{\theta})\, e^{-i \bm{k} \cdot \bm{\theta}} \, d\bm{\theta},
\]
where $i^2 = -1$ and $\bm{k} \cdot \bm{\theta} = k_1\theta_1 + \cdots + k_d\theta_d$.

For any $\bm{n} \in \mathbb{N}^d$, the $\bm{n}$-th Toeplitz matrix associated with $f$ is the $d$-level Toeplitz matrix, denoted by $T_{\bm{n}}(f)$, whose entries are the Fourier coefficients of $f$, arranged along the diagonals according to the following definition:
\[
T_{\bm{n}}(f) := [f_{\bm{i}-\bm{j}}]_{\bm{i},\bm{j}=\bm{1}}^{\bm{n}} \in M_{N(\bm{n})}(\mathbb{C}).
\]
The function $f$ is referred to as the generating function of $T_{\bm{n}}(f)$, since its Fourier coefficients determine the matrix entries. In the broader GLT framework, the term symbol instead denotes the function that describes the asymptotic singular value distribution of a matrix-sequence and the asymptotic eigenvalue distribution under certain condition; see \textbf{GLT1} and \textbf{GLT5}. In this direction, we notice that $f$ is also the GLT symbol of the matrix-sequences $\{T_{\bm{n}}(f)\}_{\bm{n}}$. Furthermore, using Equation (\ref{T comp}), $T_{\bm{n}}(f)$ can also be represented as
\[
T_{\bm{n}}(f) = \sum_{\bm{k} = \bm{1-n}}^{\bm{n-1}} f_{\bm{k}} J_{\bm{n}}^{(\bm{k})}= \sum_{k_1 = 1 - n_1}^{n_1 - 1} \cdots \sum_{k_d = 1 - n_d}^{n_d - 1} f_{(k_1, \ldots, k_d)} \left( J_{n_1}^{(k_1)} \otimes \cdots \otimes J_{n_d}^{(k_d)} \right).
\]
\subsection{Relationship between the Toeplitz matrix and its generating function}
By constructing a Toeplitz matrix $T_{\bm{n}}(f)$ from a given multivariate function $f$ in $L^1([-\pi, \pi]^d)$, we are essentially creating a map that takes functions and produces matrices:
\[
T_{\bm{n}}(\cdot) : L^1([-\pi, \pi]^d) \to M_{N(\bm{n})}(\mathbb{C})
\]
for any fixed $\bm{n} \in \mathbb{N}^d$. The properties of the function $f$ are closely related to those of the Toeplitz matrix.
To start, the operator $T_{\bm{n}}(\cdot)$ is linear. This follows directly from the linearity of the Fourier coefficient operation.
\begin{proposition}
For any $\alpha, \beta \in \mathbb{C}$ and $f, g \in L^1([-\pi, \pi]^d)$, it holds that
\[
T_{\bm{n}}(\alpha f + \beta g) = \alpha\, T_{\bm{n}}(f) + \beta\, T_{\bm{n}}(g).
\]
Moreover, further results follow from the identities
\[
(f(-\bm{\theta}))_{\bm{k}} = f_{-\bm{k}}
\quad \text{and} \quad
(\overline{f})_{\bm{k}} = \overline{f_{-\bm{k}}},
\quad \text{for all } \bm{k} \in \mathbb{Z}^d,
\]
where $\overline{f}$ denotes the complex conjugate of $f$.
\end{proposition}

\begin{proposition}
For any $f \in L^1([-\pi, \pi]^d)$, the following properties hold:
\[
T_{\bm{n}}(f(\bm{\theta)})^{\top} = T_{\bm{n}}(f(-\bm{\theta})),
\]
\[
T_{\bm{n}}(f)^* = T_{\bm{n}}(\overline{f}),
\]
where $T_{\bm{n}}(f)^*$ is the Hermitian (conjugate transpose) of $T_{\bm{n}}(f)$.
\end{proposition}
Specifically, if the function $f$ is real-valued a.e., then the corresponding Toeplitz matrix $T_{\bm{n}}(f)$ is Hermitian.  This indicates that the matrix is identical to its conjugate transpose, a characteristic that guarantees all of its eigenvalues are real numbers.  Consequently, when initiating with a real function, the Toeplitz matrix generated will always possess real eigenvalues.\\
Furthermore, the spectrum of $T_{\bm{n}}(f)$, comprising its eigenvalues, will always exist within the range of values assumed by $f$, namely inside the interval defined by the essential minimum and essential maximum of $f$.  This offers an effective method to estimate or constrain the potential eigenvalues of the Toeplitz matrix based exclusively on the characteristics of the generating function.
\begin{theorem}\label{the 2.3}
Given $f \in L^1([-\pi, \pi]^d)$, real a.e., let $m_f := \operatorname{ess\,inf} f$ and $M_f := \operatorname{ess\,sup} f$. Then, for any $\bm{n} \in \mathbb{N}^d$, it holds that
\[
\lambda_j \big(T_{\bm{n}}(f)\big) \in [m_f, M_f], \quad \forall\, j = 1, \ldots, N(\bm{n}).
\]
If, in addition, $m_f < M_f$, or equivalently $f$ is not constant a.e., then it holds
\[
\lambda_j \big(T_{\bm{n}}(f)\big) \in (m_f, M_f), \quad \forall\, j = 1, \ldots, N(\bm{n}).
\]
\end{theorem}
An important extension of the previous theorem concerns the behavior of the extreme eigenvalues of the Toeplitz matrix $T_{\bm{n}}(f)$ as the matrix dimension increases.  Specifically, for any integer $j \geq 1$, as $\bm{n} \to \infty$, the $j$ smallest eigenvalues of $T_{\bm{n}}(f)$ converge to the basic infimum of the generating function $f$.  As the matrix size increases, its smallest eigenvalues converge towards the minimum value that $f$ attains almost uniformly.  Likewise, the $j$ largest eigenvalues converge to the fundamental supremum of $f$, indicating they approach the maximum value of $f$ within its domain.\\
 This outcome offers an exact asymptotic characterization of the spectrum of large Toeplitz matrices, indicating that their extreme eigenvalues are determined by the minimum and maximum values of the generating function.
\begin{theorem}
Given $f \in L^1([-\pi, \pi]^d)$, real a.e., let $m_f := \operatorname{ess\,inf} f$ and $M_f := \operatorname{ess\,sup} f$, and let
\[
\lambda_1 \big(T_{\bm{n}}(f)\big) \leq \lambda_2 \big(T_{\bm{n}}(f)\big) \leq \cdots \leq \lambda_{N(\bm{n})} \big(T_{\bm{n}}(f)\big)
\]
be the eigenvalues of $T_{\bm{n}}(f)$ arranged in increasing order. Then, for any fixed $j \geq 1$, it holds that
\[
\lim_{\bm{n} \to \infty} \lambda_j\big(T_{\bm{n}}(f)\big) = m_f, \qquad
\lim_{\bm{n} \to \infty} \lambda_{N(\bm{n})-j+1}\big(T_{\bm{n}}(f)\big) = M_f.
\]
\end{theorem}
From Theorem~\ref{the 2.3}, we immediately deduce the positivity and monotonicity of the Toeplitz operator.
\begin{corollary}
For any $f, g \in L^1([-\pi, \pi]^d)$, the following holds:
\[
f \geq 0 \text{ a.e.} \iff T_{\bm{n}}(f) \geq O_{N(\bm{n})}, \quad \forall \bm{n} \in \mathbb{N}^d,
\]
and, consequently,
\[
f \geq g \text{ a.e.} \iff T_{\bm{n}}(f) \geq T_{\bm{n}}(g) \quad \forall \bm{n} \in \mathbb{N}^d.
\]
\end{corollary}
These results remain valid even if all of the inequalities are replaced by strict inequalities. In other words, if you consider situations where $f \geq 0$ a.e., or $f \geq g$ a.e., the corresponding conclusions about the Toeplitz matrices still hold when using the strict inequality signs.

The spectral norm of a Toeplitz matrix is connected with the $\infty$-norm of its generating function through the following inequality, proved in \cite{serra97}.

\begin{proposition}
For any $f \in L^{\infty}((-\pi, \pi]^d)$ and any $\bm{n} \in \mathbb{N}^d$, it holds that
\[
\|T_{\bm{n}}(f)\| \leq \|f\|_{\infty}.
\]
\end{proposition}
Finally, we present a noteworthy result that clarifies how the separability of the function $f$ affects the structure of the associated multi-level Toeplitz matrices.

\begin{proposition}
Let $f \in L^1([-\pi, \pi]^d)$. If $f$ is separable, i.e.,
\[
f = f_1 \otimes \cdots \otimes f_d,
\]
then, for any $\bm{n} \in \mathbb{N}^d$, it holds that
\[
T_{\bm{n}}(f) = T_{n_1}(f_1) \otimes \cdots \otimes T_{n_d}(f_d).
\]
\end{proposition}
If we allow the multi-index $\bm{n} \in \mathbb{N}^d$ instead of keeping it fixed, we obtain a whole sequence of multi-level Toeplitz matrices, denoted by $\{T_{\bm{n}}(f)\}_{\bm{n}}$. There exists a deep and important relationship between the singular values and eigenvalue distributions of this sequence and the properties of the generating function $f$. Over the years, this connection has been established and presented in several different forms within the mathematical literature.

In particular, the result is elegantly stated as in  \cite[Theorem~3.5]{garoni2018}, where the proof relies on the notion of approximation in the sense of almost convergence in the strong sense (a.c.s.). For readers interested in the earliest formulation of this result, reference \cite{tyrtyshnikov1998} provides the original statement, while further generalizations specifically to cases where the generating function $f$ is matrix-valued rather than scalar can be found in \cite{tilli1998}.
\begin{theorem}
With reference to Definition \ref{99}, if $f \in L^1([-\pi, \pi]^d)$, then
\[
\{T_{\bm{n}}(f)\}_{\bm{n}} \sim_\sigma (f,\, [-\pi, \pi]^d).
\]
If, moreover, $f$ is real valued, then
\[
\{T_{\bm{n}}(f)\}_{\bm{n}} \sim_\lambda (f,\, [-\pi, \pi]^d).
\]
\end{theorem}
This theorem states that, as the size of the multi-level Toeplitz matrices grows, their singular values (and, if $f$ is real, also their eigenvalues) become distributed according to the generating function $f$ over the domain $[-\pi, \pi]^d$.
\section{Circulant matrices}\label{cir mat}
A multi-level (or $d$-level) circulant matrix for $\bm{n} \in \mathbb{N}^d$ is a structured matrix defined as,
\[
C_{\bm{n}} := [a_{(\bm{i}-\bm{j}) \bmod \bm{n}}]_{\bm{i},\bm{j}=\bm{1}}^{\bm{n}} \in M_{N(\bm{n})}(\mathbb{C}),
\]
where $(i - j) \bmod \bm{n}$ denotes the remainder obtained when $(i - j)$ is divided by $\bm{n}$. This means that the entry in position $(\bm{i}, \bm{j})$ depends only on the difference $(\bm{i} - \bm{j})$ mod $\bm{n}$, so the matrix repeats its values cyclically along each diagonal. Each row of the circulant matrix is simply a cyclic shift of the previous row, making circulant matrices a special case of Toeplitz matrices with even higher symmetry.

When $d = 1$ and $\bm{n} = n$, this reduces to a classic circulant matrix, where each column is a one-position downward shift of the previous one:
\[
C_n = \bigr[a_{(i-j) \bmod n}\bigr]_{i,j=1}^n =
\begin{bmatrix}
a_0 & a_{n-1} & a_{n-2} & \cdots & \cdots & a_{1} \\
a_1 & a_0 & a_{n-1} & a_{n-2}&  & \vdots \\
a_2 & a_1 & \ddots &\ddots & \ddots & \vdots \\
\vdots & a_2&\ddots&\ddots&\ddots&a_{n-2}\\
\vdots &  & \ddots & \ddots & \ddots &a_{n-1} \\
a_{n-1} & \cdots & \cdots & a_{2} & a_1&a_{0}
\end{bmatrix}
\in M_n(\mathbb{C}).
\]
$C_n$ is entirely determined by the $n$ coefficients in its first column, $\{a_k\}_{k=0}^{n-1}$. Similarly, each row is a one-element rightward shift of the previous row, so that the matrix is fully specified by its first row.

For $d = 2$ and $\bm{n} = (n_1, n_2)$, we obtain a 2-level circulant matrix:
\[
C_{\bm{n}} = \Bigr[\bigr[ a_{((i_1-j_1) \bmod n_1,\, (i_2-j_2) \bmod n_2)} \bigr]_{i_2,j_2=1}^{n_2}\Bigr]_{i_1,j_1=1}^{n_1}
= \bigr[ A_{(i_1-j_1) \bmod n_1} \bigr]_{i_1,j_1=1}^{n_1}\]
\[
=
\begin{bmatrix}
A_0 & A_{n_1-1} & A_{n_1-2} & \cdots & \cdots & A_{1} \\
A_1 & A_0 & A_{n_1-1} & A_{n_1-2}&  & \vdots \\
A_2 & A_1 & \ddots &\ddots & \ddots & \vdots \\
\vdots & A_2&\ddots&\ddots&\ddots&A_{n_1-2}\\
\vdots &  & \ddots & \ddots & \ddots &A_{n_1-1} \\
A_{n_{1}-1} & \cdots & \cdots & A_{2} & A_1&A_{0}
\end{bmatrix}
\in M_{n_1 n_2}(\mathbb{C}),
\]
in which each block $A_k$ is itself a circulant matrix:
\[
A_k =
\left[ a_{(k,\ (i_2-j_2) \bmod n_2 )} \right]_{i_2, j_2=1}^{n_2}\]
\[
=
\begin{bmatrix}
a_{(k, 0)} & a_{(k, n_2-1)} & a_{(k, n_2-2)} & \cdots & \cdots & a_{(k, 1)} \\
a_{(k, 1)} & a_{(k, 0)} & a_{(k, n_2-1)} & a_{(k, n_2-2)}&  & \vdots \\
a_{(k, 2)} & a_{(k, 1)} & \ddots &\ddots & \ddots & \vdots \\
\vdots & a_{(k, 2)}&\ddots&\ddots&\ddots&a_{(k, n_2-2)}\\
\vdots &  & \ddots & \ddots & \ddots &a_{(k, n_2-1)} \\
a_{(k, n_{2}-1)} & \cdots & \cdots & a_{(k, 2)} & a_{(k, 1)}&a_{(k, 0)}
\end{bmatrix}\in M_{n_2}(\mathbb{C}).
\]
This structure is often called a \emph{block circulant matrix with circulant blocks (BCCB)}.

For $d = 3$, a 3-level circulant matrix is constructed as a block circulant matrix, where each block is itself a 2-level circulant matrix. More generally, a $d$-level circulant matrix can be defined recursively as a block circulant matrix whose blocks have $(d-1)$-level circulant structure. Just as in the classic (uni-level) case, the entire matrix is fully specified by the coefficients in its first column, or, equivalently, by those in its first row. 

A compact expression for a uni-level circulant matrix $C_n$ can be derived by using
\[
[Z_n]_{i,j} := \delta_{(i-j-1) \bmod n} =
\begin{cases}
1 & \text{if } i - j = 1 \mod{n}, \\
0 & \text{otherwise},
\end{cases}
\qquad i, j = 1, \ldots, n,
\]
where $n \in \mathbb{N}$ and $\delta$ is the Dirac delta function.
More explicitly, $Z_n$ is the matrix that has all ones on the first and $(1-n)$-th diagonals, with zeros elsewhere:
\[
Z_n =
\begin{bmatrix}
0      &        &        &        &       &   1    \\
1      & \ddots &        &        &       &        \\
       & \ddots & \ddots &        &       &        \\
       &        & \ddots & \ddots &       &        \\
       &        &        & \ddots & \ddots&        \\
       &        &        &        &1      &   0    \\
\end{bmatrix}
\in M_n(\mathbb{C}).
\]
Raising $Z_n$ to the power $k$, for $k = 1, \ldots, n-1$, cyclically shifts the diagonals by $k$ positions, with entries given by
\[
[Z_n^k]_{i,j} := \delta_{(i-j-k) \bmod n} =
\begin{cases}
1 & \text{if } i-j = k \mod{n}, \\
0 & \text{otherwise},
\end{cases}
\qquad i, j = 1, \ldots, n.
\]

More explicitly,
\[
Z_n^2 =
\begin{bmatrix}
0      &        &        &        &  1    &   0   \\
0      & \ddots &        &        &       &   1   \\
1      & \ddots & \ddots &        &       &       \\
       &  \ddots& \ddots & \ddots &       &      \\
       &        & \ddots & \ddots & \ddots&       \\
       &        &        &    1   &   0   &   0   \\
\end{bmatrix}, \qquad
Z_n^3 =
\begin{bmatrix}
0      &        &        &     1  &    0  &   0   \\
0      & \ddots &        &        &    1  &   0   \\
0      & \ddots & \ddots &        &       &   1   \\
1      & \ddots & \ddots & \ddots &       &      \\
       & \ddots & \ddots & \ddots & \ddots&       \\
       &        &  1      &    0   &   0   &   0   \\
\end{bmatrix}, \quad \cdots
\]
Hence, any uni-level circulant matrix can be expressed as a linear combination of these powers:
\begin{equation}\label{2.2} 
C_n = \sum_{k=0}^{n-1} a_k Z_n^k. 
\end{equation}
Conversely, since $Z_n^n = I_n$ and, more generally, $Z_n^k = Z_n^{k \bmod n}$ for any $k \in \mathbb{Z}$, any finite linear combination of the form
\[
\sum_{k \in K} a_k Z_n^k, \quad a_k \in \mathbb{C},\, K \subset \mathbb{Z},
\]
is a circulant matrix. For this reason, $Z_n$ is commonly referred to as the generator of circulant matrices of order $n$.

The multi-level generalization of circulant matrices is straightforward. For $\bm{n} \in \mathbb{N}^d$ and $\bm{k} \in \mathbb{Z}^d$, we define
\[
Z_{\bm{n}}^{\bm{k}} := Z_{n_1}^{k_1} \otimes Z_{n_2}^{k_2} \otimes \cdots \otimes Z_{n_d}^{k_d}.
\]
 Using this definition, any multi-level circulant matrix can be written as a linear combination of these Kronecker products:
\[
C_{\bm{n}} = \sum_{\bm{k}=\bm{0}}^{\bm{n-1}} a_{\bm{k}} Z_{\bm{n}}^{\bm{k}}.
\]
This formula is directly analogous to the uni-level case and demonstrates that, just as for classic circulant matrices, the structure of multi-level circulant matrices is fully captured by these shift matrices and the set of coefficients $a_{\bm{k}}$. 
\subsection{Diagonalization and computational efficiency of circulant matrices}
Circulant matrices hold a special place in numerical analysis because of their strong relationship with the Discrete Fourier Transform (DFT) and the Fast Fourier Transform (FFT). The DFT is an operation that maps a vector $\bm{x} \in \mathbb{C}^n$ to its product with the Fourier matrix $F_n$, i.e., $F_n \bm{x}$. The FFT is a highly efficient algorithm that computes the DFT in only $O(n \log n)$ operations, which is a substantial improvement over the $O(n^2)$ operations required for a standard matrix-vector multiplication, as discussed in \cite{vanloan1992}. The key feature that enables this efficiency is the fact that every circulant matrix can be diagonalized by the Fourier matrix. As a result, circulant matrices are particularly advantageous for both theoretical analysis and computational tasks, allowing for rapid matrix-vector products and simplified spectral analysis.

Let us begin from the definition of the Fourier matrix of order $n \in \mathbb{N}$:
\[
F_n := \frac{1}{\sqrt{n}}\left[ e^{-2\pi i \frac{rs}{n}} \right]_{r,s=0}^{n-1} \in M_n(\mathbb{C}).
\]
For convenience, this expression involves a slight abuse of notation, as the row and column indices start from $0$ rather than $1$. It is easy to verify that $F_n$ is both symmetric and unitary.

Extending the definition to multiple dimensions, the multi-level Fourier matrix for $\bm{n} \in \mathbb{N}^d$ is given by
\[
F_{\bm{n}} := F_{n_1} \otimes \cdots \otimes F_{n_d}.
\]
$F_{\bm{n}}$ inherits its symmetry and unitarity from its components $F_{n_i}$, $i = 1, \ldots, d$, by virtue of the properties of the Kronecker product.

For any multi-level circulant matrix, we have the following spectral decomposition.
\begin{theorem}\label{2.9}
For all $\bm{n} \in \mathbb{N}^d$,
\begin{equation}\label{2.3}
C_{\bm{n}} = \left[ a_{(\bm{i-j}) \bmod \bm{n}} \right]_{\bm{i,j=1}}^{\bm{n}} = F_{\bm{n}} \Lambda_{\bm{n}} F_{\bm{n}}^*, 
\end{equation}
where
\[
\Lambda_{\bm{n}} := \underset{\bm{j=0,\ldots,n-1}}{\operatorname*{diag}}\left( p\left( \frac{2\pi \bm{j}}{\bm{n}} \right) \right),
\]
and
\[
p(\bm{\theta}) := \sum_{\bm{k=0}}^{\bm{n-1}} a_{\bm{k}} e^{i \bm{k} \cdot \bm{\theta}}, \qquad \bm{k} \cdot \bm{\theta} = k_1 \theta_1 + \cdots + k_d \theta_d.
\]
\end{theorem}

The proof of this theorem relies on finding a diagonalization of the uni-level generator $Z_n$, which can be obtained explicitly by computing its eigenvalues and eigenvectors. The diagonalization for $C_n$ then follows directly from equation~(\ref{2.2}), and the extension to the multi-level case is achieved by exploiting the properties of the Kronecker product.

Theorem \ref{2.9} offers numerous significant discoveries. 
 All multi-level circulant matrices are normal, as $\Lambda_{\bm{n}}$ is a diagonal matrix with complex entries. Secondly, the eigenvalues of \( C_{\bm{n}} \) are clearly defined by 
 \[ \lambda_{\bm{j}} (C_{\bm{n}}) = \sum_{\bm{k=0}}^{\bm{n-1}} a_{\bm{k}} e^{i \bm{k} \cdot \bm{\theta}_{\bm{j}}}, \quad \bm{\theta_j} := \frac{2\pi \bm{j}}{\bm{n}}, \quad \bm{j} = \bm{0}, \ldots, \bm{n} - \bm{1}, \]
 indicating that they may be calculated by computing the discrete Fourier transform (DFT) of the initial column of $C_{\bm{n}}$.  The associated eigenvectors are only the columns of the Fourier matrix $F_{\bm{n}}$. Third, circulant matrices constitute an algebra and all matrices within this algebra are uniformly diagonalized by $F_{\bm{n}}$.

 Last but not least, many computations involving circulant matrices can be significantly accelerated by using the FFT.  The computational cost of the matrix-vector product $C_{\bm{n}} \bm{x}$, where $\bm{x} \in \mathbb{C}^{N(\bm{n})}$, is reduced from the conventional $O(N(\bm{n})^2)$ operations necessary for dense matrices to only $O(N(\bm{n}) \log\\ N(\bm{n}))$ operations, by using the following approach:
\begin{enumerate}

 \item  Compute \( \bm{y} = F_{\bm{n}}^* \bm{x} \) via the FFT, in \( O(N(\bm{n}) \log N(\bm{n})) \) flops; 
 \item  Compute \( \bm{z} = \Lambda_{\bm{n}} \bm{y} \) in \( O(N(\bm{n})) \) flops; 
 \item  Compute \( F_{\bm{n}} \bm{z} \) via the FFT, in \( O(N(\bm{n}) \log N(\bm{n})) \) flops.
 \end{enumerate}
 In a similar manner, $C_{\bm{n}}^{-1} \bm{x}$ may be calculated as $F_{\bm{n}} \Lambda_{\bm{n}}^{-1} F_{\bm{n}}^* \bm{x}$.  This methodology is applicable to other fundamental matrix operations, including multiplication, inversion, solving linear problems, and computing the spectrum, rendering circulant matrices remarkably efficient and potent in real applications.
 \begin{remark}
This algorithm can be extended to accommodate general multi-level Toeplitz matrices through the use of an immersion (or embedding) technique. The central idea is to embed the original Toeplitz matrix into a larger circulant matrix, and then perform the matrix-vector multiplication using the circulant structure as described above. The required result is then obtained by appropriately extracting the relevant portion from the output. By adopting this strategy, the computational complexity of the Toeplitz matrix-vector product is also reduced to $O(N(\bm{n}) \log N(\bm{n}))$.
\end{remark}
\begin{remark}
Finally, it is important to note that each step of the process is fully parallelizable. Therefore, in a parallel model of computation, the theoretical computational cost of all the aforementioned algorithms can be further reduced to $O(\log N(\bm{n}))$.
\end{remark}
\section{\texorpdfstring{$\omega$-circulant matrices}{omega-circulant matrices}}\label{omega cir mat}
By introducing a nonzero parameter $\omega \in \mathbb{C} \setminus \{0\}$, we obtain a generalization of uni-level circulant matrices known as \emph{$\omega$-circulant matrices}. Unlike the classical circulant case, this concept is typically limited to the uni-level setting rather than being extended to multiple levels. The structure of a $\omega$-circulant matrix is specified by the following rule:
\[
\left[C^{\omega}_n\right]_{i,j} :=
\begin{cases}
a_{(i-j) \bmod n}      & \text{if } i > j, \\
\omega\, a_{(i-j) \bmod n} & \text{if } i \leq j,
\end{cases}
\qquad i, j = 1, \ldots, n.
\]
where $a_k \in \mathbb{C}$ for $k = 0, \ldots, n-1$. Here, the parameter $\omega$ acts as a scaling factor for certain matrix entries, setting these matrices apart from the standard circulant case.

Essentially, all entries above and on the main diagonal are multiplied by the factor $\omega$, resulting in the following structure:
\[
C_n^\omega =
\begin{bmatrix}
a_0 & \omega a_{n-1} & \omega a_{n-2} & \cdots & \cdots & \omega a_{1} \\
a_1 & a_0 &\omega a_{n-1} & \omega a_{n-2}&  & \vdots \\
a_2 & a_1 & \ddots &\ddots & \ddots & \vdots \\
\vdots & a_2&\ddots&\ddots&\ddots& \omega a_{n-2}\\
\vdots &  & \ddots & \ddots & \ddots & \omega a_{n-1} \\
a_{n-1} & \cdots & \cdots & a_{2} & a_1&a_{0}
\end{bmatrix}
\in M_n(\mathbb{C}).
\]
Note that $C_n^{\omega}$ still retains a Toeplitz structure, and that for $\omega = 1$ it reduces to the standard circulant matrix. To generalize equation (\ref{2.2}), we define the $\omega$-circulant generator as
\[
\left[ Z_n^{\omega} \right]_{i,j} :=
\begin{cases}
\delta_{(i-j-1) \bmod n}      & \text{if } i > j, \\
\omega \, \delta_{(i-j-1) \bmod n} & \text{if } i \leq j,
\end{cases}
\qquad i, j = 1, \ldots, n.
\]
Equivalently,
\[
\left[ Z_n^{\omega} \right]_{i,j} =
\begin{cases}
1     & \text{if } i - j = 1, \\
\omega & \text{if } i = 1 \wedge j = n, \\
0     & \text{otherwise},
\end{cases}
\qquad i, j = 1, \ldots, n.
\]
In other words, $Z_n^{\omega}$ has all ones on the first superdiagonal and $\omega$ in the top-right corner:
\[
Z_n^{\omega} =
\begin{bmatrix}
0      &        &        &        &       &   \omega   \\
1      & \ddots &        &        &       &        \\
       & \ddots & \ddots &        &       &        \\
       &        & \ddots & \ddots &       &        \\
       &        &        & \ddots & \ddots&        \\
       &        &        &        &1      &   0    \\
\end{bmatrix}
\in M_n(\mathbb{C}).
\]
In this framework, $C_n^{\omega}$ can be expressed as a linear combination of powers of $Z_n^{\omega}$:
\[
C_n^{\omega} = \sum_{k=0}^{n-1} a_k (Z_n^{\omega})^k.
\]
A diagonalization formula, paralleling equation (\ref{2.3}), naturally extends to the case of $\omega$-circulant matrices:
\[
C_n^{\omega} = F_n^{\omega} \Lambda_n^{\omega} (F_n^{\omega})^{-1},
\]
where the diagonal matrix of eigenvalues is given by
\[
\Lambda_n^{\omega} := \underset{j=0,\ldots,n-1}{\operatorname*{diag}}\left( p\left( \frac{2\pi j}{n} \right) \right)
\]
with
\[
p(\theta) := \sum_{k=0}^{n-1} \omega^{\frac{k}{n}} a_k e^{ik\theta},
\]
and the adapted Fourier matrix is defined as
\[
F_n^{\omega} := D_n^{\omega} F_n, \qquad D_n^{\omega} := \underset{k=0,\ldots,n-1}{\operatorname*{diag}}\left( \omega^{-\frac{k}{n}} \right).
\]
It should be highlighted that, as in the standard case, the eigenvalues of $C_n^{\omega}$ can once again be computed by performing the product $F_n^{\omega}\bm{c}$ with the first column $\bm{c}$ of $C_n^{\omega}$.

Extending the efficient computation of matrix-vector products to $\omega$-circulant matrices, as well as applying the other algorithms described previously, is straightforward:
\begin{enumerate}
    \item Compute $\bm{y} = (D_n^{\omega} F_n)^{-1} \bm{x} = F_n^* (D_n^{\omega})^{-1} \bm{x}$ using the FFT, which requires $O(n \log n)$ flops;
    \item Compute $\bm{z} = \Lambda_n^{\omega} \bm{y}$, in $O(n)$ flops;
    \item Compute $D_n^{\omega} F_n \bm{z}$ via the FFT, again in $O(n \log n)$ flops.
\end{enumerate}
Altogether, this sequence of operations achieves a total computational complexity of $O(n \log n)$ when performed sequentially, or $O(\log n)$ steps in a parallel setting. However, the inclusion of the parameter $\omega$ introduces potential numerical stability issues. Since $F_n$ is unitary, the condition number of the transformation $D_n^{\omega} F_n$ is determined by that of $D_n^{\omega}$. In particular, as $n \to +\infty$, the condition number of $D_n^{\omega}$ tends to $\max\{|\omega|, |\omega|^{-1}\}$, diverging to $+\infty$ if $\omega$ approaches $0$ or becomes unbounded. 
\subsection{\texorpdfstring{ Block $\omega$-circulant matrices}{Block omega-circulant matrices}}
To prepare for practical applications involving the solution of linear systems, it is useful to extend both the definition and the diagonalization procedure for $\omega$-circulant matrices to the block setting. For $m, n \in \mathbb{N}$, we define an $m$-block $\omega$-circulant matrix $C_{n,m}^{\omega} \in M_{mn}(\mathbb{C})$ as follows:
\[
\left[ C_{n,m}^{\omega} \right]_{i,j} :=
\begin{cases}
A_{(i-j) \bmod n}         & \text{if } i > j, \\
\omega\, A_{(i-j) \bmod n} & \text{if } i \leq j,
\end{cases}
\qquad i, j = 1, \ldots, n,
\]
where $A_k \in M_{m}(\mathbb{C})$ for $k = 0, \ldots, n-1$ are $n$ general (not necessarily structured) matrices. Explicitly:

\[
C_{n,m}^\omega =
\begin{bmatrix}
A_0 & \omega A_{n-1} & \omega A_{n-2} & \cdots & \cdots & \omega A_{1} \\
A_1 & A_0 &\omega A_{n-1} & \omega A_{n-2}&  & \vdots \\
A_2 & A_1 & \ddots &\ddots & \ddots & \vdots \\
\vdots & A_2&\ddots&\ddots&\ddots& \omega A_{n-2}\\
\vdots &  & \ddots & \ddots & \ddots & \omega A_{n-1} \\
A_{n-1} & \cdots & \cdots & A_{2} & A_1& A_{0}
\end{bmatrix}.
\]
A concise representation for $C_{n,m}^{\omega}$ is given by
\[
C_{n,m}^{\omega} = \sum_{j=0}^{n-1} (Z_n^{\omega})^j \otimes A_j.
\]
Moreover, $C_{n,m}^{\omega}$ can be block-diagonalized as
\[
C_{n,m}^{\omega} = F_{n,m}^{\omega} \Lambda_{n,m}^{\omega} (F_{n,m}^{\omega})^{-1},
\]
where
\[
F_{n,m}^{\omega} := F_n^{\omega} \otimes I_m,
\]
and
\[
\Lambda_{n,m}^{\omega} :=  \underset{j=0,\ldots,n-1}{\operatorname*{diag}}\left( p\left( \frac{2\pi j}{n} \right) \right)
\]
with
\[
p(\theta) := \sum_{k=0}^{n-1} \omega^{\frac{k}{n}} A_k e^{ik\theta} \in M_m(\mathbb{C}).
\]
It is important to highlight that $C_{n,m}^{\omega}$ is not a two-level $\omega$-circulant matrix, since the blocks $A_k$ are arbitrary and may not possess any special structure, and the function $p(\theta)$ is a matrix-valued function of size $m \times m$, rather than a scalar. However, if $m$ is fixed, the matrix-vector multiplication involving $C_{n,m}^{\omega}$ can still be performed efficiently in $O(n \log n)$ flops for the sequential case, or $O(\log n)$ steps in a parallel setting, with the constant hidden in the complexity depending on $m$. If $m$ is allowed to grow and the individual matrices $A_k$ can each be diagonalized efficiently for instance, in $O(m \log m)$ operations, then the overall complexity for the matrix-vector product becomes $O(mn \log(mn))$ flops for sequential computation, or $O(\log(mn))$ steps in parallel, taking both block size and block count into account.
\section{\texorpdfstring{$\tau$}{tau} matrices}\label{tau-mat}
A uni-level $\tau$ matrix is defined as a square matrix $S_n \in M_n(\mathbb{R})$ whose entries $s_{i,j}$ satisfy the so-called cross-sum condition:
\[
s_{i-1,j} + s_{i+1,j} = s_{i,j-1} + s_{i,j+1}, \qquad i, j = 1, \ldots, n,
\]
where it is understood that $s_{n+1,j} = s_{i,n+1} = s_{0,j} = s_{i,0} = 0$. This condition ensures that $S_n$ is centrosymmetric, meaning it is symmetric with respect to the center of the matrix. In particular, $S_n$ exhibits both symmetry with respect to the main diagonal and persymmetry with respect to the main anti-diagonal. Furthermore, the elements in the first row or column of $S_n$ are sufficient to generate all other entries of the matrix. For further details and additional properties, see~\cite{BiniCapovani1983}.

Much like circulant matrices, there is a generator for the set of $\tau$ matrices, defined by
\[
[W_n]_{i,j} := \delta_{|i-j|-1} =
\begin{cases}
1 & \text{if } i - j = \pm 1, \\
0 & \text{otherwise},
\end{cases}
\qquad i, j = 1, \ldots, n,
\]
or, in explicit matrix form
\[
W_n =
\begin{bmatrix}
0      &     1   &        &        &       &       \\
1      & 0 &  1      &        &       &        \\
       & 1 & \ddots &   \ddots     &       &        \\
       &        & \ddots & \ddots & \ddots      &        \\
       &        &        & \ddots & \ddots& 1       \\
       &        &        &        &1      &   0    \\
\end{bmatrix}
\in M_n(\mathbb{R}).
\]
The $\tau$ algebra is defined as the set of matrices generated by $W_n$, meaning that every $\tau$ matrix can be represented as a linear combination of the powers $W_n^0 = I_n$, $W_n$, $W_n^2$, $\ldots$, $W_n^{n-1}$. As is characteristic of matrix algebras, this set is closed under inversion, a property that follows from the Hamilton–Cayley theorem~\cite{Bhatia1997}. Furthermore, for any $n, k \in \mathbb{N}$, the matrix $W_n^k$ is symmetric and is invertible if and only if $n$ is even.

This construction can be extended to the multi-level setting by defining the multi-level generators for $\bm{n} \in \mathbb{N}^d$ as
\begin{equation}
\begin{aligned}
W_{\bm{n}}^{(1)} &= W_{n_{1}} \otimes \left( \bigotimes_{l=2}^{d} I_{n_{l}} \right), \\[6pt]
W_{\bm{n}}^{(2)} &= I_{n_{1}} \otimes W_{n}^{(2)} \otimes \left( \bigotimes_{l=3}^{d} I_{n_{l}} \right), \\[6pt]
&\;\;\vdots \\[6pt]
W_{\bm{n}}^{(d)} &= \left( \bigotimes_{l=1}^{d-1} I_{n_{l}} \right) \otimes W_{n}^{(d)}.
\end{aligned}
\end{equation}

Matrices $S_{\bm{n}}$ that belong to the multi-level $\tau$ algebra are referred to as multi-level $\tau$ matrices.

\subsection{Spectral properties and the discrete sine transform}
Similar to the approach used for circulant matrices, it is also possible to diagonalize $\tau$ matrices by leveraging a closely related fast transform. The Discrete Sine Transform (DST) associates to each vector $\bm{x} \in \mathbb{C}^n$ the product $Q_n \bm{x}$, where the sine matrix $Q_n$ is defined as
\[
Q_n := \sqrt{\frac{2}{n + 1}} \left[ \sin \left( \frac{r s \pi}{n + 1} \right) \right]_{r,s = 0}^{n-1}.
\]
This transformation is intimately connected to the Fast Sine Transform, which enables the DST to be computed efficiently in $O(n \log n)$ real operations—roughly half the computational cost of the traditional Fast Fourier Transform (FFT)~\cite{vanloan1992}. The matrix $Q_n$ possesses several notable properties: it is real, symmetric, and orthogonal.

The construction extends naturally to the multi-level case: for $\bm{n} \in \mathbb{N}^d$, we define the multi-level sine matrix as
\[
Q_{\bm{n}} := Q_{n_1} \otimes \cdots \otimes Q_{n_d},
\]
and the $d$-level DST is given by the product $Q_n \bm{x}$, where $\bm{x} \in \mathbb{C}^{N(\bm{n})}$. By properties of the Kronecker product, $Q_{\bm{n}}$ remains real, symmetric, and orthogonal, so that $Q_{\bm{n}} = Q_{\bm{n}}^{\top} = Q_{\bm{n}}^{-1}$. This structure allows us to derive a spectral decomposition for multi-level $\tau$ matrices, analogous to what is obtained for circulant matrices.
\begin{theorem}
For all $\bm{n} \in \mathbb{N}^d$, it holds that
\[
S_{\bm{n}} = Q_{\bm{n}} \Lambda_{\bm{n}} Q_{\bm{n}},
\]
where the diagonal entries of $\Lambda_{\bm{n}}$ are given by the $d$-level Discrete Sine Transform (DST) of the first column (or, equivalently, the first row) of $S_{\bm{n}}$.
\end{theorem}
All observations made after Theorem~\ref{2.9} for circulant matrices carry over to this setting, with appropriate modifications. In particular, every multi-level $\tau$ matrix is real and symmetric, and its eigenvalues are explicitly determined by applying the DST to its first column. The associated eigenvectors are given by the columns of $Q_{\bm{n}}$. All key matrix operations such as multiplication, inversion, solving linear systems, and computing the spectrum can be executed with a computational cost of $O(N(\bm{n}) \log N(\bm{n}))$ sequentially, or $O(\log N(\bm{n}))$ steps in a parallel computation model.
\section{Hankel matrices}\label{hank matr}
Given $\bm{n} \in \mathbb{N}^d$, a matrix of the form
\[
H_{\bm{n}} := \left[ a_{\bm{i}+\bm{j}} \right]_{\bm{i},\bm{j}=\bm{1}}^{\bm{n}}
\in M_{N(\bm{n})}(\mathbb{C}),
\]
where the entry in position $(\bm{i}, \bm{j})$ depends only on the sum $\bm{i} + \bm{j}$, is known as a multi-level (or $d$-level) Hankel matrix.

In practice, such a matrix exhibits a mirrored Toeplitz structure. For example, in the uni-level case with $d=1$ and $\bm{n}=n$, a 1-level (or standard) Hankel matrix is characterized by entries that remain constant along each anti-diagonal:
\[
H_n =
\big[ a_{i+j} \big]_{i,j=1}^{n}
=
\begin{bmatrix}
a_2 & a_3 & a_4 & \cdots & a_{n+1} \\
a_3 & a_4 &  & \reflectbox{$\ddots$} & \vdots \\
a_4 &  & \reflectbox{$\ddots$} &  & a_{2n-2} \\
\vdots & \reflectbox{$\ddots$} &  & a_{2n-2} & a_{2n-1} \\
a_{n+1} & \cdots & a_{2n-2} & a_{2n-1} & a_{2n}
\end{bmatrix}
\in M_n(\mathbb{C}).
\]
Just like Toeplitz matrices, $H_n$ is completely determined by its $2n-1$ coefficients $a_k$ for $k=2,\ldots,2n$, or, equivalently, by its first row and last column.

For $d=2$ and $\bm{n} = (n_1, n_2)$, a 2-level Hankel matrix takes the form
\[
H_{\bm{n}} =
\big[ a_{\bm{i}+\bm{j}} \big]_{\bm{i,j=1}}^{\bm{n}}
=
\bigg[ \left[a_{(i_1+j_1,\, i_2+j_2)} \right]_{i_2,j_2=1}^{n_2}\bigg]_{i_1,j_1=1}^{n_1}
=
\big[ A_{i_1+j_1} \big]_{i_1,j_1=1}^{n_1}
\]
\[
=
\begin{bmatrix}
A_2 & A_3 & A_4 & \cdots & A_{n+1} \\
A_3 & A_4 &  & \reflectbox{$\ddots$} & \vdots \\
A_4 &  & \reflectbox{$\ddots$} &  & A_{2n-2} \\
\vdots & \reflectbox{$\ddots$} &  & A_{2n-2} & A_{2n-1} \\
A_{n+1} & \cdots & A_{2n-2} & A_{2n-1} & A_{2n}
\end{bmatrix}
\in M_{n_{1}n_{2}}(\mathbb{C}).
\]
where each block $A_k$ is itself a Hankel matrix,\\
\[A_k =
\big[ a_{(k,\, i_2+j_2)} \big]_{i_2,j_2=1}^{n_2}\]

\[
=
\begin{bmatrix}
a_{(k,2)} & a_{(k,3)} & a_{(k,4)} & \cdots & a_{(k,n+1)} \\
a_{(k,3)} & a_{(k,4)} &  & \reflectbox{$\ddots$} & \vdots \\
a_{(k,4)} &  & \reflectbox{$\ddots$} &  & a_{(k,2n-2)} \\
\vdots & \reflectbox{$\ddots$} &  & a_{(k,2n-2)} & a_{(k,2n-1)} \\
a_{(k,n+1)} & \cdots & a_{(k,2n-2)} & a_{(k,2n-1)} & a_{(k,2n)}
\end{bmatrix}
\in M_{n_2}(\mathbb{C}).
\]
The general $d$-level Hankel matrix can be regarded as a block Hankel matrix whose blocks themselves possess $(d-1)$-level Hankel structure. Such a matrix is entirely determined by the coefficients in its first row and last column.

To obtain a compact representation of $H_{\bm{n}}$, we introduce, for $n \in \mathbb{N}$ and $k \in \mathbb{Z}$, the matrix
\[
\big[ K_n^{(k)} \big]_{i,j} := \delta_{i+j - k} =
\begin{cases}
1 & \text{if } i + j = k, \\
0 & \text{otherwise},
\end{cases}
\qquad i, j = 1, \ldots, n,
\]
which places ones on the $(k-1)$-th anti-diagonal and zeros elsewhere.

The multi-level generalization for $\bm{n} \in \mathbb{N}^d$ and $\bm{k} \in \mathbb{Z}^d$ is given by
\[
K_{\bm{n}}^{(\bm{k})} := K_{n_1}^{(k_1)} \otimes K_{n_2}^{(k_2)} \otimes \cdots \otimes K_{n_d}^{(k_d)}.
\]
This leads to the following compact expression for $H_{\bm{n}}$:

\begin{equation}\label{2.4}
H_{\bm{n}} = \left[ a_{\bm{i}+\bm{j}} \right]_{\bm{i},\bm{j}=\bm{1}}^{\bm{n}}
= \sum_{\bm{k}=\bm{2}}^{2\bm{n}} a_{\bm{k}} K_{\bm{n}}^{(\bm{k})}. 
\end{equation}
As with the Toeplitz and circulant cases, this representation can be verified directly by checking individual matrix entries.

\chapter{GLT structures in matrix-sequences}\label{GLT}
We are now ready to discuss the main ideas of GLT matrix-sequences. This is an important theory for understanding how structured matrices behave, especially those that come from solving differential equations. 

In this thesis, we consider the general case where both $r, d \geq 1$. While the case of scalar-valued GLT symbols is presented in detail in \cite{garoni2017, garoni2018}, here we focus on matrix-valued GLT symbols, following the more general framework extensively developed in \cite{barbarino2020block1d, barbarino2020blockmulti}. The existence of two books and two extensive research papers on this subject demonstrates the depth and interest of this field of study. For any choice of positive integers $d$ (which tells how many levels of structure we have) and $r$ (which is the size of the basic matrix block), there is a class of GLT matrix-sequences with this $d$-level, $r \times r$ block structure. If we fix $d$ and $r$, the symbols in this theory are measurable functions $\kappa : [0, 1]^d \times [-\pi, \pi]^d \rightarrow M_r(\mathbb{C})$. There are still many open questions, and researchers have also studied other types of GLT sequences, such as the reduced and rectangular cases \cite{Barbarino2022, Barbarino2022RectangularGLT}.

Rather than exploring the full constructive development of the multilevel GLT class, we will focus on the principal operative properties that completely characterize it and are most relevant for practical applications. These properties, commonly known as “GLT axioms”, are presented and labeled as \textbf{GLT No}, in Section~\ref{ast algebra}.

\section{The GLT class}\label{GLT Clas}
Let us first review some basic aspects to better understand GLT sequences. The theory of LT sequences provides a powerful framework for analyzing the asymptotic singular value and eigenvalue distributions of structured matrix-sequences that arise in numerical analysis. In the context of partial differential equation (PDE) discretization, the matrices resulting from numerical approximations often fail to retain the classical Toeplitz structure due to the presence of variable coefficients. Traditionally, Toeplitz matrices have played a central role in problems involving translation-invariant differential operators, where constant coefficients ensure a globally repeating structure. Instead, the resulting matrices exhibit a Toeplitz-like structure that varies smoothly along the diagonals as the matrix size increases, leading to the broader class of LT sequences \cite{Tilli1998, Tyrtyshnikov1996}.\\
A classical Toeplitz matrix generated by a function $f: [-\pi, \pi]^d \to \mathbb{C}^{r \times r}$ belonging to $L^1([-\pi, \pi]^d)$, where the constant diagonal entries correspond to the Fourier coefficients of $f$. However, for differential operators with non-constant coefficients, the resulting matrices do not maintain this global Toeplitz structure but instead belong to the LT class. Specifically, an LT matrix $A_{\bm{n}}$ is obtained by modulating the entries of a Toeplitz matrix using an almost everywhere (a.e.) continuous function $a(x): [0,1]^d \to \mathbb{C}^{r \times r}$, scaling the entries along the same diagonal
$ a\left(\frac{\bm{i}}{\bm{n}}\right)$ for $\bm{i}= \bm{1},\dots,\bm{n}.$ As $\bm{n} \to \infty$, the differences between consecutive diagonal elements vanish due to continuity, making the structure asymptotically Toeplitz-like. More precisely, this structure can be represented as the product of a Toeplitz matrix and a diagonal matrix, where the diagonal entries correspond to evaluations of the function $a(x)$, such that $A_{\bm{n}} = D_{\bm{n}}(a)T_{\bm{n}},$
where $D_{\bm{n}}(a)$ is a diagonal matrix whose entries are given by $D_{\bm{n}}(a) = \mathrm{diag}\left(a\left(\frac{\bm{i}}{\bm{n}}\right)\right)_{\bm{i}=\bm{1}}^{\bm{n}}.$
This formulation allows LT sequences to accurately capture local variations in variable-coefficient differential operators while preserving the structural properties of Toeplitz matrices. Formally, an LT sequence $\{A_{\bm{n}}\}_{\bm{n}}$ can be expressed as the tensor product of the weight function $a(x)$ and the generating function $f(\theta)$, leading to the LT symbol $
a \otimes f: [0,1]^d \times [-\pi, \pi]^d \to \mathbb{C}^{r \times r}$.

Building on this concept, a GLT matrix-sequence is defined as the limit, in the sense of the approximating class of sequences (a.c.s.), as described in Definition~\ref{def:acs}, of a finite sum of LT sequences. The construction of GLT sequences was formalized through a.c.s. limits in \cite{SerraCapizzano2003} and refined to a \textbf{d-level $*$-algebra of structured matrix-sequences} satisfying specific axioms, listed in \cite[Chapter~6, pp.~118--120]{garoni2018}. The theory of GLT sequences provides a fundamental framework for analyzing structured matrix-sequences arising from discretizations of differential equations. 
\section{Matrix-sequences with explicit or hidden (asymptotic) structure}
In this section, we introduce three types of matrix structures that serve as the fundamental building blocks of the GLT $*$-algebras. Specifically, for any positive integers $d$ and $r$, we consider the set of $d$-level $r$-block GLT matrix-sequences. This set forms a $*$-algebra of matrix-sequences, which is both maximal and isometrically equivalent to the maximal $*$-algebra of $2d$-variate $r \times r$ matrix-valued measurable functions (with respect to the Lebesgue measure) that are naturally defined over $[0,1]^d \times [-\pi,\pi]^d$; see~\cite{barbarino2020block1d, barbarino2020blockmulti, garoni2017, garoni2018} and references therein.

The reduced version of this structure plays a crucial role in approximating integro-differential operators, including their fractional versions, particularly when defined over general (non-Cartesian) domains. This concept was initially introduced in~\cite{SerraCapizzano2003,SerraCapizzano2006} and later extensively developed in~\cite{Barbarino2022}, where GLT symbols are again defined as measurable functions over $\Omega \times [-\pi, \pi]^d$, with $\Omega$ being Peano-Jordan measurable and contained within $[0,1]^d$. Additionally, the reduced versions also form maximal $*$-algebras that are isometrically equivalent to their corresponding maximal $*$-algebras of measurable functions.

These GLT $*$-algebras provide a rich framework of hidden (asymptotic) structures, built upon three fundamental classes of explicit algebraic structures: $d$-level $r$-block Toeplitz matrix-sequences and sampling diagonal matrix-sequences (discussed in Sections~\ref{blcktop} and~\ref{blckdiag}), along with asymptotic structures described by zero-distributed matrix-sequences (see Section~\ref{zero dist}). Notably, the latter class serves an analogous role to compact operators in relation to bounded linear operators, forming a two-sided ideal of matrix-sequences within any of the GLT $*$-algebras.

\section{The Generators of the GLT algebras}
For any given $r,d\geq1$, there exists the $*$-algebra of $d$-level $r$-block GLT matrix-sequences. In the following, we describe the generators.
\subsection{Zero-distributed sequences}\label{zero dist}

Zero-distributed sequences are defined as matrix-sequences $\{A_n\}_n$ such that $\{A_n\}_n \sim_\sigma 0$.
Note that, for any $r \geq 1$, $\{A_n\}_n \sim_\sigma 0$ is equivalent to $\{A_n\}_n \sim_\sigma O_r$, where $O_r$ is the $r \times r$ zero matrix.
The following theorem, taken from~\cite{SerraCapizzano2001,garoni2017}, provides a useful characterization for detecting this type of sequence.

\begin{theorem}\label{th 3.1}
Let $\{A_n\}_n$ be a matrix-sequence, with $A_n$ of size $d_n$, and let $p \in [1,\infty]$, with $\|X\|_p$ being the Schatten $p$-norm of $X$, that is, the $\ell^p$ norm of the vector of its singular values. Let $\|\cdot\| = \|\cdot\|_\infty$ be the spectral norm. With the natural convention $1/\infty = 0$, we have:
\begin{itemize}
    \item $\{A_n\}_n \sim_\sigma 0$ if and only if $A_n = R_n + N_n$ with $\mathrm{rank}(R_n)/d_n \to 0$ and $\|N_n\| \to 0$ as $n \to \infty$;
    \item $\{A_n\}_n \sim_\sigma 0$ if there exists $p \in [1,\infty]$ such that
    \[
        \frac{\|A_n\|_p}{(d_n)^{1/p}} \to 0 \quad \text{as } n \to \infty.
    \]
\end{itemize}
\end{theorem}

As in Section \ref{acs}, the same definition can be given and the corresponding result (with obvious changes) holds when the involved matrix-sequences show a multilevel structure. In that case, $n$ is replaced by $\bm{n}$ uniformly in $A_n$, $N_n$, $R_n$, $d_n$.

Further results connecting Schatten $p$-norms, unitarily invariant norms (see~\cite{bhatia2007}), variational characterization, and Toeplitz/Hankel structures can be found in~\cite{SerraCapizzanoTilli-LPO}.\\
\begin{remark}
The multilevel block Toeplitz matrix defined below is a class of GLT matrix-sequences with a $d$-level, $r \times r$ block structure. This generalization extends the Toeplitz and BTTB cases discussed in Section \ref{TM} and naturally includes the possibility of block entries at each level.
\end{remark}
\subsection{Multilevel block Toeplitz matrices}\label{blcktop}

Given $\bm{n} \in \mathbb{N}^d$, a matrix of the form
\[
[A_{\bm{i}-\bm{j}}]_{\bm{i},\bm{j}=1}^{\bm{n}} \in M_{N(\bm{n})r} \mathbb{(C)},
\]
with blocks $A_{\bm{k}} \in M_{r \times r}\mathbb{(C)}$, $\bm{k} \in \{- (\bm{n} - 1), \ldots, \bm{n} - 1\}$, is called a \textit{multilevel block Toeplitz matrix}, or, more precisely, a $d$-level $r$-block Toeplitz matrix.

Given a matrix-valued function $f : [ -\pi, \pi ]^d \to \mathbb{C}^{r \times r}$ belonging to $L^1( [ -\pi, \pi ]^d )$, the $\bm{n}$-th Toeplitz matrix associated with $f$ is defined as
\[
T_{\bm{n}}(f) := [ \hat{f}_{\bm{i}-\bm{j}} ]_{\bm{i},\bm{j}=1}^{\bm{n}} \in M_{N(\bm{n})r}\mathbb{(C)},
\]
where
\[
\hat{f}_{\bm{k}} = \frac{1}{(2\pi)^d} \int_{ [ -\pi, \pi ]^d } f(\theta) e^{- i (\bm{k} \cdot \bm{\theta})} d\bm{\theta} \in \mathbb{C}^{r \times r}, \quad \bm{k} \in \mathbb{Z}^d,
\]
are the Fourier coefficients of $f$.

The family $\{ T_{\bm{n}}(f) \}_{ \bm{n} \in \mathbb{N}^d }$ is called the family of (multilevel block) Toeplitz matrices associated with $f$, which is referred to as the \emph{generating function}.
\subsection{Block diagonal sampling matrices}\label{blckdiag}
Given $d \geq 1$, $\bm{n} \in \mathbb{N}^d$ and a function $a : [0,1]^d \to \mathbb{C}^{r \times r}$, we define the multilevel block diagonal sampling matrix $D_{\bm{n}}(a)$ as the block diagonal matrix
\[
D_n(a) = \underset{\bm{i}=1,\ldots,\bm{n}}{\operatorname*{diag}}\left( a\left(\frac{\bm{i}}{\bm{n}}\right) \right) \in M_{N(\bm{n})r}\mathbb{(C)}.
\]
Essentially, $D_{\bm{n}}(a)$ is a diagonal matrix whose entries are the evaluations of $a(x)$ at the points of a uniform mesh over the domain $[0, 1]^d$, with the points arranged in lexicographical order. 
\section{The \texorpdfstring{$\ast$}{*}-algebra of \texorpdfstring{$d$}{d}-level \texorpdfstring{$r$}{r}-block GLT matrix-sequences}\label{ast algebra}
Let $r \geq 1$ be a fixed integer. A multilevel $r$-block GLT sequence, or simply a GLT sequence if $r$ does not need to be specified, is a special multilevel $r$-block matrix-sequence equipped with a measurable function $\kappa : [0, 1]^d \times [-\pi, \pi]^d \to \mathbb{C}^{r \times r}$, $d \geq 1$, called the \emph{symbol}. It is formally expressed as
\[
\{A_n\}_n \sim_{\mathrm{GLT}} \kappa.
\]
We write to denote that $\{A_n\}_n$ is a GLT sequence with symbol $\kappa$. A fundamental property of the GLT symbol is that it is unique: if two symbols are associated with the same matrix-sequence, then they must coincide almost everywhere as a direct consequence of Axiom \textbf{GLT 3}, second item.
\begin{proposition}
If $\{A_{\bm{n}}\}_{\bm{n}} \sim_{\mathrm{GLT}} \kappa$ and $\{A_{\bm{n}}\}_{\bm{n}} \sim_{\mathrm{GLT}} \xi$, then $
\kappa=\xi$ \text{a.e.}
\end{proposition}
The symbol plays a central role in determining the singular value distribution and, in the Hermitian setting, the eigenvalue distribution of GLT matrix-sequences.

It can be proven that the set of multilevel block GLT sequences is the $\ast$-algebra generated by the three classes of sequences defined in Sections \ref{zero dist}, \ref{blcktop}, \ref{blckdiag}: zero-distributed sequences, multilevel block Toeplitz sequences, and block diagonal sampling matrix-sequences. The GLT class satisfies several algebraic and topological properties, which are discussed in detail in \cite{barbarino2020block1d, barbarino2020blockmulti, garoni2017, garoni2018}. Here, we focus on the main operative properties listed below, which provide a complete characterization of GLT sequences, equivalent to the full constructive definition.

\subsection*{GLT axioms}
\begin{itemize}
    \item \textbf{GLT 1.} In the setting of Definition~\ref{99}, with \( t = 2d \) and $D = [0, 1]^d \times [-\pi, \pi]^d$,  if $\{A_{\bm{n}}\}_{\bm{n}}\sim_{\mathrm{GLT}} \kappa$, then $\{A_{\bm{n}}\}_{\bm{n}} \sim_{\sigma} \kappa$. Moreover, if all the matrices $A_{\bm{n}}$ are Hermitian, then $\{A_{\bm{n}}\}_{\bm{n}} \sim_{\lambda} \kappa,$ again in the sense of Definition~\ref{99}, with \( t = 2d \).
    \item \textbf{GLT 2.} We have
    \begin{itemize}
        \item $\{T_{\bm{n}}(f)\}_{\bm{n}} \sim_{\mathrm{GLT}} \kappa(\bm{x}, \bm{\theta}) = f(\bm{\theta)}$ if $f : [-\pi, \pi]^d \to \mathbb{C}^{r \times r}$ is in $L^1([- \pi, \pi]^d)$;
        \item $\{D_{\bm{n}}(a)\}_{\bm{n}} \sim_{\mathrm{GLT}} \kappa(\bm{x}, \bm{\theta}) = a(\bm{x})$ if $a : [0, 1]^d \to \mathbb{C}^{r \times r}$ is Riemann-integrable;
        \item $\{Z_{\bm{n}}\}_{\bm{n}} \sim_{\mathrm{GLT}} \kappa(\bm{x}, \bm{\theta}) = O_r$ if and only if $\{Z_{\bm{n}}\}_{\bm{n}} \sim_\sigma 0$.
    \end{itemize}
    \end{itemize}
One of the most fundamental properties of the GLT class is that it forms a $\ast$-algebra, meaning it is closed under the natural operations of conjugate transposition, linear combination, and product. While the previous axiom shows that multilevel Toeplitz, diagonal sampling, and zero-distributed sequences are all included in the GLT class, the set GLT of sequences is, in fact, precisely the $\ast$-algebra generated by the union of these three types of sequences. This means that any GLT sequence can be constructed from these basic building blocks through the $\ast$-algebra operations. Furthermore, the map that associates each GLT sequence with its corresponding symbol preserves the operations; in other words, it is an algebra homomorphism.

    \begin{itemize}
    \item \textbf{GLT 3.} If $\{A_{\bm{n}}\}_{\bm{n}} \sim_{\mathrm{GLT}} \kappa$ and $\{B_{\bm{n}}\}_{\bm{n}} \sim_{\mathrm{GLT}} \xi$, then:
    \begin{itemize}
        \item $\{A_{\bm{n}}^*\}_{\bm{n}} \sim_{\mathrm{GLT}} \kappa^*$;
        \item $\{\alpha A_{\bm{n}} + \beta B_{\bm{n}}\}_{\bm{n}} \sim_{\mathrm{GLT}} \alpha \kappa + \beta \xi$ for all $\alpha, \beta \in \mathbb{C}$;
        \item $\{A_{\bm{n}} B_{\bm{n}}\}_{\bm{n}} \sim_{\mathrm{GLT}} \kappa \xi$;
        \item $\{A_{\bm{n}}^\dagger\}_{\bm{n}} \sim_{\mathrm{GLT}} \kappa^{-1}$, provided that $\kappa$ is invertible almost everywhere.
    \end{itemize}
    \item \textbf{GLT 4.} \(\{A_{\bm{n}}\}_{\bm{n}} \sim_{\mathrm{GLT}} \kappa\) if and only if there exist \(\{B_{{\bm{n}},j}\}_{\bm{n}} \sim_{\text{GLT}} \kappa_j\) such that \(\{\{B_{{\bm{n}},j}\}_{\bm{n}}\}_j \xrightarrow{\text{a.c.s.wrt }j} \{A_{\bm{n}}\}_{\bm{n}}\) and \(\kappa_j \rightarrow \kappa\) in measure.
    \item \textbf{GLT 5.} If $\{A_{\bm{n}}\}_n \sim_{\mathrm{GLT}} \kappa$ and $A_{\bm{n}} = X_{\bm{n}} + Y_{\bm{n}}$, where
\begin{itemize}
    \item every $X_{\bm{n}}$ is Hermitian,
    \item $ ||X_{\bm{n}}||, \, ||Y_{\bm{{n}}}|| \leq C$ for some constant $C$ independent of $n$,
    \item $N(\bm{n})^{-1} \|Y_{\bm{n}}\|_1 \to 0$,
\end{itemize}
then $\{A_{\bm{n}}\}_n \sim_\lambda \kappa$.
   \item \textbf{GLT 6.}  If $\{A_{\bm{n}}\}_n \sim_{\mathrm{GLT}} \kappa$ and each $A_{\bm{n}}$ is Hermitian, then $\{f(A_{\bm{n}})\}_n \sim_{\mathrm{GLT}} f(\kappa)$ for every continuous function $f : \mathbb{C} \to \mathbb{C}$.

\end{itemize}
Note that, by \textbf{GLT 1}, it is always possible to obtain the singular value distribution from the GLT symbol, while the eigenvalue distribution can only be deduced either if the involved
matrices are Hermitian or the related matrix-sequence is quasi-Hermitian in the sense of \textbf{GLT 5}.

\chapter{Geometric means in matrix-sequences with hidden structures}\label{GM GL1}
Having established the theoretical foundations of matrix-sequences and their asymptotic properties, this chapter presents the results published in \cite{ahmad2025matrix}. We investigate the spectral distribution of the geometric mean of two or more matrix-sequences formed by HPD matrices, focusing on cases where all sequences belong to the same GLT $\ast$-algebra. 
The geometric mean of two positive numbers, $a$ and $b$, is the square root of their product, $\sqrt{ab}$. We extend this concept to Hermitian positive definite (HPD) matrices by employing the Ando–Li–Mathias (ALM) geometric mean, defined as
\begin{equation}\label{1}
    G(A, B) = A^{1/2} \left( A^{-1/2} B A^{-1/2} \right)^{1/2} A^{1/2},
\end{equation}
where $A$ and $B$ are HPD. Here, $(\cdot)^{1/2}$ denotes the unique Hermitian positive definite square root of a Hermitian positive definite matrix. This approach maintains key properties such as symmetry, invariance under congruence transformations, and consistency with the classical scalar mean in the commutative case.

When analyzing the geometric mean of GLT matrix-sequences, our results show that for the scalar unilevel case (\( r = d = 1 \)), if \(\{A_{\bm{n}}\}_{\bm{n}}\) and \(\{B_{\bm{n}}\}_{\bm{n}}\) are sequences of Hermitian positive definite matrices with GLT symbols \(\kappa\) and \(\xi\), then the sequence formed by their geometric mean, \(\{G(A_{\bm{n}}, B_{\bm{n}})\}_{\bm{n}}\), is itself a scalar unilevel GLT matrix-sequence whose symbol is \((\kappa \xi)^{1/2}\), provided at least one of the symbols is nonzero almost everywhere. We further demonstrate that this result extends to the multilevel scalar case (\( r = 1,\, d > 1 \)), where the same relation holds for the GLT symbols. In the general block multilevel case (\( r,\, d > 1 \)), if at least one of the minimal eigenvalues of the GLT symbols \(\kappa\) or \(\xi\) is nonzero almost everywhere, then the sequence of geometric means remains a GLT matrix-sequence, with its symbol given by the matrix geometric mean \(G(\kappa, \xi)\), defined as
\[
G(\kappa, \xi) = \kappa^{1/2} \left( \kappa^{-1/2} \xi \kappa^{-1/2} \right)^{1/2} \kappa^{1/2}.
\]

For the case of more than two matrices, the geometric mean is generalized by the Karcher mean. Specifically, the Karcher mean of $k$ Hermitian positive definite matrices $A_1, \ldots, A_k$ is defined as the unique positive definite solution $X$ of the matrix equation $\sum_{i=1}^k \log(A_i^{-1} X) = 0$.

In the setting of GLT matrix-sequences, suppose we have $k$ sequences of Hermitian positive definite matrices, each associated with a GLT symbol $\kappa_i$, which by construction are nonnegative definite almost everywhere. We conjecture that the sequence of Karcher means $
\left\{ G\left(A^{(1)}_{\bm{n}}, \ldots, A^{(k)}_{\bm{n}} \right) \right\}_{\bm{n}}$ forms a new GLT matrix-sequence, with its symbol being the geometric mean of the individual symbols. Specifically, when all symbols commute, the resulting symbol is $(\kappa_1 \cdots \kappa_k)^{1/k}$, and in the noncommutative case, the symbol is the Karcher mean $G(\kappa_1, \ldots, \kappa_k)$ defined through the same matrix equation at the symbol level. In this respect, we notice that two main difficulties are present when dealing with Karcher means: The first is the lack of an explicit formula, which prevents the use of the GLT axioms in a direct way; The second relies on the pointwise convergence results, which do not allow us to employ powerful tools such as those in Theorem~\ref{th 3.1}.

To verify these theoretical results, we present a range of numerical experiments, including examples where the input matrix-sequences are of Toeplitz type $r$-block, $d$-level GLT matrix-sequences arising from the discretizations of differential operators by finite differences, finite elements, isogeometric analysis, and Galerkin methods. We address both one- and two-dimensional settings and analyze the asymptotic behavior of minimal eigenvalues and conditioning of the geometric mean sequences.

\section{Geometric mean of GLT matrix-sequences}\label{21}
This section discusses the geometric mean of positive definite matrices, starting with the well-established case of two matrices and then considering  multiple matrices. In particular we give distribution results in the case of GLT matrix-sequences.
\subsection{Means of two matrices}
The geometric mean of two positive numbers \( a \) and \( b \) is simply \( \sqrt{ab} \), a fact well-known from basic arithmetic. However, extending this concept to HPD matrices introduces a number of challenges, as matrix multiplication is not commutative. The question of how to define a geometric mean for matrices in a way that preserves key properties such as congruence invariance, consistency with scalars, and symmetry was solved by ALM \cite{Ando2004}. Their work presented an axiomatic approach to defining the geometric mean of two HPD matrices previously defined in Equation (\ref{1}).
ALM formalized the geometric mean for matrices by establishing a set of ten essential properties, known as the ALM axioms. These axioms ensure the geometric mean behaves appropriately in the matrix setting. Here we report three key properties, even if the first one is a consequence of the second and of the third via the choice of a proper congruence. \\
1. \textit{Permutation invariance}: \( G(A, B) = G(B, A) \) for all \( A, B \in \mathcal{P}_n \).\\
2. \textit{Congruence invariance}: \( G(M^* A M, M^* B M) = M^* G(A, B) M \) for all \( A, B \in \mathcal{P}_n \) and all invertible matrices \( M \in \mathbb{C}^{n \times n} \).\\
3. \textit{Consistency with scalars}: \( G(A, B) = (AB)^{1/2} \) for all commuting \( A, B \in \mathcal{P}_n \) (note that \( AB \in \mathcal{P}_n \) for all commuting \( A, B \in \mathcal{P}_n \) because \( (AB)^* = B^* A^* = B A = AB \) and \( AB \) is similar to the HPD matrix \( A^{1/2} B A^{1/2} = A^{-1/2} (AB) A^{1/2} \)).\\
When considering sequences of matrices, particularly in the framework of GLT sequences, the geometric mean operation is well-preserved under the structure of GLT sequences. If $r=d=1$ and we consider two scalar unilevel GLT matrix-sequences, that is, 
\(\{A_n\}_n \sim_{\text{GLT}}\kappa\) and \(\{B_n\}_n \sim_{\text{GLT}} \xi\), where \(A_n, B_n \in \mathcal{P}_n\) are HPD matrices for every $n$, the matrix-sequence of their geometric mean \(\{G(A_n, B_n)\}_n\) also forms a scalar unilevel GLT matrix-sequence. The symbol of the resulting sequence is the geometric mean of the individual symbols \(\kappa\) and \(\xi\).

\begin{theorem}[\!\!\cite{garoni2017}, Theorem 10.2]\label{4: th:two - r=d=1 GM1}
Let $r=d=1$. Suppose \(\{A_n\}_n \sim_{\mathrm{GLT}} \kappa\) and \(\{B_n\}_n \sim_{\mathrm{GLT}} \xi\), where \(A_n, B_n \in \mathcal{P}_n\) for every $n$. Assume that at least one between \(\kappa\) and \(\xi\) is nonzero almost everywhere. Then
\begin{equation}\label{r=d=1 GLT}
\{G(A_n, B_n)\}_n \sim_{\mathrm{GLT}}(\kappa \xi)^{1/2},
\end{equation}
and
\begin{equation}\label{r=d=1 distributions}
\{G(A_n, B_n)\}_n \sim_{\sigma, \lambda} (\kappa \xi)^{1/2}.
\end{equation}
\end{theorem}

The previous result is easily extended to the case of matrix-sequence resulting from the geometric mean of multilevel GLT matrix-sequences, thanks to powerful $*$-algebra structures of considered spaces described in Section~\ref{ast algebra}. Indeed the following two generalizations of Theorem \ref{4: th:two - r=d=1 GM1} hold.

\begin{theorem}\label{4: th:two - r=1,d general GM1}
Let $r=1$ and $d\ge 1$.  Suppose \(\{A_{\bm{n}}\}_{\bm{n}} \sim_{\mathrm{GLT}} \kappa\) and \(\{B_{\bm{n}}\}_{\bm{n}} \sim_{\mathrm{GLT}} \xi\), where \(A_{\bm{n}}, B_{\bm{n}} \in \mathcal{P}_{\nu(\bm{n})}\) for every multi-index $\bm{n}$. Assume that at least one between \(\kappa\) and \(\xi\) is nonzero almost everywhere. Then
\begin{equation}\label{r=1, d general GLT}
\{G(A_{\bm{n}}, B_{\bm{n}})\}_{\bm{n}} \sim_{\mathrm{GLT}}(\kappa \xi)^{1/2},
\end{equation}
and
\begin{equation}\label{r=1, d general distributions}
\{G(A_{\bm{n}}, B_{\bm{n}})\}_{\bm{n}} \sim_{\sigma, \lambda} (\kappa \xi)^{1/2}.
\end{equation}
\end{theorem}
{\bf Proof.} \ \ 
Since both $A_{\bm{n}}$ and $B_{\bm{n}}$ are positive definite for every multi-index $\bm{n}$, the matrix-sequence 
$\{G(A_{\bm{n}}, B_{\bm{n}})\}_{\bm{n}}$ is well defined according to formula (\ref{1}) since $A_{\bm{n}}^{-1/2}$ is well defined for every multi-index $\bm{n}$. According to the assumption, we start with the case where \(\kappa\) is nonzero almost everywhere. Hence the matrix-sequence $\{A_{\bm{n}}^{-1}\}_{\bm{n}}$ is a GLT matrix-sequence with GLT symbol $\kappa^{-1}$ by Axiom {\bf GLT 3}, part 3. Since the square root is continuous and well defined over positive definite matrices also the matrix-sequence  $\{A_{\bm{n}}^{-1/2}\}_{\bm{n}}$ is a GLT matrix-sequence with GLT symbol $\kappa^{-1/2}$ by virtue of Axiom {\bf GLT 6}. 

Now using two times \textbf{GLT 3}, part 2, we infer that $\{A_{\bm{n}}^{-1/2} B_{\bm{n}} A_{\bm{n}}^{-1/2}\}_{\bm{n}}$ is a GLT matrix-sequence with GLT symbol $\kappa^{-1} \xi $, where $X_{\bm{n}} = A_{\bm{n}}^{-1/2} B_{\bm{n}} A_{\bm{n}}^{-1/2}$ is positive definite because of the Sylvester inertia law. Hence, the square root of $X_{\bm{n}}$ is well defined and by exploiting again Axiom \textbf{GLT 6} we deduce that $\left\{\left(A_{\bm{n}}^{-1/2} B_{\bm{n}} A_{\bm{n}}^{-1/2}\right)^{1/2}\right\}_{\bm{n}}$ is a GLT matrix-sequence with GLT symbol $\left(\kappa^{-1} \xi \right)^{1/2}$. Finally, by exploiting Axiom \textbf{GLT 3}, part 3, we have $\{A_{\bm{n}}^{1/2}\}_{\bm{n}}$ is a GLT matrix-sequence with GLT symbol $\kappa^{1/2}$, and the application of \textbf{GLT 3}, part 2, two times leads to the desired conclusion
\[
\{G(A_{\bm{n}}, B_{\bm{n}})\}_{\bm{n}} \sim_{\mathrm{GLT}}(\kappa \xi)^{1/2},
\]
where the latter and Axiom \textbf{GLT 3} imply $\{G(A_{\bm{n}}, B_{\bm{n}})\}_{\bm{n}} \sim_{\sigma, \lambda} (\kappa \xi)^{1/2}$.

Finally the other case where \(\xi\) is nonzero almost everywhere has the very same proof since $G(\cdot, \cdot)$ is invariant under permutations and hence
\[
\resizebox{1.0\hsize}{!}{$
G(A, B) = A^{1/2} \left( A^{-1/2} B A^{-1/2} \right)^{1/2} A^{1/2} = G(B,A) = B^{1/2} \left( B^{-1/2} A B^{-1/2} \right)^{1/2} B^{1/2},
$}
\]
so that the the same steps can be repeated by exchanging $A_{\bm{n}}$ and $B_{\bm{n}}$.
\hfill $\bullet$

\begin{theorem} \label{4: th:two - r,d general GM1}
Let $r,d\ge 1$. 
Suppose \(\{A_{\bm{n}}\}_{\bm{n}} \sim_{\mathrm{GLT}} \kappa\) and \(\{B_{\bm{n}}\}_{\bm{n}} \sim_{\mathrm{GLT}} \xi\), where \(A_{\bm{n}}, B_{\bm{n}} \in \mathcal{P}_{\nu(\bm{n})}\) for every multi-index $\bm{n}$. Assume that at least one between the minimal eigenvalue of \(\kappa\) and the minimal eigenvalue of \(\xi\) is nonzero almost everywhere. Then
\begin{equation}\label{r,d general GLT GM1}
\{G(A_{\bm{n}}, B_{\bm{n}})\}_{\bm{n}} \sim_{\mathrm{GLT}} G(\kappa, \xi),
\end{equation}
and
\begin{equation}\label{r,d general distributions}
\{G(A_{\bm{n}}, B_{\bm{n}})\}_{\bm{n}} \sim_{\sigma, \lambda} G(\kappa,\xi).
\end{equation}
Furthermore $G(\kappa,\xi)=(\kappa \xi)^{1/2}$ whenever the GLT symbols  \(\kappa\) and \(\xi\) commute.
\end{theorem}
{\bf Proof.} \ \ 
The case of $r=1$ is already contained in Theorem \ref{4: th:two - r=1,d general GM1}, so we assume $r>1$ i.e. a true GLT block setting. The proof is in fact a repetition of that of the previous theorem with the only attention to GLT symbol part where the multiplication is noncommutative. \\
Since both $A_{\bm{n}}$ and $B_{\bm{n}}$ are positive definite for every multi-index $\bm{n}$, the matrix-sequence 
$\{G(A_{\bm{n}}, B_{\bm{n}})\}_{\bm{n}}$ is well defined according to formula (\ref{1}) since $A_{\bm{n}}^{-1/2}$ is well defined for every multi-index $\bm{n}$. According to the assumption, we start with the case where \(\kappa\) is invertible almost everywhere so that 
$\{A_{\bm{n}}^{-1}\}_{\bm{n}}\sim_{\mathrm{GLT}} \kappa^{-1}$ by Axiom {\bf GLT 3}, part 3, and  $\{A_{\bm{n}}^{-1/2}\}_{\bm{n}}\sim_{\mathrm{GLT}}\kappa^{-1/2}$ thanks to Axiom {\bf GLT 6}. \\
Now using two times {\bf GLT 3}, part 2, we have $\{A_{\bm{n}}^{-1/2} B_{\bm{n}} A_{\bm{n}}^{-1/2}\}_{\bm{n}} \sim_{\mathrm{GLT}} \kappa^{-1/2} \xi \\ \kappa^{-1/2}$, where 
\[
X_{\bm{n}} = A_{\bm{n}}^{-1/2} B_{\bm{n}} A_{\bm{n}}^{-1/2}
\]
is positive definite because of the Sylvester inertia law. Hence the square root of $X_{\bm{n}}$ is well defined and by exploiting again Axiom {\bf GLT 6}, we obtain 
$\left\{\left(A_{\bm{n}}^{-1/2} B_{\bm{n}} A_{\bm{n}}^{-1/2}\right)^{1/2}\right\}_{\bm{n}} \sim_{\mathrm{GLT}} 
\left(\kappa^{-1/2} \xi \kappa^{-1/2}\right)^{1/2}$. Finally, by exploiting Axiom {\bf GLT 3}, part 3, on the matrix-sequence $\{A_{\bm{n}}^{1/2}\}_{\bm{n}}$ and using {\bf GLT 3}, part 2, two times, we conclude
\[
\{G(A_{\bm{n}}, B_{\bm{n}})\}_{\bm{n}} \sim_{\mathrm{GLT}} \kappa^{1/2}\left(\kappa^{-1/2}\xi\kappa^{-1/2}\right)^{1/2}\kappa^{1/2},
\]
where the symbol $\kappa^{1/2}\left(\kappa^{-1/2}\xi\kappa^{-1/2}\right)^{1/2}\kappa^{1/2}$ is exactly $G(\kappa, \xi)$. Furthermore, relation (\ref{r,d general GLT GM1}) and Axiom {\bf GLT 3} imply $\{G(A_{\bm{n}}, B_{\bm{n}})\}_{\bm{n}} \sim_{\sigma, \lambda} G(\kappa, \xi)$ where $G(\kappa, \xi)=(\kappa \xi)^{1/2}$, whenever \(\kappa\) and \(\xi\) commute.\\
Finally the remaining case where \(\xi\) is invertible almost everywhere has the very same proof since $G(\cdot, \cdot)$ is invariant under permutations and hence
\[
\resizebox{1.0\hsize}{!}{$
G(A, B) = A^{1/2} \left( A^{-1/2} B A^{-1/2} \right)^{1/2} A^{1/2} = G(B,A) = B^{1/2} \left( B^{-1/2} A B^{-1/2} \right)^{1/2} B^{1/2},
$}
\]
so that the the same steps can be repeated by exchanging $A_{\bm{n}}$ and $B_{\bm{n}}$ and by exchanging \(\kappa\) and \(\xi\).
\hfill $\bullet$

\begin{remark}
In Theorem~\ref{4: th:two - r,d general GM1}, we assume that at least one between $\kappa$ and $\xi$ is invertible almost everywhere and the same is true in Theorem~\ref{4: th:two - r=1,d general GM1}. The assumption has in fact only a technical nature for allowing the possibility of using part 4 of Axiom \textbf{GLT 3}, when a GLT matrix-sequence is inverted for guaranteeing that the sequence of inverses is still of GLT nature. However, if we assume commutativity then no inversion is required and the technical assumption is no longer needed. The theoretical problem has been solved in paper \cite{GM2_GLT2025}, as described in Chapter~\ref{GM GLT2}.
\end{remark}

\subsection{Mean of more than two matrices}\label{ssec: Karcher}
If $k>2$ then the ALM-mean is obtained through a recursive iteration process where at each step the geometric means of \( k \) matrices are computed by reducing them to \( k-1 \) matrices. However, a significant limitation of this method is its linear convergence leading to a high computational cost due to the large number of iterations required at each recursive step. As a result, the computation of the ALM-mean using this approach becomes quite expensive. However, despite the elegance of the ALM geometric mean for two matrices, it becomes computationally infeasible to extend this formula to more than two matrices \cite{Poloni2010}.\\
To overcome these limitations, the Karcher mean \cite{Bini2013}, was introduced as a generalization of the geometric mean for more than two matrices. The Karcher mean of HPD matrices $A_1, A_2, \dots, A_k$ is defined as the unique positive definite solution to the matrix equation:
\begin{equation}\label{3}
    \sum_{i=1}^{k} \log(A_i^{-1} X) = 0,
\end{equation}
as established by Moakher \cite[Proposition 3.4]{Moakher2005}. This equation can be equivalently expressed in other forms, such as
\begin{equation}\label{6}
    \sum_{i=1}^{k} \log(XA_i^{-1}) = 0, \quad \text{or} \quad \sum_{i=1}^{k} \log(X^{1/2} A_i^{-1} X^{1/2}) = 0,
\end{equation}
by utilizing the formula $M^{-1} \log(K) M = \log(M^{-1}KM)$, which holds for any invertible matrix $M$ and any matrix $K$ with real positive eigenvalues. This formulation arises from Riemannian geometry, where the space of positive definite matrices forms a Riemannian manifold with non-positive curvature. The Karcher mean represents the center of mass (or barycenter) on this manifold \cite{bhatia2007}. In this manifold, the distance between two positive definite matrices \( A \) and \( B \) is defined as
$
\delta(A, B) = \| \log(A^{-1/2} B A^{-1/2}) \|_F,
$
where \( \| \cdot \|_F \) denotes the Frobenius norm.\\
Several numerical methods have been proposed for solving the Karcher mean equation. Initially, fixed-point iteration methods were used, but these methods suffered from slow convergence, especially in cases where the matrices involved were poorly conditioned \cite{Moakher2005}. Later, methods based on gradient descent in the Riemannian manifold were introduced. A common iteration scheme for approximating the Karcher mean is
\begin{equation}\label{4}
    X_{v+1} = X_v \exp\left( -\theta_{v} \sum_{i=1}^{k} \log(A_i^{-1} X_v) \right), \quad X_{0} \in \mathcal{P}_n, v\ge 0,
\end{equation}
where \( X_v \) is the approximation at the \( v \)-th step, and the exponential and logarithmic functions are matrix operations.
Although this method improves convergence, it can exhibit slow linear convergence in certain cases. The iteration with
$X_0 = A_1$ or $X_0 = I$ and $\theta_v = 1/k$, as considered in \cite{manton2004,Moakher2005}, can fail to converge for some matrices $A_1, \dots, A_k$. Furthermore, similar iterations have been proposed in \cite{barbaresco2009,rathi2007}, but without specific recommendations on choosing the initial value or step length. While an optimal $\theta_v$ could theoretically be determined using a line search strategy, this approach is often computationally expensive. Heuristic strategies for selecting the step size, as discussed in \cite{fletcher2007}, may result in slow convergence in many cases.\\
To further enhance the convergence rate, the Richardson iteration method is employed. Indeed, the considered method improves the convergence by using a parameter \( \theta \), which controls the step size in each iteration \cite{Bini2013}. More precisely, given $X_0\in \mathcal{P}_n$, the Richardson iteration is given by
\begin{equation}\label{5}
   X_{v+1} = X_v - \theta X_v \sum_{i=1}^{k} \log(A_i^{-1} X_v)\equiv T(X_v),\quad \ v\ge 0.
\end{equation}
Any solution of (\ref{6}) is a fixed point of the map $T$ in (\ref{5}). The iterative formula can also be rewritten as
\begin{equation}\label{7}
    X_{\nu+1} = X_{\nu} - \theta X_{\nu}^{1/2} \left( \sum_{i=1}^{k} \log \left( X_{\nu}^{1/2} A_i^{-1} X_{\nu}^{1/2} \right) \right) X_{\nu}^{1/2}, \quad \ v\ge 0,
\end{equation}
provided that all the iterates $X_{\nu}$ remain positive definite. Equation (\ref{7}) further demonstrates that if $X_{\nu}$ is Hermitian, then $X_{\nu+1}$  is also a Hermitian matrix.\\
The parameter $\theta$ plays a crucial role in controlling the step size of each iteration and choosing an optimal value of $\theta$ can significantly influence the convergence behavior of the iteration. If $\theta$ is small enough, the iteration is guaranteed to produce positive definite matrices and converge towards the solution. In particular, when the matrices $A_1, \dots, A_k$ commute, setting $\theta = \frac{1}{k}$ ensures at least quadratic convergence.\\
The explicit value of $\theta$ used in the iteration is given by
\begin{equation}
\theta = \frac{2}{\gamma + \beta},
\label{eq:theta}
\end{equation}
with \(\beta\) and \(\gamma\) computed as
\begin{equation}
[\beta, \gamma] = \left[ \sum_{j=1}^{k} \frac{\log(c_j)}{c_j - 1}, \sum_{j=1}^{k} c_j \frac{\log(c_j)}{c_j - 1} \right],
\label{eq:beta_gamma}
\end{equation}
where $c_j = \frac{\lambda_{\text{max}}(M_j)}{\lambda_{\text{min}}(M_j)}, $ and $M_j = G^{1/2} A_j^{-1} G^{1/2},$ $G$ is the current approximation of the Karcher mean. The closer the eigenvalues of $M_j$ are to 1, the faster the convergence is.

The Richardson-like iteration can be implemented in different equivalent forms
\begin{align}
X &= X - \theta X^{1/2} \left( \sum_{i=1}^{k} \log\left( X^{1/2} A_i^{-1} X^{1/2} \right) \right) X^{1/2} \qquad & (G = X)  \label{10} \\
X &= X^{1/2} \left( I - \theta \sum_{i=1}^{k} \log\left( X^{1/2} A_i^{-1} X^{1/2} \right) \right) X^{1/2} \qquad & (G = X)  \label{11} \\
Y &= Y - \theta Y^{1/2} \left( \sum_{i=1}^{k} \log\left( Y^{1/2} A_i Y^{1/2} \right) \right) Y^{1/2} \qquad & (G = Y^{-1}) \label{12} \\
Y &= Y^{1/2} \left( I - \theta \sum_{i=1}^{k} \log\left( Y^{1/2} A_i Y^{1/2} \right) \right) Y^{1/2} \qquad & (G = Y^{-1}) \label{13}
\end{align}
Among these equivalent formulations, the first one, Equation (\ref{10}), is the most practical for implementation. It avoids matrix inversions, which can introduce numerical instabilities and increase computational complexity. While formulations Equation (\ref{12}) and Equation (\ref{13}) aim to reduce the number of matrix inversions, the final step requires inverting the result, which can lead to inaccuracies, especially for poorly conditioned matrices. Additionally, Equation (\ref{11}) retains the simplicity of direct matrix operations without introducing unnecessary complications.\\
A numerically more efficient approach uses Cholesky factorization to reduce the computational cost of forming matrix square roots at every step, enhancing efficiency as forming the Cholesky factor costs less than computing a full matrix square root \cite{Higham2008}.
Suppose \( X_{\nu} = R_{\nu}^{T} R_{\nu} \) is the Cholesky decomposition of \( X_{\nu} \), where \( R_{\nu} \) is an upper triangular matrix. The iteration step can be rewritten as
\begin{equation}\label{14}
X_{\nu+1} = X_{\nu} + \theta R_{\nu}^{T} \left( \sum_{i=1}^{k} \log\left( R_{\nu}^{-T} A_i R_{\nu}^{-1} \right) \right) R_{\nu}.
\end{equation}
In this formulation, the Cholesky factor \( R_{\nu} \) is updated at each iteration. The condition number of the Cholesky factor \( R_{\nu} \) in the spectral norm is the square root of the condition number of \( X_{\nu} \), thus ensuring good numerical accuracy. For this heuristic to be effective, it is essential that \( X_0 \) provides a good approximation of \( G \). If the initial guess $X_0$ is close enough to the solution and is positive definite, the sequence ${X_v}$ generated by the iteration remains well-defined and converges to the desired solution. However, if the initial iterate is not positive definite, adjusting the value of $\theta$ or modifying the iteration scheme may be necessary to ensure that all iterates remain positive. Therefore, selecting $X_0$ as the Cheap mean is critical. An adaptive version of this iteration has been proposed and implemented in the Matrix Means Toolbox~\cite{BiniIannazzo2010}. \\
Of course the Richardson-like iteration is relevant for computing efficiently the Karcher mean: we also exploit its formal expression for theoretical purposes when dealing with sequences of matrices, particularly those involving GLT sequences. For the theory we come back at relation (\ref{10}) and we consider the associated iteration
\begin{equation}\label{iter-karcher}
X_{\nu+1} = X_{\nu} - \theta X_{\nu}^{1/2} \left( \sum_{i=1}^{k} \log\left( X_{\nu}^{1/2} A_i^{-1} X_{\nu}^{1/2} \right) \right) X_{\nu}^{1/2}, 
\end{equation}
with $X_0$ given positive definite matrix. We know that $X_{\nu}$ converges to the geometric mean of $A_1,\ldots,A_k$ as $\nu$ tends to infinity for every 
fixed positive definite initial guess $X_0$.\\
Fix $r,d\ge 1$. Suppose now that the block multitivel sequence of matrices \( \{ A_{\bm{n}}^{(i)} \}_{\bm{n}}\sim_{\text{GLT}}\kappa_i \) for \( i = 1, \dots, k \) are given, where \( A_{\bm{n}}^{(1)}, \dots, A_n^{(k)} \) are positive definite for every multi-index ${\bm{n}}$. 
Due to the positive definiteness of the matrices $\{A_{\bm{n}}^{(i)}\}_{\bm{n}}$ and because $\{ A_{\bm{n}}^{(i)} \}_{\bm{n}} \sim_{\mathrm{GLT}} \kappa_i$, from Axiom~\textbf{GLT~1} it follows that each $\kappa_i$ is nonnegative definite almost everywhere (see Remark~\ref{rem: range}).

In this setting, it is conjectured that the sequence of Karcher means \( \{ G(A_{\bm{n}}^{(1)},\\ \dots, A_{\bm{n}}^{(k)}) \}_{\bm{n}} \) forms a new GLT matrix-sequence whose symbol is the geometric mean of the individual symbols \( \kappa_1, \dots, \kappa_k \), specifically \( (\kappa_1 \cdots \kappa_k)^{1/k} \) if all symbols commute and $G(\kappa_1, \dots, \kappa_k)$ in the general case.
In order to attack the problem the initial guess matrix-sequence  \( \{X_{\bm{n},0}\}_{\bm{n}} \) must be of GLT type with nonnegative definite GLT symbol. In this way thanks to (\ref{iter-karcher}) and using the GLT axioms in the way it is done in Theorem \ref{4: th:two - r,d general GM1}, we deduce easily that \( \{X_{\bm{n},\nu}\}_{\bm{n}} \) is still a GLT matrix-sequence with symbol $g_\nu$ converging to $G(\kappa_1, \dots, \kappa_k)$.

Finally, Theorem~\ref{th: acs} and Axiom~\textbf{GLT 4} could be applied if we prove that $\{\{X_{\bm{n},\nu}\}_{\bm{n}}\}_{\nu}$ is an a.c.s.\ for the limit sequence $\{ G(A_{\bm{n}}^{(1)}, \dots, A_{\bm{n}}^{(k)}) \}_{\bm{n}}$. This could be proven using Schatten estimates like those in the second item of Theorem~\ref{th 3.1}, but at the moment this is not easy because the known convergence proofs for the Karcher iterations are all based on pointwise convergence, which makes difficult to obtain $O(\nu(\bm{n}))$ bounds for the $p$ power of the Schatten $p$-norm of error matrix-sequences; see Theorem~\ref{th 3.1}. 

\section{Numerical experiments}\label{22}
In this section, we present and critically analyze several selected examples, by considering matrix-sequences of geometric means $k$ HPD matrices for $k=2,3$. In particular our numerics show asymptotical spectral properties of the resulting matrix-sequences in accordance with the theoretical results (and conjectures) of the previous section. We introduce few examples in which the input matrix-sequences are either of Toeplitz type or are general $r$-block $d$-level GLT matrix-sequences, stemming from the approximations of differential operators via local methods such finite differences, finite elements, isogeometric analysis: in the first group we explore the geometric mean of two matrix-sequences and in the second group we consider the Karcher mean of three matrix-sequences, both taking into account one-dimensional (1D) and two-dimensional (2D) settings, i.e., $d=1,2$ and $r=1$; in the final group we deal with $r$-block GLT matrix-sequences with $r=2$. We anticipate the strong agreement of the numerical evidences with the theoretical results in Theorem \ref{4: th:two - r=d=1 GM1}, Theorem \ref{4: th:two - r=1,d general GM1}, Theorem \ref{4: th:two - r,d general GM1}, and with the conjecture regarding the Karcher mean, also when the matrix-sizes are quite moderate. The latter is a nontrivial numerical finding since all the theoretical results have an asymptotic spectral nature. 

\subsection{Example 1 (1D)}\label{30}
Let \( A_n = T_n\left((2 - 2 \cos(\theta))^2\right) \) according to Section~\ref{TM} and let \( B_n \) be the finite difference discretization of the differential operator \( (\alpha(x) u'')'' \) on the interval \( (0, 1) \), with boundary conditions \( u(0) = u(1) = u'(0) = u'(1) = 0 \), where \( \alpha(x) \) is positive on \( (0, 1) \). For the fourth order boundary value problem
\begin{equation*}
     (\alpha(x)u^{''})^{''}=f(x), \quad
u(0)=u(1)=u^{'}(0)=u^{'}(1)=0,
\end{equation*}
we approximate the derivative $u^{(4)}(x)$ by using the second-order central FD scheme characterized by the stencil $(1,-4,6,-4,1)$. 
More specifically, for $\alpha, u$ smooth enough, we have $\left. (\alpha(x)u^{''}(x))^{''} \right|_{x=x_i}$ equal to
\begin{align*}
    \frac{1}{h^{4}}&\bigg
    (\alpha_{i-2}u_{i-2}-2(\alpha_{i-1}+\alpha_{i})u_{i-1}+(\alpha_{i-1} +4\alpha_{i}+\alpha_{i+1})u_{i}\\-&2(\alpha_{i+1} +\alpha_{i})u_{i+1} +\alpha_{i+1}u_{i+2}\bigg) + O(h^2),
\end{align*}
for all $i=2,3,....,n+1$; here, $h=\frac{1}{n+3}$ and $x_{i}=ih$ for $i=0,1,....,n+3$.
\begin{center} 
\begin{tikzpicture}
\draw[thick] (0,0) -- (7,0);
\fill[blue] (0,0) circle (2pt) node[below] {$x_0$} node[above] {$0$};
\fill[blue] (1,0) circle (2pt) node[below] {$x_1$};
\fill[blue] (6,0) circle (2pt) node[below] {$x_{n+2}$};
\fill[blue] (7,0) circle (2pt) node[below] {$x_{n+3}$} node[above] {$1$};

\fill[red] (2,0) circle (2pt) node[below] {$x_2$};
\fill[red] (3,0) circle (2pt) node[below] {$x_3$};
\fill[red] (4,0) circle (2pt) node[below] {$x_n$};
\fill[red] (5,0) circle (2pt) node[below] {$x_{n+1}$};

\node at (2.5, 0) {$\cdot$};
\node at (3.5, 0) {$\cdot$};
\node at (4.5, 0) {$\cdot$};

\draw [decorate,decoration={brace,amplitude=10pt,raise=4pt},yshift=0pt]
(2,0) -- (5,0) node [midway,yshift=15pt] {};

\end{tikzpicture}
\end{center}
By taking into account the homogeneous boundary conditions and by neglecting the $O(h^2)$ approximation error, we approximate the nodal value $u(x_{i})$ with the value of $u_{i}$ for $i=0,1,....,n+3$, where $u_{0}=u_{1}=u_{n+2}=u_{n+3}=0$ and $u=(u_{2},....,u_{n+1})^{T}$ is the solution of the linear system
\begin{align*}
    \alpha_{i-2}&u_{i-2}-2(\alpha_{i-1}+\alpha_{i})u_{i-1}+(\alpha_{i-1} +4\alpha_{i}+\alpha_{i+1})u_{i}\\-&2(\alpha_{i+1} +\alpha_{i})u_{i+1} +\alpha_{i+1}u_{i+2}=h^{4}f(x),
\end{align*}
for all $i=2,3,....,n+1$.

The structure of the resulting matrix $B_{n}=B_{n}(\alpha)$ is as reported below
\begin{equation}\label{eq: 4th-der-var}
\scalebox{0.74}{$
\begin{pmatrix}
    (\alpha_3 + 4\alpha_2 + \alpha_1) &-2(\alpha_3 + \alpha_2) &\alpha_3&  & & & \\
    -2(\alpha_2 + \alpha_3) & (\alpha_4 + 4\alpha_3 + \alpha_2) & -2(\alpha_4 + \alpha_3) & \alpha_4  & &  &\\

    \alpha_3 & -2(\alpha_3 + \alpha_4) & (\alpha_5 + 4\alpha_4 + \alpha_3) & -2(\alpha_5 + \alpha_4)  &  \alpha_5 & &\\

   \ddots& \ddots & \ddots & \ddots & \ddots &  &  \\
    \alpha_{n-2} & -2(\alpha_{n-2} + \alpha_{n-1}) & (\alpha_{n-2} + 4\alpha_{n-1} + \alpha_{n}) & -2(\alpha_{n} + \alpha_{n-1}) & \alpha_{n} & \\

   & \alpha_{n-1} & -2(\alpha_{n-1} + \alpha_{n}) & (\alpha_{n-1} + 4\alpha_{n} + \alpha_{n+1}) & -2(\alpha_{n+1} + \alpha_{n})  \\
   & & \alpha_{n} & -2(\alpha_{n} + \alpha_{n+1}) & (\alpha_n + 4\alpha_{n+1} + \alpha_{n+2})
\end{pmatrix}.
$}
\end{equation}
Looking at $B_{n}$ in (\ref{eq: 4th-der-var}), we observe that it can be written as 
\[
B_{n}=B_{n}(\alpha)=D_{n}^{+}K_{n}^{+} + D_{n}K_{n} + D_{n}^{-}K_{n}^{-},
\]
with
\[ K_{n}^{+}=
\begin{pmatrix}
1 & -2 & 1 & & & & \\
& 1 & -2 & 1 & & & \\
& & 1 & -2 & 1 & & \\
& & & 1 & -2 & \ddots & \\
& & & & 1 & \ddots & 1 \\
& & & & & \ddots & -2 \\
& & & & & & 1 \\
\end{pmatrix}= T_{n}(1-2e^{-i\theta} +e^{-2i\theta} ),
\]
\[
K_{n}=
\begin{pmatrix}
4 & -2 &  &  &  &  &  \\
-2 & 4 & -2 &  &  &  &  \\
 & -2 & 4 & -2 &  &  &  \\
 &  & -2 & 4 & -2 &  &  \\
 &  &  & \ddots & \ddots & \ddots &  \\
 &  &  &  & -2 & 4 & -2 \\
 &  &  &  &  & -2 & 4 \\
\end{pmatrix}
=T_{n}(4-2e^{-i\theta} -2e^{i\theta} ),
\]
\[
K_{n}^{-}=
\begin{pmatrix}
1 & & & & & & \\
-2 & 1 & & & & & \\
1 & -2 & 1 & & & & \\
& 1 & -2 & 1 & & & \\
& & \ddots & \ddots & \ddots & & \\
& & & 1 & -2 & 1 & \\
& & & & 1 & -2 & 1 \\
\end{pmatrix}
=T_{n}(1-2e^{i\theta} +e^{2i\theta} ).
\]
\begin{eqnarray*}
D_{n}^{+} & = & \underset{i=1,2....,n }{ \text{diag}}\alpha_{i+2}=\underset{i=1,2....,n }{ \text{diag}}
\alpha(x_{i+2}), \\
D_{n}^{}  & = & \underset{i=1,2....,n }{ \text{diag}}\alpha_{i+1} =\underset{i=1,2....,n }{ \text{diag}} \alpha(x_{i+1}), \\
D_{n}^{-} & = &\underset{i=1,2....,n }{ \text{diag}} \alpha_{i} =\underset{i=1,2....,n }{ \text{diag}} \alpha(x_{i}).
\end{eqnarray*}
It is easy to check that grids $\mathcal{G}_{n}^{+}=\{x_{i+2}\}_{i=1,....,n},  \mathcal{G}_{n}=\{x_{i+1}\}_{i=1,....,n}$ and $\mathcal{G}_{n}^{-}=\{x_{i}\}_{i=1,....,n }$ are asymptotically uniform in $[0,1]$, according to the notion given in \cite{glt-au-grids}.
In fact, for $ \mathcal{G}_{n}^{+}=\{x_{i+2}\}_{i=1,....,n}$ we have
\begin{align*}
 \underset{i=1,2,\dots,n}{\text{max}} \left| x_{i,n} - \frac{i}{n} \right|
= &\underset{i=1,2,\dots,n}{\text{max}} \left| x_{i+2} - \frac{i}{n} \right|
= \underset{i=1,2,\dots,n}{\text{max}} \left| (i+2)h - \frac{i}{n} \right| 
\\ =& \underset{i=1,2,\dots,n}{\text{max}} \left| \frac{i+2}{n+3} - \frac{i}{n} \right|
=\underset{i=1,2,\dots,n}{\text{max}} \left| \frac{ni + 2n - i(n+3)}{n(n+3)} \right|\\
=& \underset{i=1,2,\dots,n}{\text{max}} \left| \frac{-i}{n(n+3)} \right|
\leq \frac{n}{n(n+3)} \rightarrow 0.
\end{align*}
\\
For $ \mathcal{G}_{n}=\{x_{i+1}\}_{i=1,....,n}$ we have
\begin{align*}
  \underset{i=1,2,\dots,n}{\text{max}} \left| x_{i,n} - \frac{i}{n} \right|
= &\underset{i=1,2,\dots,n}{\text{max}} \left| x_{i+1} - \frac{i}{n} \right|
= \underset{i=1,2,\dots,n}{\text{max}} \left| (i+1)h - \frac{i}{n} \right|  
\\ =& \underset{i=1,2,\dots,n}{\text{max}} \left| \frac{i+1}{n+3} - \frac{i}{n} \right|
= \underset{i=1,2,\dots,n}{\text{max}} \left| \frac{ni + n - i(n+3)}{n(n+3)} \right|\\
=& \underset{i=1,2,\dots,n}{\text{max}} \left| \frac{-2i}{n(n+3)} \right|
\leq \frac{2n}{n(n+3)} \rightarrow 0.
\end{align*}
For $ \mathcal{G}_{n}^{-}=\{x_{i}\}_{i=1,....,n}$ we have
\begin{align*}
 \underset{i=1,2,\dots,n}{\text{max}} \left| x_{i,n} - \frac{i}{n} \right|
=& \underset{i=1,2,\dots,n}{\text{max}} \left| x_{i} - \frac{i}{n} \right|
= \underset{i=1,2,\dots,n}{\text{max}} \left| ih - \frac{i}{n} \right|  
\\ =& \underset{i=1,2,\dots,n}{\text{max}} \left| \frac{i}{n+3} - \frac{i}{n} \right|
= \underset{i=1,2,\dots,n}{\text{max}} \left| \frac{ni - i(n+3)}{n(n+3)} \right|\\
=& \underset{i=1,2,\dots,n}{\text{max}} \left| \frac{-3i}{n(n+3)} \right|
\leq \frac{3n}{n(n+3)} \rightarrow 0.
\end{align*}
Hence, by~\textbf{GLT~2}, part 2, and by \cite{glt-au-grids}, we deduce that $\{D_{n}^{+}\}_{n}\sim_{\text{GLT}} \alpha(x) $, $\{D_{n}^{}\}_{n}\sim_{\text{GLT}} \alpha(x) $ and $\{D_{n}^{-}\}_{n}\sim_{\text{GLT}} \alpha(x) $. In conclusion, by invoking \textbf{GLT~2}-\textbf{GLT~3}, we infer that
\begin{align*}
    &\{ B_{n}\}_{n} \sim_{\rm GLT} \alpha(x)(1-2e^{-i\theta} +e^{-2i\theta} )+\alpha(x)(4-2e^{-i\theta} -2e^{i\theta} )  + \alpha(x)(1-2e^{i\theta} +e^{2i\theta} ) \\
 &\{ B_{n}\}_{n} \sim_{\rm GLT} \alpha(x)(2-2\cos(\theta))^2.
\end{align*}
\subsection*{Eigenvalue distribution}

We begin by numerically verifying the eigenvalue distribution of the matrix-sequence  $\{G(A_{{n}}, B_{{n}})\}_{{n}}$ with respect to its GLT symbol
$(\kappa \xi)^{1/2}$ according to Theorem \ref{r=d=1 distributions}, with $A_n=T_n\left((2 - 2 \cos(\theta))^2\right)$, $B_n=B_n(\alpha)$ as in (\ref{eq: 4th-der-var}) with $\alpha(x)=x$, $\kappa(x,\theta)=(2 - 2 \cos(\theta))^2$, $\xi(x,\theta))=\alpha(x)(2-2\cos(\theta))^2$. In Figure \ref{fig:exp1(1D)}, we compare the eigenvalues of the geometric mean with a uniform sampling of the symbol. It is evident that, as \( n \) increases, the symbol provides a better and better approximation of the eigenvalues. Similar results for the two-dimensional case are shown in Example \ref{31}, Figure \ref{fig:exp2(2D)}.
\begin{figure}[H]
   \begin{minipage}{0.54\textwidth}
     \centering
     \includegraphics[width=\linewidth, height=6cm]{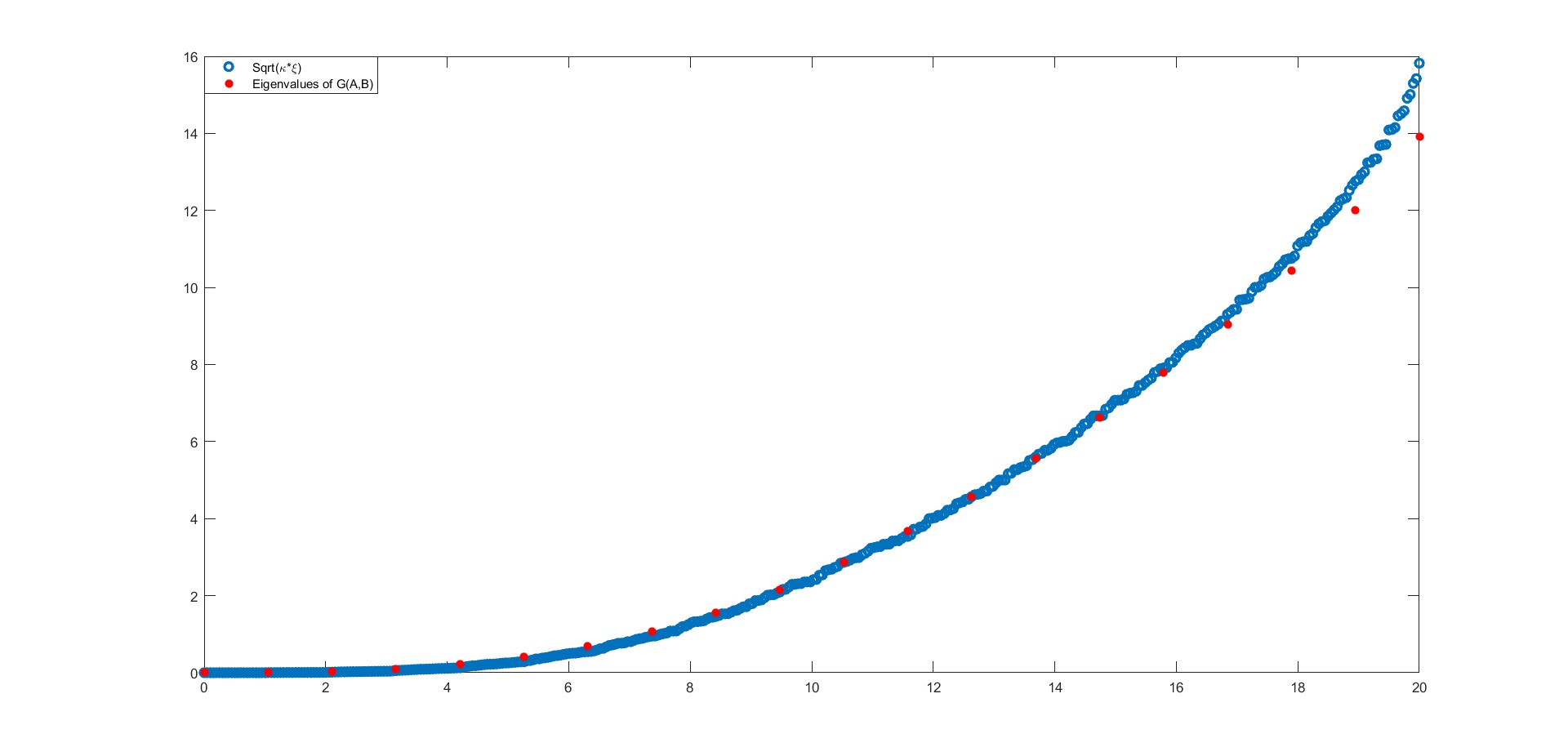} 
     \captionsetup{labelformat=empty}
     \caption*{}
   \end{minipage}\hfill
   \begin{minipage}{0.52\textwidth}
     \centering
     \includegraphics[width=\linewidth, height=6cm]{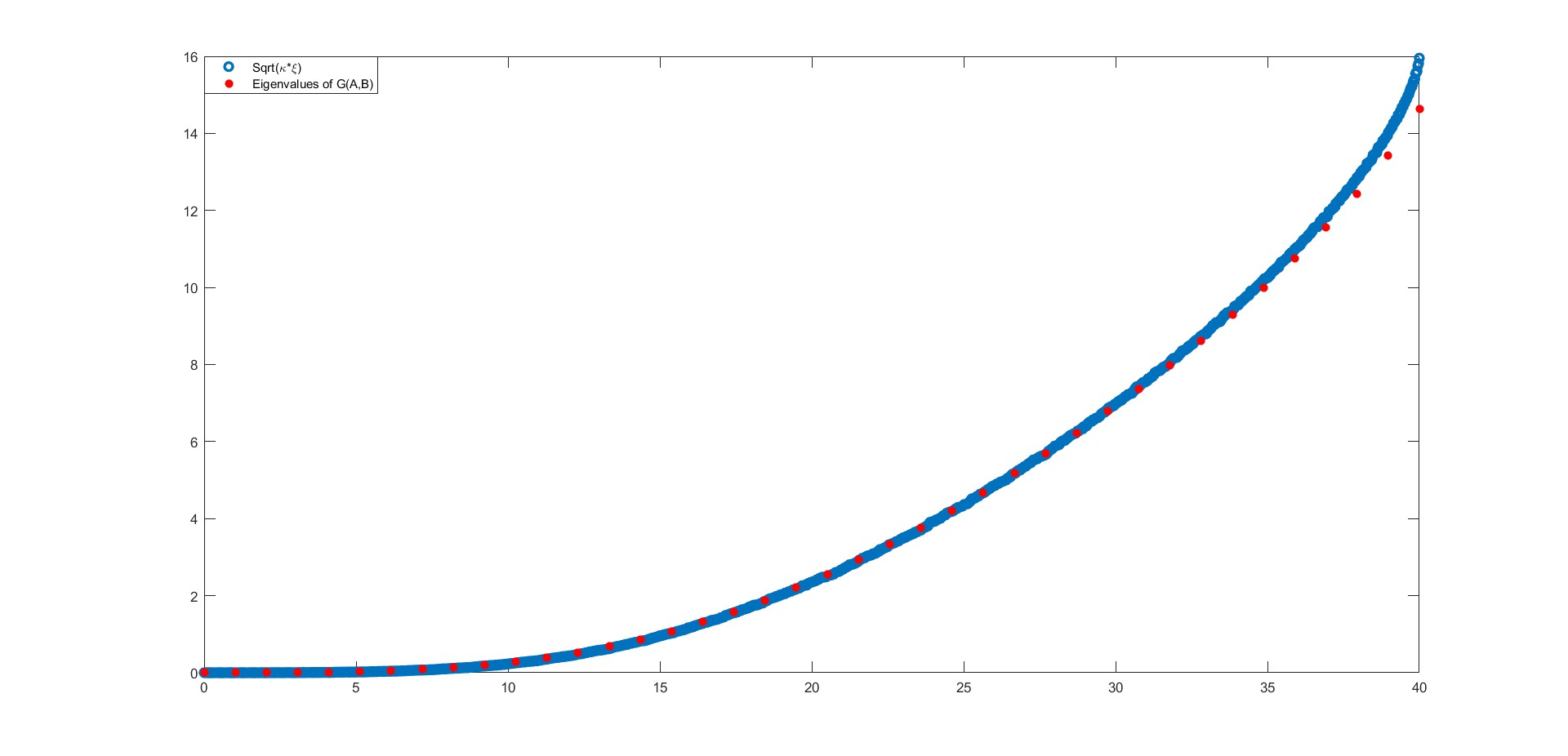} 
     \captionsetup{labelformat=empty}
     \caption*{}
   \end{minipage}
   \caption{{\textbf{ Example 1 (1D).}
Comparison between eigenvalues of \(G(A_n,B_n)\) (red stars; \(n=20,40\))
versus the symbol \( (\kappa\xi)^{\frac{1}{2}} \) (blue circles)}}
\label{fig:exp1(1D)}
\end{figure}
\subsection{Example 2 (2D)}\label{31}
Let $A_{\bm{n}} = T_{\bm{n}}(F(\theta_{1}, \theta_{2}))$ according to the notation in Section~\ref{TM} with $d=2$ and $B_{\bm{n}}$ be the finite difference discretization of the differential operator 
\[
\frac{\partial^2}{\partial x^2}
   \left(a(x,y)\frac{\partial^2}{\partial x^2}\right)+
\frac{\partial^2}{\partial y^2}
   \left(b(x,y)\frac{\partial^2}{\partial y^2}\right),
\]
on the open domain \(\Omega=(0, 1)^2\), with \(a(x, y), b(x,y)\) nonnegative on the closure of $\Omega$, with homogeneous Dirichlet boundary conditions on $u(x,y)$, and zero normal derivatives at $\partial \Omega$.

Regarding $A_{\bm{n}}$, the generating function \(F(\theta_{1}, \theta_{2})\) is given by
\[
F(\theta_{1}, \theta_{2}) = f(\theta_{1}) + f(\theta_{2}),
\]
with $f(\theta) = (2 - 2 \cos(\theta))^2$.
Thus, we have
\[
A_{\bm{n}} = T_{\bm{n}}(f(\theta_{1}) + f(\theta_{2})) = T_{\bm{n}}(8 + 4 \cos^2(\theta_{1}) + 4 \cos^2(\theta_{2}) - 8 \cos(\theta_{1}) - 8 \cos(\theta_{1})).
\]
As in the one-dimensional setting in Example \ref{30}, we apply the second-order central finite difference scheme separately to the
$x$- and $y$-directions, in perfect analogy with the 1D case, and take $a(x,y)=\alpha(x)$, $b(x,y)=\alpha(y)$.
Choosing $\alpha(t)=t$ the related problem is semielliptic, but it is of separable nature. This separable nature is reflected algebraically in a tensor decomposition of the whole approximation and in fact we find that resulting global matrix is given by the Kronecker sum
\[
B_{\bm{n}}=B_{(n_1,n_2)} = B_{n_1}(\alpha) \otimes I_{n_2} + I_{n_1} \otimes B_{n_2}(\alpha),
\]
where $I_n$ denotes the identity matrix of size $n$ and $B_{n}(\alpha)$ is exactly the structure displayed in (\ref{eq: 4th-der-var}). By exploiting the GLT analysis in the one-dimensional case, the symbol for the two-dimensional matrix-sequence can be derived similarly. More precisely, we have
\begin{equation*}
    \{ B_{\bm{n}}\}_{\bm{n}} \sim_{\rm GLT} \alpha(x)(2-2\cos(\theta_{1}))^2 + \alpha(y)(2-2\cos(\theta_{2}))^2.
\end{equation*}
As shown in Figure~\ref{fig:exp2(2D)}, the agreement is remarkable even for the largest eigenvalues for which the absolute discrepancy is higher.
\begin{figure}[H]
   \begin{minipage}{0.52\textwidth}
     \centering
     \includegraphics[width=\linewidth, height=6cm]{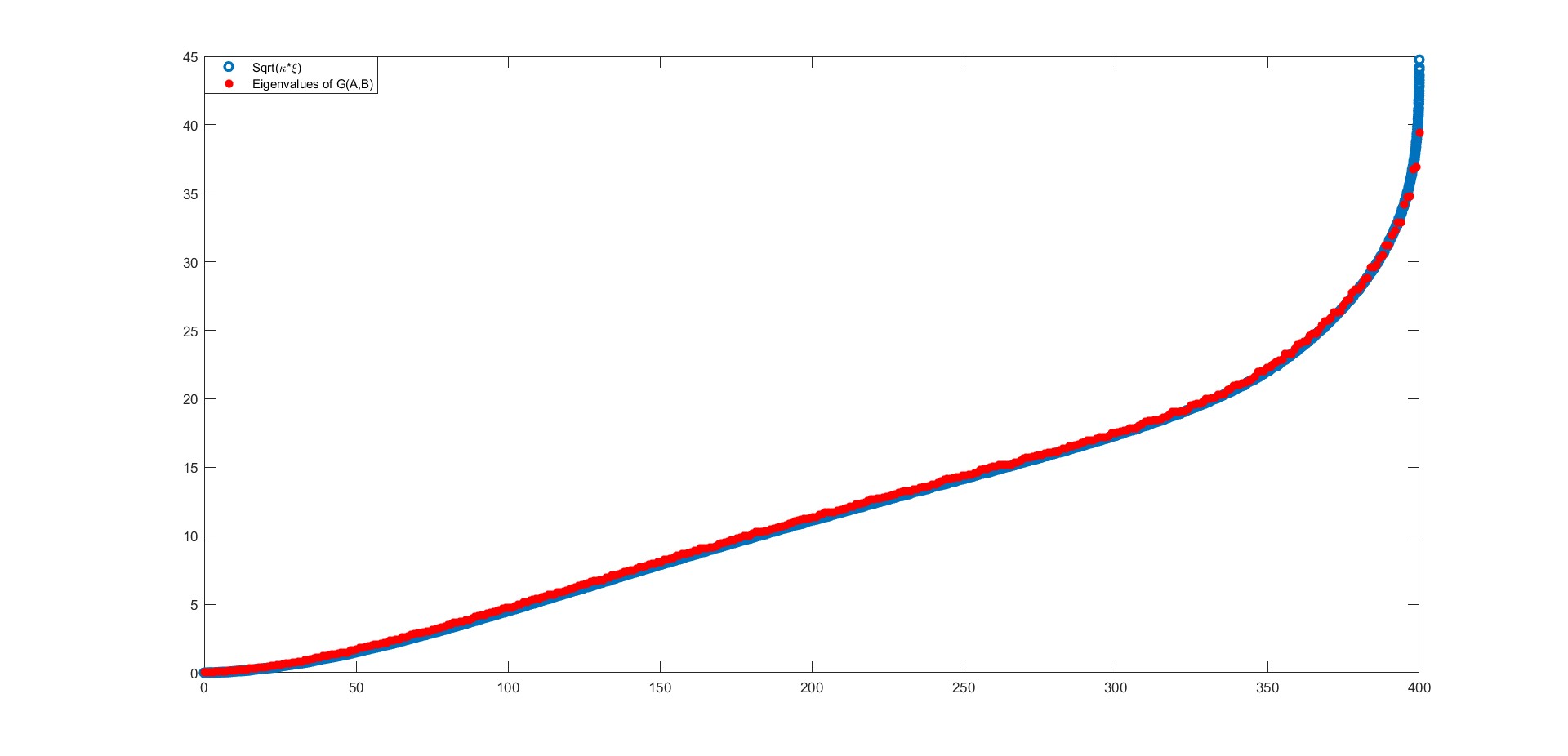} 
     \captionsetup{labelformat=empty}
     \caption*{}
   \end{minipage}\hfill
   \begin{minipage}{0.52\textwidth}
     \centering
     \includegraphics[width=\linewidth, height=6cm]{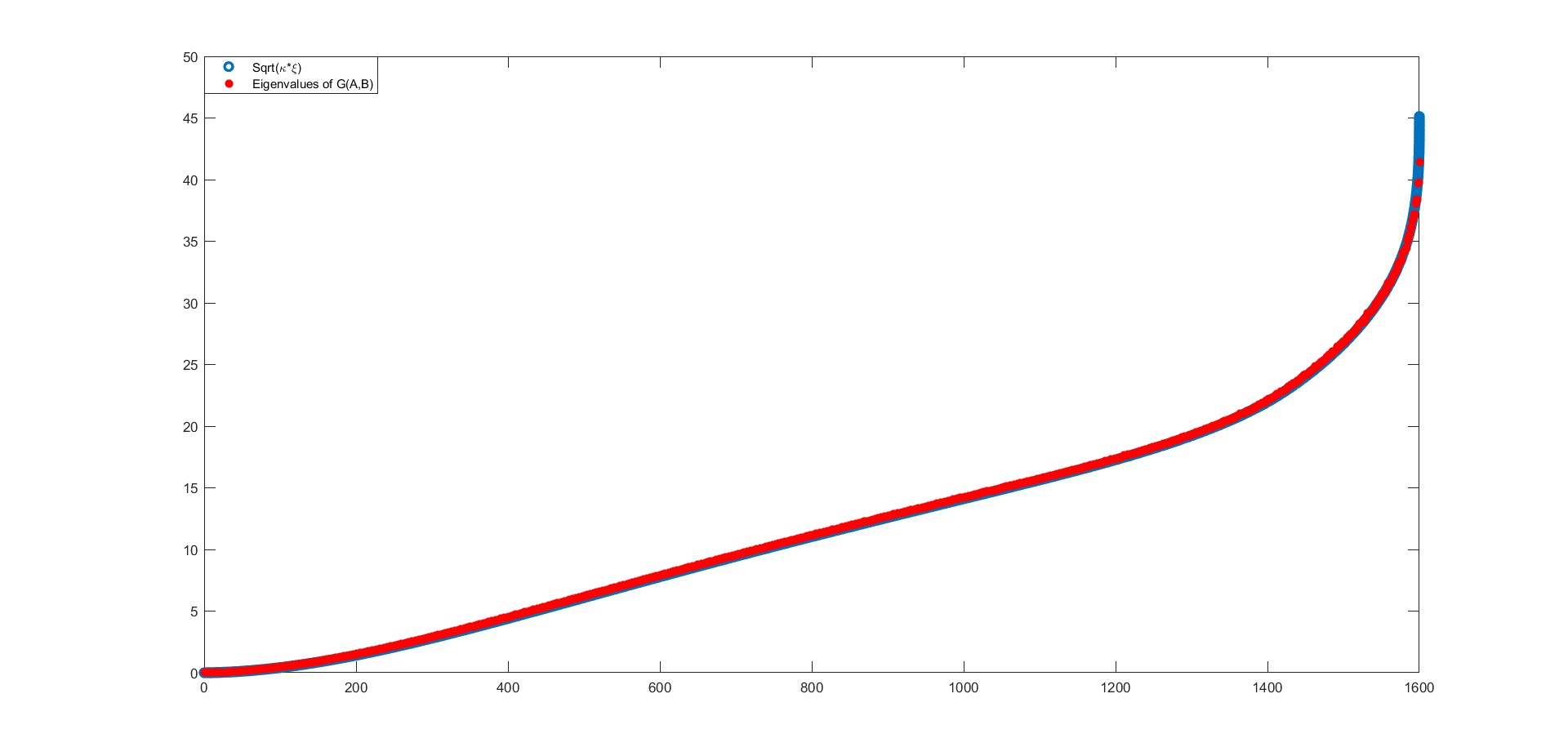} 
     \captionsetup{labelformat=empty}
     \caption*{}
   \end{minipage}
    \caption{{\textbf{ Example 2 (2D).}
Comparison between eigenvalues of \(G(A_n,B_n)\) (red stars; \(n=20,40\))
versus the symbol \( (\kappa\xi)^{\frac{1}{2}} \) (blue circles)}}
\label{fig:exp2(2D)}
\end{figure}
\subsection{Example 3 (1D)}\label{32}
Let \( A_n^{(1)} = T_n(3 + 2 \cos(\theta)) \), \( A_n^{(2)} = D_n(a) \), and \( A_n^{(3)} = D_n(a) T_n(4 - 2 \cos(\theta)) D_n(a) \), where \( T_n(\cdot) \) is the Toeplitz operator for $d=1$ as in Section~\ref{TM}  and \( D_n(a) \) is the diagonal matrix generated by the continuous function \( a(x)=x^2 \), according to the notation in Section~\ref{blckdiag}.
\subsection*{Eigenvalue distribution}
First, we aim to verify the eigenvalue distribution of more than two matrices using the Karcher mean. In Figure~\ref{fig:exp3(1D)}, we compare the eigenvalues of the Karcher mean with a uniform sampling of the resulting limit GLT symbol. It can be observed that the symbol provides a reliable approximation of the eigenvalues, and in fact, as \( n \) increases, the spectral distribution holds asymptotically. Similar results for the two-dimensional case are shown in Example~\ref{33}, Figure~\ref{fig:exp4(2D)}. Both types of result corroborate the conjecture on the GLT nature of the limit matrix-sequence of the Karcher means, as discussed at the end of Section~\ref{ssec: Karcher}.

\begin{figure}[H]
   \begin{minipage}{0.52\textwidth}
     \centering
     \includegraphics[width=\linewidth, height=6cm]{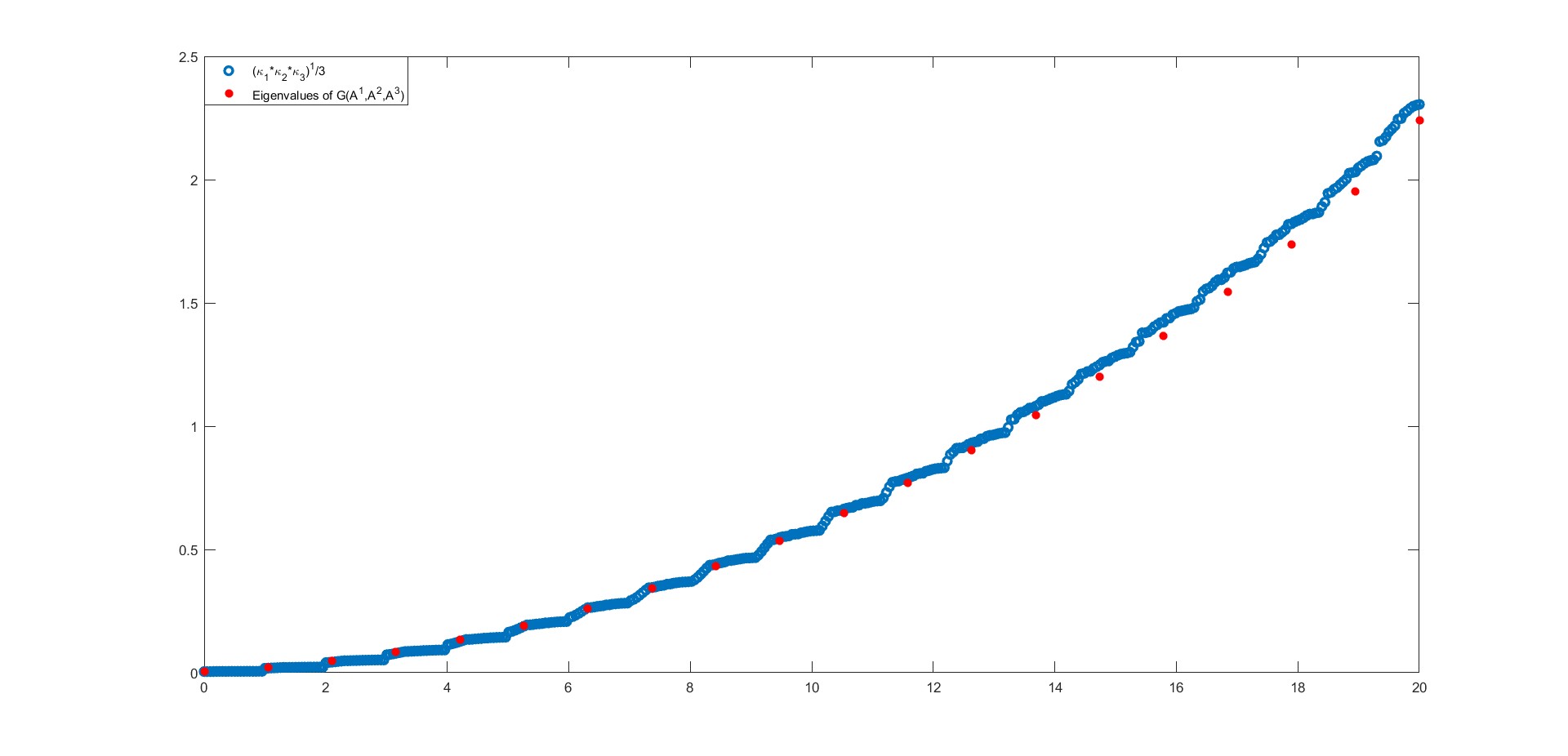} 
     \captionsetup{labelformat=empty}
     \caption*{}
   \end{minipage}\hfill
   \begin{minipage}{0.52\textwidth}
     \centering
     \includegraphics[width=\linewidth, height=6cm]{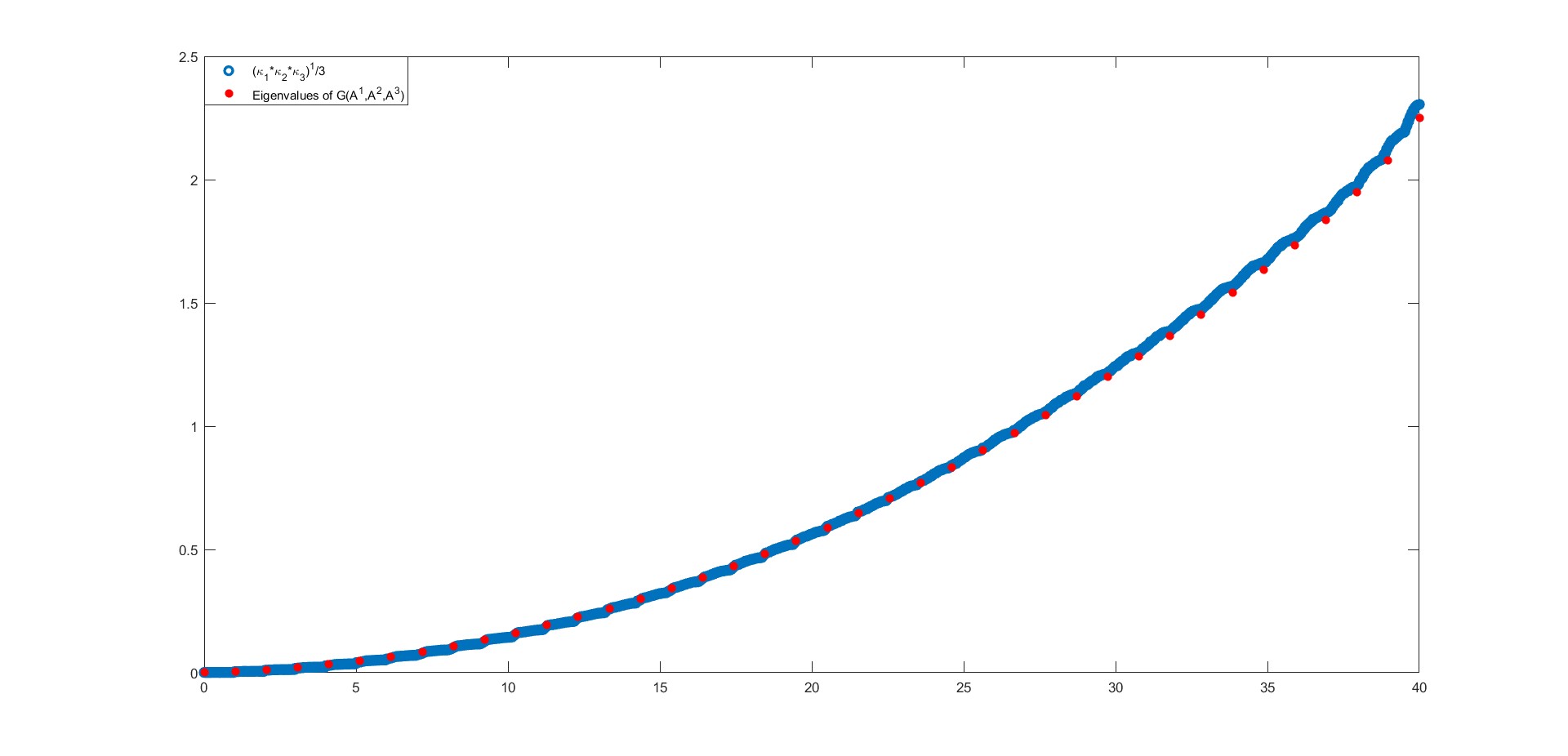} 
     \captionsetup{labelformat=empty}
     \caption*{}
   \end{minipage}
   \caption{{\textbf{ Example 3 (1D).}
Comparison between eigenvalues of \( (G(A_{n}^{(1)},A_{n}^{(2)},\) \(A_{n}^{(3)})) \) (red stars; \(n=20,40\))
versus the symbol \( (\kappa_{1}\kappa_{2}\kappa_{3})^{\frac{1}{3}}  \) (blue circles)}}
\label{fig:exp3(1D)}
\end{figure}
\subsection{Example 4 (2D)}\label{33}
Consider the matrix $A_{\bm{n}}^{(1)} = T_{\bm{n}}(F(\theta_{1}, \theta_{2}))$, where $T_{\bm{n}}(\cdot)$ is the Toeplitz operator as in Section~\ref{TM} with $d=2$, and the function $F(\theta_{1}, \theta_{2})$ is defined as $F(\theta_{1}, \theta_{2}) = f(\theta_{1}) + f(\theta_{2})$, with $f(\theta) = 3 + 2\cos(\theta)$, resulting in $A_{\bm{n}}^{(1)} = T_{\bm{n}}(6 + 2\cos(\theta_{1}) + 2\cos(\theta_{2}))$.

Also, consider the diagonal matrix $A_{\bm{n}}^{(2)} = D_{\bm{n}}(a)$, where, according to the notation in Section~\ref{blckdiag}, $D_{\bm{n}}(a)$ is the diagonal sampling matrix generated by a continuous function $a(x, y)=x^2+y^2$, which is positive on the domain $(0, 1)^2$.\\
We take the matrix $A_{\bm{n}}^{(3)} = D_{\bm{n}}(b) T_{\bm{n}}(G(\theta_{1}, \theta_{2})) D_{\bm{n}}(b)$, where $b(x, y)=1/x+1/y$ is positive and unbounded on the domain $(0, 1)^2$, and where the generating function $G(\theta_{1}, \theta_{2}) = f(\theta_{1}) + f(\theta_{2})$, with $f(\theta) = 4 - 2\cos(\theta)$ implies that $A_{\bm{n}}^{(3)}$ is $T_{\bm{n}}(8 - 2\cos(\theta_{1}) - 2\cos(\theta_{2}))$.

Also in the current example the agreement between the limit GLT symbol and the displayed eigenvalues is remarkably good, so giving ground to the 
conjecture on the Karcher means reported in the final part of Section \ref{ssec: Karcher}. 
\begin{figure}[H]
   \begin{minipage}{0.52\textwidth}
     \centering
     \includegraphics[width=\linewidth, height=6cm]{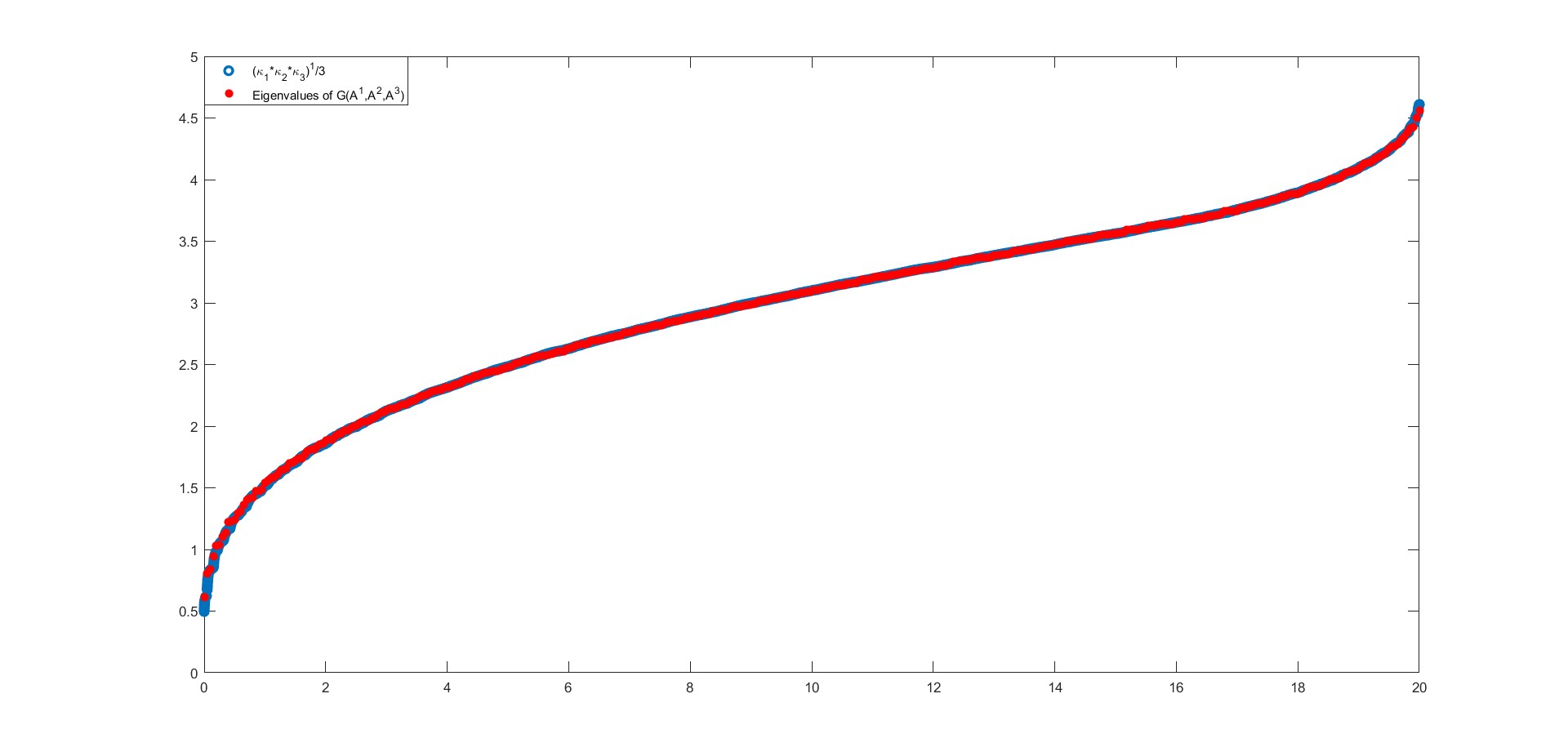} 
     \captionsetup{labelformat=empty}
     \caption*{}
   \end{minipage}\hfill
   \begin{minipage}{0.52\textwidth}
     \centering
     \includegraphics[width=\linewidth, height=6cm]{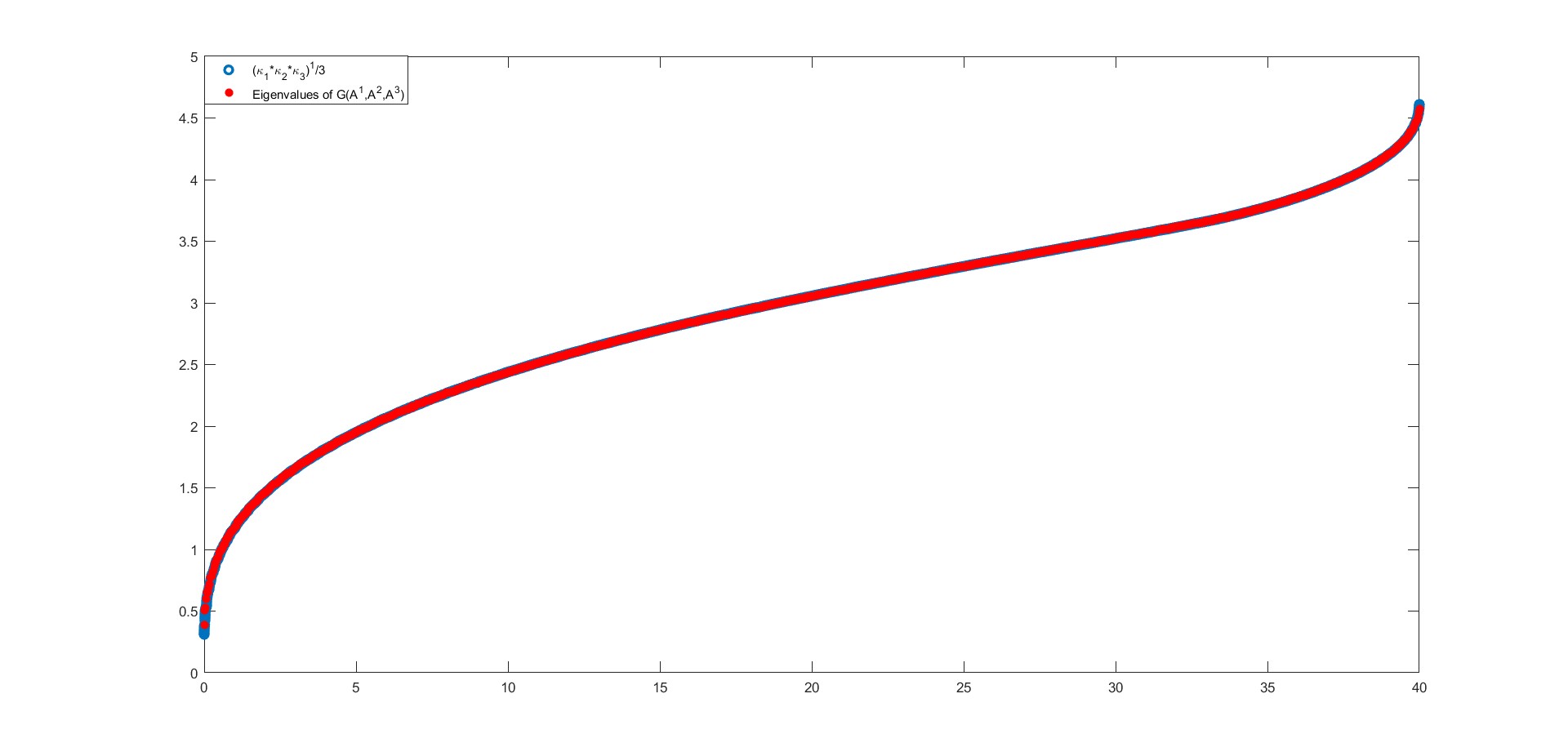} 
     \captionsetup{labelformat=empty}
     \caption*{}
   \end{minipage}
   \caption{{\textbf{ Example 4 (2D).}
Comparison between eigenvalues of \( (G(A_{n}^{(1)},A_{n}^{(2)},\) \(A_{n}^{(3)})) \) (red stars; \(n=20,40\))
versus the symbol \( (\kappa_{1}\kappa_{2}\kappa_{3})^{\frac{1}{3}}  \) (blue circles)}}
\label{fig:exp4(2D)}
\end{figure}
\subsection{Galerkin discretization of the Laplacian eigenvalue problem}\label{101}
The one-dimensional Laplacian eigenvalue problem is given by the differential equation
\begin{equation}\label{102}
  -u_j''(x) = \lambda_j u_j(x), \quad x \in (0,1),  
\end{equation}
with Dirichlet boundary conditions:
\[
u_j(0) = u_j(1) = 0.
\]
The goal is to find the eigenvalues $\lambda_j \in \mathbb{R^{+}}$ and corresponding eigenfunctions $u_j \in H_0^1([0,1])$, for $j=1,2,....\infty$, where
$H_0^1([0,1])$ is the standard Sobolev space of $L^2$ functions with $L^2$ derivatives and vanishing at the boundaries. 
\subsubsection{Weak formulation}
To derive the weak formulation, we multiply both sides of the differential equation by a test function $v \in H_0^1([0,1])$ and integrate over the interval $[0,1]$ i.e. $-\int_0^1 u_j''(x) v(x) \, dx = \int_0^1 \lambda_j u_j(x) v(x) \, dx$. Using integration by parts on the left-hand side and noting that the boundary terms vanish, we deduce $\int_0^1 u_j''(x) v(x) \, dx = \int_0^1 u_j'(x) v'(x) \, dx$, so that the weak form becomes
\[
\int_0^1 u_j'(x) v'(x) \, dx = \lambda_j \int_0^1 u_j(x) v(x) \, dx,
\]
for every test function $v \in H_0^1([0,1])$. The latter is rewritten compactly as
\begin{equation}\label{103}
  a(u_j, v) = \lambda_j (u_j, v),  
\end{equation}
where the bilinear form $a(u_j, v)$ and the $L^2$ inner product $(u_j, v)$ are defined as
\[
a(u_j, v) := \int_0^1 u_j'(x) v'(x) \, dx, \quad (u_j, v) := \int_0^1 u_j(x) v(x) \, dx.
\]

\subsubsection{Galerkin approximation}
The weak formulation allows us to use the Galerkin method to approximate the solution. Let $\mathcal{W}_n = \text{span}\{\phi_1, \dots, \phi_{N_n}\}$ be a finite-dimensional subspace of $H_0^1([0,1])$. The weak problem now becomes: find approximate eigenvalues $\lambda_{j,n} \in \mathbb{R^{+}}$ and eigenfunctions $u_{j,n} \in \mathcal{W}_n$, for $j=1,2....,N_n$ such that, for all $v_n \in \mathcal{W}_n$, we have $a(u_{j,n}, v_n) = \lambda_{j,n} (u_{j,n}, v_n)$. By expanding $u_{j,n}$ and $v_n$ in terms of the basis functions $\{u_i\}$, we obtain 
$u_{j,n} = \sum_{i=1}^{N_n} c_i \phi_i$, $v_n = \sum_{k=1}^{N_n} d_k \phi_k$, and substituting into the bilinear forms, the generalized eigenvalue problem 
\[
K_n \mathbf{c}_j = \lambda_{j,n} M_n \mathbf{c}_j
\]
is defined, where the stiffness matrix $K_n$ and mass matrix $M_n$ are defined as
\begin{equation}\label{104}
 K_n= [a(\phi_j, \phi_i)]_{ij=1}^{N_n} = \Bigr[\int_0^1 \phi_j'(x) \phi_i'(x) \, dx\Bigr]_{ij=1}^{N_n},   
\end{equation}
\begin{equation}\label{105}
 M_n= [(\phi_j, \phi_i)]_{ij=1}^{N_n} = \Bigr[\int_0^1 \phi_j(x) \phi_i(x) \, dx\Bigr]_{ij=1}^{N_n}. 
\end{equation}
Both $K_n$ and $M_n$ are symmetric and positive definite, due to the coercive character of the underlying bilinear forms.
\subsection{\texorpdfstring{Quadratic C$^{0}$ B-Spline discretization}{Quadratic C\string^0 B-Spline Discretization}}\label{107}

In the quadratic C$^{0}$ B-spline discretization of the one-dimensional Laplacian eigenvalue problem, the basis functions $\phi_1, \dots, \phi_{N_n}$  are chosen as B-splines of degree 2 defined on a uniform mesh with step size $\frac{1}{n}$. The basis functions are explicitly constructed on the knot sequence $\{0, 0, 0,  \frac{1}{n}, \frac{1}{n}, \frac{2}{n}, \frac{2}{n} \dots, \frac{n-1}{n}, \frac{n-1}{n}, 1, 1,\\ 1\},
$ (excluding the first and last B-splines, which do not vanish on the boundary of $[0,1]$); see \cite{tom}. The resulting normalized stiffness and mass matrices are given by
\begin{equation}\label{108}
\includegraphics[width=0.7\textwidth]{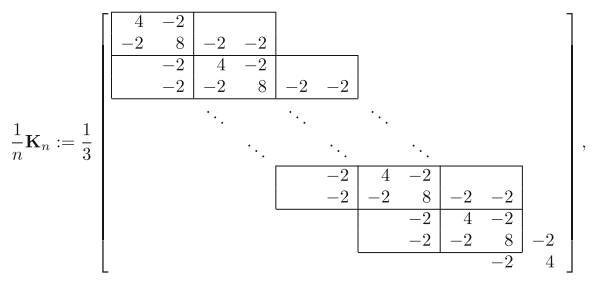} 
\end{equation}
\begin{equation}\label{109}
\includegraphics[width=0.7\textwidth]{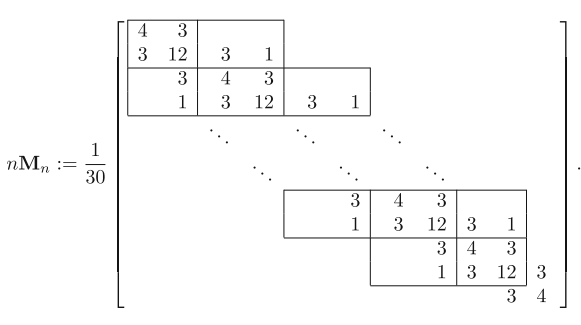} 
\end{equation}

The stiffness matrix \(\frac{1}{n} K_n\) and the mass matrix \(n M_n\) contain, as principal submatrices, the unilevel $2$-block Toeplitz matrices \(T_{n-1}(f)\) and \(T_{n-1}(h)\), according to the notation in Section~\ref{blcktop}  with $d=1$ and $r=2$, where the generating functions \(f(\theta)\) and \(h(\theta)\) are given by
\[
f(\theta) := \frac{1}{3} 
\left(
\begin{bmatrix}
0 & -2 \\
0 & -2
\end{bmatrix}
e^{i\theta}
+
\begin{bmatrix}
4 & -2 \\
-2 & 8
\end{bmatrix}
+
\begin{bmatrix}
0 & 0 \\
-2 & -2
\end{bmatrix}
e^{-i\theta}
\right),
\]
\[
f(\theta) = \frac{1}{3}
\begin{bmatrix}
4  & -2 - 2 e^{i\theta} \\
-2 - 2 e^{-i\theta} & 8 - 4 \cos \theta
\end{bmatrix},
\]
and 
\[
h(\theta) := \frac{1}{30} 
\left(
\begin{bmatrix}
0 & 3 \\
0 & 1
\end{bmatrix}
e^{i\theta}
+
\begin{bmatrix}
4 & 3 \\
3 & 12
\end{bmatrix}
+
\begin{bmatrix}
0 & 0 \\
3 & 1
\end{bmatrix}
e^{-i\theta}
\right),
\]
\[
h(\theta) = \frac{1}{30}
\begin{bmatrix}
4  & 3 + 3e^{i\theta} \\
3 + 3e^{-i\theta} & 12 + 2\cos \theta
\end{bmatrix}.
\]
Since \( \{T_n(g)\}_n \sim_{\text{GLT}} g \) for any Lebesgue integrable generating function \(g\) (Axiom \textbf{GLT 2}, part 1), the theory of GLT sequences and a basic use of the extradimensional approach \cite{prequel-extra,extra} lead to
\begin{equation}\label{114}
    \left\{ \frac{1}{n} K_n \right\}_n \sim_{\text{GLT}} f(\theta),
\end{equation}
\begin{equation}\label{115}
    \{n M_n\}_n \sim_{\text{GLT}} h(\theta).
\end{equation}
By the Axiom \textbf{GLT 3}, the linear combination of two GLT sequences is again a GLT sequence, with the symbol being the corresponding linear combination of their symbols. The matrix-sequence  \( \{L_n\}_n\) with
\begin{equation}
    L_n := \frac{1}{n}K_n + nM_n,
\end{equation}
is a linear combination of the sequences 
\(\left\{\frac{1}{n}K_n\right\}_n\) and \(\{nM_n\}_n\). Consequently, 
\[
\{L_n\}_n \sim_{\text{GLT}} f(\theta) + h(\theta)=e(\theta),
\]
where \(f(\theta)\) and \(h(\theta)\) are the symbols of the stiffness matrix \(\frac{1}{n}K_n\) and the mass matrix \(nM_n\), respectively.

\begin{figure}[H]
   \begin{minipage}{0.52\textwidth}
     \centering
     \includegraphics[width=\linewidth, height=6cm]{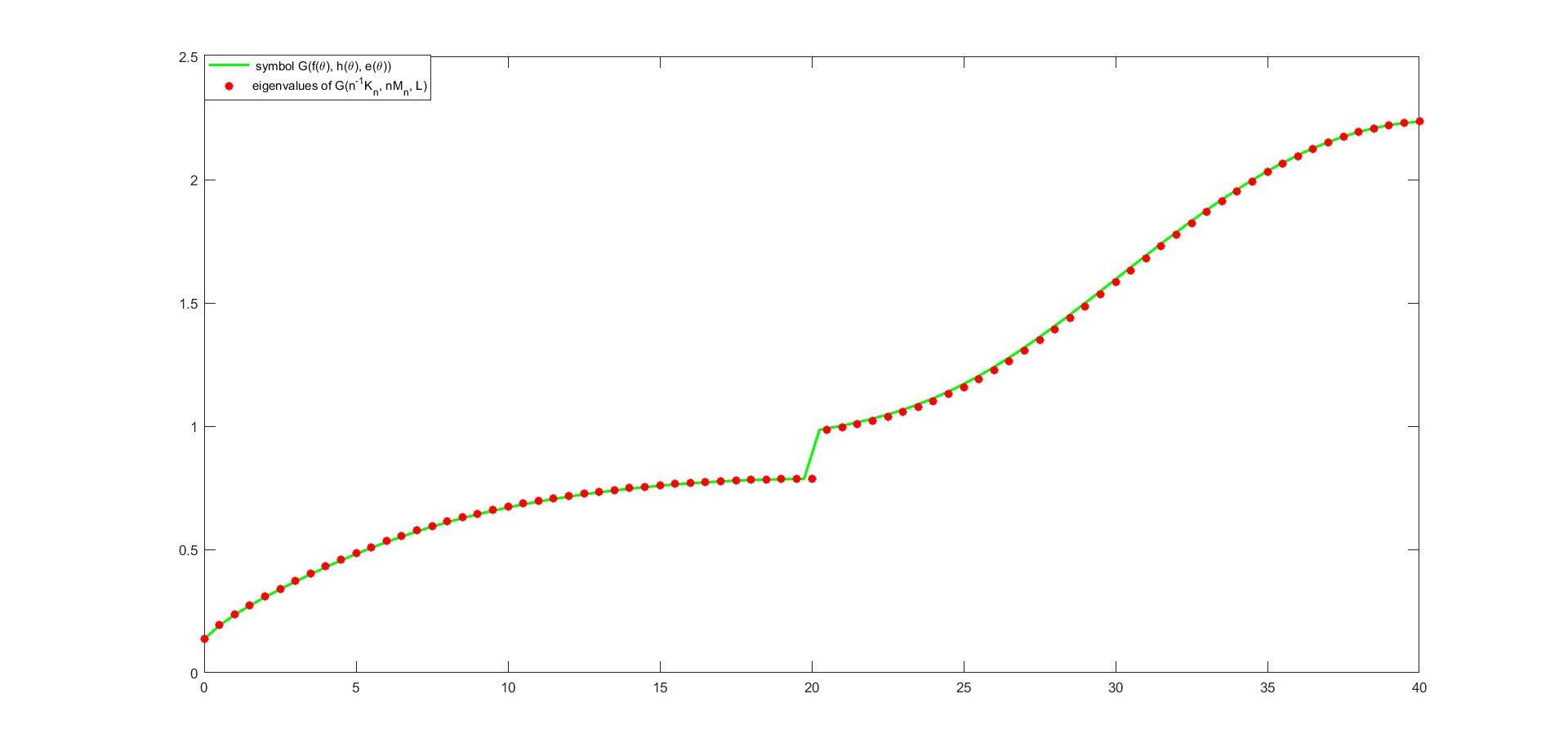} 
     \captionsetup{labelformat=empty}
     \caption*{}
   \end{minipage}\hfill
   \begin{minipage}{0.52\textwidth}
     \centering
     \includegraphics[width=\linewidth, height=6cm]{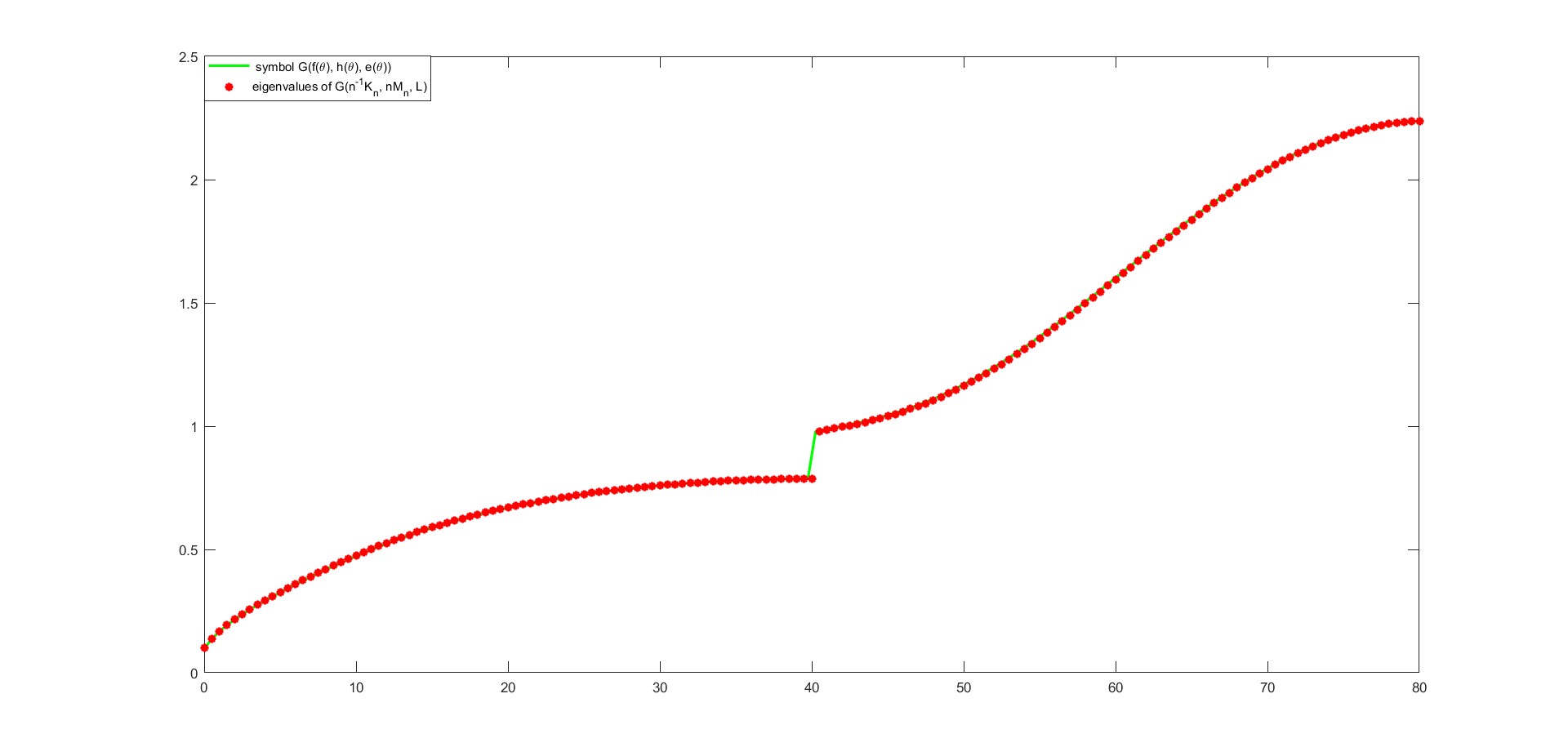} 
     \captionsetup{labelformat=empty}
     \caption*{}
   \end{minipage}
   \caption{{\textbf{For Quadratic case: $r=2, d=1$ .}
Comparison between eigenvalues of \( (G(n^{-1}K_{n}, nM_{n}, L_n)) \) (red stars; \(n=40,80\))
versus the symbol \( G(f(\theta),h(\theta),e(\theta)) \) (blue circles)}}
\label{fig:exp5(1D)}
\end{figure}
The numerical evidence is convincing: furthermore, from the related figure, we observe that there exist two branches of the spectrum in accordance with Theorem \ref{4: th:two - r,d general GM1} with $r=2$. In fact, when a spectral symbol is $r \times r$ matrix-valued, in according with definition~\ref{99}, we observe $r$ branches in the spectra of the associated matrix-sequences.

\subsection{\texorpdfstring{Cubic C$^{1}$ B-Spline discretization}{Cubic C\string^1 B-Spline Discretization}}

In the cubic C$^{1}$ B-spline discretization on a uniform mesh with step size $\frac{1}{n}$, the basis functions $\phi_1, \dots, \phi_{N_n}$ are chosen as the B-splines; see \cite{tom}. The resulting normalized stiffness and mass matrices are given by
\begin{equation}\label{116}
\includegraphics[width=0.7\textwidth]{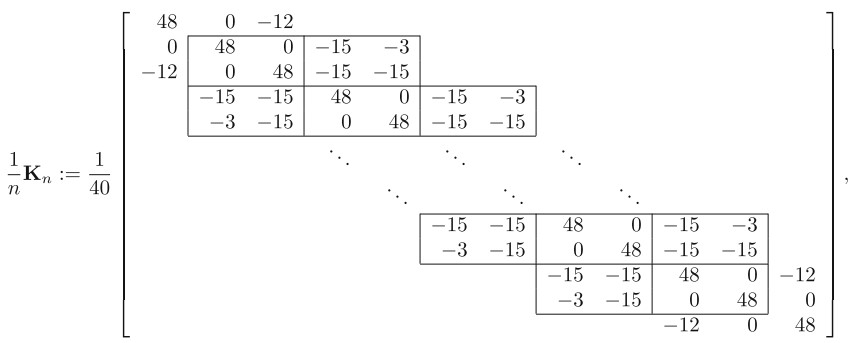} 
\end{equation}
\begin{equation}\label{117}
\includegraphics[width=0.7\textwidth]{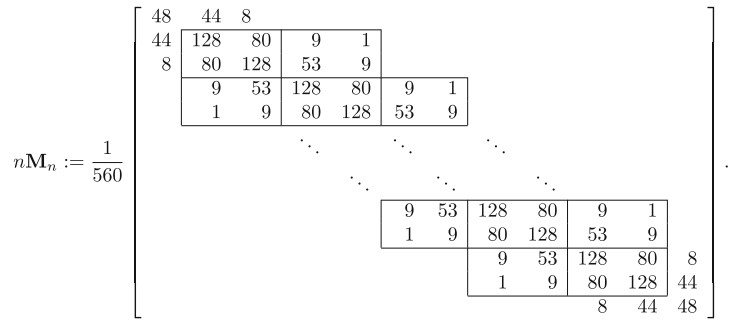} 
\end{equation}
According to the notation in Section~\ref{blcktop}  with $d=1$ and $r=2$, the stiffness matrix \(\frac{1}{n} K_n\) and the mass matrix \(n M_n\) contain, as principal submatrices, the unilevel $2$-block Toeplitz matrices \(T_{n-1}(f)\) and \(T_{n-1}(h)\), where the generating functions \(f(\theta)\) and \(h(\theta)\) are $2 \times 2$ matrix-valued functions given by
\[
f(\theta) := \frac{1}{40} 
\left(
\begin{bmatrix}
-15 & -15 \\
-3 & -15
\end{bmatrix}
e^{i\theta}
+
\begin{bmatrix}
48 & 0 \\
0 & 48
\end{bmatrix}
+
\begin{bmatrix}
-15 & -3 \\
-15 & -15
\end{bmatrix}
e^{-i\theta}
\right),
\]
\[
f(\theta) = \frac{1}{40}
\begin{bmatrix}
48-30\cos(\theta)  & -15 -3e^{-i\theta} \\
-3e^{i\theta} - 15e^{-i\theta} & 48 - 30\cos \theta
\end{bmatrix},
\]
and 
\[
h(\theta) := \frac{1}{560} 
\left(
\begin{bmatrix}
9 & 53 \\
1 & 9
\end{bmatrix}
e^{i\theta}
+
\begin{bmatrix}
128 & 80 \\
80 & 128
\end{bmatrix}
+
\begin{bmatrix}
9 & 1 \\
53 & 9
\end{bmatrix}
e^{-i\theta}
\right),
\]
\[
h(\theta) = \frac{1}{560}
\begin{bmatrix}
128+18\cos(\theta)  & 80 + 53e^{i\theta}+e^{-i\theta} \\
80 + 53e^{-i\theta}+e^{i\theta} & 128+18\cos(\theta)
\end{bmatrix}.
\]
As discussed in Section \ref{107}, the same reasoning applies to the matrix-sequence  \( \{L_n\}_n\) with
\[
L_n := \frac{1}{n}K_n + nM_n,
\]
which results in the GLT sequence symbol \( e(\theta) = f(\theta) + h(\theta) \).
The numerical evidence is again strong and we can see from the related figure that there exist two branches of the spectrum in accordance with Theorem \ref{4: th:two - r,d general GM1} with $r=2$.

\begin{figure}[H]
   \begin{minipage}{0.52\textwidth}
     \centering
     \includegraphics[width=\linewidth, height=6cm]{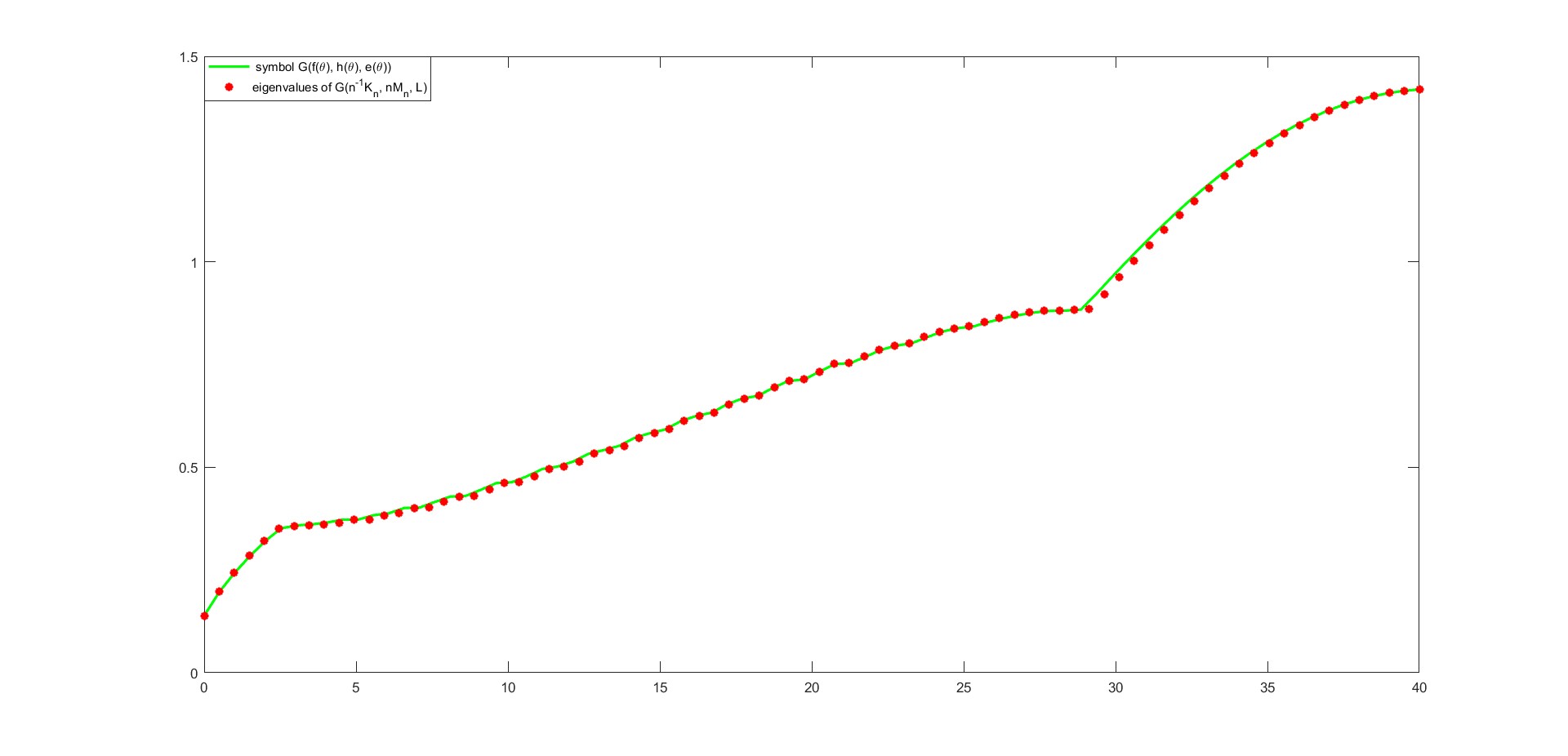} 
     \captionsetup{labelformat=empty}
     \caption*{}
   \end{minipage}\hfill
   \begin{minipage}{0.52\textwidth}
     \centering
     \includegraphics[width=\linewidth, height=6cm]{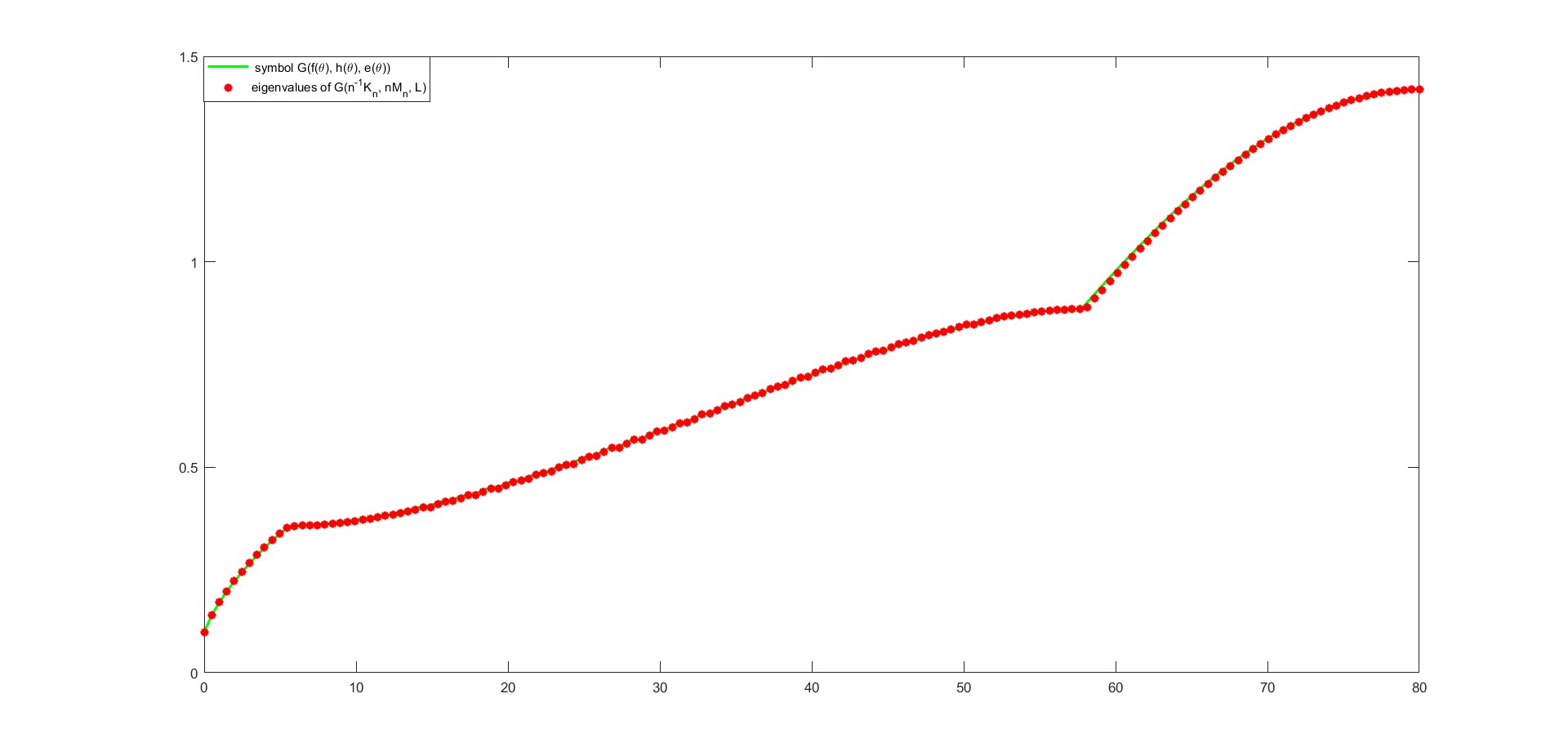} 
     \captionsetup{labelformat=empty}
     \caption*{}
   \end{minipage}
    \caption{{\textbf{For Cubic case: $r=2, d=1$ .}
Comparison between eigenvalues of \( (G(n^{-1}K_{n}, nM_{n}, L_n)) \) (red stars; \(n=40,80\))
versus the symbol \( G(f(\theta),h(\theta),e(\theta)) \) (blue circles)}}
\label{fig:exp6(1D)}
\end{figure}

\subsection{Minimal eigenvalues and conditioning}\label{ssec: extremal}

In the last part of the current numerical section, we consider the problem of understanding the extremal behavior of the spectrum related to the conditioning of
$\{G(A^{(1)}_{n},A^{(2)}_{n},A^{(3)}_{n})\}_n$ in dependence of the analytical features of the corresponding GLT symbol. The idea is borrowed from the literature, where we find papers dealing with the extremal eigenvalues in a Toeplitz setting \cite{extreme1,extreme2,extreme3,extreme4}, in a $r$-block Toeplitz setting \cite{SerraCapizzano1999a,SerraCapizzano1999b}, in a differential setting \cite{extreme diff1,extreme diff2}, including multilevel cases of $d>1$.
Here we restrict our attention to the unilevel scalar setting with $d=r=1$ and we consider again the examples in Section~\ref{30} and in Section~\ref{32}.

\subsubsection{Example 1 (1D): minimal eigenvalue}\label{exp1 (1D)}

We consider the example in Section \ref{30}. The generating function $\kappa(x,\theta)=(2 - 2 \cos(\theta))^2$ of $A_n$ has a unique zero of order $4$ at $\theta=0$, while the matrix-sequence  $\{B_{n}\}_n$ shows the GLT symbol $\xi(x,\theta))=x(2-2\cos(\theta))^2$ with combinations of zeros at $x=0$ of order $1$ and at $\theta=0$ of order $4$. According to the theory, the minimal eigenvalue of $A_n$ is positive and tends to zero as $n^{-4}$ (see \cite{extreme1,extreme2,extreme3,extreme4}). On the other hand, according to \cite{SerraCapizzano1999b,extreme diff1}, we know that  eigenvalue of $B_n$ is positive and tends to zero as $n^{-4}$. 
If we consider the geometric mean of the two symbols, we deduce that it has again zeros of order at most $4$: as a consequence, heuristically we expect that the minimal eigenvalue of $X_{n}=G(A_{n},B_{n})$ is positive and tends to zero as $n^{-4}$ as it is perfectly verified in the table below with $\alpha_j$ tending to $4$.

\begin{itemize}
    \item $X_{n}=G(A_{n},B_{n})$
    \item Take $n_{j}=40*2^{j}, \quad\quad j=0,1,2,3, $
    \item Compute $\tau_{j}= \lambda_{\min}(X_{n_{j}}), \quad \quad j=0,1,2,3, $
    \item Compute $\alpha_{j}=\log_{2}\left(\frac{\tau_{j}}{\tau_{j+1}}\right), \quad \quad j=0,1,2. $ 
\end{itemize}
\vspace{9pt}

\begin{table}[H]
  \centering
  \caption{Numerical behaviour of the minimal eigenvalue \\of Example 1, Section \ref{exp1 (1D)}}
  \label{tab:tau_alpha2}
 
  \begin{tabular}{|c |S[table-format=1.4e+4] |S[table-format=1.4]|}
    \hline
    & {$\tau_j$} & {$\alpha_j$}\\
    \hline
  $n_{0}=40$ &0.00010994 & 3.8698  \\
  \hline
  $n_{1}=80$&0.00000752  & 3.9398  \\
  \hline
  $n_{2}=160$ &0.00000049  &4.0297   \\
  \hline
  $n_{3}=320$ &0.00000003  &    \\
  \hline
  \end{tabular}
\end{table}

\subsubsection{Example 3 (1D): minimal eigenvalue}\label{exp3 (2D)}

We consider the example in Section~\ref{32}. The generating function $\kappa_1$ of $A_n^{(1)}$ is strictly positive, the generating function $\kappa_2(x,\theta)=x^2$ of $A_n^{(2)}$ has a unique zero of order $2$ at $x=0$, while the matrix-sequence  $\{A^{(3)}_{n}\}_n$ shows the GLT symbol $\kappa_3(x,\theta))=x^4(2-2\cos(\theta))^4$ with combinations of zeros at $x=0$ of order $4$ and at $\theta=0$ of order $4$. According to the theory, the minimal eigenvalue of $A_n^{(1)}$ is positive and tends to $1=\min \kappa_1$ as (see \cite{extreme1,extreme2,extreme3,extreme4}). On the other hand, according to \cite{extreme diff1,extreme diff2}, we know that  eigenvalue of $A^{(3)}_{n}$ is positive and tends to zero as $n^{-4}$, while for 
$A^{(2)}_{n}$ it is trivial tp check that the minimal eigenvalue tends to zero as $n^{-2}$. 
If we consider the geometric mean of the two symbols, we observe that it has zeros of order at most $2$: hence, heuristically we expect that the minimal eigenvalue of $X_{n}=G(A^{(1)}_{n},A^{(2)}_{n},A^{(3)}_{n})$ is positive and tends to zero as $n^{-2}$ as it is perfectly verified in the table below with $\alpha_j$ tending to $2$.

\begin{itemize}
\item $X_{n}=G(A^{(1)}_{n},A^{(2)}_{n},A^{(3)}_{n})$
    \item Take $n_{j}=40*2^{j}, \quad\quad j=0,1,2,3, $
    \item Compute $\tau_{j}= \lambda_{\min}(X_{n_{j}}), \quad \quad j=0,1,2,3, $
    \item Compute $\alpha_{j}=\log_{2}\left(\frac{\tau_{j}}{\tau_{j+1}}\right), \quad \quad j=0,1,2. $ 
\end{itemize}
\vspace{9pt}
\begin{table}[H]
  \centering
  \caption{Numerical behaviour of the minimal eigenvalue\\ of Example 3, Section \ref{exp3 (2D)}}
  \label{tab:tau_alpha3}
  \begin{tabular}{|c
                  |S[table-format=1.4e+2]
                  |S[table-format=1.4]|}
                  \hline
    & {$\tau_j$} & {$\alpha_j$}\\
    \hline
 $n_{0}=40$ &0.0014 & 2.2223  \\
  \hline
  $n_{1}=80$&0.00034797  & 2  \\
  \hline
  $n_{2}=160$ &0.000086993  & 2  \\
  \hline
  $n_{3}=320$ & 0.000021748 &    \\
  \hline
  \end{tabular}
\end{table}

\chapter{Maximal results for geometric means of GLT sequences}\label{GM GLT2}


Building upon the results established in Chapter~\ref{GM GL1}, this chapter presents and discusses the main contributions of our recent work~\cite{GM2_GLT2025}. In the previous chapter, we analyzed the spectral distribution of the geometric mean for two HPD GLT matrix-sequences, proving that when the GLT symbols $\kappa$ and $\xi$ are nonzero almost everywhere (or at least one is invertible a.e.), the geometric mean sequence remains GLT, with symbol $(\kappa\xi)^{1/2}$ in the scalar or commuting case, and more generally $G(\kappa,\xi)=\kappa^{1/2}\left(\kappa^{-1/2}\xi\kappa^{-1/2}\right)^{1/2}\kappa^{1/2}$ in the noncommutative case. In this chapter, we address the open question of whether these assumptions can be relaxed. Motivated by the observation that invertibility is primarily required due to non-commutativity, we rigorously prove that when the GLT symbols $\kappa$ and $\xi$ commute, the invertibility condition can be completely removed in the case where $r=1$ and $d \geq 1$.

Furthermore, we investigate the limitations of this maximal result through both theoretical analysis and extensive numerical experiments. We show that, in the general block case ($r > 1$), if the GLT symbols $\kappa$ and $\xi$ are both non-invertible and do not commute, the geometric mean sequence $\{G(A_n, B_n)\}_n$ still forms a GLT matrix-sequence with some symbol $\psi$; however, this symbol may not coincide with the classical geometric mean $G(\kappa, \xi)$, or, in some cases, $G(\kappa, \xi)$ may not be well defined. These findings demonstrate that our main results are indeed maximal: relaxing either the commutativity or invertibility assumptions can lead to fundamentally different or ill-posed spectral behavior.

A significant advance in this paper is the introduction and use of momentary GLT symbols (see \cite{bolten2022toeplitz,bolten2023note,new momentary}), which provide a finer description of the spectral distribution, especially for finite-size matrices and cases where the product $\kappa\xi$ is zero or the symbols are rank-deficient. Momentary symbols incorporate the effects of small-norm perturbations, allowing the observed richer spectral structures such as secondary eigenvalue distributions in numerical experiments to be explained within an extended GLT framework.

The work further investigates the extremal spectral behavior of the geometric mean, with particular focus on the asymptotic rate at which the minimal eigenvalues decay as the matrix size increases. Through detailed numerical experiments, we observe that the decay rate of the minimal eigenvalue is intimately connected to the order of vanishing (i.e., the multiplicity of zeros) of the associated GLT symbol.

\section{Geometric mean of GLT matrix-sequences}\label{sec GM}

We start by recalling Theorem~\ref{4: th:two - r=d=1 GM1} from Chapter~\ref{GM GL1}, which corresponds to \cite[Theorem 4]{ahmad2025matrix} and generalizes \cite[Theorem 10.2]{garoni2017}.

\begin{theorem}[\!\!\cite{ahmad2025matrix}, Theorem 4]\label{th:two - r=1,d general}
Let $r=1$ and $d\ge 1$.  Suppose \(\{A_{\bm{n}}\}_{\bm{n}} \sim_{\mathrm{GLT}} \kappa\) and \(\{B_{\bm{n}}\}_{\bm{n}} \sim_{\mathrm{GLT}} \xi\), where \(A_{\bm{n}}, B_{\bm{n}} \in \mathcal{P}_{\nu(\bm{n})}\) for every multi-index $\bm{n}$,{ with $\mathcal{P}_{\bm{(n)}}$ denoting the set of Hermitian positive definite (HPD) matrices of multi-index $\bm{n}$}. Assume that at least one of between \(\kappa\) and \(\xi\) is nonzero almost everywhere. Then
\begin{equation}\label{r=1, d general GLT}
\{G(A_{\bm{n}}, B_{\bm{n}})\}_{\bm{n}} \sim_{\mathrm{GLT}}(\kappa \xi)^{1/2},
\end{equation}
and
\begin{equation}\label{r=1, d general distributions}
\{G(A_{\bm{n}}, B_{\bm{n}})\}_{\bm{n}} \sim_{\sigma, \lambda} (\kappa \xi)^{1/2}.
\end{equation}
\end{theorem}

Using the topological a.c.s. notion, we now show that the thesis of the previous theorem holds without assuming that at least one between $\kappa$ and $\xi$ is invertible almost everywhere (a.e.), thus proving \cite[Conjecture 10.1]{garoni2017}. We make use of Axioms \textbf{GLT 1}, \textbf{GLT 2}, \textbf{GLT 3}, \textbf{GLT 4}, \textbf{GLT 5}, \textbf{GLT 6} as reported in Section~\ref{ast algebra}.

\begin{theorem}\label{theorem 1}
Let $r=1$ and $d\ge 1$. 
Assume \(\{A_{\bm{n}}\}_{\bm{n}} \sim_{\mathrm{GLT}} \kappa\) and \(\{B_{\bm{n}}\}_{\bm{n}}\sim_{\mathrm{GLT}} \xi\), where \(A_{\bm{n}}, B_{\bm{n}} \in \mathcal{P}_{\nu(\bm{n})}\) for every multi-index $\bm{n}$.  Then
\[
\{G(A_{\bm{n}}, B_{\bm{n}})\}_{\bm{n}} \sim_{\mathrm{GLT}} (\kappa \xi)^{1/2},
\]
\[
\{G(A_{\bm{n}}, B_{\bm{n}})\}_{\bm{n}} \sim_{\sigma, \lambda} (\kappa \xi)^{1/2}.
\]
\end{theorem}

\begin{proof}
We have
\begin{equation}\label{kappa equ}
    \{A_{\bm{n}}\}_{\bm{n}} \sim_{\mathrm{GLT}} \kappa,
\end{equation}
\begin{equation}\label{eta equ}
    \{B_{\bm{n}}\}_{\bm{n}} \sim_{\mathrm{GLT}} \xi,
\end{equation}
where $\kappa, \xi: [0,1]^d \times [-\pi,\pi]^d \to \mathbb{C}$ and $A_{\bm{n}}, B_{\bm{n}}$ are Hermitian positive definite matrices $\forall \bm{n}$. 
Therefore $\kappa, \xi: [0,1]^d \times [-\pi,\pi]^d \to \mathbb{R}^+_0$ almost everywhere: furthermore, we assume that $ \mu_{2d}(\kappa \equiv 0) > 0$ and $\mu_{2d}(\xi \equiv 0) > 0$, so that the hypotheses of \cite[Theorem 10.2]{garoni2017} for $d=1$ and of Theorem \ref{th:two - r=1,d general} for $d>1$ are violated. Let $\varepsilon > 0$, let $\kappa_\varepsilon = \kappa + \varepsilon$ and let $A_{\bm{n}, \varepsilon} = A_{\bm{n}} + \varepsilon I_{{\nu}(\bm{n})},$ where $I_{\nu(\bm{n})}$ denotes the identity matrix of size $\nu(\bm{n})$. By the first part of Axiom \textbf{GLT 2} (the identity is a special Toeplitz matrix with GLT symbol $1$), we have  $\{ \varepsilon I_{{\nu}(\bm{n})}=T_{\bm{n}}(\varepsilon) \}_{\bm{n}} \sim_{\mathrm{GLT}}\varepsilon,$ where $T_{\bm{n}}(\varepsilon)$ denotes the Toeplitz matrix generated by the constant function $\varepsilon$. Now, by exploiting linearity, i.e. the second item of Axiom \textbf{GLT 3}, it follows that
\[ 
\{ A_{\bm{n}, \varepsilon} \}_{\bm{n}} \sim_{\mathrm{GLT}} \kappa_\varepsilon, \quad \kappa_\varepsilon \geq \varepsilon \text{ almost everywhere}, 
\]
since $\kappa, \xi \geq 0$ almost everywhere due to \eqref{kappa equ} and \eqref{eta equ}.
Hence, we are again in the framework of Theorem \ref{th:two - r=1,d general}. Therefore, by  Theorem \ref{th:two - r=1,d general} we conclude that the sequence of geometric means satisfies the GLT relation
\[ 
\{G(A_{\bm{n},\varepsilon}, B_{\bm{n}}) \}_{\bm{n}} \sim_{\mathrm{GLT}} (\kappa_{\varepsilon} \xi)^{\frac{1}{2}}.
\]
Furthermore, we apply the full $*$-algebra framework of the GLT matrix-sequences as in \cite[Section 2.9]{ahmad2025matrix}. More precisely, by third item of Axiom \textbf{GLT 3} followed by Axiom \textbf{GLT 6} with the function $f(z)=z^{\frac{1}{4}}$, we obtain
\[ \{ (A_{\bm{n}, \varepsilon} B_{\bm{n}}^{2} A_{\bm{n}, \varepsilon})^{\frac{1}{4}} \}_{\bm{n}} \sim_{\mathrm{GLT}} (\kappa_\varepsilon^2 \xi^{2})^{\frac{1}{4}}=(\kappa_\varepsilon \xi)^{\frac{1}{2}}, \]
which the very same GLT symbol of $ \{G(A_{\bm{n},\varepsilon}, B_{\bm{n}}) \}_{\bm{n}} $. As a result, again by the second item of Axiom \textbf{GLT 3}, the difference between these sequences satisfies the following asymptotic GLT relation
\begin{equation}\label{rw inversion}
\{G(A_{\bm{n}, \varepsilon}, B_{\bm{n}}) - (A_{\bm{n}, \varepsilon} B_{\bm{n}}^{2} A_{\bm{n}, \varepsilon})^{\frac{1}{4}} \}_{\bm{n}} \sim_{{\mathrm{GLT}} } 0.
\end{equation}
The previous relation is the key step since we have found a new GLT matrix-sequence having the same GLT symbol as the geometric mean 
matrix-sequence $\{G(A_{\bm{n}, \varepsilon}, B_{\bm{n}})\}_{\bm{n}}$, but where no inversion is required.
Now given the structure of the performed operations, by virtue of Definition~\ref{def:acs}, the class $$ \{\{ (A_{\bm{n}, \varepsilon} B_{\bm{n}}^{2} A_{\bm{n}, \varepsilon})^{\frac{1}{4}} \}_{\bm{n}},  \varepsilon>0\}$$ is obviously an a.c.s. for $\{ (A_{\bm{n}} B_{\bm{n}}^2 A_{\bm{n}})^{\frac{1}{4}} \}_{\bm{n}} $. As a consequence, by invoking (\ref{rw inversion}), also the class $ \{\{ G(A_{\bm{n}, \varepsilon}, B_{\bm{n}}) \}_{\bm{n}},  \varepsilon>0\} $
is an a.c.s for $\{ (A_{\bm{n}} B_{\bm{n}}^2 A_{\bm{n}})^{\frac{1}{4}} \}_{\bm{n}}, $
with all the involved sequences being GLT matrix-sequences, with symbols $ (\kappa_\varepsilon \xi)^{\frac{1}{2}}, (\kappa\xi)^{\frac{1}{2}}$, $\varepsilon>0$.
Therefore, since  $ \exists \lim_{\varepsilon \to 0} (\kappa_\varepsilon \xi)^{\frac{1}{2}}= (\kappa\xi)^{\frac{1}{2}}$, by using the powerful Theorem~\ref{th: acs}, we deduce that $ \exists \lim_{\varepsilon \to 0} \text{ (a.c.s.) } \{\{ G(A_{\bm{n}, \varepsilon}, B_{\bm{n}}) \}_{\bm{n}}, \varepsilon \}= \{ (A_{\bm{n}} B_{\bm{n}}^2 A_{\bm{n}})^{\frac{1}{4}} \}_{\bm{n}} \sim_{\mathrm{GLT}} (\kappa\xi)^{\frac{1}{2}}$ and this limit coincides with $ \{ G(A_{\bm{n}}, B_{\bm{n}}) \}_{\bm{n}}$ in the a.c.s. topology, that is, 
$\{G(A_{\bm{n}}, B_{\bm{n}}) - (A_{\bm{n}} B_{\bm{n}}^2 A_{\bm{n}})^{\frac{1}{4}} \}_{\bm{n}}\sim_{{\mathrm{GLT}}} 0.$
Finally $\{(A_{\bm{n}} B_{\bm{n}}^2 A_{\bm{n}})^{\frac{1}{4}} \}_{\bm{n}}\sim_{{\mathrm{GLT}}} (\kappa\xi)^{\frac{1}{2}}$. Hence
\[
\{G(A_{\bm{n}}, B_{\bm{n}})\}_{\bm{n}} \sim_{\mathrm{GLT}} (\kappa \xi)^{1/2},
\]
so that, by Axiom \textbf{GLT 1}, we infer
\[
\{G(A_{\bm{n}}, B_{\bm{n}})\}_{\bm{n}} \sim_{\sigma, \lambda} (\kappa \xi)^{1/2}.
\]
\end{proof}
\begin{remark}[Intuition on the generalization]
\label{gen vs barrier}
One of the basic but key ingredients of the previous proof is that scalar-valued functions commute. Hence, it is reasonable to expect that the same proof also works in the case of $d$-level $r$-block GLT matrix-sequences, when assuming that the symbols commute. We collect the result in Theorem \ref{theorem 2} whose proof follows verbatim that of Theorem \ref{theorem 1}, except for minimal changes. On the other hand, when both $\kappa$ and $\xi$ are not invertible almost everywhere (degenerate) and do not commute, the expression $G(\kappa,\xi)$ is not well defined. In fact, we could replace the inversion with the standard pseudoinversion, which is denoted as $X^+$ if $X$ is any complex matrix. However, we stress that the expression
$\kappa^{\frac{1}{2}} ([\kappa^{\frac{1}{2}}]^+ \xi [\kappa^{\frac{1}{2}}]^+)^{\frac{1}{2}} \kappa^{\frac{1}{2}}$ is not the same as $\xi^{\frac{1}{2}} ([\xi^{\frac{1}{2}}]^+ \kappa [\xi^{\frac{1}{2}}]^+)^{\frac{1}{2}} \xi^{\frac{1}{2}}$ in general and this is a serious indication that the commutation between the GLT symbols is essential, when the symbols are both degenrate.
\end{remark}

We now recall the result proven in \cite{ahmad2025matrix}, concerning the general GLT setting with $d,r \ge 1$ and when at least one of the involved GLT symbols is invertible almost everywhere.

\begin{theorem}[\!\!\cite{ahmad2025matrix}, Theorem 5]\label{th:two - r,d general}
Let $r,d\ge 1$. 
Suppose \(\{A_{\bm{n}}\}_{\bm{n}} \sim_{\mathrm{GLT}} \kappa\) and \(\{B_{\bm{n}}\}_{\bm{n}} \sim_{\mathrm{GLT}} \xi\), where \(A_{\bm{n}}, B_{\bm{n}} \in \mathcal{P}_{\nu(\bm{n})}\) for every multi-index $\bm{n}$. Assume that at least one of the minimal eigenvalues of \(\kappa\) and the minimal eigenvalue of \(\xi\) is nonzero almost everywhere. Then
\begin{equation}\label{r,d general GLT}
\{G(A_{\bm{n}}, B_{\bm{n}})\}_{\bm{n}} \sim_{\mathrm{GLT}}G(\kappa, \xi),
\end{equation}
and
\begin{equation}\label{r,d general distributions}
\{G(A_{\bm{n}}, B_{\bm{n}})\}_{\bm{n}} \sim_{\sigma, \lambda} G(\kappa,\xi).
\end{equation}
Furthermore $G(\kappa,\xi)=(\kappa \xi)^{1/2}$ whenever the GLT symbols  \(\kappa\) and \(\xi\) commute.
\end{theorem}

\begin{theorem}\label{theorem 2}
Let $r>1$ and $d\ge 1$. 
Assume that  \(\{A_{\bm{n}}\}_{\bm{n}} \sim_{\mathrm{GLT}} \kappa\) and \(\{B_{\bm{n}}\}_{\bm{n}} \sim_{\mathrm{GLT}} \xi\), where \(A_{\bm{n}}, B_{\bm{n}} \in \mathcal{P}_{\nu(\bm{n})}\) for every multi-index $\bm{n}$, with $\mathcal{P}_{\nu\bm{(n)}}$ denoting the set of Hermitian positive definite (HPD) matrices of size $r\bm{n}$. Under the assumption that $\kappa$ and $\xi$ commute we infer 
\[
\{G(A_{\bm{n}}, B_{\bm{n}})\}_{\bm{n}} \sim_{\mathrm{GLT}} (\kappa \xi)^{1/2},
\]
\[
\{G(A_{\bm{n}}, B_{\bm{n}})\}_{\bm{n}} \sim_{\sigma, \lambda} (\kappa \xi)^{1/2},
\]
\end{theorem}
\begin{proof}
Since $\{A_{\bm{n}}\}_{\bm{n}} \sim_{\mathrm{GLT}} \kappa$, $\{B_{\bm{n}}\}_{\bm{n}} \sim_{\mathrm{GLT}} \xi$, we deduce that $\kappa, \xi: [0,1]^d \times [-\pi,\pi]^d \to \mathbb{C}^{r\times r}$ are Hermitian nonnegative definite, while, by the assumptions, $A_{\bm{n}}, B_{\bm{n}}$ are Hermitian positive definite matrices $\forall \bm{n}$. 
Therefore, the minimal eigenvalue of $\kappa$ and the minimal eigenvalues of $\xi$ are nonnegative almost everywhere. Here, we assume that $ \mu_{2d}(\lambda_{\min}(\kappa) \equiv 0) > 0$ and $\mu_{2d}(\lambda_{\min}(\xi) \equiv 0) > 0$, in such a way that the hypotheses of Theorem \ref{th:two - r,d general} are violated. 
Let $\varepsilon > 0$, let $\kappa_\varepsilon = \kappa + \varepsilon I_r$, $I_r$ being the identity of size $r$, and let $A_{\bm{n}, \varepsilon} = A_{\bm{n}} + \varepsilon I_{\mathcal{P}(\bm{n})}$. Now $\{ \varepsilon I_{\mathcal{P}(\bm{n})}=T_{\bm{n}}(\varepsilon I_r) \}_{\bm{n}} \sim_{\mathrm{GLT}} \varepsilon I_r$ by the first part of Axiom \textbf{GLT 2} (the identity $I_{r{\nu}(\bm{n})}$ is a special multilevel block Toeplitz matrix with GLT symbol $I_r$). Consequently, by exploiting linearity i.e. the second item of Axiom \textbf{GLT 3}, it follows that
\[ 
\{ A_{\bm{n}, \varepsilon} \}_{\bm{n}} \sim_{\mathrm{GLT}} \kappa_\varepsilon, \quad \kappa_\varepsilon \geq \varepsilon I_r \text{ almost everywhere.} 
\]
Since $\kappa, \xi$ are both nonnegative definite almost everywhere.
Hence, by  Theorem \ref{th:two - r,d general}, we have $\{G(A_{\bm{n},\varepsilon}, B_{\bm{n}}) \}_{\bm{n}} \sim_{\mathrm{GLT}} (\kappa_{\varepsilon} \xi)^{\frac{1}{2}}$, 
since the commutation between $\kappa$ and $\xi$ implies the commutation between $\kappa_\varepsilon = \kappa + \varepsilon I_r$ and $\xi$. 
By exploiting the $*$-algebra features of the GLT matrix-sequences, and specifically the third item of Axiom \textbf{GLT 3} and Axiom \textbf{GLT 6} with the function $f(z)=z^{\frac{1}{4}}$, we deduce
$ \{ (A_{\bm{n}, \varepsilon} B_{\bm{n}}^{2} A_{\bm{n}, \varepsilon})^{\frac{1}{4}} \}_{\bm{n}} \sim_{\mathrm{GLT}} (\kappa_\varepsilon^2 \xi^{2})^{\frac{1}{4}}=(\kappa_\varepsilon \xi)^{\frac{1}{2}}$, 
which is the very same GLT symbol of $ \{G(A_{\bm{n},\varepsilon}, B_{\bm{n}}) \}_{\bm{n}}$. As a result, again by the second item of Axiom \textbf{GLT 3}, the difference between these sequences satisfies the following asymptotic relation
\begin{equation}\label{reducing without inversion}
\{G(A_{\bm{n}, \varepsilon}, B_{\bm{n}}) - (A_{\bm{n}, \varepsilon} B_{\bm{n}}^{2} A_{\bm{n}, \varepsilon})^{\frac{1}{4}} \}_{\bm{n}} \sim_{{\mathrm{GLT}} } 0.
\end{equation}
In other words, we have written a new GLT matrix-sequence having the same symbol as the geometric mean 
matrix-sequence $\{G(A_{\bm{n}, \varepsilon}, B_{\bm{n}})\}_{\bm{n}}$, but where no inversion is required.
By virtue of Definition~\ref{def:acs} , the class $ \{\{ (A_{\bm{n}, \varepsilon} B_{\bm{n}}^{2} A_{\bm{n}, \varepsilon})^{\frac{1}{4}} \}_{\bm{n}},  \varepsilon\\>0\}$ is an a.c.s. for $\{ (A_{\bm{n}} B_{\bm{n}}^2 A_{\bm{n}})^{\frac{1}{4}} \}_{\bm{n}} $.\\ Therefore, by (\ref{reducing without inversion}), also the class $ \{\{ G(A_{\bm{n}, \varepsilon}, B_{\bm{n}}) \}_{\bm{n}},  \varepsilon>0\} $
is an a.c.s for $\{ (A_{\bm{n}} B_{\bm{n}}^2 A_{\bm{n}})^{\frac{1}{4}} \}_{\bm{n}}, $
with all the involved sequences being GLT matrix-sequences, with symbols $ (\kappa_\varepsilon \xi)^{\frac{1}{2}}, (\kappa\xi)^{\frac{1}{2}}$, $\varepsilon>0$. Since  $ \exists \lim_{\varepsilon \to 0} (\kappa_\varepsilon \xi)^{\frac{1}{2}}= (\kappa\xi)^{\frac{1}{2}}$, Theorem~\ref{th: acs}  implies $ \exists \lim_{\varepsilon \to 0} \text{ (a.c.s.) } \{\{ G(A_{\bm{n}, \varepsilon}, B_{\bm{n}}) \}_{\bm{n}}, \varepsilon \}= \{ (A_{\bm{n}} B_{\bm{n}}^2 A_{\bm{n}})^{\frac{1}{4}} \}_{\bm{n}} \sim_{\mathrm{GLT}} (\kappa\xi)^{\frac{1}{2}}$. It is now clear that this limit coincides with $ \{ G(A_{\bm{n}}, B_{\bm{n}}) \}_{\bm{n}} $
in the a.c.s. topology, that is, 
$\{G(A_{\bm{n}}, B_{\bm{n}}) - (A_{\bm{n}} B_{\bm{n}}^2 A_{\bm{n}})^{\frac{1}{4}} \}_{\bm{n}}\sim_{{\mathrm{GLT}}} 0.$
In this manner $\{(A_{\bm{n}} B_{\bm{n}}^2 A_{\bm{n}})^{\frac{1}{4}} \}_{\bm{n}}\\ \sim_{{\mathrm{GLT}}} (\kappa\xi)^{\frac{1}{2}}$. Consequently
$\{G(A_{\bm{n}}, B_{\bm{n}})\}_{\bm{n}} \sim_{\mathrm{GLT}} (\kappa \xi)^{1/2}$, so that, by Axiom \textbf{GLT 1}, we finally obtain
$\{G(A_{\bm{n}}, B_{\bm{n}})\}_{\bm{n}} \sim_{\sigma, \lambda} (\kappa \xi)^{1/2}$.
\end{proof}
\section{Numerical experiments}\label{Num_Exp}

As is well known, many localization, extremal, and distribution results hold when $d$-level, $r$-block Toeplitz matrix-sequences are considered and these results are somehow summarized in specific analytic features of the generating function of the corresponding matrix-sequence. In turn, the generating function is also the $d$-variate, $r\times r$ matrix-valued GLT symbol of the $d$-level, $r$-block Toeplitz matrix-sequence. 

Although the distribution results are also valid for general $d$-level, $r$-block GLT matrix-sequences, this is no longer true in general for the extremal behavior and for the localization results, unless we add supplementary assumptions, like the request that the matrix-sequence is obtained via a matrix-valued linear positive operator (LPO) \cite{SerraCapizzanoTilli-LPO,LPO-rev}. We observe that the geometric mean can be seen as a monotone operator with respect to its two variables, and the monotonicity is implied when we consider an LPO, even if the converse is not true. Hence, it is also interesting to verify which properties are maintained by a geometric mean of two $d$-level, $r$-block GLT matrix-sequences in terms of its GLT symbol when it exists.

According to the previous discussion, the remainder of the section considers numerical experiments in the following directions:

\begin{itemize}
\item the validation of the distribution results in the commuting setting as in Theorem~\ref{theorem 1} and Theorem~\ref{theorem 2}. Notice that the commutation of the GLT symbols do not imply that the matrices commute; in general, in fact, they do not commute.
\item connections of the previous results with the notion of Toeplitz and GLT momentary symbols and the extremal behavior compared to the GLT symbol; 
\item evidence that Theorem~\ref{theorem 1} and Theorem~\ref{theorem 2} are maximal, by taking GLT matrix-sequences with non-commuting symbols which are both not invertible almost everywhere. 
\end{itemize}
  All numerical experiments are performed using MATLAB R2022b on a laptop equipped with an 11th Gen Intel(R) Core(TM) i5-1155G7 CPU running at 2.50 GHz, with 16 GB of RAM. The operating system was Windows 11 Pro (version 23H2, build 22631.5189).
\subsection{Validation of the distribution results} \label{ssec:num1}
  
\subsubsection{Example 1 }
Let $ d = 1 $ and consider the following two matrix-sequences $\{A_{n}\}_{n}$, $\{B_{n}\}_{n}$, with $
A_n = D_n(a) + \frac{1}{n^4} I_n$ and  $ B_n = T_n(3 + 2\cos(\theta)),$
 where $T_n(\cdot) $ denotes the Toeplitz operator for $d=1 $, as introduced in Section~\ref{TM},  $D_n(a) $ represents the diagonal matrix generated by the continuous function
\[
a(x) =
\begin{cases}
0, & x \in [0, \frac{1}{2}), \\
1, & x \in [\frac{1}{2},1].
\end{cases}
\]
Because of Theorem~\ref{theorem 1}, the geometric mean of these two sequences, that is $\{ G(A_n, B_n) \} $, satisfies the GLT relation:
\[
\left\{ G\left(A_n, B_n\right) \right\}_n \sim_{\mathrm{GLT}} \sqrt{a(x)(3+2\cos(\theta))},
\]
where the corresponding GLT symbols are explicitly given by
$
\kappa = a(x)$ and $ \xi=3+2\cos(\theta).
$
\subsubsection*{Eigenvalue distribution}
We numerically analyze the spectral behavior of the geometric mean $G(A_n, B_n) $ from \textbf{Example 1}. The eigenvalues of this geometric mean are computed for various increasing dimensions $n $ and compared with the uniformly sampled points from the GLT symbol $\sqrt{a(x)(3+2\cos(\theta))} $. In Figure \ref{fig:exp1}, numerical results strongly indicate that the GLT symbol accurately characterizes the eigenvalue distribution of the geometric mean. As $ n $ grows, the eigenvalue distribution closely matches the GLT symbol, giving evidence of the theoretical findings in Theorem~\ref{theorem 1}.
\begin{figure}[H]
   \begin{minipage}{0.52\textwidth}
     \centering
     \includegraphics[width=\linewidth, height=6cm]{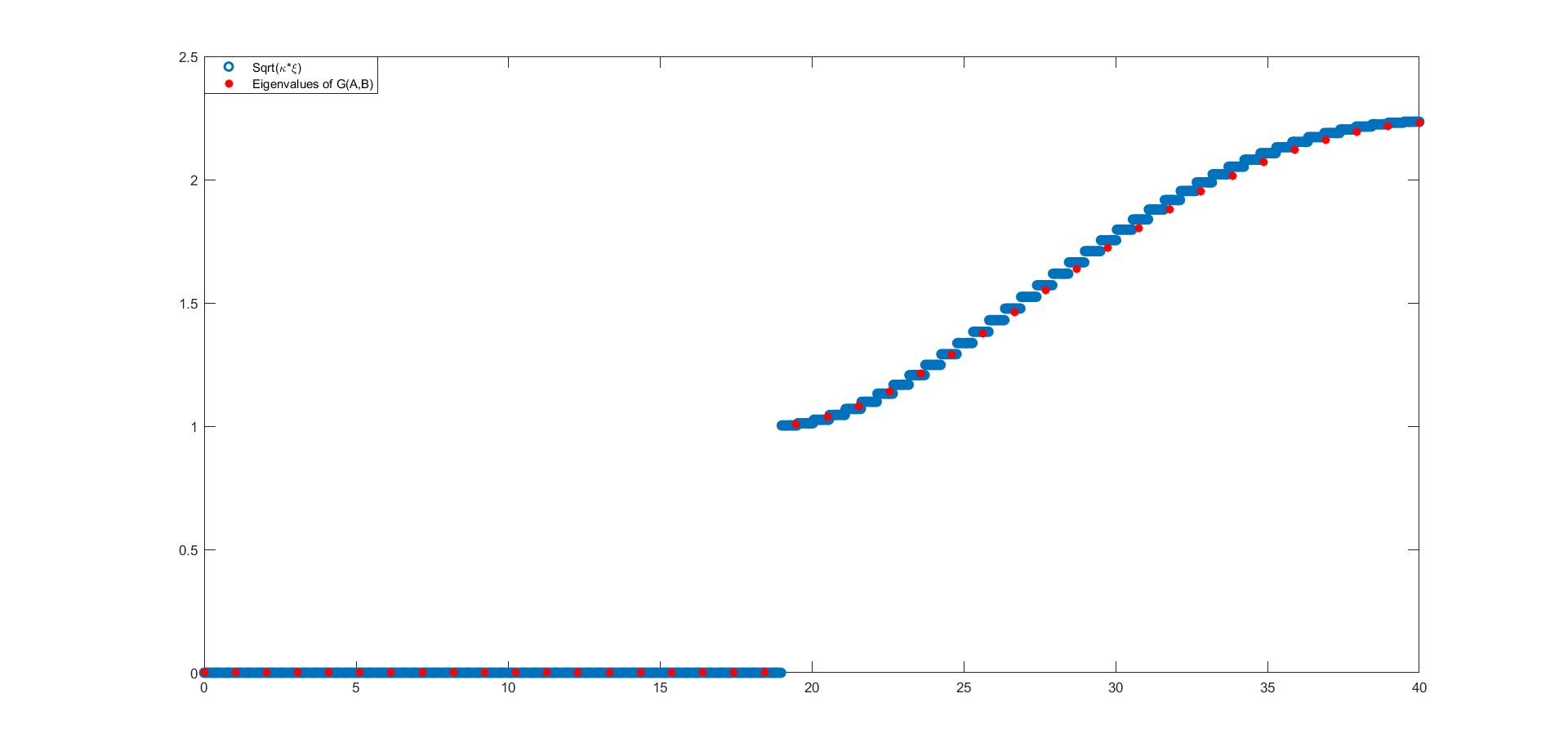} 
     \captionsetup{labelformat=empty}
     \caption*{}
   \end{minipage}\hfill
   \begin{minipage}{0.52\textwidth}
     \centering
     \includegraphics[width=\linewidth, height=6cm]{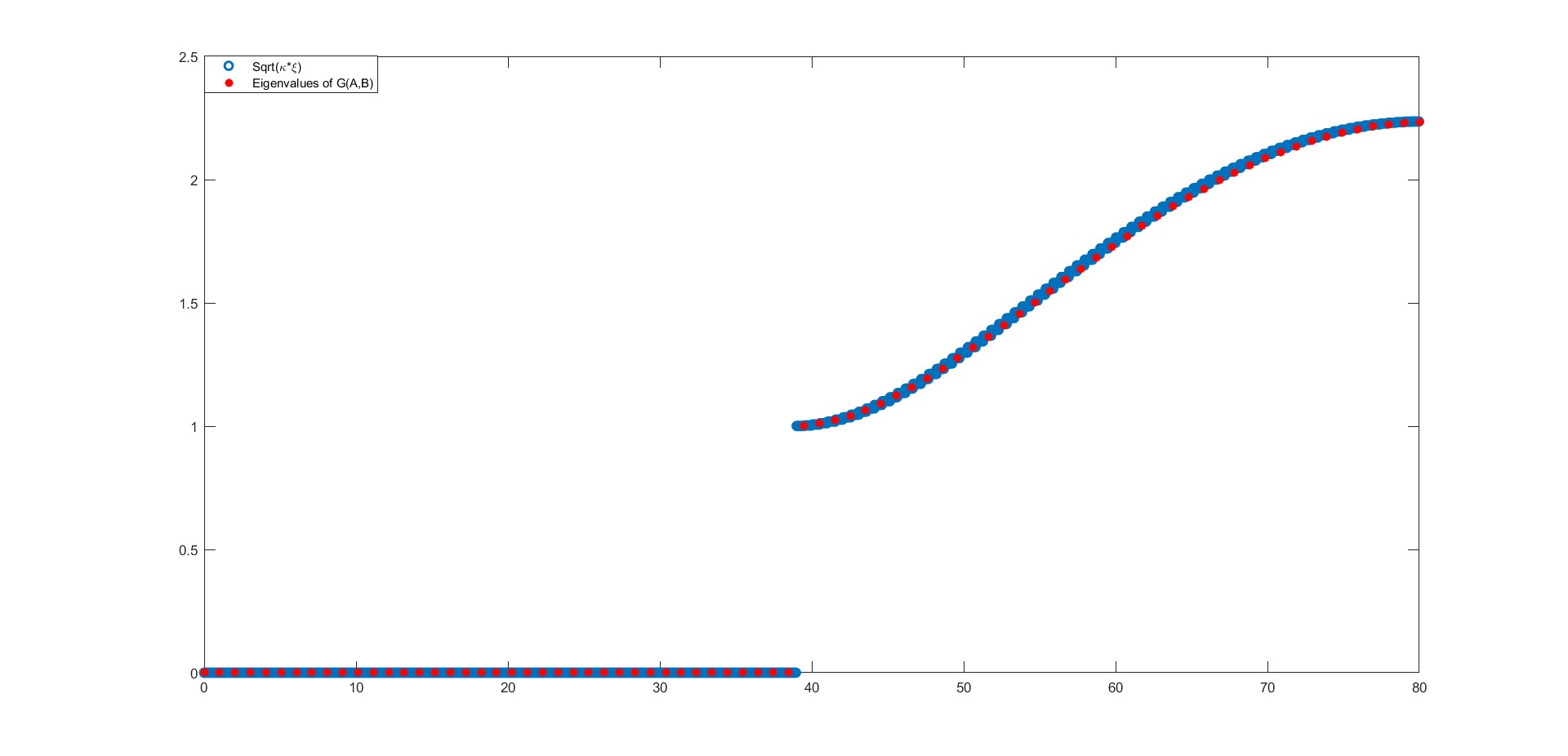} 
     \captionsetup{labelformat=empty}
     \caption*{}
   \end{minipage}
   \caption{{\textbf{ Example 1.}
Comparison between eigenvalues of \(G(A_n,B_n)\) (red stars; \(n=40,80\)) versus the symbol \( (\kappa\xi)^{\frac{1}{2}} \) (blue circles)}}
\label{fig:exp1}
\end{figure}

\subsection{Distribution results and momentary symbols} \label{ssec:num2}

\subsubsection{Example 2 }
Let $d = 1 $ and define the matrix-sequences $\{A_{n}\}_{n}$, $\{B_{n}\}_{n}$, with $
A_n = D_n(a) + \frac{1}{n^4} I_n,$  $ C_n = T_n(3 + 2\cos(\theta))
$ and $
B_n = \left(D_n(1 - a) + \frac{1}{n^4} I_n \right) C_n \left(D_n(1 - a) + \frac{1}{n^4} I_n \right),
$ where \( D_n(a) \) is the diagonal matrix generated by the piecewise continuous function \( a(x) \) given in \textbf{Example 1}. Then, the geometric mean sequence explicitly satisfies the GLT relation:
\begin{align*}
\left\{ G\left(A_n, B_n\right) \right\}_n &\sim_{\mathrm{GLT}} \sqrt{a(x)(1 - a(x))(3 + 2\cos\theta)} \\
&= a(x)(1 - a(x))\sqrt{3 + 2\cos\theta} = 0
\end{align*}
where
\[
\kappa = a(x) \quad \xi=(1-a(x))(3+2\cos(\theta)).
\]
The zero-valued symbol arises naturally due to the definition of the piecewise function \( a(x) \):
\begin{itemize}
    \item For \( x \in [0, \frac{1}{2}) \), we have \( a(x) = 0 \). Thus, the term \( a(x)(1 - a(x)) \) becomes zero.
    \item For \( x \in [\frac{1}{2},1] \), we have \( a(x) = 1 \). In this interval, the factor \( 1 - a(x) \) becomes zero, making \( a(x)(1 - a(x)) = 0 \).
\end{itemize}
\subsubsection*{Eigenvalue distribution}
The eigenvalue distribution of \( \{G(A_n, B_n)\}_{n} \) is numerically analyzed for increasing matrix dimensions \( n \). Since the GLT symbol for this example is zero, by Theorem~\ref{theorem 1} we expect the eigenvalues to concentrate along the zero line. Interestingly, numerical results reveal an additional spectral structure beyond Theorem~\ref{theorem 1}.
{For smaller values of $ n $, eigenvalues align closely with zero, as predicted by the GLT symbol. However, some eigenvalues appear above this level, forming a secondary symbol. This phenomenon is attributed to {momentary symbols}, as discussed in the introduction (see \cite{bolten2022toeplitz,bolten2023note,new momentary}). As \( n \) increases, these elevated eigenvalues shift upwards reaching approximately \( 10^{-3} \) for \( n = 40 \) (shown in the Figure \ref{fig:exp2}) and \( 10^{-4} \) for \( n = 80 \).}
\begin{figure}[H]
   \begin{minipage}{0.52\textwidth}
     \centering
     \includegraphics[width=\linewidth, height=6cm]{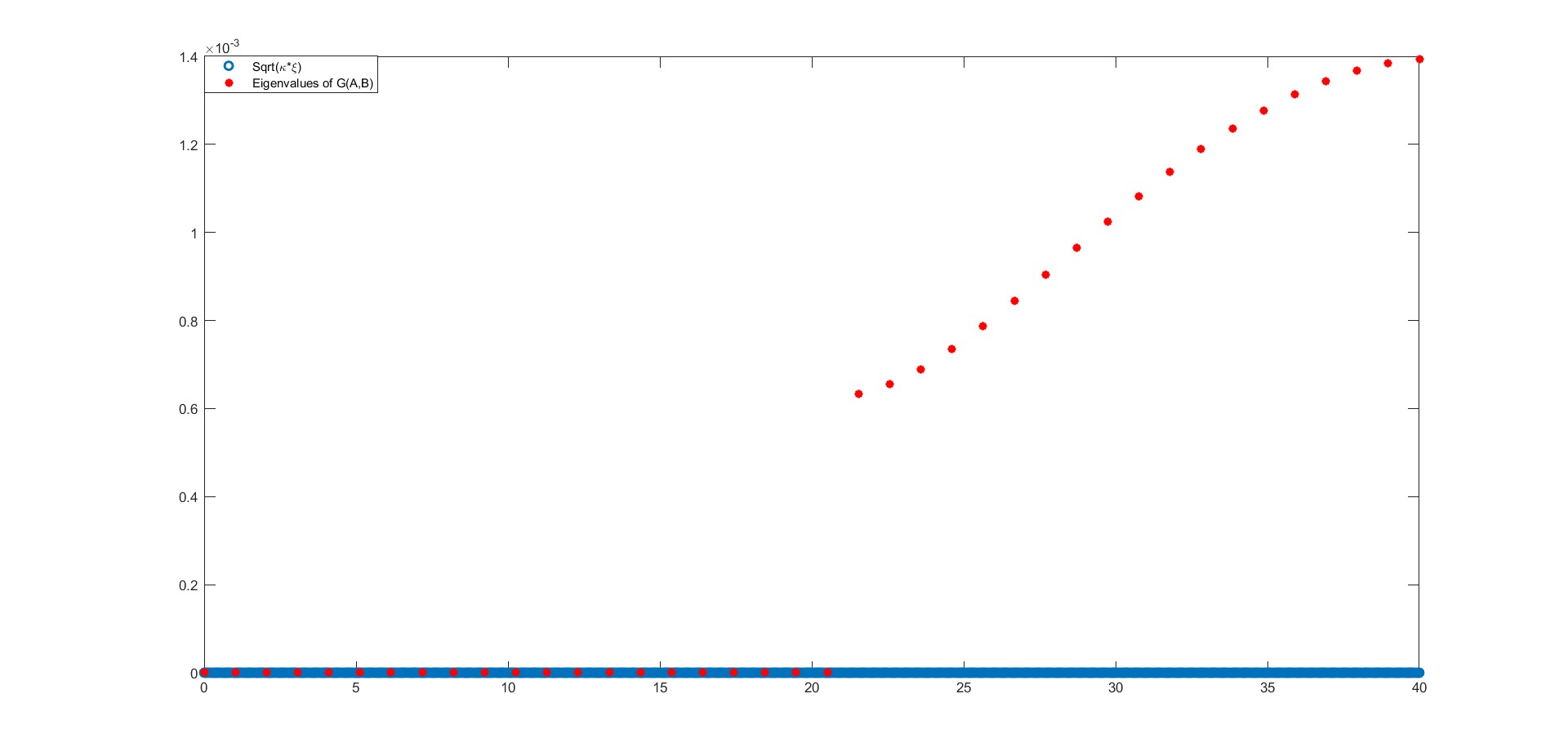} 
     \captionsetup{labelformat=empty}
     \caption*{}
   \end{minipage}\hfill
   \begin{minipage}{0.52\textwidth}
     \centering
     \includegraphics[width=\linewidth, height=6cm]{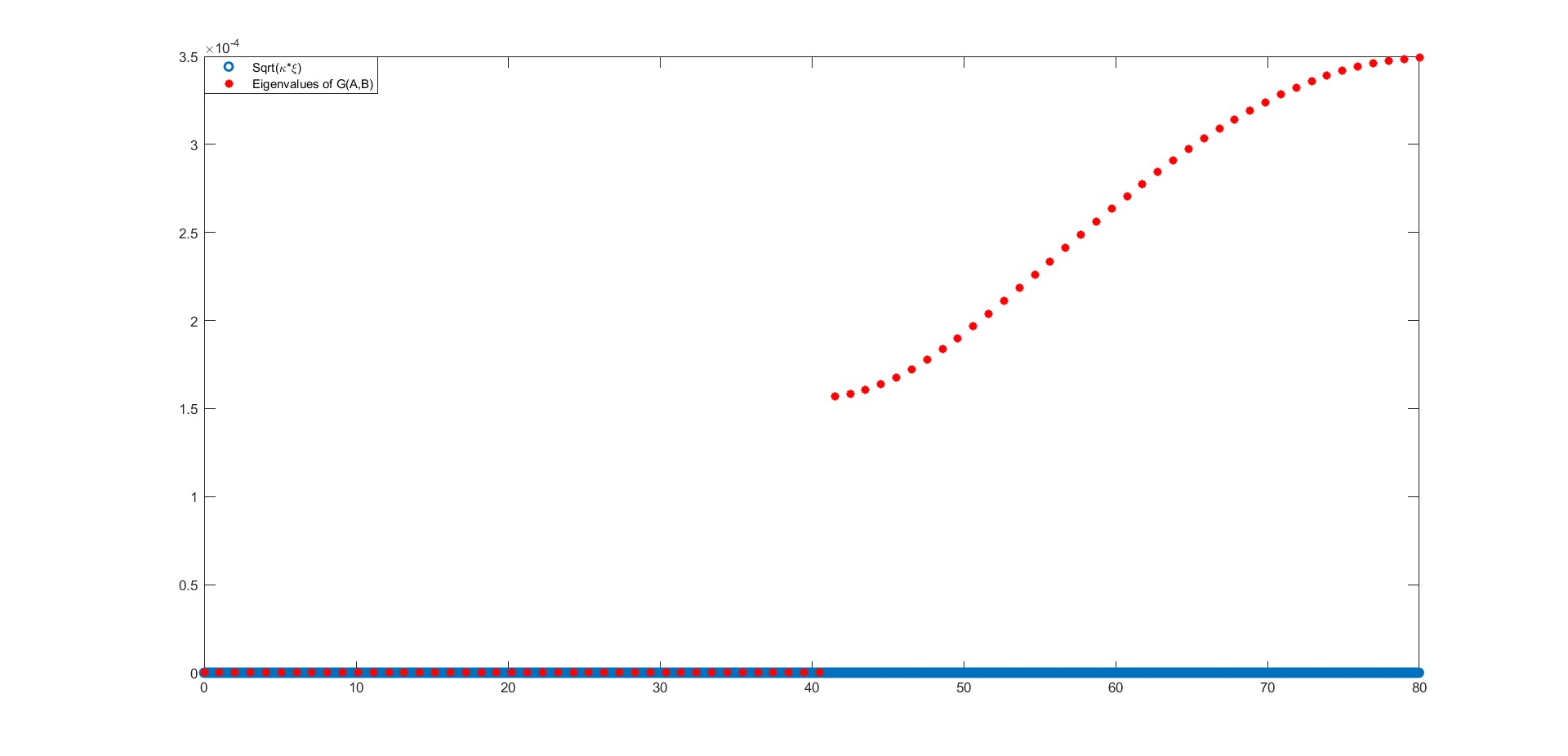} 
     \captionsetup{labelformat=empty}
     \caption*{}
   \end{minipage}
    \caption{{\textbf{ Example 2.}
Comparison between eigenvalues of \(G(A_n,B_n)\) (red stars; \(n=40,80\)) versus the symbol \( (\kappa\xi)^{\frac{1}{2}} \) (blue circles)}}
\label{fig:exp2}
\end{figure}

\subsection{Minimal eigenvalues and conditioning} \label{ssec:num3}

In this section, we analyze the extremal spectral behavior and conditioning of the geometric mean sequence \( \{ G(A_n, B_n) \}_n \), focusing on its dependence on the analytical properties of the corresponding GLT symbol. This approach follows prior studies on extremal eigenvalues in structured matrix settings, particularly in Toeplitz matrices \cite{extreme3,extreme1,extreme2,extreme4} and block Toeplitz matrices \cite{SerraCapizzano1999a,SerraCapizzano1999b}, as well as variable coefficient differential operators, including multilevel cases with \( d > 1 \) \cite{extreme diff1,extreme diff2,extr2-glt}. In the case of variable coefficient differential operators it is worth noticing that extremal spectral results do not stem from the GLT theory but from a combination of GLT tools and properties which are typical of linear positive operators 
\cite{SerraCapizzanoTilli-LPO,LPO-rev}.\\
Here, we restrict our attention to the unilevel scalar setting with \( d = r = 1 \), considering \textbf{Example 1} and \textbf{Example 2}.
\subsubsection{Example 1: minimal eigenvalue}
\begin{itemize}
    \item $X_{n}=G(A_{n},B_{n})$
    \item Take $n_{j}=40.2^{j}, \quad\quad j=0,1,2,3, $
    \item Compute $\tau_{j}= \lambda_{\min}(X_{n_{j}}), \quad \quad j=0,1,2,3, $
    \item Compute $\alpha_{j}=\log_{2}(\frac{\tau_{j}}{\tau_{j+1}}), \quad \quad j=0,1,2. $ 
\end{itemize}
\vspace{9pt}

\begin{table}[H]
  \centering
  \caption{Numerical behaviour of the minimal \\eigenvalue.}
  \label{tab:tau_alpha}
  \begin{tabular}{|c
                  |S[table-format=1.4e+2]
                  |S[table-format=1.4]|}
    \hline
    {$n$} & {$\tau_j$} & {$\alpha_j$}\\
   \hline
    40  & 6.3265e-04 & 2.0132 \\
    80  & 1.5673e-04 & 2.0064 \\
    160 & 3.9093e-05 & 2.0074 \\
    320 & 9.7675e-06 & \multicolumn{1}{c|}{}     \\
    \hline
  \end{tabular}
\end{table}
As it can be seen, the quantity $\alpha_j$ stabilizes around $2$ as $n$ increases. This is in perfect agreement with the fact that the minimal eigenvalue of $A_n$ converges to zero as $n^{-4}$, while the minimal eigenvalue of $B_n$ converges monotonically from above to $\min \xi=1$. The key point is that the minimal eigenvalue of the geometric mean behaves asymptotically as the geometric mean of the minimal eigenvalues of $A_n$ and $B_n$, respectively. Similar remarks can be made in the subsequent case with respect to \textbf{Example 2}.
 
 The observed numerical evidences are not implied by the theoretical derivations and this is an interesting fact that deserves to be investigated theoretically in the future. In particular, we would like to prove formally that that $\alpha_j$ converges to $2$ for the first example, Table~\ref{tab:tau_alpha2}, and it converges to 4 in the second example, Table~\ref{tab:tau_alpha3}, even from the preliminary
numerical results the convergence does not look monotonic.

\subsubsection{Example 2: minimal eigenvalue}
\begin{itemize}
    \item $X_{n}=G(A_{n},B_{n})$
    \item Take $n_{j}=40.2^{j}, \quad\quad j=0,1,2,3, $
    \item Compute $\tau_{j}= \lambda_{\min}(X_{n_{j}}), \quad \quad j=0,1,2,3, $
    \item Compute $\alpha_{j}=\log_{2}(\frac{\tau_{j}}{\tau_{j+1}}), \quad \quad j=0,1,2. $ 
\end{itemize}
\vspace{9pt}
\begin{table}[H]
  \centering
  \caption{Numerical behaviour of the minimal\\ eigenvalue.}
  \label{tab:tau_alpha1}
  \begin{tabular}{|c
                  |S[table-format=1.4e+2]
                  |S[table-format=1.4]|}
    \hline
     {$n$} & {$\tau_j$} & {$\alpha_j$} \\
   \hline
    40  & 3.9177e-07 & 4.0003 \\
    80  & 2.4480e-08 & 4.0040 \\
    160 & 1.5259e-09 & 4.0009 \\
    320 & 9.5367e-11 & \multicolumn{1}{c|}{}      \\
    \hline
  \end{tabular}
\end{table}


\subsection{Numerical study: non-commuting, rank-deficient symbols}
\label{sec:non_commuting_rank_def}

In this section, through numerical test, we investigate the eigenvalue distribution of the geometric means $\{G(A_{n}, B_{n})\}_{n}$ with
\[
G(A_n,B_n)\;=\;
A_n^{1/2}\!\bigl(A_n^{-1/2}B_nA_n^{-1/2}\bigr)^{1/2}\!A_n^{1/2},
\]
where $\{A_n\}_n$ and $\{B_n\}_n$ are HPD GLT matrix-sequences with $r\times r$ matrix-valued symbols $\kappa,\xi$.
The GLT matrix-sequences are specifically chosen with GLT symbols so that:

\begin{enumerate}
\item they are rank-deficient on a subset of positive Lebesgue measure; and
\item they do not commute pointwise on a set of positive measure.
\end{enumerate}

These hypotheses clearly set the problem outside the scope of commuting or almost everywhere invertible cases treated in Theorems~\ref{theorem 1} and in Theorem~\ref{theorem 2}. \\
Our tests provide evidence for two phenomena:
\begin{itemize}
\item[\textbf{(i)}]
The support of the asymptotic eigenvalue distribution of $\{G(A_n,B_n)\}_{n}$ coincides with the intersection set $\operatorname{ess\,Ran}(\kappa)$ $\cap$ $\operatorname{ess\,Ran}(\xi)$ see Section~\ref{subsubsec:GM_conjecture}.
\item[\textbf{(ii)}] The eigenvalues appear to converge in distribution to a candidate symbol which is rank-deficient and, in general, distinct from the classical geometric mean $G(\kappa,\xi)$.
\end{itemize}
Clearly, the relation \( G(A_n, B_n) \sim_{\mathrm{GLT}} G(\kappa, \xi) \) cannot hold under hypotheses (1.)-(2.), and the experiments show the substantial differences in this case compared to the commuting or invertible scenarios. If this case can be studied within the GLT framework, it requires a more technical application of GLT and a.c.s. theory, along with a deeper functional calculus of matrix means.

\subsubsection{Setup of the numerical experiments}
\label{sec:NumSetup}
Below, we provide the setup of the numerical tets. \\
We construct four pairs of unilevel GLT matrix-sequences $\{A_n\}_n, \{B_n\}_n,$ whose symbols
simultaneously satisfy the hypothesis stated at the beginning of
Section~\ref{sec:non_commuting_rank_def} (rank deficiency and
non-commutation). These examples are organized according to the property of rank loss:

\begin{itemize}\setlength\itemsep{4pt}
\item \textbf{Case 1} - each symbol is full rank on a set of positive
measure;
\item \textbf{Case 2} - each symbol is rank-deficient almost everywhere.
\end{itemize}

\medskip\noindent
\textbf{Case 1, Example 1.}

Define
\[
f(\theta)\; =\; \begin{cases}
0, & \theta\in[-\pi,0],\\[4pt]
\theta, & \theta\in(0,\pi],
\end{cases}
\qquad
g(\theta) = f(-\theta),
\]
and set
\[
F(\theta)\;=\;f(\theta)\otimes
\begin{bmatrix}2&1\\[2pt]1&2\end{bmatrix},\qquad
G(\theta)\;=\;g(\theta)\otimes
\begin{bmatrix}3&1\\[2pt]1&1\end{bmatrix}.
\]
With the $2n\times2n$ Toeplitz matrices $T_n(F)$ and $T_n(G)$, define
\begin{equation}\label{eq:case1_ex1}
A_n=T_n(F)+\frac{1}{n^{3}}I_{2n},\qquad
B_n=T_n(G)+\frac{1}{n^{3}}I_{2n}.
\end{equation}

\medskip\noindent
\textbf{Case 1, Example 2.}

Let $\chi_{[-a,a]}$ be the characteristic function of $[-a,a]$, $0<a<\pi$, and set
\[
f(\theta)\,=\,\chi_{[-1/2,\,1/2]}(\theta),\qquad
g(\theta)\,=\,\chi_{[-1/4,\,1/4]}(\theta),
\]
\[
A\;=\;\begin{bmatrix}2&1\\[2pt]1&2\end{bmatrix},\qquad
B\;=\;\begin{bmatrix}3&1\\[2pt]1&1\end{bmatrix},
\]
\[
F(\theta)\;=\;f(\theta)\otimes A,\qquad
G(\theta)\;=\;g(\theta)\otimes B.
\]
Define
\begin{equation}\label{eq:case1_ex2}
A_n=T_n(F)+\frac{1}{n^{3}}I_{2n},\qquad
B_n=T_n(G)+\frac{1}{n^{3}}I_{2n}.
\end{equation}

\medskip\noindent
\textbf{Case 2, Example 1.}

With $f(\theta)=2-\cos\theta$, $g(\theta)=3+\cos\theta$ and
rank one blocks
\[
A\;=\;\begin{bmatrix}1&1\\[2pt]1&1\end{bmatrix},\qquad
B\;=\;\begin{bmatrix}1&2\\[2pt]2&4\end{bmatrix},
\]
let
\(F(\theta)\;=\;f(\theta)\otimes A\) and
\(G(\theta)\;=\;g(\theta)\otimes B\).
Set
\[
A_n=T_n(F)+\frac{1}{n^{2}}I_{2n},\qquad
B_n=T_n(G)+\frac{1}{n^{2}}D_n(b),\qquad
b(x)=(1+x) I_2.
\]
Here, both symbols are rank $1$ almost everywhere.

\medskip\noindent
\textbf{Case 2, Example 2.}

Define piecewise-linear scalar functions in $[0,1]$:
\[
a(x)\;=\;\begin{cases}
1-2x, & 0\le x\le\tfrac12,\\[2pt]
0, & \tfrac12<x\le1,
\end{cases}
\qquad
b(x)\;=\;\begin{cases}
0, & 0\le x\le\tfrac13\text{ or } \tfrac23\le x\le1,\\[2pt]
x-\tfrac13, & \tfrac13<x\le\tfrac12,\\[2pt]
\tfrac23-x, & \tfrac12<x<\tfrac23,
\end{cases}
\]
and let $D_n(a)$, $D_n(b)$ be the corresponding sampling matrices.
With $f(\theta)=2+\cos\theta$, $g(\theta)=3+\cos\theta$ and
\[
A\;=\;\begin{bmatrix}2 & 0 & 1\\[2pt]
0 & 2 & 1\\[2pt]
1 & 1 & 1\end{bmatrix},\qquad
B\;=\;\begin{bmatrix}2 & 1 & 0\\[2pt]
1 & 1 & 1\\[2pt]
0 & 1 & 2\end{bmatrix},
\]
define
\(F(\theta)\;=\;f(\theta)\otimes A\) and
\(G(\theta)\;=\;g(\theta)\otimes B\), and set
\begin{align}\label{eq:case2_ex2}
A_n &\;=\; \bigl(D_n^{1/2}(a)\,\otimes\, I_3\bigr)\,
T_n(F)\, \bigl(D_n^{1/2}(a)\,\otimes\, I_3\bigr)
+\frac{1}{5n}\,I_{3n}, \\
B_n&\;=\; \bigl(D_n^{1/2}(b)\,\otimes\, I_3\bigr)\,
T_n(G)\,
\bigl(D_n^{1/2}(b)\,\otimes\, I_3\bigr)
+\frac{1}{5n}\,I_{3n}.
\end{align}

\medskip
In every example we denote by $\kappa$ and $\xi$ the GLT symbols of
$\{A_n\}_n$ and $\{B_n\}_n$, respectively.

\subsubsection{A candidate symbol for the geometric mean}
\label{subsubsec:GM_conjecture}

Let $\{A_n\}_n$ and $\{B_n\}_n$ be GLT matrix-sequences whose matrix-valued symbols
\[
\kappa,\;\xi : [0,1]^d\times[-\pi,\pi]^d \;\longrightarrow\; \mathbb{C}^{\,r\times r},
\]
are positive-semidefinite. \\
For any $\varepsilon>0$, define the strictly positive symbols
\(
\kappa_\varepsilon=\kappa+\varepsilon I_r,\;
\xi_\varepsilon=\xi+\varepsilon I_r.
\)
Since the matrix geometric mean is monotone continuous in each argument, the limit below exists point-wise {and it is unique} (\cite[p.\,3]{Ando2004}):
\begin{definition}[Candidate symbol]\label{def:candidate}
For almost every $(x,\theta)$ set
\begin{equation}\label{eq:cand_symbol}
\widetilde G(\kappa,\xi)(x,\theta)
:=\lim_{\varepsilon\rightarrow 0}
G\bigl(\kappa_\varepsilon(x,\theta),\,
\xi_\varepsilon (x,\theta)\bigr).
\end{equation}
\end{definition}
The matrix-valued symbol $\widetilde G(\kappa,\xi)$ is \emph{essentially} positive semidefinite (see \cite{Ando2004}) and satisfies the essential support identity
\[
\operatorname{ess\,Ran}\,\widetilde G(\kappa,\xi)\;=\;
\operatorname{ess\,Ran}\!\bigl[(x,\theta)\mapsto
\operatorname{ran}\kappa(x,\theta)\;\cap\;\operatorname{ran}\xi(x,\theta)\bigr].
\]

Here $\operatorname{ran}M$ denotes the column space of an individual matrix $M$, whereas
$\operatorname{ess\,Ran}(.)$ is the essential range of a Grassmannian valued measurable map. The essential formulation of the subspace condition can also be formulated as the pre-image of the essential numerical range in the style of [\cite{nomi-mazza}, Definition\, 1].

Taking into account the Definition \ref{def:candidate} of the candidate symbol, we propose the following conjecture:
\begin{conjecture}[GLT closure under geometric mean]
\label{conj:GM_glt}
Let $\{A_n\}_n$ and $\{B_n\}_n$ be Hermitian positive definite
GLT matrix-sequences with
\[
\{A_n\}_n\sim_{\mathrm{GLT}}\kappa,
\qquad
\{B_n\}_n\sim_{\mathrm{GLT}}\xi.
\]
Then the geometric mean matrix-sequence is again GLT and
\[
\{\,G(A_n,B_n)\,\}_n \;\sim_{\operatorname{GLT}}\; \widetilde G(\kappa,\xi),
\]
where $\widetilde G(\kappa,\xi)$ is the candidate symbol
defined in~\eqref{eq:cand_symbol}.
\end{conjecture}

The heuristic justification of the conjecture is the following. For each fixed $\varepsilon>0$ the shifted sequences
$\{A_n+\varepsilon I\}_n$ and $\{B_n+\varepsilon I\}_n$ are GLT sequences symbols
$\kappa_\varepsilon$ and $\xi_\varepsilon$ and their geometric mean sequences satisfy
\(
\{\,G(A_n+\varepsilon I,B_n+\varepsilon I)\,\}_n
\sim_{\operatorname{GLT}}
G(\kappa_\varepsilon,\xi_\varepsilon).
\)
{Properties of monotone functions and the matrix geometric mean axioms described in \cite{Ando2004}}, combined with closeness of the GLT $\ast$-algebra (Axiom \textbf{GLT 4}) under a.c.s convergence, suggest that letting $\varepsilon\rightarrow0$ the limit of $ \{\,G(A_n+\varepsilon I,B_n+\varepsilon I)\,\}_n$ is well defined in the {a.c.s topology and this limit is $ \{\,G(A_n,B_n)\,\}_n$, up to zero-distributed perturbations }, leading to the GLT algebraic part of Conjecture~\ref{conj:GM_glt}, a full proof when
$\kappa$ and $\xi$ do not commute point-wise remains open. Meanwhile, all
numerical evidence below supports the conjecture on the spectral distribution part of our hypothesis. 
\subsubsection*{Candidate symbols predicted by Conjecture~\ref{conj:GM_glt}}

For the four test pairs introduced in Section \ref{sec:NumSetup}, Conjecture~\ref{conj:GM_glt} predicts
the following GLT symbols for $\{G(A_n,B_n)\}_n$.
Throughout
\(
\mathbf 1=(1,1,1)^{\mathsf *},\;
J=\mathbf 1\,\mathbf 1^{\mathsf *}\in\mathbb{C}^{3\times3}.
\)

\paragraph{Case 1, Example 1.}
\[
\widetilde G(\kappa,\xi)(x,\theta)=0_{2\times2}.
\]

\paragraph{Case 1, Example 2.}
\[
\widetilde G(\kappa,\xi)(x,\theta)=
\chi_{[-\frac14,\frac14]}(\theta)\;
\otimes C,
\]
\[
C \; =\;
\frac{1}{6^{1/4}\sqrt{\,2+\sqrt6\,}}
\begin{bmatrix}
2\sqrt2+3\sqrt3 & \sqrt2+\sqrt3\\[4pt]
\sqrt2+\sqrt3 & 2\sqrt2+\sqrt3
\end{bmatrix}.
\]

\paragraph{Case 2, Example 1.}
\[
\widetilde G(\kappa,\xi)(x,\theta)=0_{2\times2}.
\]

\paragraph{Case 2, Example 2.}
\[
\widetilde G(\kappa,\xi)(x,\theta)=
\sqrt{\,\bigl(2+\cos\theta\bigr)\,\bigl(3+\cos\theta\bigr)\,h(x)\,k(x)}\;
\otimes J.
\]

\subsubsection{Numerical verification of the Conjecture}

Let
\[
G_n := G(A_n,B_n)\in\mathbb{C}^{d_n\times d_n},
\qquad d_n=\text{size}(G_n).
\]
We test whether the empirical eigenvalue distribution of $\{G_n\}_n$
converges to the distribution induced by the
candidate symbol
$\widetilde G(\kappa,\xi)(x,\theta)$ defined in
Section~\ref{subsubsec:GM_conjecture}.
The comparison follows the standard rearrangement strategy.

\begin{itemize}\setlength\itemsep{6pt}

\item[\textbf{1.}] \textbf{Sampling the symbol.}
Evaluate $\widetilde G(\kappa,\xi)$ on a tensor grid
$\{x_j\}_{j=1}^{M_x}\!\subset[0,1]$,
$\{\theta_i\}_{i=1}^{M_\theta}\!\subset[-\pi,\pi]$ with a total number of
$M_x \times M_\theta\simeq2000$ samples.
For each point $(x_j,\theta_i)$ compute the
$r$ eigenvalues of the function; then merge and sort in non-decreasing order all the computed values.
This yields the empirical quantile function of the eigenvalues.

\item[\textbf{2.}] \textbf{Eigenvalue computation.}
Matrices $A_n,B_n$ are generated for
$n\in\{40,80,160,320\}$, and eigenvalues are ordered in non-decreasing order, to obtain the quantile spectral distribution function of $\{G_n\}_n$. \\
We also record
$\lambda_{\min}(G_n)$, $\lambda_{\max}(G_n)$ and the condition
number~$\mathrm{cond}_2(G_n)$ for later stability analysis.

\item[\textbf{3.}] \textbf{Quantile plot matching.}
For each $n$ we plot the sorted eigenvalues
$\bigl(\lambda_i(G_n)\bigr)$ against the corresponding rearranged sample values of $\widetilde G$.
A superposition of the two plots for growing values of $n$ gives an experimental indication of the asymptotic convergence of the empirical spectral distribution of $\{G_n\}_n$ to that of the symbol.

\item[\textbf{4.}] \textbf{Support analysis.}
To estimate the measure of the zero set
$\{(x,\theta):\widetilde G=0\}$, we count the fraction of eigenvalues below the threshold $0.1$  compared to the matrix size; the value is an estimate of the complement of the predicted measure of the support as $n$ grows.
\end{itemize}
All tests for this section are run on an Intel Core Ultra 7 155H CPU (22 threads, 30.9 GiB RAM) under Ubuntu 24.10, Linux 6.11.0-29, with Python~3.12.7, NumPy~1.26.4 and SciPy~1.13.1 linked against OpenBLAS.
\medskip
\begin{remark}[Measure-theoretic meaning of the rearrangement test]
\label{rearrangement}
As anticipated implicitly, the comparison of sorted eigenvalue versus sorted symbol has a precise
measure-theoretic foundation based on the quantile approximation theory \cite{Bogoya2016}. For completeness, we report the construction in \cite{Bogoya2016} for the
unilevel setting $d=1$; the multilevel case is analogous.

Let $r$ denote the size of the matrix
$\widetilde G(\kappa,\xi)(x,\theta)$ and put
$D=[0,1]\times[-\pi,\pi]$ (so $\mu(D)=2\pi$).
Define the probability space defined by the triple
\[
(\Omega,\mathcal F,\mathbb P),\qquad
\Omega=D\times\{1,\dots,m\},\qquad
\mathcal F=\mathcal B(D)\otimes2^{\{1,\dots,m\}}.
\]
Here, $\mathcal B(D)$ denotes the Borel $\sigma$-algebra on $D$, and
$2^{\{1,\dots,m\}}$ denotes the power set of $\{1,\dots,m\}$,
with product measure
\(
\mathbb P(A\times B)=\dfrac{\mu(A)}{2\pi}\,\dfrac{|B|}{{r}}.
\)\\
Introduce the random variable
\[
X(x,\theta,i):=\lambda_i\!\bigl(\widetilde G(\kappa,\xi)(x,\theta)\bigr),
\qquad (x,\theta,i)\in\Omega,
\]
where the eigenvalues are ordered non-decreasingly.
The quantile function of $X$ is precisely the \emph{non-decreasing rearrangement} $\widetilde G^{\dagger}$ of the matrix-valued symbol.\\
Because spectral distributions are defined only up to
measure-preserving rearrangements, we deduce that
\[
\{G_n\}_n\sim_{\lambda} \widetilde G
\;\Longrightarrow\;
\{G_n\}_n\sim_{\lambda} \widetilde G^{\dagger},
\]
so verifying convergence to the quantile $\widetilde G^{\dagger}$ is equivalent and numerically simpler than verifying convergence to the symbol
$\widetilde G(\kappa,\xi)$ itself. \\
Furthermore, whenever $\widetilde G^{\dagger}$ is continuous at $t\in(0,1)$, the sorted
eigenvalues satisfy
\(
\lambda_{\lceil t d_n\rceil}(G_n)\to \widetilde G^{\dagger}(t)
\)
as $n\to\infty$, and the empirical proportion of eigenvalues below a small threshold
$t_0$ (we take $t_0=0.1$) converges to
\(
\mathbb P[X\le t_0],
\)
therefore it can be used as an estimate of the measure of the rank-deficient subset
$\{(x,\theta):\widetilde G(\kappa,\xi)(x,\theta)=0\}$.
\end{remark}

\subsubsection{Numerical results}
\label{sec:numerical-results}
In this section, we compare the sorted eigenvalues of each geometric mean matrix
\(G(A_n,B_n)\) with the rearranged eigenvalue distribution predicted by the
candidate symbol \(\widetilde{G}(\kappa,\xi)\). Plots are provided for each block size.
\paragraph{Case 1.}
Figures~\ref{fig:case1_ex1}--\ref{fig:case1_ex2} display the ordered
eigenvalues of \(G(A_n,B_n)\) (colored markers) together with the ordered
samples of \(\widetilde{G}(\kappa,\xi)\) (solid red line).

\smallskip
\emph{Example 1.}
Because the supports of \(\kappa\) and \(\xi\) are essentially disjoint,
\(\widetilde{G}(\kappa,\xi)\equiv 0\).
Most of the eigenvalues match the zero line, showing convergence to the zero-distribution. Only a small number of positive outliers remain above the line. For small values of $n$, the eigenvalue plot appears to be governed by a GLT {momentary} symbol that, when rearranged, follows a power law that rapidly collapses to the zero function as \(n \to \infty\).

\smallskip
\emph{Example 2.}
Here, the supports of \(\kappa\) and \(\xi\) intersect on \(\theta \in [-0.25,0.25]\). As a result, \(\widetilde{G}(\kappa,\xi)\) takes two values: zero outside the overlap and a positive constant matrix inside it. The eigenvalues reflect exactly this behavior: they form a long plateau at zero, followed by two shorter clusters at the positive eigenvalues determined by \(\widetilde{G}(\kappa,\xi)\). As we will see, the length of this zero plateau matches the respective Lebesgue measure of the complement of the support of \(\widetilde G\).
\begin{figure}[H]
\centering
\begin{subfigure}[t]{0.49\textwidth}
\centering
\includegraphics[width=\textwidth]{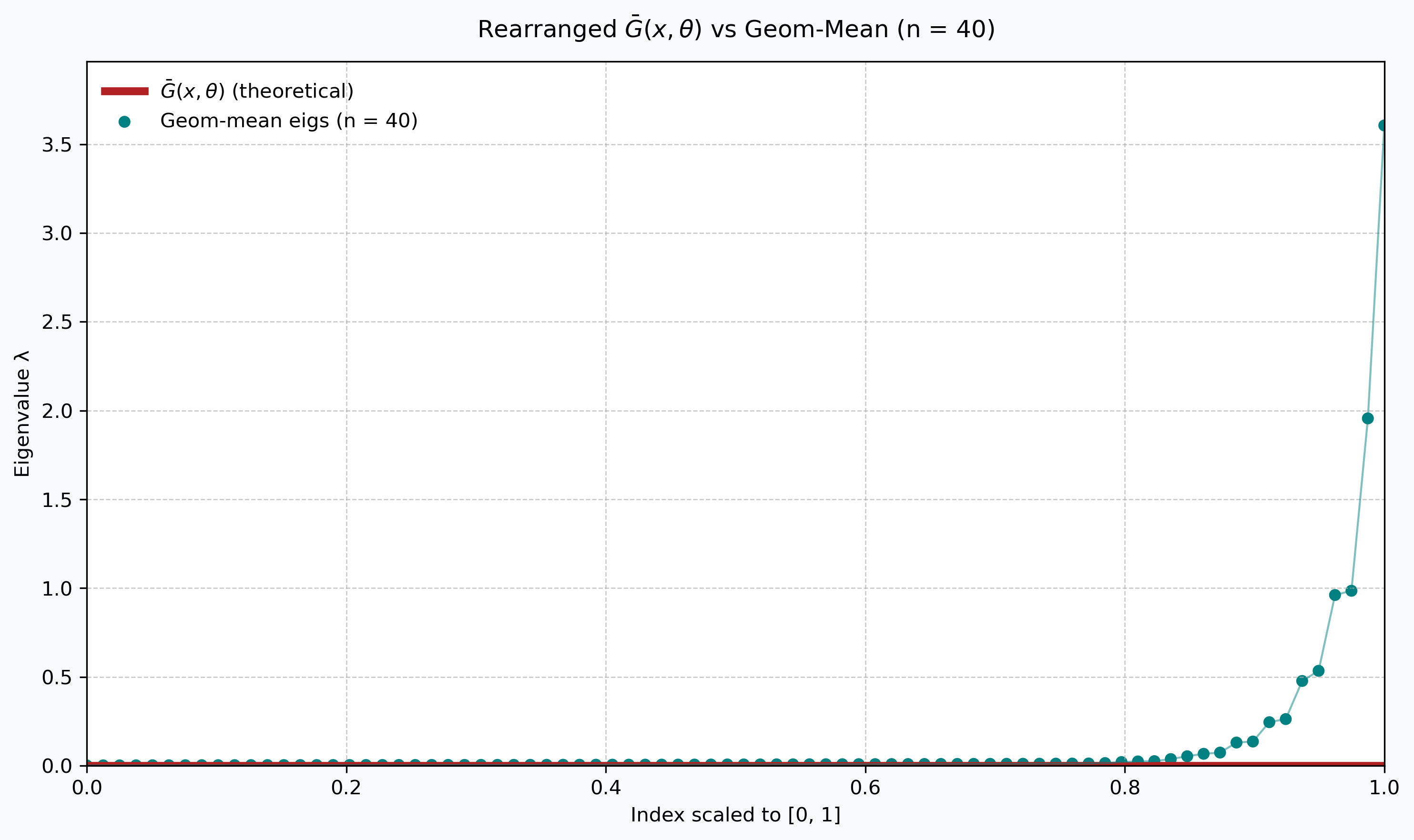}
\end{subfigure}
\hfill
\vspace{0.19cm}
\begin{subfigure}[t]{0.49\textwidth}
\centering
\includegraphics[width=\textwidth]{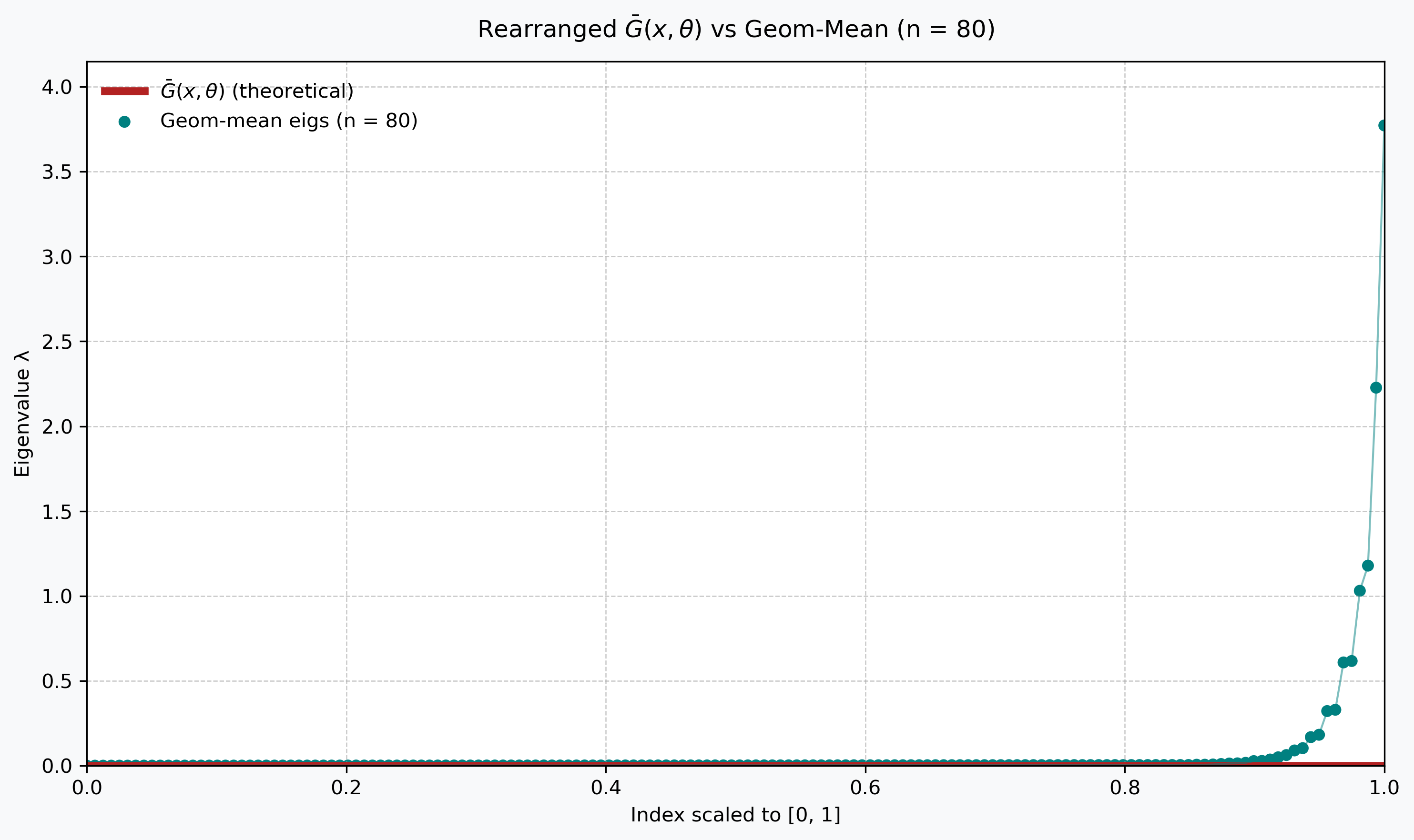}
\end{subfigure}
\\[-1em] 

\begin{subfigure}[t]{0.49\textwidth}
\centering
\includegraphics[width=\textwidth]{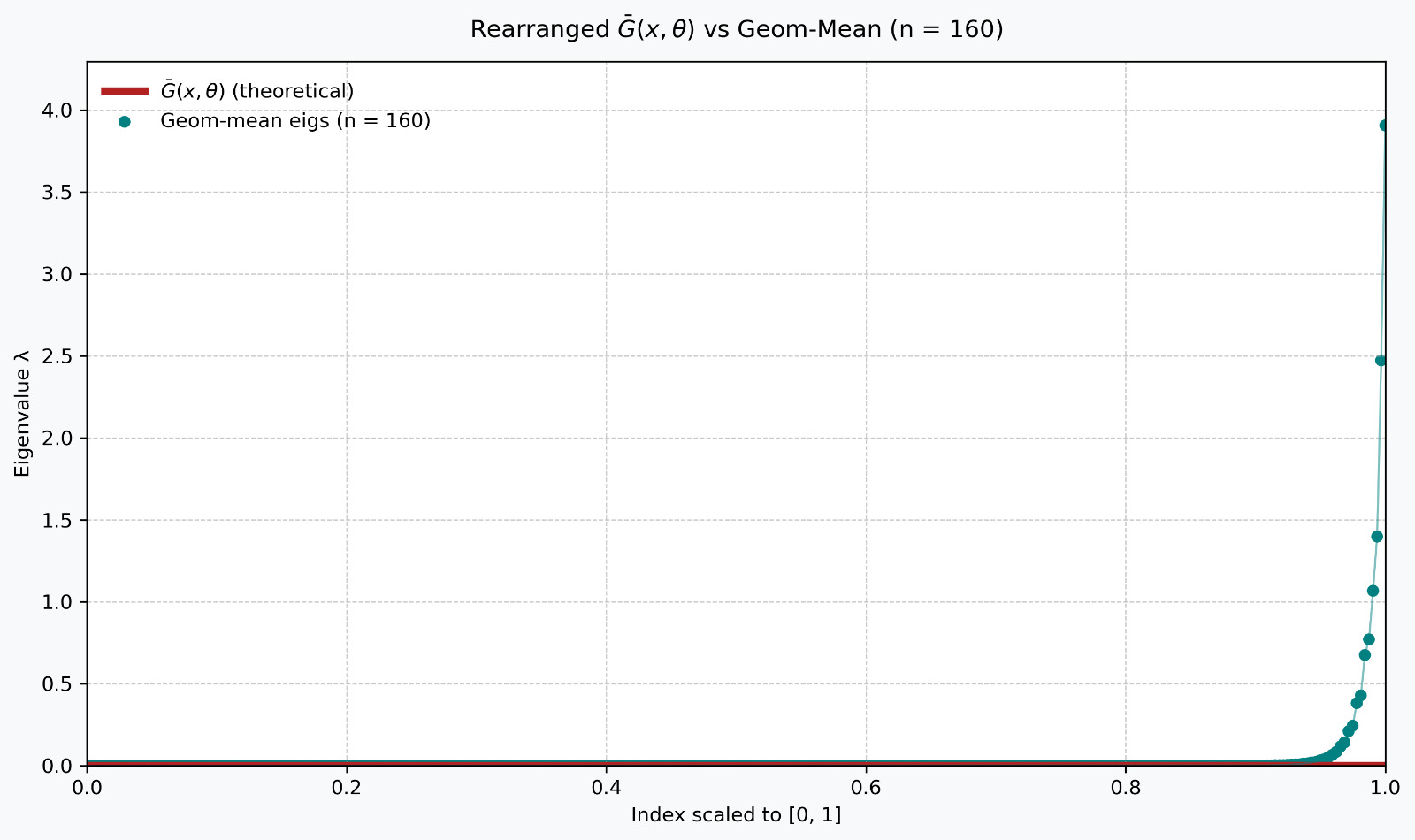}
\end{subfigure}
\hfill
\begin{subfigure}[t]{0.49\textwidth}
\centering
\includegraphics[width=\textwidth]{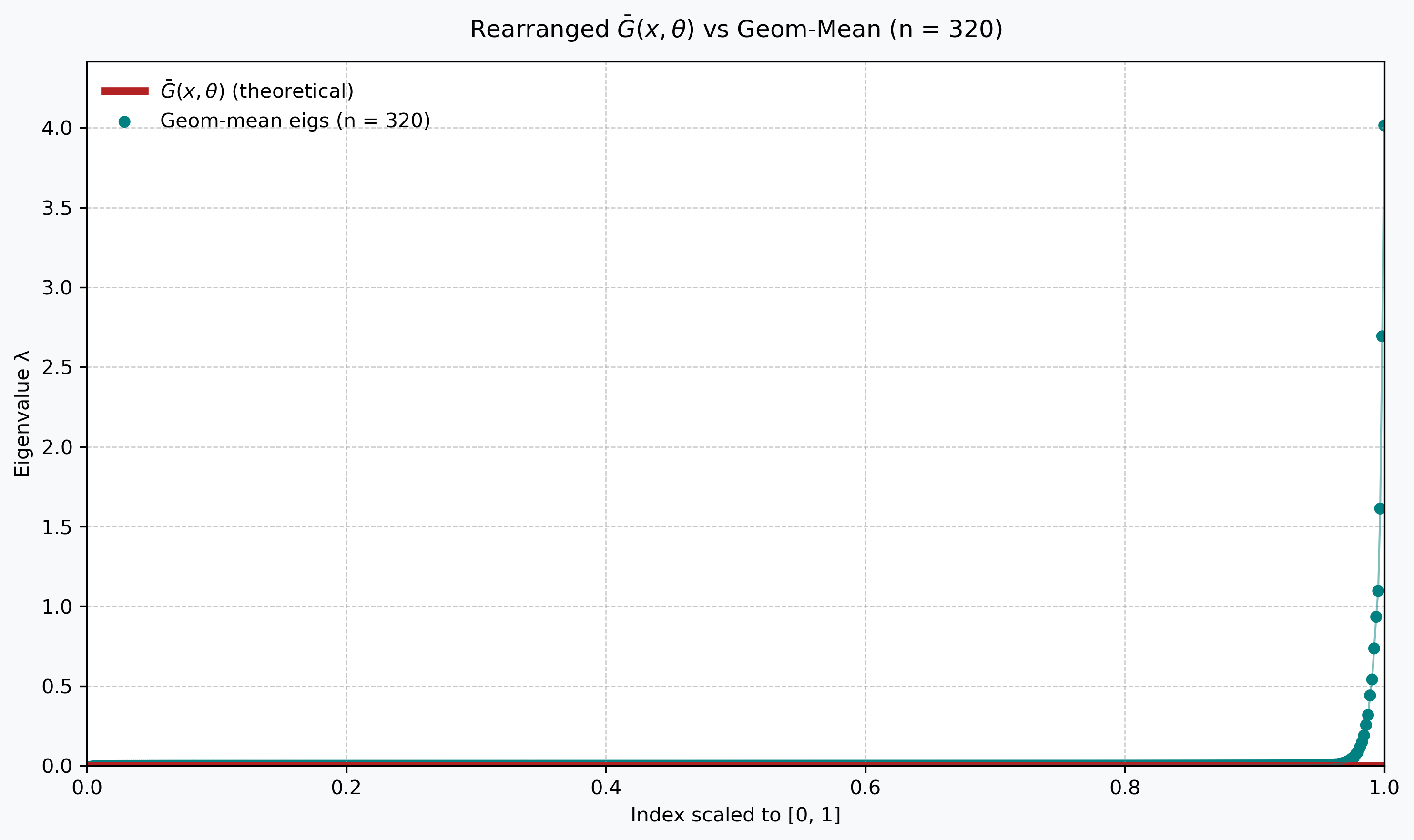}
\end{subfigure}
\caption{{\textbf{Case 1 - Example 1.}
Sorted eigenvalues of \(G(A_n,B_n)\) (colored markers; \(n=40,80,160,320\))
versus the rearranged distribution of the symbol
\(\widetilde G(\kappa,\xi)\) (solid red line)}}.
\label{fig:case1_ex1}
\end{figure}

\begin{figure}[H]
\centering
\begin{subfigure}[t]{0.49\textwidth}
\centering
\includegraphics[width=\textwidth]{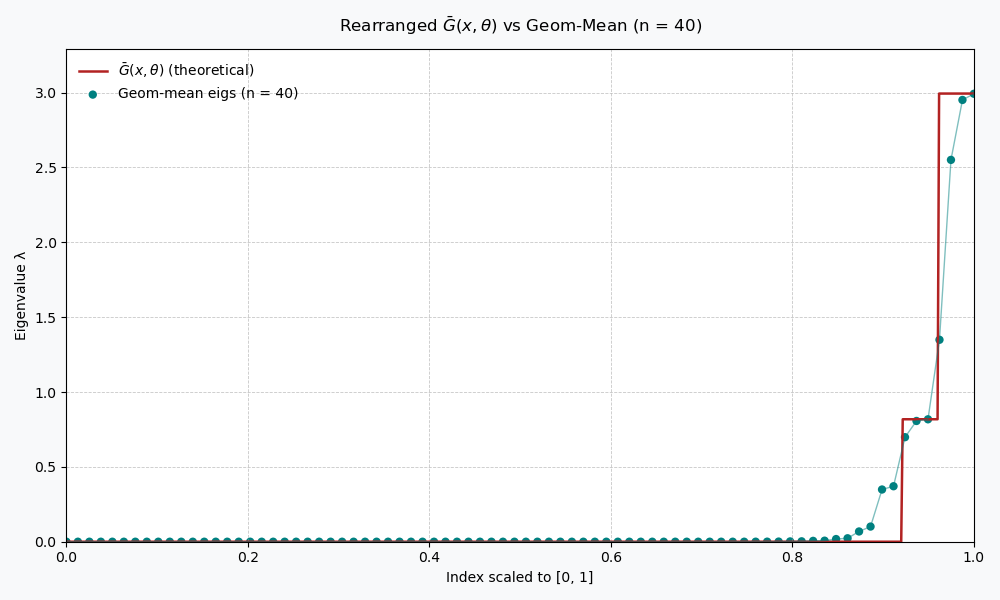}
\end{subfigure}
\hfill
\begin{subfigure}[t]{0.49\textwidth}
\centering
\includegraphics[width=\textwidth]{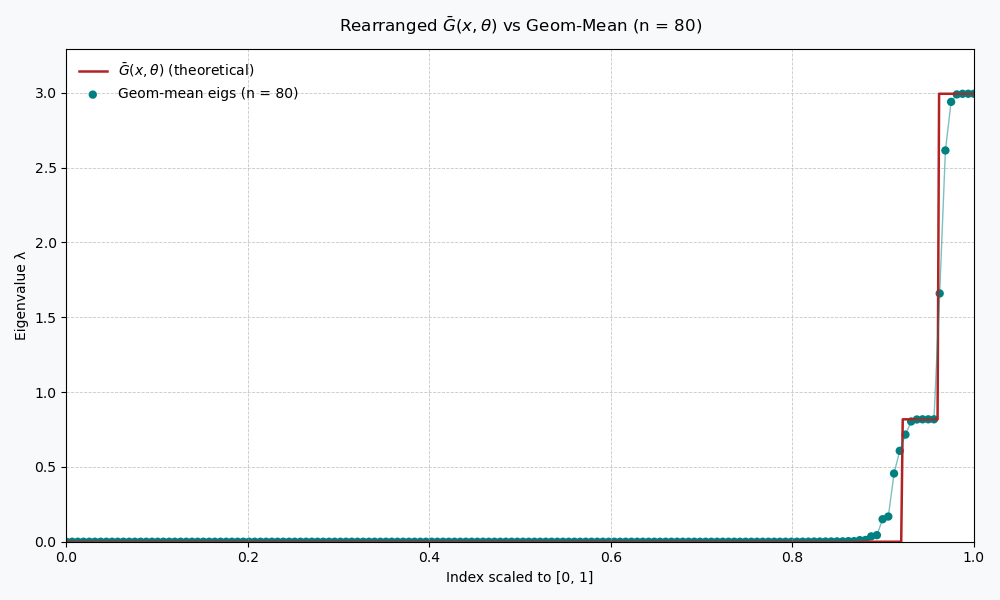}
\end{subfigure}

\begin{subfigure}[t]{0.49\textwidth}
\centering
\includegraphics[width=\textwidth]{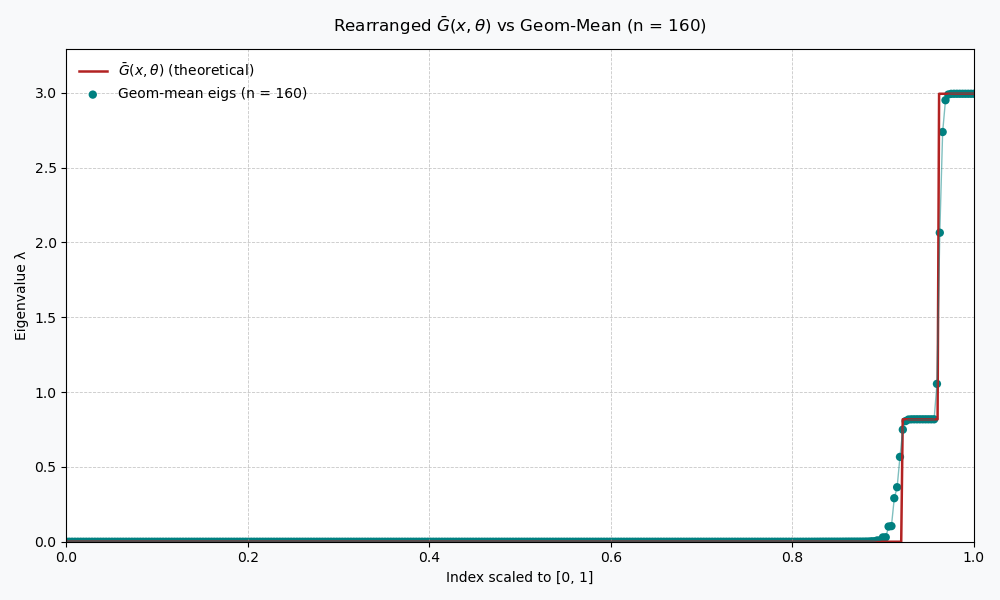}
\end{subfigure}
\hfill
\begin{subfigure}[t]{0.49\textwidth}
\centering
\includegraphics[width=\textwidth]{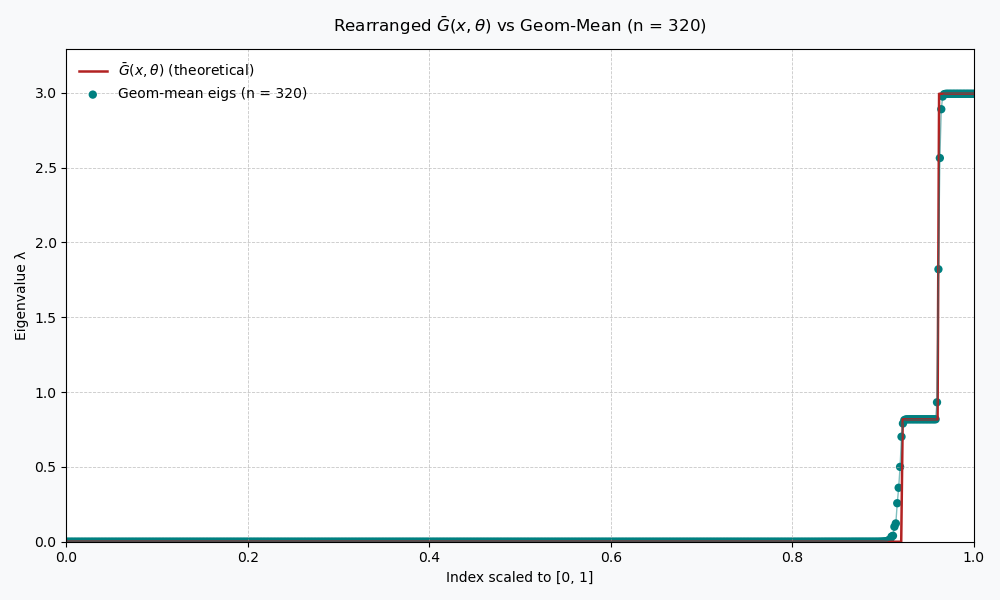}
\end{subfigure}

\caption{{\textbf{Case 1 - Example 2.}
Comparison of the ordered eigenvalues of \(G(A_n,B_n)\)
(colored markers) with the rearranged values of
\(\widetilde G(\kappa,\xi)\) (solid red line) for
\(n=40,80,160,320\).}}
\label{fig:case1_ex2}
\end{figure}
\FloatBarrier

\paragraph{Case 2.}
Figures~\ref{fig:case2_ex1}--\ref{fig:case2_ex2} repeat the experiment for
symbols that are rank-deficient on sets of full measure.

\smallskip
\emph{Example 1.}
Because the supports are disjoint, we have \(\widetilde{G}(\kappa,\xi)\equiv 0\).
The eigenvalues converge to the zero symbol. Despite the momentary symbol may seem to cause a rougher perturbation, we remark that, unlike in \textbf{Case 1 - Example 1}, this perturbation does not result in proper outliers since all perturbed eigenvalues converge to zero. In practice, we observe a better convergence.

\smallskip
\emph{Example 2.}
Here, the intersection set has rank one, leading to exactly one positive eigenvalue
in \(\widetilde{G}(\kappa,\xi)\).
The plot shows convergence despite the rank deficiency of the support. Because the norm of the diagonal perturbation introduced in \eqref{eq:case2_ex2} decays slowly, a momentary GLT perturbation \cite{new momentary,bolten2023note} is still observed, even in this non-zero case.

\begin{figure}[H]
\centering
\begin{subfigure}[t]{0.49\textwidth}
\centering
\includegraphics[width=\textwidth]{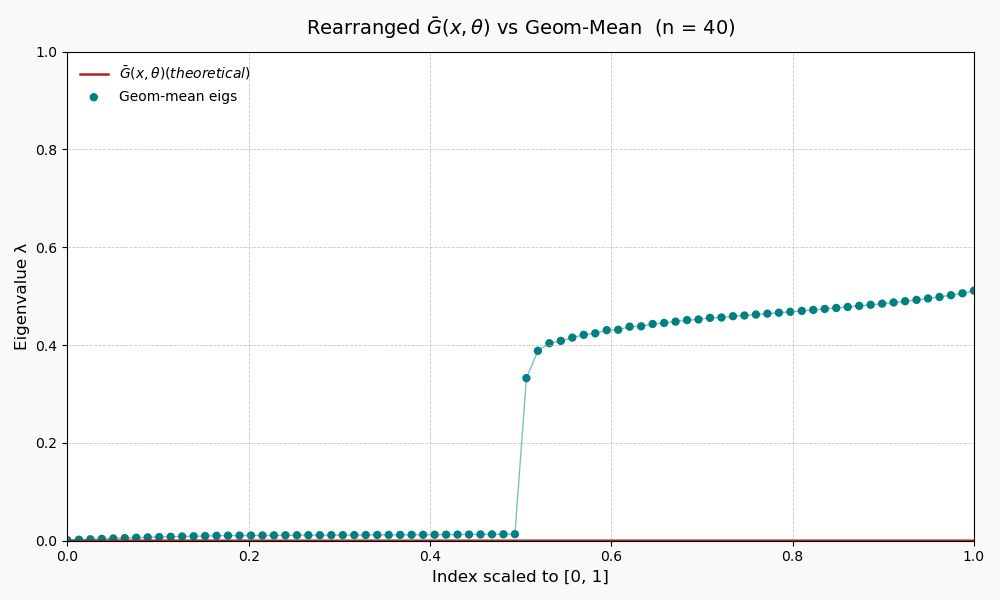}
\end{subfigure}
\hfill
\begin{subfigure}[t]{0.49\textwidth}
\centering
\includegraphics[width=\textwidth]{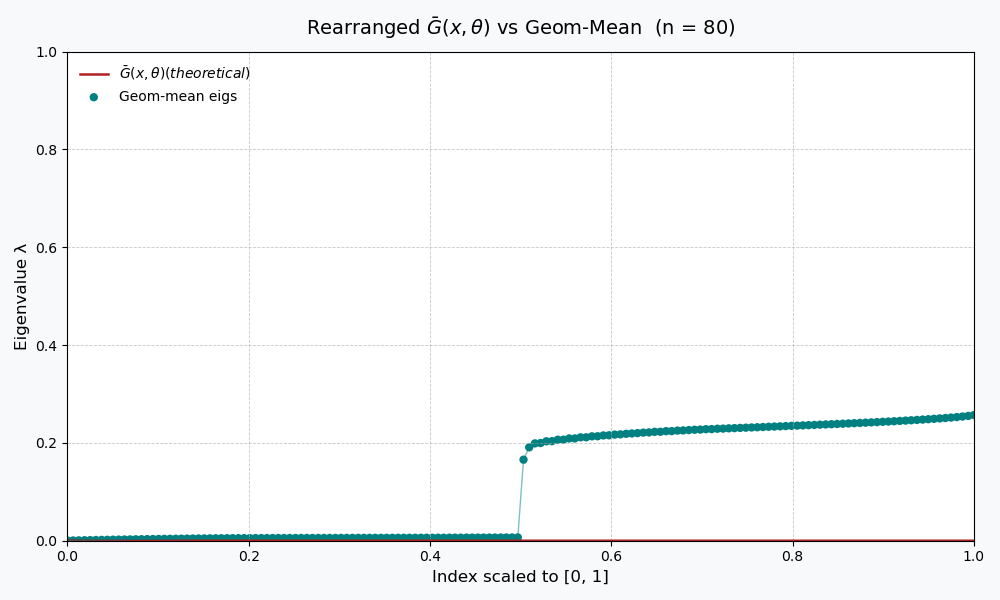}
\end{subfigure}
\\[-1em]

\begin{subfigure}[t]{0.49\textwidth}
\centering
\includegraphics[width=\textwidth]{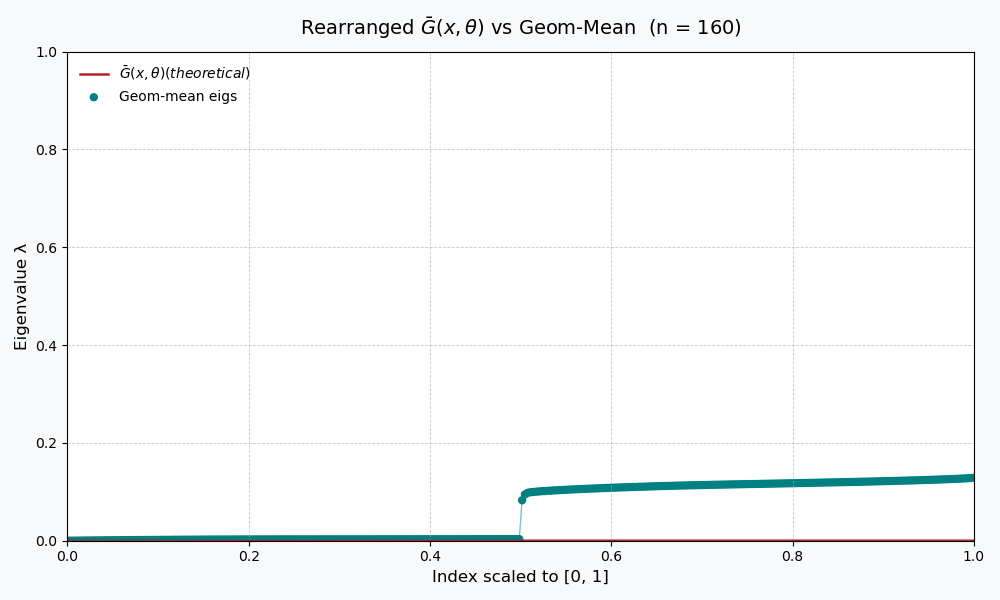}
\end{subfigure}
\hfill
\begin{subfigure}[t]{0.49\textwidth}
\centering
\includegraphics[width=\textwidth]{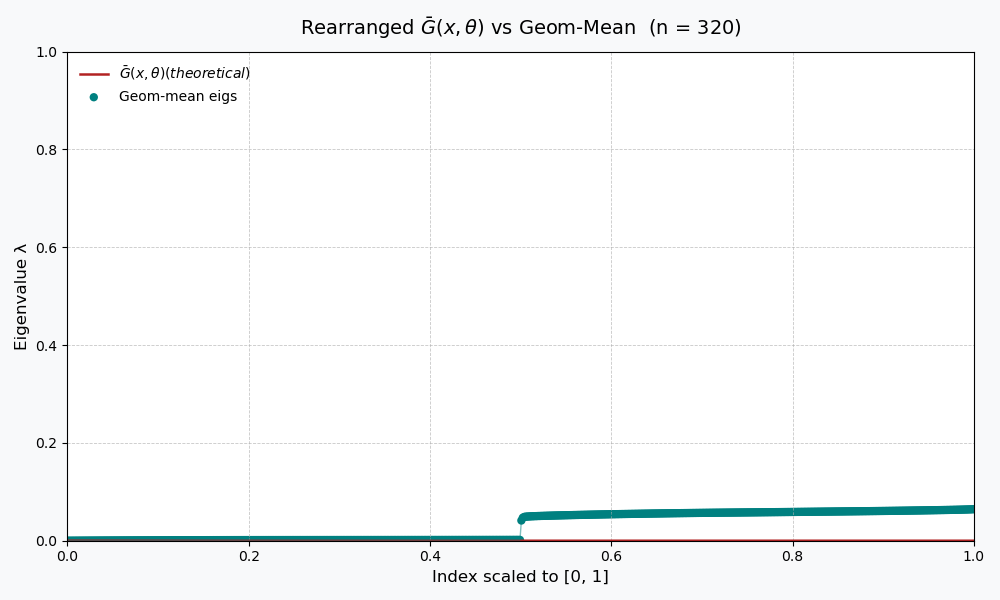}
\end{subfigure}
\caption{{\textbf{Case 2, Example 1.} 
Sorted eigenvalues of \(G(A_n,B_n)\) (colored markers) versus the rearranged eigenvalue distribution predicted by \(G(\kappa,\xi)\) (solid red line), for
\(n=40,80,160,320\).}}
\label{fig:case2_ex1}
\end{figure}

\begin{figure}[H]
\centering
\begin{subfigure}[t]{0.49\textwidth}
\centering
\includegraphics[width=\textwidth]{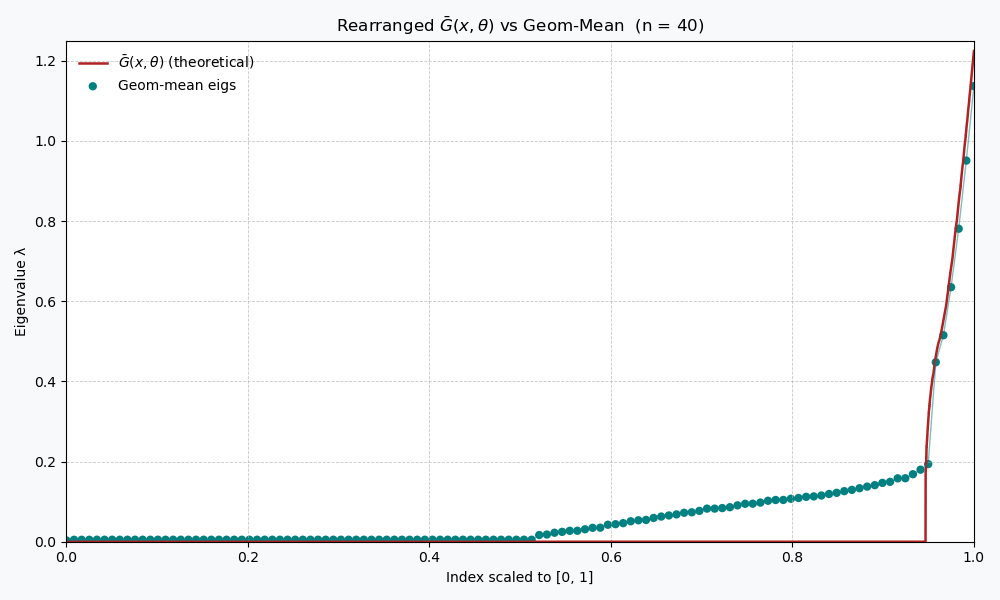}
\end{subfigure}
\hfill
\begin{subfigure}[t]{0.49\textwidth}
\centering
\includegraphics[width=\textwidth]{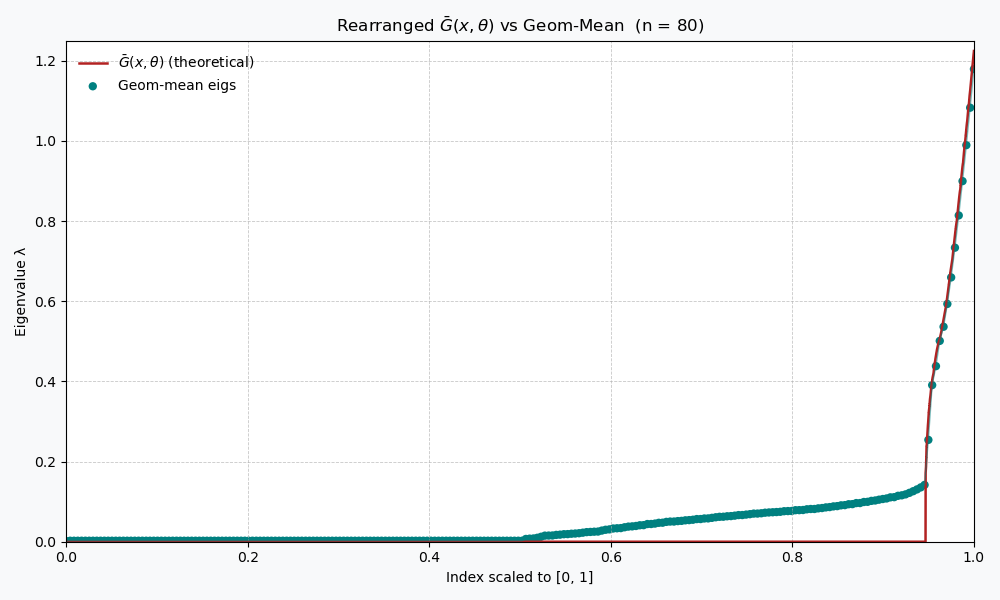}
\end{subfigure}
\\[-1em]

\begin{subfigure}[t]{0.49\textwidth}
\centering
\includegraphics[width=\textwidth]{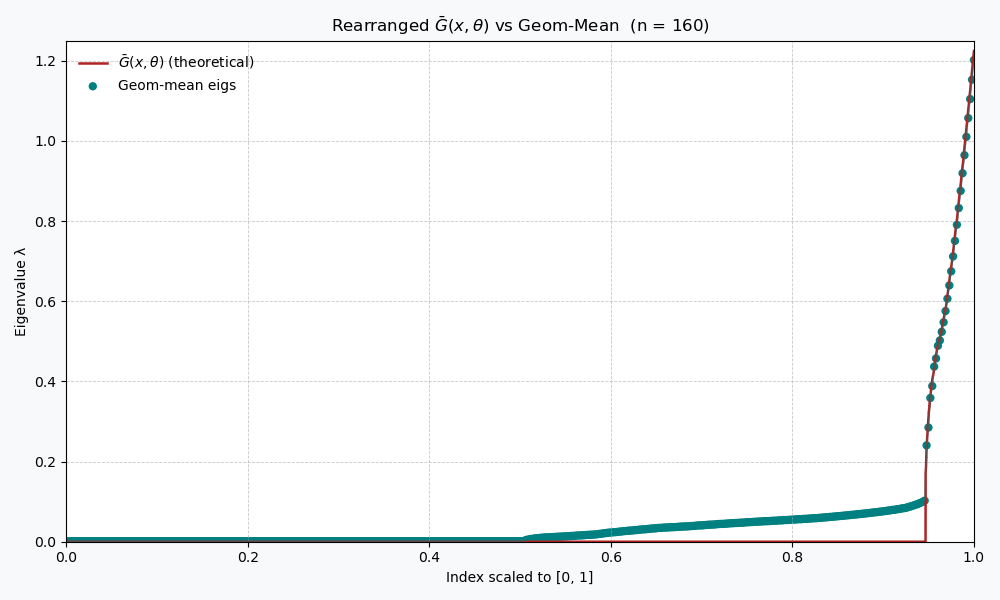}
\end{subfigure}
\hfill
\begin{subfigure}[t]{0.49\textwidth}
\centering
\includegraphics[width=\textwidth]{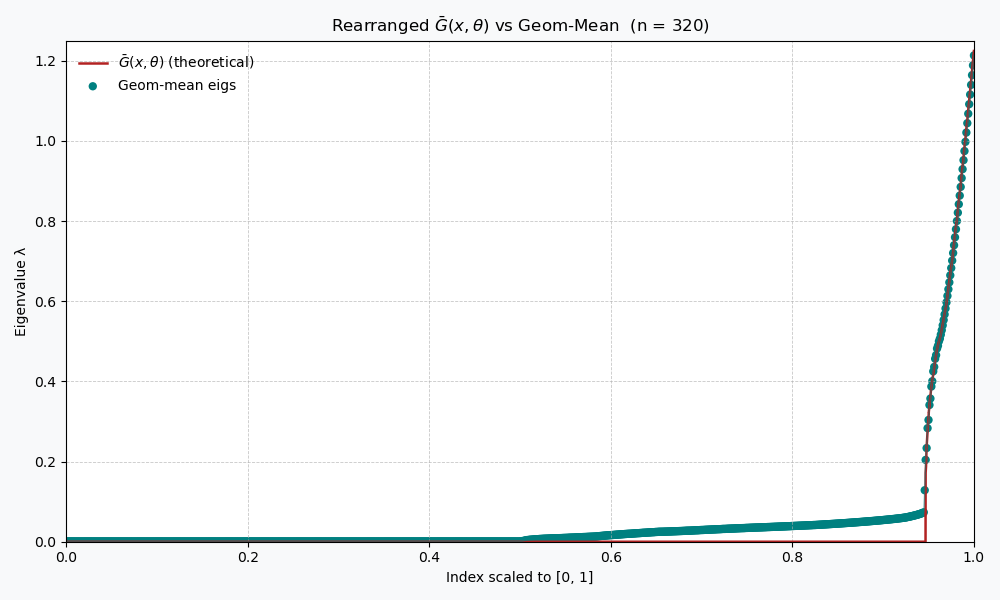}
\end{subfigure}

\caption{{\textbf{Case 2, Example 2.} Sorted eigenvalues of \(G(A_n,B_n)\) (colored markers) versus the rearranged eigenvalue distribution from \(G(\kappa,\xi)\) (solid red line), for
\(n=40,80,160,320\).}}
\label{fig:case2_ex2}
\end{figure}
\FloatBarrier
\paragraph{Extremal eigenvalues.}
For each block size \(n\) we report
\[
\lambda_{\min}(G_n), \qquad
\lambda_{\max}(G_n), \qquad
\mathrm{cond}_2(G_n)=\frac{\lambda_{\max}(G_n)}{\lambda_{\min}(G_n)},
\]
together with the essential infimum and supremum of the conjectured
symbol, \(\widetilde G_{\min}\) and \(\widetilde G_{\max}\).
In this way, we can observe how rapidly the minimum eigenvalue decays and, in most cases, how the maximum eigenvalues converge to the essential supremum of the spectral symbol.
This provides a clear measure of the conditioning behaviour as \(n\) grows.

\medskip\noindent

\begin{table}[H]
\centering
\caption{Extremal eigenvalues and condition numbers for
\textbf{Case 1 - Example 1} (symbol zero almost everywhere).}
\label{tab:case1_ex1_range}
\renewcommand{\arraystretch}{1.2}
\setlength{\tabcolsep}{8pt}
\begin{tabular}{|c
|>{\centering\arraybackslash}m{1.6cm}
|>{\centering\arraybackslash}m{1.6cm}
|>{\centering\arraybackslash}m{2.4cm}
|>{\centering\arraybackslash}m{2.3cm}
|>{\centering\arraybackslash}m{2.4cm}|}
\hline
{$n$} & {\(\widetilde G_{\min}\)}
& {\(\widetilde G_{\max}\)}
& {Min.\ eig.}
& {Max.\ eig.}
& {\(\mathrm{cond}_2(G_n)\)}\\
\hline
40  & 0 & 0 & \(8.714 \times 10^{-3}\) & 3.912 & \(4.490 \times 10^{2}\) \\
80  & 0 & 0 & \(3.684 \times 10^{-3}\) & 4.036 & \(1.095 \times 10^{3}\) \\
160 & 0 & 0 & \(1.492 \times 10^{-3}\) & 4.133 & \(2.769 \times 10^{3}\) \\
320 & 0 & 0 & \(5.833 \times 10^{-4}\) & 4.211 & \(7.220 \times 10^{3}\) \\
\hline
\end{tabular}
\end{table}
\paragraph{Case 1 - Example 2.}
Here \(\widetilde G_{\max}=2.9939\) is strictly positive. In this example, we observe a fast convergence of \(\lambda_{\max}(G_n)\) and an exponential decay of
\(\lambda_{\min}(G_n)\), which clearly causes the condition number \(\mathrm{cond}_2(G_n)\) to escalate, reaching
\(\mathcal{O}(10^{8})\) at \(n=320\).

\begin{table}[H]
\centering
\caption{Extremal eigenvalues and condition numbers for
\textbf{Case 1 - Example 2} (symbol positive on an overlap set).}
\label{tab:case1_ex2_range}
\renewcommand{\arraystretch}{1.2}
\setlength{\tabcolsep}{8pt}
\begin{tabular}{|c
|>{\centering\arraybackslash}m{1.6cm}
|>{\centering\arraybackslash}m{1.8cm}
|>{\centering\arraybackslash}m{2.4cm}
|>{\centering\arraybackslash}m{1.8cm}
|>{\centering\arraybackslash}m{2.4cm}|}
\hline
{$n$} & {\(\widetilde G_{\min}\)}
& {\(\widetilde G_{\max}\)}
& {Min.\ eig.}
& {Max.\ eig.}
& {\(\mathrm{cond}_2(G_n)\)}\\
\hline
40  & 0 & 2.993 & \(1.562 \times 10^{-5}\) & 2.992 & \(1.915 \times 10^{5}\) \\
80  & 0 & 2.993 & \(1.953 \times 10^{-6}\) & 2.993 & \(1.533 \times 10^{6}\) \\
160 & 0 & 2.993 & \(2.441 \times 10^{-7}\) & 2.993 & \(1.226 \times 10^{7}\) \\
320 & 0 & 2.993 & \(3.051 \times 10^{-8}\) & 2.993 & \(9.811 \times 10^{7}\) \\
\hline
\end{tabular}
\end{table}

\paragraph{Case 2 - Example 1.}
The minimum eigenvalue appear to decay to 0 at a quadratic rate, while the largest eigenvalue converges to zero linearly, as determined by the momentary symbol. As a consequence, the growth of
\(\mathrm{cond}_2(G_n)\) is moderate.
\medskip
\smallskip

\begin{table}[H]
\centering
\caption{Extremal eigenvalues and condition numbers for
\textbf{Case 2 - Example 1} (candidate symbol
\(\widetilde G \equiv 0\)).}
\label{tab:case2_ex1_range}
\renewcommand{\arraystretch}{1.2}
\setlength{\tabcolsep}{8pt}
\begin{tabular}{|c
|>{\centering\arraybackslash}m{1.6cm}
|>{\centering\arraybackslash}m{1.6cm}
|>{\centering\arraybackslash}m{2.4cm}
|>{\centering\arraybackslash}m{1.8cm}
|>{\centering\arraybackslash}m{2.4cm}|}
\hline
{$n$} & {\(\widetilde G_{\min}\)}
& {\(\widetilde G_{\max}\)}
& {Min.\ eig.}
& {Max.\ eig.}
& {\(\mathrm{cond}_2(G_n)\)}\\
\hline
40  & 0 & 0 & \(1.261 \times 10^{-3}\) & 0.511 & \(4.053 \times 10^{2}\) \\
80  & 0 & 0 & \(3.210 \times 10^{-4}\) & 0.256 & \(8.002 \times 10^{2}\) \\
160 & 0 & 0 & \(8.099 \times 10^{-5}\) & 0.128 & \(1.589 \times 10^{3}\) \\
320 & 0 & 0 & \(2.034 \times 10^{-5}\) & 0.064 & \(3.169 \times 10^{3}\) \\
\hline
\end{tabular}
\end{table}

\medskip
\smallskip

\paragraph{Case 2 - Example 2.}
Here \(\widetilde G_{\max}\simeq1.22\).
\(\lambda_{\max}(G_n)\) converges to this value and is included in the range of $\widetilde G$, whereas
\(\lambda_{\min}(G_n)\) decays linearly. This better conditioning compared to \textbf{Example~2} of the first case is reasonable, since the positive perturbations we introduced to the Toeplitz matrices of this example (Equation \ref{eq:case2_ex2} decay in norm more slowly than the ones in \textbf{Case 1 - Example 1}).

\medskip

\begin{table}[H]
\centering
\caption{Extremal eigenvalues and condition numbers for
\textbf{Case 2 - Example 2} (symbol positive on an overlap set).}
\label{tab:case2_ex2_range}
\renewcommand{\arraystretch}{1.2}
\setlength{\tabcolsep}{8pt}
\begin{tabular}{|c
|>{\centering\arraybackslash}m{1.6cm}
|>{\centering\arraybackslash}m{1.8cm}
|>{\centering\arraybackslash}m{2.4cm}
|>{\centering\arraybackslash}m{1.8cm}
|>{\centering\arraybackslash}m{2.4cm}|}
\hline
{$n$} & {\(\widetilde G_{\min}\)}
& {\(\widetilde G_{\max}\)}
& {Min.\ eig.}
& {Max.\ eig.}
& {\(\mathrm{cond}_2(G_n)\)}\\
\hline
40  & 0 & 1.224 & \(3.24 \times 10^{-3}\) & 1.136 & \(3.50 \times 10^{2}\) \\
80  & 0 & 1.224 & \(1.46 \times 10^{-3}\) & 1.179 & \(8.10 \times 10^{2}\) \\
160 & 0 & 1.224 & \(5.65 \times 10^{-4}\) & 1.201 & \(2.13 \times 10^{3}\) \\
320 & 0 & 1.224 & \(8.19 \times 10^{-5}\) & 1.213 & \(1.48 \times 10^{4}\) \\
\hline
\end{tabular}
\end{table}

\paragraph{Zero-related statistics.}
For each block size \(n\), we report the fraction of eigenvalues whose
modulus does not exceed the fixed cutoff \(0.1\); see
Tables~\ref{tab:case1_ex1_zero}-\ref{tab:case2_ex2_zero}.
Each table lists the empirical proportion
(\emph{Prop.\(\le0.1\)}), the theoretical measure of the zero set of the
rearranged symbol (\emph{Target}), and the absolute difference (\emph{Error}).

\smallskip
\paragraph{Case 1 - Example 1.}
The conjectured symbol is the zero function, so the target measure is~1. The empirical value increases monotonically with~\(n\), remaining  within a small error.
Because the number of outliers is at most \(o(n)\), we expect that the proportion will converge to 1 as $n \to \infty $.
\paragraph{Case 1 - Example 2.}
Here, the zero set of the rearranged spectral symbol has a measure
\[
1-\frac{1}{4\pi}\approx 0.9204.
\]%
The measured proportions appear to approach this value with an
\(\mathcal O(n^{-1})\) rate of decay, with the error falling below \(10^{-3}\) by \(n=160\).

\medskip
\paragraph{Case 2 - Example 1.}
Here the expected spectral symbol is zero, but the momentary symbol perturbation,
as observed in Figure~\ref{fig:case2_ex1}, produces an apparent
plateau \(0.5\) up to \(n=320\), where the proportion jumps to~1 and the measured error vanishes.
Since the momentary perturbation does not create proper outliers, as observed in \textbf{Case 1 - Example 2}, a faster convergence can reasonably be expected in this case.

\paragraph{Case 2 - Example 2.}
In this case, the complement of the support of the rearranged distribution has measure
\[
\frac{17}{18}\approx 0.9444.
\]
The empirical measure of the zero cluster converges to this value, with the error appearing to follow the same power law observed in \textbf{Case 2 - Example 1}.

\medskip
\begin{table}[H]
\centering
\caption{Case 1, Example 1: proportion of eigenvalues below \(0.1\);\\ target \(0\).}
\label{tab:case1_ex1_zero}
\begin{tabular}{|c|S|S|S|}
\hline
{$n$} & {Prop.\(\le0.1\)} & {Target} & {Error} \\
\hline
40 & 0.6375 & 1.0000 & 0.3625 \\
80 & 0.8938 & 1.0000 & 0.1062 \\
160 & 0.9438 & 1.0000 & 0.0562 \\
320 & 0.9688 & 1.0000 & 0.0312 \\
\hline
\end{tabular}
\end{table}

\begin{table}[H]
\centering
\caption{Case 1, Example 2: proportion of eigenvalues below \(0.1\);\\
target \(1-\frac{1}{4\pi}\approx0.92042\).}
\label{tab:case1_ex2_zero}
\begin{tabular}{|c|S|S|S|}
\hline
{$n$} & {Prop.\(\le0.1\)} & {Target} & {Error} \\
\hline
40 & 0.8750 & 0.9204 & 0.0454 \\
80 & 0.8938 & 0.9204 & 0.0266 \\
160 & 0.9031 & 0.9204 & 0.0173 \\
320 & 0.9109 & 0.9204 & 0.0095 \\
\hline
\end{tabular}
\end{table}

\begin{table}[H]
\centering
\caption{Case 2, Example 1: proportion of eigenvalues below \(0.1\);\\ target \(0\).}
\label{tab:case2_ex1_zero}
\begin{tabular}{|c|S|S|S|}
\hline
{$n$} & {Prop.\(\le0.1\)} & {Target} & {Error} \\
\hline
40 & 0.5000 & 1.0000 & 0.5000 \\
80 & 0.5000 & 1.0000 & 0.5000 \\
160 & 0.5156 & 1.0000 & 0.4844 \\
320 & 1.0000 & 1.0000 & 0.0000 \\
\hline
\end{tabular}
\end{table}

\begin{table}[H]
\centering
\caption{Case 2, Example 2: proportion of eigenvalues below \(0.1\);\\
target \(\tfrac{17}{18}\approx0.9444\).}
\label{tab:case2_ex2_zero}
\begin{tabular}{|c|S|S|S|}
\hline
{$n$} & {Prop.\(\le0.1\)} & {Target} & {Error} \\
\hline
40 & 0.7667 & 0.9444 & 0.1777 \\
80 & 0.8833 & 0.9444 & 0.0611 \\
160 & 0.9438 & 0.9444 & 0.0006 \\
320 & 0.9448 & 0.9444 & 0.0004 \\
\hline
\end{tabular}
\end{table}

\chapter{\texorpdfstring{Hidden GLT structures in quantum spin models}
                          {Hidden GLT Structures in Quantum Spin Models}}
\label{chapter:quantum}

This chapter presents important results about the connection between mathematical models and physical systems, as published in \cite{GLTSpin2025}. The authors thoroughly investigate structured matrix-sequences originating from mean-field quantum spin systems, employing the GLT algebra framework to deepen the theoretical understanding of these systems. Specifically, the study delves into the quantum Curie-Weiss (CW) model, an important paradigm for mean-field interactions, recognized in quantum statistical mechanics due to its ability to illustrate phenomena such as spontaneous symmetry breaking and critical phase transitions.

The research systematically addresses two variants of the Curie-Weiss model: the unrestricted and restricted versions. For the general model, the paper identifies the normalized Curie-Weiss Hamiltonian as a zero-distributed GLT sequence. This finding indicates that the sequence exhibits trivial asymptotic spectral and singular value distributions, effectively characterized by a zero symbol within GLT theory.

In contrast, the restricted version of the Curie-Weiss model presents a significantly richer structure. The corresponding normalized matrix-sequences are demonstrated to form GLT sequences with a nontrivial symbol. This GLT symbol, explicitly derived in the paper, describes the intricate asymptotic behavior of eigenvalues and singular values for these matrix-sequences.

Numerical experiments are rigorously conducted to validate theoretical assertions, including precise visualizations of spectral distributions. These numerical analyses underscore a remarkable consistency between theoretical predictions and observed numerical results, further confirming the robustness of the GLT approach.

\section{Curie-Weiss model}\label{sec: CW-model}
We consider the Hamiltonian for the quantum Curie-Weiss model for ferromagnetism, which takes the form
\begin{align}
    H_{\Lambda_N}^{\text{CW}}=-\frac{\Gamma}{2|\Lambda_N|}\sum_{x,y\in \Lambda_N}\sigma_3(x)\sigma_3(y)-B\sum_{x\in\Lambda_N}\sigma_1(x), \label{eq:curieweiss}
\end{align}
where $\Lambda_N$ is an arbitrary finite subset of $\mathbb{Z}^\ell$, $\Gamma>0$ scales the spin-spin coupling, and $B$ is an external magnetic field. This model describes a chain of $N$ immobile spin-$1/2$ particles with ferromagnetic coupling in a transverse magnetic field. The Hamiltonian acts on the Hilbert space $\mathcal{H}_{\Lambda_N}=\otimes_{x\in\Lambda_N}H_x$, where $H_x=\mathbb{C}^2$.
The operator $\sigma_i(x), \dots (i=1,2,3)$ acts as the Pauli matrix $\sigma_i$ on $H_x$ and acts as the unit matrix $\mathbb{1}_2$ elsewhere. The single Pauli matrices are explcitly given by
\begin{align}\label{basic matrices}
    \sigma_1=\begin{pmatrix}
    0 & 1 \\ 1 & 0
\end{pmatrix},\ \ \ \sigma_2=\begin{pmatrix}
    0 & -i \\ i & 0
\end{pmatrix},\ \ \ \sigma_3=\begin{pmatrix}
    1 & 0 \\ 0 & -1
\end{pmatrix}.
\end{align}
\noindent
In contrast to locally interacting quantum spin models, the spatial dimension of this model does not influence the behaviour. This follows from the fact that for the averages
\begin{align}\label{spinss}
    S_i^{\Lambda_N}=\frac{1}{|\Lambda_N|}\sum_{x\in\Lambda_N}\sigma_i(x), \ \ \ (i=1,2,3),
\end{align}
we can write the Hamiltonian \eqref{eq:curieweiss} (see e.g. \cite{Ven_2024}) as
\begin{align}\label{hamcw}
    H_{\Lambda_N}^{\text{CW}}=-|\Lambda_N|\bigg{(}\frac{\Gamma}{2}(S_3^{\Lambda_N})^2+BS_1^{\Lambda_N})\bigg{)}=
    |\Lambda_N|\bigg{(}h_0^{\text{CW}}({\bf S}^{\Lambda_N})\bigg{)},
\end{align}
for ${\bf S}^{\Lambda_N}=(S_1^{\Lambda_N},S_2^{\Lambda_N},S_3^{\Lambda_N})$, and the choice
\begin{align}\label{classical CW}
B^3\ni (x,y,z)\mapsto h_0^{\text{CW}}(x,y,z)=-\bigg{(}\frac{\Gamma}{2}z^2+Bx\bigg{)},
\end{align}
with $B^3=\{(x,y,z)\ | \ x^2+y^2+z^2\leq 1\}$ the unit three-dimensional sphere in $\mathbb{R}^3$.

Using spherical coordinates ${\bf e}(\Omega)=(\sin{\vartheta}\cos{\varphi},\sin{\vartheta}\sin{\varphi},\cos{\vartheta})$,  with $\vartheta\in [0,\pi]$ and $\varphi\in [0,2\pi)$; for radius $u\in [0,1]$, we may now rewrite
 \begin{align}\label{classicalCWmodel}
 h_0^{CW}(u\cdot {\bf e}(\Omega))=-\bigg{(}\frac{\Gamma}{2}(u\cos{\vartheta})^2+Bu\sin{\vartheta}\cos{\varphi}\bigg{)},
\end{align}
as being a polynomial function on $[0,1]\times\mathbb{S}^2$.

As a result, we may as well consider the quantum Curie–Weiss
Hamiltonian \eqref{eq:curieweiss} in $\ell=1$, so that we may simply write $|\Lambda_N|=N$ and $H_{N}^{\text{CW}}:= H_{\Lambda_N}^{\text{CW}}$.

\subsection{The CW model as GLT-sequence}
Here we show that the real symmetric matrix-sequence $\{\bar{H}_{N}^{\text{CW}}\}_N$, $d_N=2^N$, where $\bar{H}_{N}^{\text{CW}}:= H_{N}^{\text{CW}}/N$, is a basic GLT matrix-sequence with GLT symbol $0$, i.e., a zero-distributed matrix-sequence (Axiom \textbf{GLT 2}, third item).

\begin{theorem}\label{main result}
The normalized  CW Hamiltonian $\bar{H}_{N}^{\text{CW}}:= H_{N}^{\text{CW}}/N$ defines $\{\bar{H}_{N}^{\text{CW}}\}_N$ as a zero-distributed GLT sequence.
\end{theorem}
\begin{proof}
The Curie-Weiss Hamiltonian is self-adjoint and in fact all the considered matrices are real symmetric. Thus in view of Definition~\ref{99} (with  $d_N=2^N$, $D=[0,1]\times\mathbb{S}^2$), it suffices to prove that  for all real-valued $F\in C_c(\mathbb{R})$
\begin{align}\label{limit1}
    \lim_{N\to\infty}\frac{1}{2^N}Tr_{2^N}[F(\bar{H}_N^{\text{CW}})]=F(0),
\end{align}
where $Tr_{2^N}$ denotes the trace on $\mathcal{H}_N=\bigotimes_{i=1}^N\mathbb{C}^2$.
The convergence in \eqref{limit1} for all possible test functions $F$ implies that the sequence $(\bar{H}_{\Lambda_N}^{\text{CW}})_N$ is a zero-distributed GLT sequence.

Since the spectrum of the normalized Curie-Weiss Hamiltonian $\bar{H}_N^{CW}$ is uniformly bounded in $N$ and contained in a connected compact subset $C \subset \mathbb{R}$ (see also \cite{Ven_2020,Ven_2022} for further details), it suffices to consider functions $F$ supported on $C$. By the Stone-Weierstrass theorem, any continuous function $F$ on $C$ can be uniformly approximated by polynomials  restricted to $C$: in fact this argument is standard and was used e.g. in \cite{GoSe} in the context of the zero distribution of orthogonal polynomials, when the Jacobi operator is perturbed by a non self-adjoint compact operator. Moreover, since the (normalized) trace operation is linear and continuous with respect to uniform convergence of continuous functions on the spectrum, it suffices to prove \eqref{limit1} for all polynomials $P$ restricted to $C$. Thus, we will establish \eqref{limit1} for all such polynomials $P$, which by approximation will imply the result for any general $F \in C_c(\mathbb{R})$ supported on $C$.

To do so, we stress that the Hamiltonian is homogeneously decomposable \cite[Section~2]{Cegla_Lewis_Raggio_1988},
implying that the model is block diagonal with the following tracial decomposition
\begin{align*}
   Tr_{2^N}[\bar{H}_N^{\text{CW}}]=\frac{1}{2^N}\sum_{J\in\mathbb{J}_N}C(J,N)Tr_{2J+1}[\bar{H}_N^{\text{CW}}(J)],
\end{align*}
where, by \cite{Mihailov_1977}, the quantities
$$C(J,N)=\frac{2J+1}{N+1}\binom{N+1}{\frac{N}{2}+J+1},$$
are the multiplicities of the $(2J+1)$-dimensional irreducible unitary representations arising in the decomposition of the $N$-fold tensor product representation of $SU(2)$ onto $\mathbb{C}^2$ with itself. Here, $\mathbb{J}_N=\{0,1,...,N/2\}$ if $N/2$ is an integer, and equals $\{1/2,3/2,...,N/2
\}$ if $N/2$ is a half-integer.
The $(2J+1)\times (2J+1)$-dimensional matrix $\bar{H}_N^{\text{CW}}(J)$ is  defined as \begin{align}\label{restricted CW}
\bar{H}_N^{\text{CW}}(J):=
h_0^{CW}({\bf S}^N)|_{2J+1}.
\end{align}

Define the function
\begin{align}\label{newfunction}
    P_N(J):=P(\bar{H}_N^{\text{CW}}(J))=P(h_0^{CW}({\bf S}^N)|_{2J+1})=P(h_0^{CW}({\bf S}^N))|_{2J+1},
\end{align}
so that
\begin{align}\label{id0}
\frac{1}{2^N}Tr_{2^N}[P(h_0^{CW}({\bf S}^N))]=\frac{1}{2^N}\sum_{J\in\mathbb{J}_N}C(J,N)Tr_{2J+1}(P_N(J)).
\end{align}

This allows us to introduce a sequence of probability measures $\nu_N$, defined for Borel measurable sets $E\subset [0,1]$, by
\begin{align}\label{idprob}
    \nu_N(E)=\frac{1}{2^N}\sum_{\{J\in\mathbb{J}_N\ |\ 2J/N\in E\}}C(J,N)(2J+1).
\end{align}
Note that indeed $2^N=\sum_{\{J\ |\ 2J/N\in [0,1]\}}C(J,N)(2J+1)$, so that $\nu_N([0,1])=1$.
We recall the well-known {\em resolution of the identity} \cite{Ven_2020}, that is,
\begin{align}\label{resolution identity}
    \mathbb{1}_{2J+1}=\frac{2J+1}{4\pi}\int_{\mathbb{S}^2}d\mu_{\mathbb{S}^2}(\Omega)Pr(J,\Omega),
\end{align}
where $d\mu_{\mathbb{S}^2}$ is the uniform measure on the 2-sphere, and $Pr(J,\Omega)$ is the one-dimensional projection onto the linear span of the $(2J+1)$-dimensional spin-coherent state vector labeled by the point $\Omega\in\mathbb{S}^2$.
Let $$p_N\bigg{(}\frac{2J}{N},\Omega\bigg{)}:=Tr_{2J+1}[P_N(J)Pr(J,\Omega)].$$

By inserting the identity \eqref{resolution identity}, we can write
\begin{align}\label{id1}
Tr_{2J+1}(P_N(J))=&\frac{2J+1}{4\pi}\int_{\mathbb{S}^2}d\mu_{\mathbb{S}^2}(\Omega)Tr_{2J+1}[P_N(J)Pr(J,\Omega)]\nonumber\\=& \frac{2J+1}{4\pi}\int_{\mathbb{S}^2}d\mu_{\mathbb{S}^2}(\Omega)p_N\bigg{(}\frac{2J}{N},\Omega\bigg{)}.
\end{align}

If we set $K_N(du,d\Omega):=\nu_N(du)\times \frac{1}{4\pi}\mu_{\mathbb{S}^2}(d\Omega)$, equations \eqref{id0}, \eqref{idprob} and \eqref{id1} imply that
$$\frac{1}{2^N}Tr_{2^N}[P(h_0^{CW}({\bf S}^N))]=\int_{[0,1]\times \mathbb{S}^2}K_N(du,d\Omega)p_N(u,\Omega),$$
where the integral over $[0,1]$ should be interpreted as the discrete integral with  respect to $\nu_N$. It is clear that $K_N$ is a probability meausure for each $N$.
Using the previous preparatory statements, we are now in a position to prove the following facts.
\begin{itemize}
    \item[(i)]
    For all $N\in\mathbb{N}_+$, it holds
\begin{align}\label{onemli1}
\sup_{J\in\mathbb{J}_N}\sup_{\Omega\in\mathbb{S}^2}\bigg{|}p_N\bigg{(}\frac{2J}{N},\Omega\bigg{)} - P\bigg{(}h_0^{\text{CW}}\bigg{(}\frac{2J}{N}{\bf e}(\Omega)\bigg{)}\bigg{)}\bigg{|}\leq \frac{C}{N},
\end{align}
    the constant $C$ being independent of $J$, $\Omega$ and $N$.
    \item[(ii)] The sequence of probability measures $(K_N)_N$ converges setwise to the probability measure $\delta_0\times \frac{1}{4\pi}\mu_{\mathbb{S}^2}$, with $\delta_0$ the Dirac measure concentrated at $0\in [0,1]$.
\end{itemize}
\noindent
For (i), we first rely on the following fact obtained in \cite{Manai_Warzel}. For any  non-commuting self-adjoint polynomial $H(J):=P_0({\bf S}^N)|_{2J+1}$ on $\mathbb{C}^{2J+1}$, it holds
\begin{align}\label{onemli2}
    \sup_{J\in \mathbb{J}_N}\sup_{\Omega\in\mathbb{S}^2}\bigg{|}Tr_{2J+1}[Pr(J,\Omega)H(J)]-P_0\bigg{(}\frac{2J}{N}{\bf e}(\Omega)\bigg{)}\bigg{|}\leq \frac{C}{N},
\end{align}
where the constant $C>0$ only depends on the chosen $P_0$, not on $J$, $\Omega$ and $N$.

To prove \eqref{onemli1}, we stress that for any given polynomial $P$, the operator $P\circ h_0^{CW}({\bf S}^N)$ is again a polynomial in three non coummuting self-adjoint spin operators. As a matter of fact, the statement \eqref{onemli1} follows from \eqref{onemli2} for the choice $P_0=P\circ h_0^{CW}$.

In order to see that (ii) holds, we rewrite
\begin{align}\label{nu}
    \nu_N(E)&=\frac{1}{2^N}\sum_{\{J\ |\ 2J/N\in E\}}\frac{(2J+1)^2}{N+1}\binom{N+1}{N/2+J+1}.
\end{align}
Define $X_N\sim Bin(N+1,1/2)$ and write $J=X_N-\frac{N}{2}-1$. Use the binomial probability mass function (pmf) $BinProb$,
$$BinProb(X_N=k)=\binom{N+1}{k}\frac{1}{2^{N+1}},$$ so that \eqref{nu} becomes
\begin{align*}
    \nu_N(E)&=2\sum_{\{X_N\ |\ \frac{2(X_N-\frac{N}{2}-1)}{N}\in E\}}\frac{(2(X_N-\frac{N}{2}-1)+1)^2}{N+1}BinProb(X_N).
\end{align*}
Consider $\epsilon>0$ arbitrary and let $E\subset [0,1]$ be a measurable set such that $E\cap [0,\epsilon]=\emptyset$.
We deduce that
$$\frac{2J}{N}\in E\ \ \ \implies \ \ \  X_N \in \bigg{(} \frac{N(\epsilon+1)+2}{2}, N+1\bigg{]}.$$
As a result,
\begin{align}\label{concentration}
&BinProb\bigg{(}\frac{2(X_N-\frac{N}{2}-1)}{N}\in E\bigg{)}\leq BinProb\bigg{(}X_N\geq \frac{N(\epsilon+1)+2}{2}\bigg{)}\nonumber\\&\leq \exp{\bigg{(}\frac{-(\frac{N\epsilon+1}{N+1})^2\frac{N+1}{2}}{3}\bigg{)}}\nonumber\\&=
\exp{\bigg{(}-\frac{(N\epsilon+1)^2}{6(N+1)}\bigg{)}},
\end{align}
where, in the second inequality, we have exploited Chernoff's inequality for the binomial distribution \cite[Theorem 4.4]{Mitz}. To derive this, consider $\mu=\frac{N+1}{2}$ and $p=1/2$, so that on account of the standard Chernoff inequality with $0<\delta< 1$, we find
$$BinProb(X_N\geq (1+\delta)\mu)\leq e^{-\frac{\delta^2}{3}\mu}.$$
If $X_N\sim Bin(N+1,1/2)$, then $\mu=\frac{N+1}{2}$, so that for the choice $\delta=(a-\mu)/\mu$ with $a=(N(\epsilon+1)+2)/2$ one indeed obtains \eqref{concentration}, as certainly $0<\delta< 1$.
The inequality \eqref{concentration} implies that the pmf $BinProb$ concentrates around zero, as $N\to\infty$, as long as $E$ is bounded away from zero.

It follows that for all $E$ with $E\cap [0,\epsilon]=\emptyset$,
\begin{align}\label{measures}
\nu_N(E)\leq 2(N+1)^2e^{-(N\epsilon+1)^2/6(N+1)}.
\end{align}
As a result, $(\nu_N)_N$ converges setwise to the Dirac mass concentrated at zero. Since the measure $\mu_{\mathbb{S}^2}$ does not depend on $N$, the same statement holds true for $K_N=\nu_N\times  \frac{1}{4\pi}\mu_{\mathbb{S}^2}$, i.e. $(K_N)_N$ converges setwise to $\delta_0\times  \frac{1}{4\pi}\mu_{\mathbb{S}^2}$. This shows the validity of $(ii)$.

We are finally in a position to prove \eqref{limit1}. To this avail, we notice that $P\circ h_0^{CW}$ is uniformly continuous. Hence, given $\varepsilon>0$ there is $\delta>0$ such that for all $u,u'\in [0,1]$ and ${\bf e}(\Omega), {\bf e}(\Omega')\in \mathbb{S}^2$, for which $||u{\bf e}(\Omega)-u'{\bf e}(\Omega')||< \delta$, it holds
\begin{align}\label{uniformly cts}
    |P(h_0^{CW}(u\cdot {\bf e}(\Omega))-P(h_0^{CW}(u'\cdot {\bf e}(\Omega'))|<\varepsilon.
\end{align}
We now estimate
\begin{align*}
    &\bigg{|}\frac{1}{4\pi}\int_{[0,1]\times \mathbb{S}^2}K_N(du,d\Omega)p_N(u,\Omega)-P(0)\bigg{|}\leq\\&\frac{1}{4\pi}\int_{[0,\delta)\times \mathbb{S}^2}K_N(du,d\Omega)|p_N(u,\Omega)-P(0)|+\frac{1}{4\pi}\int_{[\delta, 1]\times \mathbb{S}^2}K_N(du,d\Omega)|p_N(u,\Omega)-P(0)|\\& \leq\underbrace{\sup_{\substack{J\in\mathbb{J}_N \\ 0\leq \frac{2J}{N}< \delta}}\sup_{\Omega\in\mathbb{S}^2}\bigg{|}p_N\bigg{(}\frac{2J}{N},\Omega\bigg{)}-P(0)\bigg{|}}_{\text{(I)}} + \underbrace{\sup_{\substack{J\in\mathbb{J}_N \\ \delta \leq \frac{2J}{N}\leq 1}}\sup_{\Omega\in\mathbb{S}^2}\bigg{|}p_N\bigg{(}\frac{2J}{N},\Omega\bigg{)}-P(0)\bigg{|}\nu_N([\delta,1])}_{\text{(II)}},
\end{align*}
To estimate term (II), we notice that following bound holds, i.e.
$$ \sup_{\substack{J\in\mathbb{J}_N \\ \delta\leq \frac{2J}{N}\leq 1}}\sup_{\Omega\in\mathbb{S}^2}\bigg{|}p_N\bigg{(}\frac{2J}{N},\Omega\bigg{)}-P(0)\bigg{|}\leq 2\|P\|_\infty,$$
since  for all $J$, $N$ and $\Omega$, one has
$$p_N\bigg{(}\frac{2J}{N},\Omega\bigg{)}\leq \|P\|_\infty.$$
On account of \eqref{measures}, it holds $\nu_N([\delta,1])< 2(N+1)e^{-(N\delta+1)^2/6(N+1)}$. It follows that there exists $N$ sufficiently large, such that $(II)$ can be made smaller than $\varepsilon$.

For (I), we first exploit the triangle inequality: for all $N\in\mathbb{N}_+$, $J\in\mathbb{J}_N$ and $\Omega\in\mathbb{S}^2$, it holds
\begin{align}\label{triangle}
    \bigg{|}p_N\bigg{(}\frac{2J}{N},\Omega\bigg{)}-P(0)\bigg{|}\leq& \bigg{|}p_N\bigg{(}\frac{2J}{N},\Omega\bigg{)}-P\bigg{(}h_0^{CW}\bigg{(}\frac{2J}{N}{\bf e}(\Omega)\bigg{)} \bigg{)} \bigg{|}\nonumber\\+&\bigg{|}P\bigg{(}h_0^{CW}\bigg{(}\frac{2J}{N}{\bf e}(\Omega)\bigg{)} \bigg{)}-P(0)\bigg{|}.
\end{align}
If we now apply the supremum over $J\in\mathbb{J}_N$, for which $0\leq\frac{2J}{N}< \delta$ and the supremum over $\Omega\in\mathbb{S}^2$, then, for $N$ large enough, the first summand in \eqref{triangle} is bounded by $\varepsilon$ on account of \eqref{onemli1}. The second summand is bounded by $\varepsilon$ due to uniform continuity  of $P\circ h_0^{CW}$, cf. \eqref{uniformly cts}, which is applied at  the zero vector, i.e., at the point $0\cdot {\bf e}(\Omega))={\bf 0}$, for which
it holds $h_0^{CW}(0\cdot {\bf e}(\Omega))=0$.

This shows the validity of \eqref{limit1}, thereby concluding the proof of the theorem.

\end{proof}

\begin{remark}
 We note that the previous result readily generalizes to any (normalized) mean-field quantum spin model expressed as a polynomial in the total spin operator. This follows from \cite[Proposition II.2]{Raggio_Werner_1989}, which states that the continuous functional calculus of a (normalized) mean-field model, representing a so-called quasi-symmetric sequence, remains quasi-symmetric. A detailed exploration of this generalization is left for future work.
 \hfill$\blacksquare$
\end{remark}

\section{Restricted Curie-Weiss model}\label{sec: our-pb}
Through direct inspection, the Curie-Weiss model $H_{N}^{\text{CW}}$ preserves the symmetric subspace $\text{Sym}^N(\mathbb{C}^2)$, which has dimension $N+1$. Consequently, the model can be restricted this subspace. With a slight abuse of notation, we denote this restricted $(N+1)\times (N+1)$ matrix by $H^{s}_N$, and, as before,  normalize it by the factor $1/N$, yielding $\Bar{H}^{s}_N$. This restricted matrix represents a single quantum spin system with spin quantum number $J=N/2$, a setting commonly analyzed in the classical limit  $J\to\infty$ \cite{Moretti_vandeVen_2020,Ven_2020}.  Specifically, the matrix aligns with the choice $J=N/2$ in \eqref{restricted CW}.
As demonstrated in the proof of Theorem~\ref{main result}, the resulting family of matrices $\{\Bar{H}^{s}_N\}_N$  a Berezin-Toeplitz operator with symbol $h_0^{CW}\in C(\mathbb{S}^2)$, corresponding to the case $u=1$ in \eqref{classicalCWmodel}.
\begin{remark}\label{change variables}
    By making the following change of variables
$$\cos{\vartheta}\mapsto 2x-1;$$
$$\varphi\mapsto \theta:=\varphi-\pi,$$
it follows that the symbol (for $u=1$) reads
\begin{equation}\label{symbol restricted}
h_0^{CW}(x,\varphi)=-\frac{\Gamma}{2}(2x-1)^2 {-}2B\sqrt{(1-x)x}\cos{\theta}, \ \ \ x\in [0,1], \ \ \ \theta\in [-\pi,\pi].
\end{equation}
\hfill$\blacksquare$
\end{remark}
In the sequel, we step-by-step prove that the Berezin-Topelitz operator $\Bar{H}^{s}_N$ is indeed a GLT sequence. For achieving this, we recall that there exists a basis in which $\Bar{H}^{s}_N$ is represented by the following matrix
\begin{align}\label{seq restricted}
    \Bar{H}^{s}_N =& \underset{1 \leq k \leq N+1}{\text{diag}} \left( -\frac{\Gamma}{2} \left( \frac{2k}{N} - 1 \right)^2 \right)
+ \text{tridiag} \bigg( -B \sqrt{1-\frac{(k-1)}{N}} \sqrt{\frac{k}{N}} \quad 0 \quad \nonumber\\-&B \sqrt{1-\frac{k}{N}} \sqrt{\frac{k+1}{N}} \bigg),
\end{align}
as proven in \cite{Ven_Groenenboom_Reuvers_Landsman}.
\begin{theorem}\label{main result-bis}
With reference to the setting in (\ref{symbol restricted}) and (\ref{seq restricted}), we have
\[
\{\Bar{H}^{s}_N \}_N \sim_{\rm{GLT},\sigma,\lambda} h_0^{CW}(x,\varphi(\theta)).
\]
\end{theorem}
\begin{proof}
By considering $d=r=1$ it is immediate to see that the matrix $\Bar{H}^{s}_N$ can be written as the sum of diagonal sampling matrices as in Section~\ref{blckdiag} and two products of sampling matrices and very basic unilevel Toeplitz matrices as in Section~\ref{TM}. In fact we have
\begin{align*}
 \Bar{H}^{s}_N =& -\frac{\Gamma}{2}D_{N+1}\Big(\Big(2x-1\Big)^2\Big)-B D_{N+1}(\sqrt{1-x}\sqrt{x})T_{N+1}(e^{\iota \theta})\\-&B T_{N+1}(e^{-\iota \theta}) D_{N+1}(\sqrt{1-x}\sqrt{x}).
\end{align*}
Now by Axiom \textbf{GLT 2}, part 1, we have $\{T_{N+1}(e^{\pm \iota \theta})\}_N\sim_{\text{GLT}} e^{\pm \iota \theta}$ and by Axiom \textbf{GLT 2}, part 2, we deduce
\[
\left\{D_{N+1}\left(\left(2x-1\right)^2\right)\right\}_N\sim_{\text{GLT}}\left( 2x-1 \right)^2,\
\{D_{N+1}(\sqrt{1-x}\sqrt{x})\}_N\sim_{\text{GLT}}  \sqrt{(1-x)x},
\]
simply because both functions $\left( 2x-1 \right)^2, \sqrt{1-x}\sqrt{x}$ are continuous on $[0,1]$ and a fortiori Riemann integrable.

Hence, using the $*$-algebra structure of the GLT sequences and more precisely Axiom \textbf{GLT 3}, part 2, Axiom \textbf{GLT 3}, part 3, we infer $\, \{\Bar{H}^{s}_N\}_N \sim_{\text{GLT}} -\frac{\Gamma}{2}  \left( 2x-1 \right)^2 -2B\cos(\theta)\sqrt{(1-x)x}$, which is compatible with the symbol
$h_0^{CW}(x,\varphi(\theta))$ indicated in Remark \ref{change variables}.
Finally \textbf{GLT 1} implies
\begin{itemize}
    \item $\, \{\Bar{H}^{s}_N\}_N \sim_{\sigma} -\frac{\Gamma}{2}  \left( 2x-1 \right)^2 -2B\cos(\theta)\sqrt{(1-x)x}$.
    \begin{itemize}
        \item Furthermore, since $\Bar{H}^{s}_N$ is real and symmetric for any $N$, again by \textbf{GLT 1}, we deduce
    \end{itemize}
    \item $\, \{\Bar{H}^{s}_N\}_N \sim_{\lambda} -\frac{\Gamma}{2}  \left( 2x-1 \right)^2 -2B\cos(\theta)\sqrt{(1-x)x}$.
\end{itemize}

Notice that $\, \{\Bar{H}^{s}_N+U_N\}_N \sim_{\text{GLT}} -\frac{\Gamma}{2}  \left( 2x-1 \right)^2 -2B\cos(\theta)\sqrt{(1-x)x}$ for any
$\{U_{n}\}_N$ such that  $\{U_{n}\}_N \sim_{\text{GLT}} 0$  by Axiom \textbf{GLT 3}, part 2, and  Axiom \textbf{GLT 2}, part 3, so that
$\, \{\Bar{H}^{s}_N+U_N\}_N \sim_{\sigma} -\frac{\Gamma}{2}  \left( 2x-1 \right)^2 -2B\cos(\theta)\sqrt{(1-x)x}$. Finally, as it happens for compact non-Hermitian perturbations of Jacobi matrix-sequences \cite{GoSe}, $\, \{\Bar{H}^{s}_N+U_N\}_N \sim_{\lambda} -\frac{\Gamma}{2}  \left( 2x-1 \right)^2 -2B\cos(\theta)\sqrt{(1-x)x}$ either if all $U_N$ are Hermitian or if the assumption in Axiom \textbf{GLT 5} is satisfied by the non-Hermitian matrix-sequence perturbation $\{U_N\}_N$.
\end{proof}

\begin{remark}\label{Schr\"{o}dinger operator}
It is interesting to comment the relations between a standard second order centered finite difference (FD) discretization of the Schr\"{o}dinger operator and Theorem \ref{main result-bis}.
Let us define
$a(x):=2B\sqrt{(1-x)x}$ and $c(x):=-\frac{\Gamma}{2}(2x-1)^2-2B\sqrt{(1-x)x}$, i.e. we have set $\cos{\theta}=1$ in the definition of the symbol. For $N\in\mathbb{N}_+$, consider the Schr\"{o}dinger operator on $L^2(0,1)$:
$$H(N)u(x)=-\frac{1}{(N+1)^2}a(x)u''(x)+c(x)u(x),$$
with $u(0)=u(1)$. If we discretize the domain in uniform steps of width $h=1/(N+1)$, it follows that the resulting FD discretization matrix is a GLT sequence with symbol $a(x)(2-2\cos{\theta})+c(x)$. The factor $1/N$ simultaneously plays the role of the discretization step-size as well as the semi-classical parameter. It is clear that the ensuing GLT matrix-sequence is equivalent to the restricted Curie-Weiss model. Hence, the restricted Curie-Weiss model defines a Schr\"{o}dinger operator with potential $c$. In the parameter regime  $\Gamma=1$ and $B=1$, $c$  has the shape of a single well, cf. Section~\ref{harmregime}, whilst for the choices  $\Gamma=1$ and $B\in (0,1)$, $c$  has the shape of a double wel, cf. Section~\ref{schrregime}.
However, the fact that the parameter factor $1/N$ plays  simultaneously the role of the discretization step-size and of the semi-classical parameter is non-standard from the numerical analysis viewpoint and the whole potential of a related theoretical study has still to be explored further.
\end{remark}

\section{Numerical results}\label{sec: num}

In the present section, we give various visualizations of the spectral features of the matrices $\Bar{H}^{s}_N$, confirming the derivations in Theorem \ref{main result-bis}. We fix the value of $\Gamma$ and $B$ and, for these fixed values, we consider the matrix-size parameter $N= 40, 80, 160, 320$. We recall that Theorem \ref{main result-bis} is an asymptotic one, as all the GLT results, but the really impressive fact is that the spectrum of  $\Bar{H}^{s}_N$ adheres to the spectral symbol already for $N=40$.

\subsection{\texorpdfstring
  {Asymptotic spectral behavior of $\Bar{H}^{\,s}_{N}$ with $\Gamma=B=1$}
  {Asymptotic spectral behavior of H \string^s\string_N with Gamma=B=1}}
\label{harmregime}

As already mentioned, we consider the matrix
\begin{align*}
 \Bar{H}^{s}_N =& -\frac{\Gamma}{2}D_{N+1}\Big(\Big(2x-1\Big)^2\Big) {-}B D_{N+1}(\sqrt{1-x}\sqrt{x})T_{N+1}(e^{\iota \theta})\\{-}&B T_{N+1}(e^{-\iota \theta}) D_{N+1}(\sqrt{1-x}\sqrt{x}),
\end{align*}
whose associated matrix-sequence has proven to have a GLT nature in Section \ref{sec: our-pb} in Theorem \ref{main result-bis}.

We show the comparison between symbol $h_0^{CW}$ (for $u=1$), cf. \eqref{classicalCWmodel},  and the eigenvalues of  $\Bar{H}_N^s$. This comparison is carried out using the concept of {\em monotone rearrangement}, which enables the interpretation of the symbol as a single-variable function
\cite[Definition 2.1]{BarBianGar}.

We notice that the agreement is very good, even for moderate matrix-sizes, which is nontrivial given the asymptotic nature of the GLT distributional results.
A further remarkable fact is that the range of the spectral symbol $h_0^{CW}$ contains all the spectra, i.e., no outliers are observed. Again, this is a highly nontrivial matter and cannot be deduced directly from distributional results. In fact, the interplay between spectral localization and distributional properties is characteristic of linear positive operators (LPOs), such as Toeplitz operators and various classes of variable-coefficient coercive differential operators (see \cite[Corollary 6.2]{garoni2017} for the Toeplitz case and \cite{LPO-rev} for a broader discussion).

\begin{figure}[H]
    \centering
    \begin{minipage}[b]{0.47\textwidth}
        \centering
        \includegraphics[width=\linewidth, height=6cm]{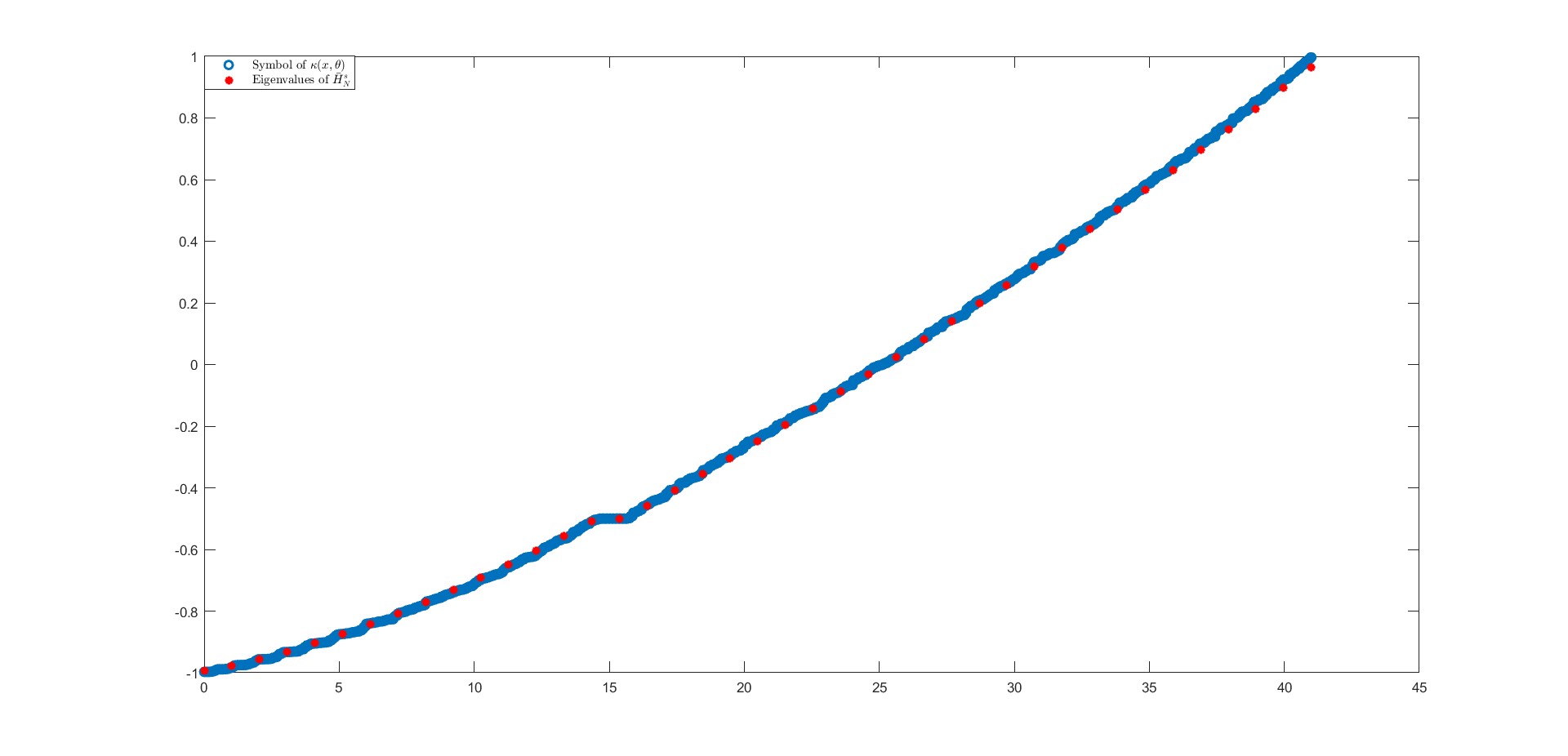}
        \caption*{$N=40$}
    \end{minipage}
    \hfill
    \begin{minipage}[b]{0.47\textwidth}
        \centering
        \includegraphics[width=\linewidth, height=6cm]{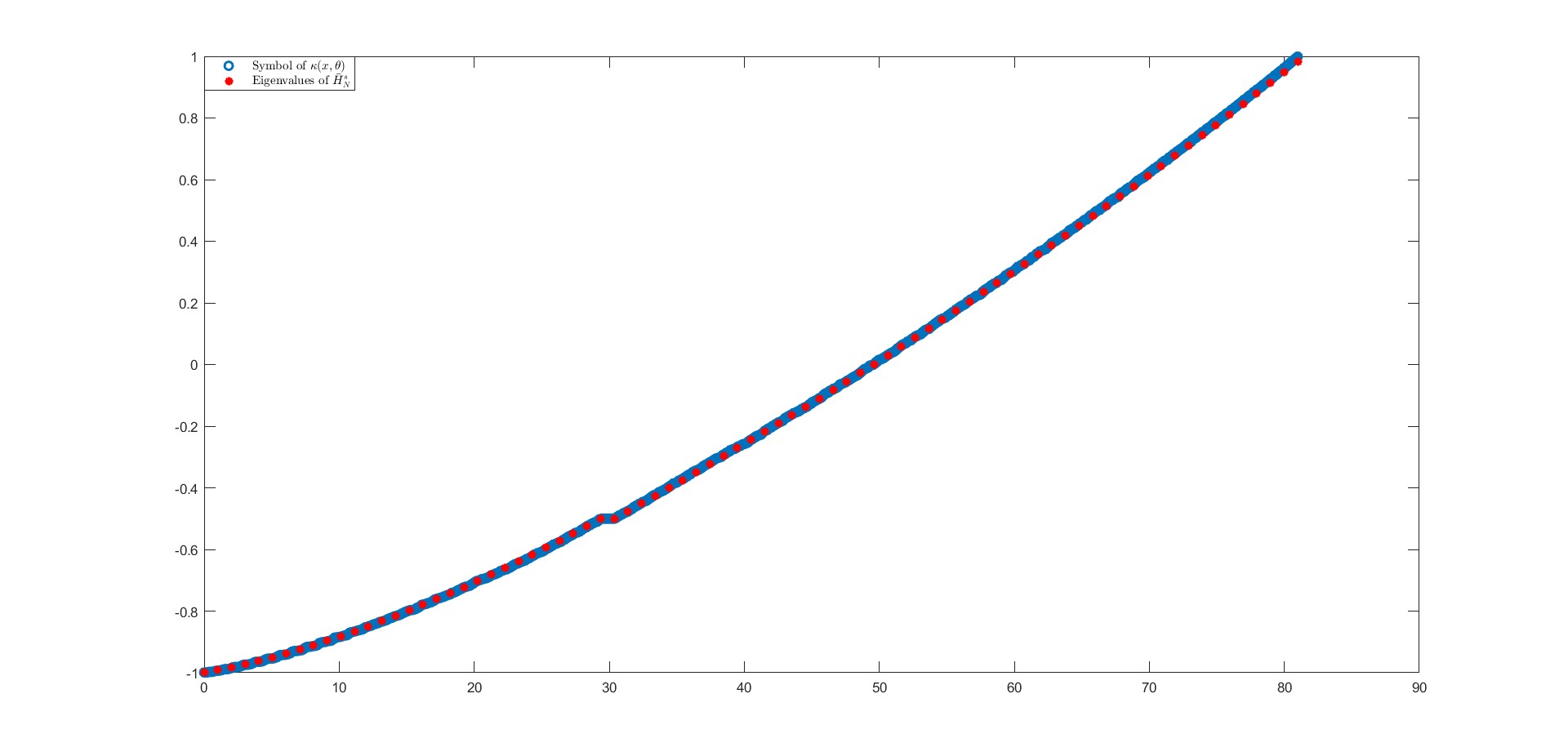}
        \caption*{$N=80$}
    \end{minipage}
    \caption{Comparison between the symbol $h_0^{CW}$ (blue circle; \(N=40,80\)), and the eigenvalues of $\Bar{H}_N^s$ (red stars).}

    \label{fig:N40-N80}
\end{figure}

\vspace{0.7em} 

\begin{figure}[H]
    \centering
    \begin{minipage}[b]{0.47\textwidth}
        \centering
        \includegraphics[width=\linewidth, height=6cm]{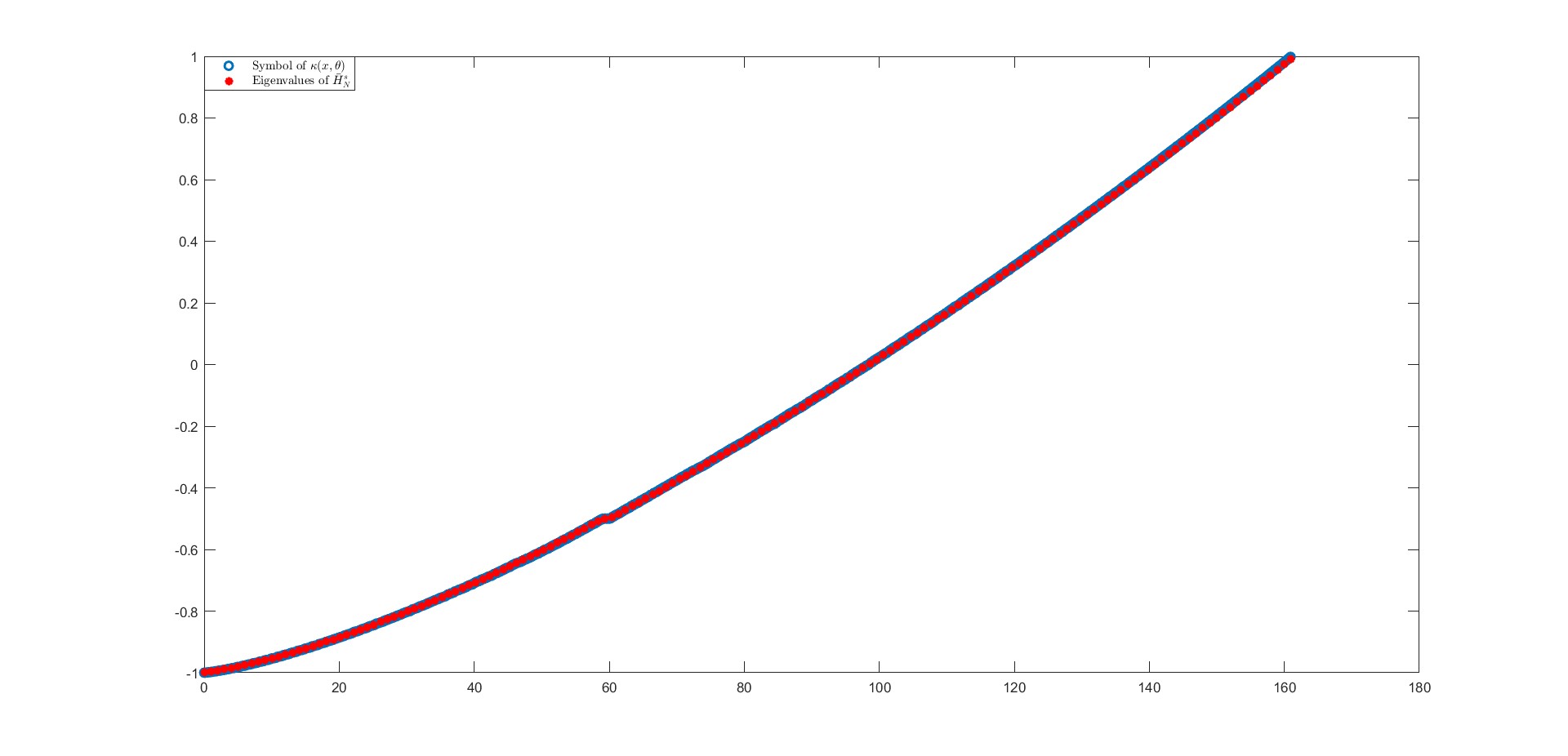}
        \caption*{$N=160$}
    \end{minipage}
    \hfill
    \begin{minipage}[b]{0.47\textwidth}
        \centering
        \includegraphics[width=\linewidth, height=6cm]{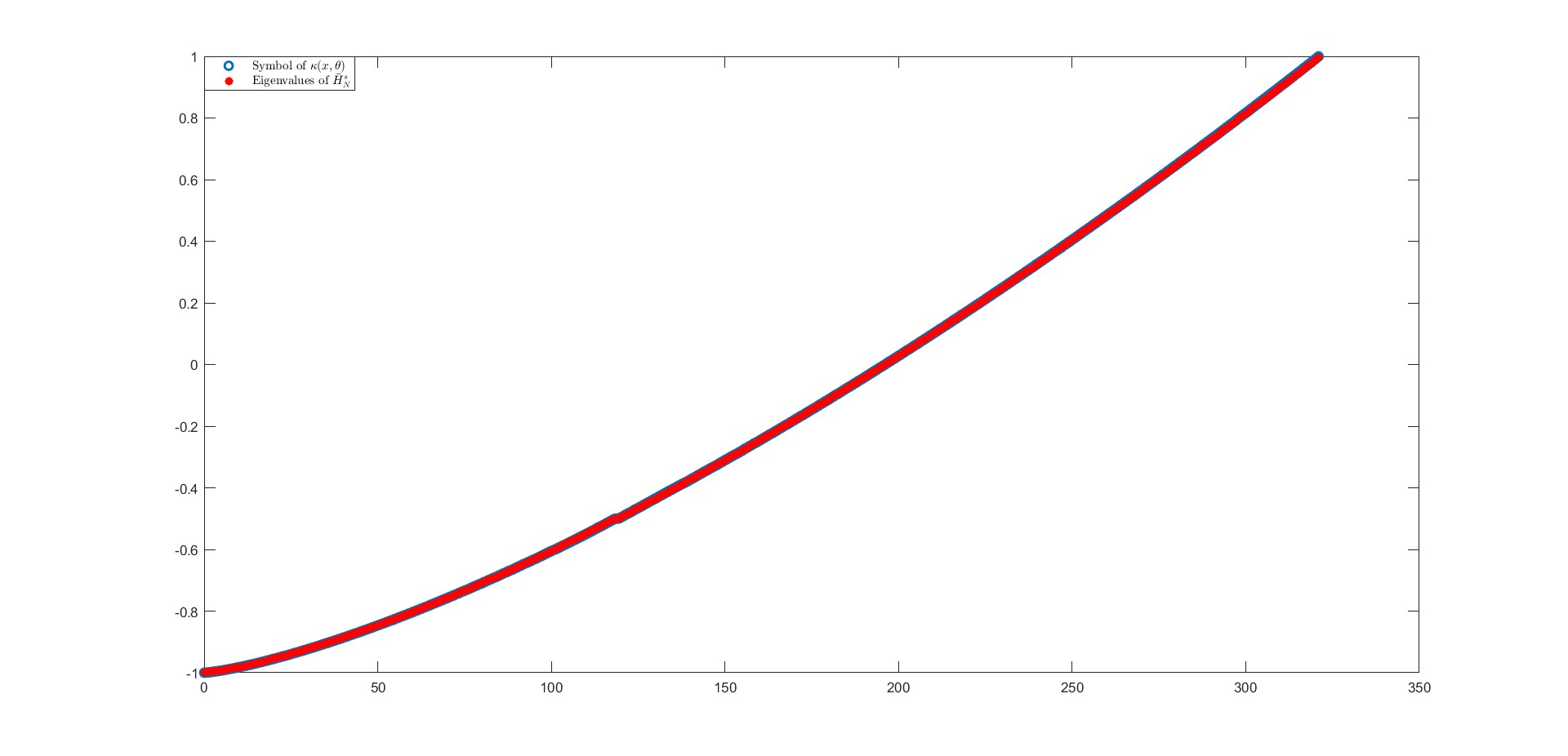}
        \caption*{$N=320$}
    \end{minipage}
    \caption{Comparison between the symbol $h_0^{CW}$ (blue circle; \(N=160,320\)), and the eigenvalues of $\Bar{H}_N^s$ (red stars).}
    \label{fig:N160-N320}
\end{figure}

\subsection{\texorpdfstring
  {Asymptotic spectral behavior of $\Bar{H}^{\,s}_{N}$ with $\Gamma=2B=1$}
  {Asymptotic spectral behavior of H \string^s\string_N with Gamma=2B=1}}
\label{schrregime}

{For $\Gamma=1$ and $B=\frac{1}{2}$ we show again the comparison between symbol and the eigenvalues.} The comments already provided in the case $\Gamma=B=1$ can be repeated verbatim also in this setting.

\begin{figure}[H]
    \centering
    \begin{minipage}[b]{0.47\textwidth}
        \centering
        \includegraphics[width=\linewidth, height=6cm]{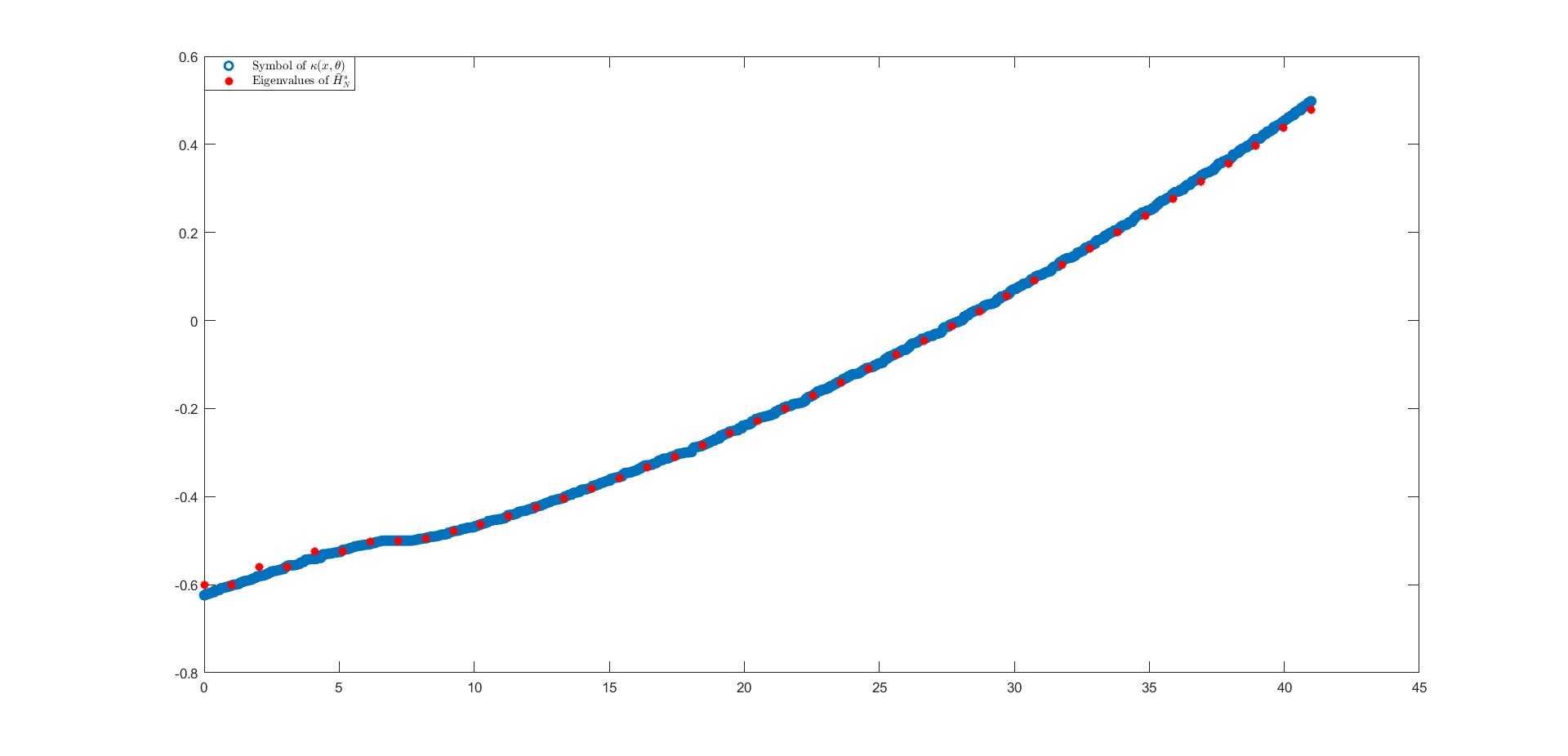}
        \caption*{$N=40$}
    \end{minipage}
    \hfill
    \begin{minipage}[b]{0.47\textwidth}
        \centering
        \includegraphics[width=\linewidth, height=6cm]{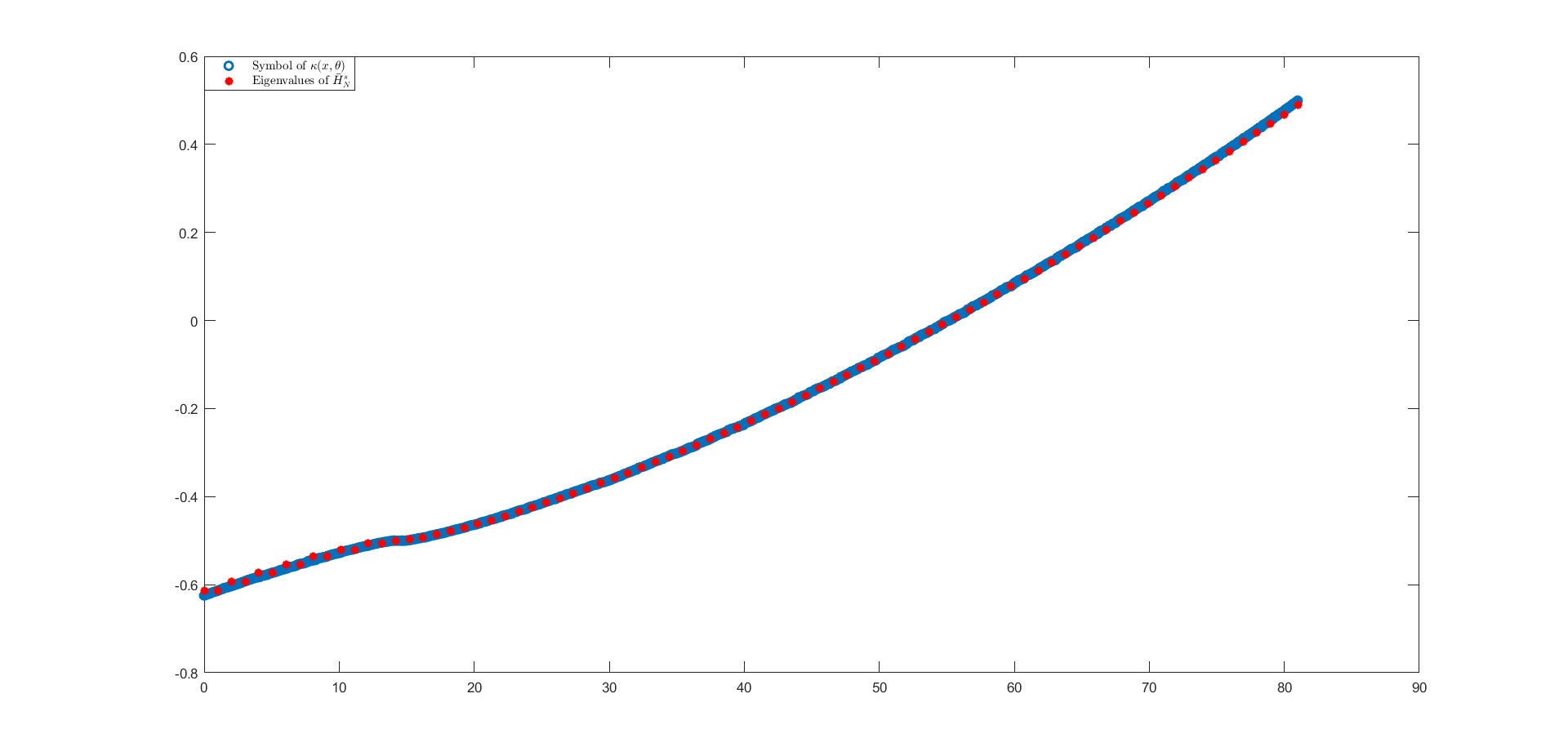}
        \caption*{$N=80$}
    \end{minipage}
    \caption{Comparison between the symbol $h_0^{CW}$ (blue circle; \(N=40,80\)), and the eigenvalues of $\Bar{H}_N^s$ (red stars).}
    \label{fig:N40-N80 U}
\end{figure}

\vspace{0.7em} 

\begin{figure}[H]
    \centering
    \begin{minipage}[b]{0.47\textwidth}
        \centering
        \includegraphics[width=\linewidth, height=6cm]{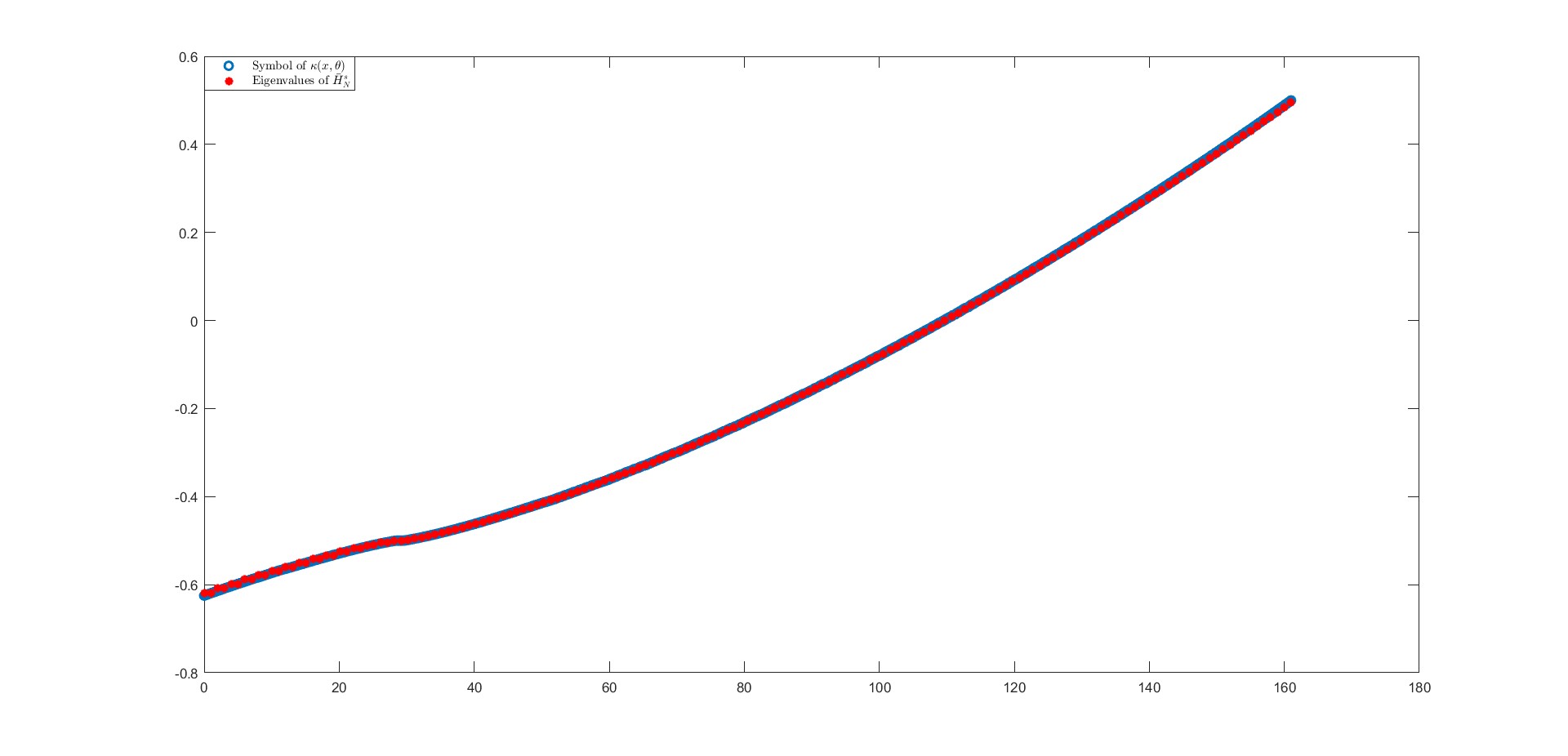}
        \caption*{$N=160$}
    \end{minipage}
    \hfill
    \begin{minipage}[b]{0.47\textwidth}
        \centering
        \includegraphics[width=\linewidth, height=6cm]{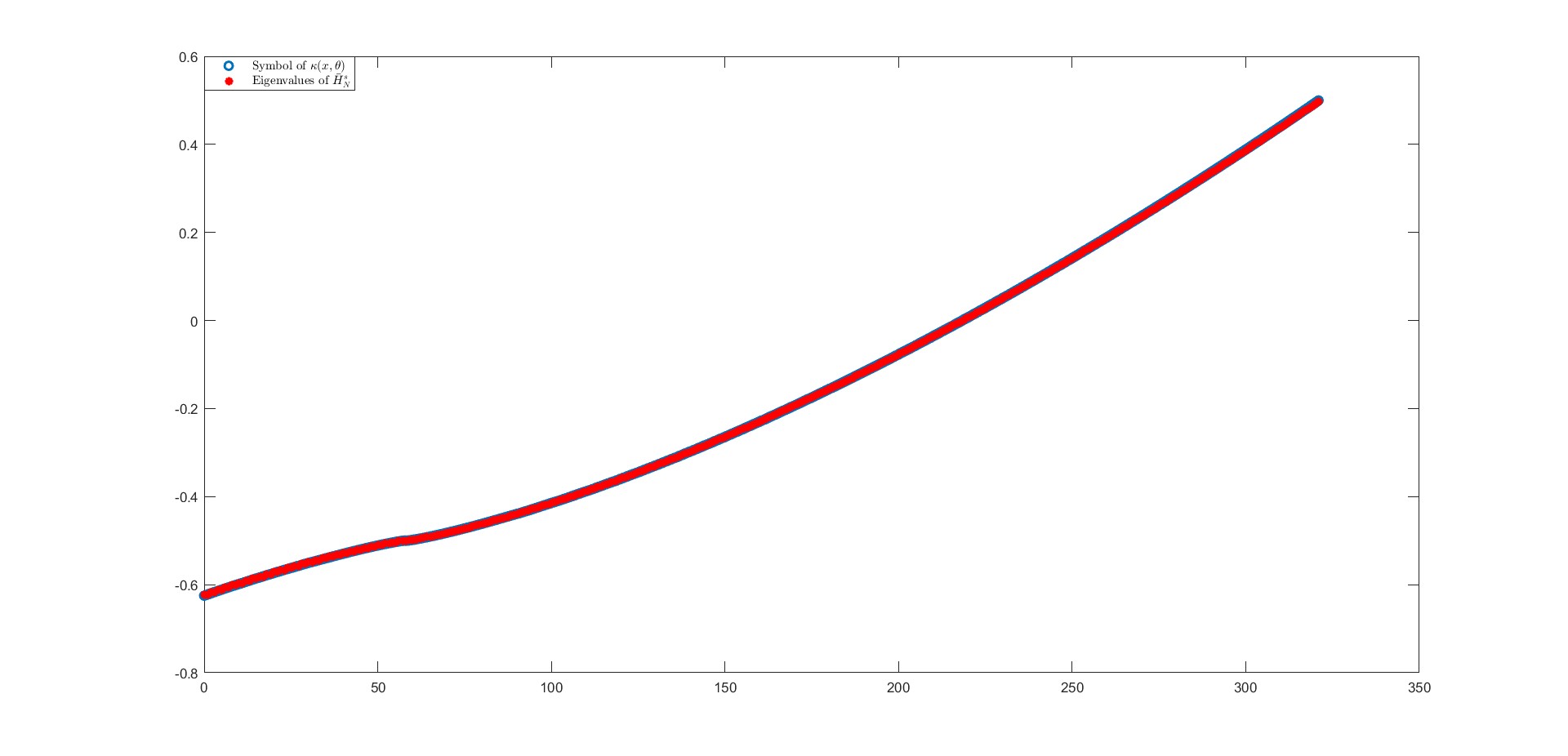}
        \caption*{$N=320$}
    \end{minipage}
    \caption{Comparison between the symbol $h_0^{CW}$ (blue circle; \(N=160,320\)), and the eigenvalues of $\Bar{H}_N^s$ (red stars).}
    \label{fig:N160-N320 U}
\end{figure}

The figures provide clear evidence that the lowest eigenvalues become nearly doubly degenerate,  a characteristic feature of  a Schr\"{o}dinger operator with a double-well potential. For a more detailed discussion on the approximation of this Schr\"{o}dinger operator by the Curie-Weiss model, we refer the reader to \cite{Ven_Groenenboom_Reuvers_Landsman}.

\subsection{\texorpdfstring
  {Extremal spectral behavior of $\Bar{H}^{\,s}_{N}$}
  {Extremal spectral behavior of H\string^s\string_N}}

The figures reported in the previous section inform of localization results, beyond the proven distributional results. Here we show that in the two considered setting of parameters ($\Gamma=B=1$ and  $\Gamma=2B=1$), not only the spectrum of $\Bar{H}^{s}_N$ is contained in the interior of the range of the GLT symbol, but the behavior of the extreme eigenvalues is very regular. In fact, both the minimal and the maximal eigenvalues converge monotonically to the minimum and to maximum of the symbol, respectively: again this phenomenon is typical of matrix-valued LPOs \cite{LPO-rev}.

For the minimum/maximum eigenvalue analysis, for $\Gamma=B=1$ and we proceed as follows.

\begin{itemize}
    \item The GLT symbol is $\kappa(x,\theta)=-\frac{\Gamma}{2}(2x-1)^{2}- 2B\cos(\theta)\sqrt{(1-x)x}$.
    \item We consider $m=\underset{(x,\theta) \in [0,1]\times [0,\pi]}{\text{min}} \kappa(x,\theta)$,  $M=\underset{(x,\theta) \in [0,1]\times [0,\pi]}{\text{max}} \kappa(x,\theta)$.
\end{itemize}
\begin{itemize}
    \item Take $N_{j} +1=40*2^{j}, \quad\quad j=0,1,2,3, $
    \item Compute $\tau_{j}= \lambda_{\min}(\bar{H}_{N_{j}}^{s})-m, \quad \quad j=0,1,2,3, $
    \item Compute $\alpha_{j}=\log\left(\frac{\tau_{j}}{\tau_{j+1}}\right), \quad \quad j=0,1,2. $
\end{itemize}
\vspace{9pt}
\begin{table}[H]
\centering
\caption{Numerical behavior of the minimal eigenvalue for \( \Gamma = 1 \),\\ \( B = 1 \), and \( m = -1 \).}
\begin{tabular}{|c|c|c|c|}
\hline
   & $\lambda_{\min}(\bar{H}_{N_{j}}^{s})$ & $\tau_{j}$ & $\alpha_{j}$ \\
  \hline
 $N_{j} +1=40$ & -0.9936 & 0.0064 & 0.4082 \\
  \hline
  $N_{j} +1=80$& -0.9975 & 0.0025 & 0.3979 \\
  \hline
  $N_{j} +1=160$ & -0.9990 & 0.001 & 0.3979 \\
  \hline
  $N_{j} +1=320$ & -0.9996 & 0.0004 &  \\
  \hline
\end{tabular}
\end{table}

\begin{itemize}
    \item Take $N_{j} +1=40*2^{j}, \quad\quad j=0,1,2,3, $
    \item Compute $\hat{\tau}_{j}= M-\lambda_{\max}(\bar{H}_{N_{j}}^{s}), \quad \quad j=0,1,2,3, $
    \item Compute $\beta{j}=\log\left(\frac{\hat{\tau}_{j}}{\hat{\tau}_{j+1}}\right), \quad \quad j=0,1,2. $
\end{itemize}
\vspace{9pt}
\begin{table}[H]
\centering
\caption{Numerical behavior of the maximal eigenvalue for \( \Gamma = 1 \),\\ \( B = 1 \), and \( M = -1 \).}
\begin{tabular}{|c|c|c|c|}
\hline
   & $\lambda_{\max}(\bar{H}_{N_{j}}^{s})$ & $\hat{\tau}_{j}$ & $\beta_{j}$ \\
  \hline
 $N_{j} +1=40$ & 0.9654 & 0.0346 & 0.2960 \\
  \hline
  $N_{j} +1=80$& 0.9825 & 0.0175 & 0.2985 \\
  \hline
  $N_{j} +1=160$ & 0.9912 & 0.0088 & 0.3010 \\
  \hline
  $N_{j} +1=320$ & 0.9956 & 0.0044 &   \\
  \hline
\end{tabular}
\end{table}

The interpretation of the above tables is very informative. In fact, if we consider $\psi$ the nondecreasing rearrangement of the GLT symbol $\kappa(x,\theta)$ defined on the standard interval $[0,1]$, then the numerical tests suggest that
\[
\lambda_{\min}\!\bigl(\Bar{H}_{N_j}^{\,s}\bigr) - m \sim \frac{1}{N^{0.4}}, \qquad
M - \lambda_{\max}\!\bigl(\Bar{H}_{N_j}^{\,s}\bigr) \sim \frac{1}{N^{0.3}}.
\]

It remains to investigate analytically if the corresponding analytic behavior holds that is
\[
\psi(t)-m \sim {t^{0.4}}, \ \ \ \ M-\psi(t)  \sim {(1-t)^{0.3}}.
\]

We now continue, along the previous reasoning, {with the computation of the minimum/maximum eigenvalues analysis, we proceed as follows in the case where $\Gamma=2B=1$.}
\begin{itemize}
    \item {The GLT symbol is $\kappa(x,\theta)=-\frac{\Gamma}{2}(2x-1)^{2}-2B\cos(\theta)\sqrt{(1-x)x}$}
    \item {We consider $m=\underset{(x,\theta) \in [0,1]\times [0,\pi]}{\text{min}} \kappa(x,\theta)$  $;$  $M=\underset{(x,\theta) \in [0,1]\times [0,\pi]}{\text{max}} \kappa(x,\theta)$}
\end{itemize}

\begin{itemize}
    \item {Take $N_{j} +1=40*2^{j}, \quad\quad j=0,1,2,3, $}
    \item {Compute $\tau_{j}= \lambda_{\min}(\Bar{H}_{N_{j}}^{s})-m, \quad \quad j=0,1,2,3, $}
    \item {Compute $\alpha_{j}=\log\left(\frac{\tau_{j}}{\tau_{j+1}}\right), \quad \quad j=0,1,2. $}
\end{itemize}

\vspace{9pt}
\begin{table}[H]
\centering
\caption{Numerical behavior of the minimal eigenvalue for \( \Gamma = 1 \),\\ \( B = \frac{1}{2} \), and \( m = -0.6241 \).}
\begin{tabular}{|c|c|c|c|}
\hline
   & {$\lambda_{\min}(\Bar{H}_{N_{j}}^{s})$} & {$\tau_{j}$} & {$\alpha_{j}$} \\
  \hline
 {$N_{j} +1=40$} &{-0.6007} & {0.0234} & {0.3521} \\
  \hline
  {$N_{j} +1=80$} & {-0.6137} & {0.0104} & {0.3542} \\
  \hline
  {$N_{j} +1=160$} & {-0.6195} & {0.0046} & {0.4074} \\
  \hline
  {$N_{j} +1=320$} & {-0.6223} & {0.0018} &  \\
  \hline
\end{tabular}
\end{table}

\begin{itemize}
    \item {Take $N_{j} +1=40*2^{j}, \quad\quad j=0,1,2,3, $}
    \item {Compute $\hat{\tau}_{j}= M-\lambda_{\max}(\Bar{H}_{N_{j}}^{s}), \quad \quad j=0,1,2,3, $}
    \item {Compute $\beta{j}=\log\left(\frac{\hat{\tau}_{j}}{\hat{\tau}_{j+1}}\right), \quad \quad j=0,1,2. $}
\end{itemize}

\vspace{9pt}
\begin{table}[H]
\centering
\caption{{Numerical behavior of the maximal eigenvalue for \( \Gamma = 1 \),\\ \( B = \frac{1}{2} \), and \( M = 0.4982 \).}}
\begin{tabular}{|c|c|c|c|}
\hline
   & {$\lambda_{\max}(\Bar{H}_{N_{j}}^{s})$} & {$\hat{\tau}_{j}$} & {$\beta_{j}$} \\
  \hline
 {$N_{j} +1=40$} & {0.4789} & {0.0193} & {0.3361} \\
  \hline
  {$N_{j} +1=80$} & {0.4893} & {0.0089} & {0.3930} \\
  \hline
  {$N_{j} +1=160$} & {0.4946} & {0.0036} & {0.4010} \\
  \hline
  {$N_{j} +1=320$} & {0.4973} & {0.0009} & {} \\
  \hline
\end{tabular}
\end{table}

The interpretation of the numerical results is of interest. Again, as observed for the previous case, both the minimal and the minimal eigenvalues converge monotonically to the minimum and to maximum of the symbol, respectively, in line with results which are usually observed for matrix-valued LPOs.

\chapter*{Conclusion}
In the present thesis, we analyzed several classes of structured matrix-sequences that arise from various contexts, including also the discretization of differential operators. In Chapter~\ref{chap:Structured matrices}, we introduced the fundamental concepts of structured matrices, including Toeplitz, circulant, $\omega$-circulant, $\tau$-algebra, and Hankel matrices, focusing on their spectral properties and efficient computational techniques.

In Chapter~\ref{GLT}, we presented the framework of GLT sequences, which extends the classical theory of structured matrices and matrix-sequences with hidden structure. We studied the fundamental properties of GLT sequences, including their symbol representation, spectral distribution, and algebraic closure under addition, multiplication and inversion. In particular, we discussed the role of zero-distributed sequences and their connection to negligible spectral components, as well as the characterization of diagonalizable approximations within the GLT setting. These results form the theoretical foundation for the subsequent chapters, where GLT tools are exploited to analyze the spectral distribution of matrix functions, the geometric means of HPD GLT sequences, and the hidden GLT structures arising in quantum spin models.

Chapter~\ref{GM GL1} has studied the spectral distribution of the geometric mean of two or more matrix-sequences constituted by HPD matrices, under the assumption that $k$ input matrix-sequences belong to the same $d$-level, $r$-block GLT $*$-algebra, where $k\ge 2$, $d,r\ge 1$. For $k=2$ an explicit formula exists and this has allowed to prove that the new matrix-sequence is of GLT type for any fixed $d,r$, with GLT symbol being the corresponding geometric mean of the $k=2$ input symbols, according to Theorems \ref{4: th:two - r=1,d general GM1} and \ref{4: th:two - r,d general GM1}. As a consequence the spectral distribution of the geometric mean of $k=2$ matrix-sequences is formally demonstrated. For $k>2$, an explicit formula is not available and we have considered the Karcher mean, for which efficient iterative procedures available in the literature can be used for computational purposes. Using the same tools as in the proof of Theorems \ref{4: th:two - r=1,d general GM1} and \ref{4: th:two - r,d general GM1}, we think that it is possible to prove that all the matrix-sequences made by the Karcher mean iterates are still GLT matrix-sequences with symbols converging the geometric mean of the $k$ symbols, provided that the matrix-sequence of the initial input matrices is a GLT matrix-sequence made by HPD matrices. The conjecture has been confirmed through various numerical experiments, where we compared the eigenvalues of the geometric mean with a uniform sampling of the geometric mean of the symbols. A very good agreement has been observed in all the numerical tests and aven for very moderate matrix-sizes. In fact, as the matrix-size increases, the symbol provides a better and better approximation of the eigenvalues, both in one-dimensional and two-dimensional cases, and both with scalar and block matrices, i.e. using GLT matrix-sequences taken from the differential world with $k=2,3$, $d,r=1,2$ and including finite difference, finite element, isogeometric analysis approximations of constant and variable differential operators.

Chapter~\ref{GM GLT2} have studied the spectral distribution of the geometric mean matrix-sequence of two $d$-level $r$-block GLT matrix-sequences 
$\{G(A_n, B_n)\}_n$  formed by HPD matrices, where $\{A_n\}_n, \{B_n\}_n$ have GLT symbols $\kappa, \xi$, respectively.
In Theorem \ref{theorem 1} and in Theorem \ref{theorem 2}, we have shown that the assumption that at least one of the input GLT symbols is invertible almost everywhere is not necessary when the symbols commute so that 
\[
\{G(A_n, B_n)\}_n \sim_{\mathrm{GLT}} G(\kappa,\xi)\equiv (\kappa \xi)^{1/2}.
\]
In this way, we have positively solved Conjecture 10.1 in \cite{garoni2017}.
 
On the other hand, the statement is generally false or even not well posed when the symbols are not invertible almost everywhere and do not commute, as shown in detail in several numerical experiments. This shows that Theorem \ref{theorem 2} is maximal, so answering this time in the negative to the second item in the conclusion of \cite{ahmad2025matrix} and to [\cite{ahmad2025matrix}, Remark 2], when both symbols are degenerate and do not commute. Further numerical experiments are presented and critically commented in connection with extremal spectral features, linear positive operators, and in connection with the notion of Toeplitz and GLT momentary symbols. The numerical evidences open the door to further theoretical studies, which we will consider in the near future. Finally, the study of the spectral distribution when we consider the geometric mean of more than two matrix-sequences is completely open. In this chapter we have proven \cite[Conjecture 10.1]{garoni2017}: however, when dealing with several matrix-sequences, \cite[Conjecture 10.2]{garoni2017} is still open and for that we need normwise error estimates for using the tools in \cite{SerraCapizzanoTilli-LPO}, while till now the convergence results in the literature have a entrywise nature (see \cite{bini2024survey} and references therein). 

Chapter~\ref{chapter:quantum} has studied the theory of GLT $\ast$-algebras for dealing with structured matrix-sequences arising from the modelling of mean-field quantum spin systems. In this chapter, we expressed the related matrix-sequences within the GLT formalism and analyzed two specific cases in detail. For both cases, we determined the spectral distributions in the context of GLT theory, and the theoretical results were confirmed through visualizations and numerical experiments. Despite these findings, several open problems remain. One important direction is the study of extremal eigenvalues and the expansion of the considered matrix-sequences using momentary GLT symbols~\cite{bolten2022toeplitz, new momentary}, which would provide a finer description of the spectrum compared to Theorem~\ref{main result-bis}. In particular, in the non-reduced case, the GLT symbol equal to zero in Theorem~\ref{main result} may hide a richer structure: there may exist a numerical sequence $\alpha_N$ converging to zero such that the scaled sequence, obtained by dividing by $\alpha_N$, is still a GLT sequence with a nonzero GLT symbol. In such a scenario, the emerging GLT structure could be multilevel (due to the geometry on the sphere as in \eqref{classical CW}, block-structured (due to the basic blocks in \eqref{basic matrices})~\cite{barbarino2020block1d, barbarino2020blockmulti, garoni2017, garoni2018}, or even reduced, since the terms in the general model have varying sizes $C(J,N)$; see~\cite[pp.~398--399]{SerraCapizzano2003}, \cite[Section~3.1.4]{SerraCapizzano2006}, and~\cite{Barbarino2022} for a complete treatment of reduced GLT matrix-sequences. Finally, it would be desirable to explore GLT theory in connection with locally interacting quantum spin models, such as the quantum Ising or quantum Heisenberg models, which are widely studied in condensed matter physics. Progress in this direction would significantly improve our understanding of the spectral properties and the fundamental structure of these realistic interacting systems.

Despite these advances, several fundamental open problems and future directions remain:
\begin{itemize}
    \item  A formal proof is still lacking for the GLT nature of the Karcher mean matrix-sequence when applied to $k > 2$ HPD GLT matrix-sequences, even under the assumption that the initial guess is itself a HPD GLT matrix-sequence. From a computational viewpoint, initializing the iteration with a GLT sequence whose symbol is the geometric mean of the input symbols may accelerate convergence, but a comprehensive theoretical justification is needed.

   \item The previous item is related to the convergence estimates in terms of the Schatten $p$-norms, so that the results in Theorem~\ref{th 3.1} could be used to show that the error term matrix-sequences are zero-distributed, i.e., they are GLT with zero GLT symbol. This type of results is still lacking and it will be the subject of future investigations.
    
    \item  A completely open problem concerns the study of the extremal eigenvalues of the geometric means of GLT matrix-sequences as a function of the analytic features of the geometric means of the GLT symbols: in this direction it should be recalled that a rich literature exists regarding the extremal eigenvalues in a Toeplitz setting~\cite{extreme1, extreme2, extreme3, extreme4}, in an $r$-block Toeplitz setting~\cite{SerraCapizzano1999a,SerraCapizzano1999b}, in a differential setting~\cite{extreme diff1,extreme diff2}, so involving all the types of examples considered in the numerical experiments. Preliminary numerical experiments in the unilevel scalar setting with $k = 2, 3$ have been performed in Section~\ref{ssec: extremal}, and the results are quite promising. Interestingly enough, and substantially mimicking the cases already studied in the literature, it seems that the order of the zeros of the GLT symbol decides the asymptotic behavior of the minimal eigenvalues and hence of the conditioning of $G(A_{1}^{n}, \ldots, A_{k}^{n})$, at least for $d = r = 1$, $k = 2, 3$.

    \item The GLT algebraic Conjecture~\ref{conj:GM_glt} for non-commuting symbols remains unresolved. While numerical experiments support the spectral distribution aspect of the conjecture, a rigorous proof in the general non-commuting case is still an open challenge.
\end{itemize}

\addcontentsline{toc}{chapter}{Conclusion}

\backmatter
\sloppy 
\bibliographystyle{plain} 

\end{document}